\documentclass[reqno]{amsart}
\usepackage{amsmath, amsthm, amssymb,amsfonts,color, dsfont, bm}
\usepackage{hyperref,comment}
\usepackage{todonotes}
\usepackage{esint}
\usepackage[latin1]{inputenc}

\usepackage{tikz-cd}

\usepackage[shortlabels]{enumitem}

\usepackage{accents}

\allowdisplaybreaks

\newcommand\bR{\mathbb{R}}

\newcommand\bS{\mathbb{S}}

\newcommand\bE{\mathbf{E}}

\newcommand\bB{\mathbf{B}}

\newcommand\cC{\mathcal{C}}

\newcommand\cD{\mathcal{D}}

\newcommand\cE{\mathcal{E}}

\newcommand\cI{\mathcal{I}}

\newcommand\cS{\mathcal{S}}

\newcommand\cR{\mathcal{R}}

\newcommand\nt{\mathcal{NT}}

\newcommand\cs{\text{Cauchy-Schwarz inequality}}

 \theoremstyle{definition}

\newtheorem{theorem}{Theorem}[section]
\newtheorem{lemma}[theorem]{Lemma}
\newtheorem{corollary}[theorem]{Corollary}

\newtheorem{proposition}[theorem]{Proposition}

\newtheorem{definition}{Definition}[section]

\theoremstyle{remark}
\newtheorem{remark}[theorem]{Remark}

\newtheorem{assumption}[theorem]{Assumption}

\numberwithin{equation}{section}

\makeatletter
\newcommand{\hathat}[1]{%
\begingroup%
  \let\macc@kerna\z@%
  \let\macc@kernb\z@%
  \let\macc@nucleus\@empty%
  \hat{\raisebox{.4ex}{\vphantom{\ensuremath{#1}}}\smash{\hat{#1}}}%
\endgroup%
}
\makeatother

\makeatletter
\newcommand{\dtilde}[1]{%
\begingroup%
  \let\macc@kerna\z@%
  \let\macc@kernb\z@%
  \let\macc@nucleus\@empty%
  \tilde{\raisebox{.4ex}{\vphantom{\ensuremath{#1}}}\smash{\tilde{#1}}}%
\endgroup%
}
\makeatother

\begin{document}

\title[Asymptotic stability for RVML]{Asymptotic Stability for the Relativistic Vlasov-Maxwell-Landau System in a Bounded Domain}

\author[H. Dong]{Hongjie Dong}
\address[H. Dong]{Division of Applied Mathematics, Brown University, 182 George Street, Providence, RI 02912, USA}
\email{Hongjie\_Dong@brown.edu }
\thanks{H. Dong was partially supported by a Simons fellowship, grant no. 007638 and the NSF under agreement DMS-2350129. Y. Guo's research was supported in part by NSF Grant  DMS-2405051.}

\author[Y. Guo]{Yan Guo}
\address[Y. Guo]{Division of Applied Mathematics, Brown University, 182 George Street, Providence, RI 02912, USA}
\email{Yan\_Guo@brown.edu}

\author[T. Yastrzhembskiy]{Timur Yastrzhembskiy}
\address[T. Yastrzhembskiy]{Department of Mathematics,  University of Wisconsin-Madison, WI 53706, USA}
\email{yastrzhembsk@wisc.edu}

\subjclass[2010]{35Q83, 35Q84, 35Q61, 35K70, 35H10, 34A12}
\keywords{Relativistic Vlasov-Maxwell-Landau system,  collisional plasma, perfect conductor boundary condition, specular reflection boundary condition, div-curl estimate, velocity averaging lemma}

\begin{abstract}
The control of plasma-wall interactions is crucial to fusion devices from both physical and mathematical perspectives.
It is well known that a magnetic field satisfying the classical perfect conducting conditions at the wall,
$$
    \bE \times n_x = 0, \quad  \bB \cdot n_x = 0,
$$
plays an important role in fusion plasma dynamics studies.
Since the early 1990s \cite{G_94}, it has been understood that the Lorentz force can penetrate into the domain at the boundary and create a singularity. Consequently, the uniqueness for any nonlinear kinetic plasma models in the presence of a perfectly conducting boundary remained open until our recent local well-posedness result \cite{VML}. In this paper, we finally establish a global well-posedness theory for the relativistic Vlasov-Maxwell-Landau system in a general $3D$ domain
with a specularly reflective, perfectly conducting boundary.
\end{abstract}

\maketitle

\tableofcontents

\section{Introduction}
\label{section 1}
The main impetus for plasma studies is nuclear fusion, 
with the tokamak serving as a key device that confines charged particles within a toroidal (doughnut-shaped) boundary in the presence of a magnetic field. Even though the plasma-wall interaction is extremely complex and challenging to control (see \cite{St_81}), a classical perfect conductor boundary condition is often imposed  on the electromagnetic field:
\begin{align}
    \label{eq1.5}
     (\bE \times n_x)|_{\partial \Omega} = 0, \quad (\bB \cdot n_x)|_{\partial \Omega} = 0.
\end{align}

It is well known that the kinetic description is fundamental  to the study of plasma in fusion devices,   which is formulated by a system of PDEs  for the density distribution functions of ions and electrons $F^{\pm} (t, x, p), x \in \Omega, p \in \bR^3,$  as well as for the electromagnetic field $\bE (t, x), \bB (t, x)$: 
\begin{align}
        \label{RVML}
&
\partial_t F^{+} +  \frac{p}{p_0^{+}} \cdot \nabla_x F^{+} + e_{+} (\bE + \frac{p}{p_0^{+}} \times \bB) \cdot \nabla_p F^{+} = \cC (F^{+}, F^{+} + F^{-}),\\
&
	\partial_t F^{-} +  \frac{p}{p_0^{+}} \cdot \nabla_x F^{-} - e_{-} (\bE + \frac{p}{p_0^{-}} \times \bB) \cdot \nabla_p F^{-} = \cC (F^{-}, F^{-}+F^{+}),\notag\\
%
			&\partial_t \bE -  \nabla_x \times \bB = -   4 \pi \int  (e_{+} \frac{p}{p_0^{+}}F^{+} - e_{-} \frac{p}{p_0^{-}} F^{-}) \, dp, \notag\\
 &
	\partial_t \bB +   \nabla_x \times \bE = 0. \notag\\
 &
	\nabla_x \cdot \bE = 4 \pi \int (e_{+} F^{+} - e_{-} F^{-}) \, dp,
	   \quad \nabla_x \cdot \bB = 0, \notag
\end{align}
In this model, the speed of light is set to $1$  for convenience.
Here,
$m_{\pm}$ and $e_{\pm}$ are masses and magnitudes of charges of electrons and ions,  $p$ is the momentum variable, and $p_0^{\pm}  = \sqrt{m_{\pm}^2 + |p|^2}$.
The terms $\cC (\cdot, \cdot)$ denote the relativistic Landau collision operators, which characterize the collision rates of charged particles  (see \eqref{eq0.1}).
 We refer to \eqref{RVML} as the relativistic Vlasov-Maxwell-Landau system (RVML).

Motivated by the tokamak device,  an essential  PDE problem is to study the well-posedness theory for the system \eqref{RVML} 
under the perfect conductor boundary condition \eqref{eq1.5} in a non-convex domain. Unfortunately, due to the presence of a notorious singularity from the grazing set
$$
   \gamma_0 : =  \{(x, p) \in \partial \Omega \times \bR^3: p\cdot n_x  = 0\}
$$
in a non-convex domain (see \cite{K_11}), such as a tokamak,  there has not been  
a single local well-posedness result until recently  \cite{VML} for any nonlinear kinetic models in $3D$ with any boundary conditions for the charged plasma in the presence of the perfect conductor condition \eqref{eq1.5}. The primary objective of this paper is to extend the solutions constructed in \cite{VML} globally in time by establishing global well-posedness and asymptotic stability of  Maxwellians for the relativistic Landau collision
\begin{align}
    \label{eq0.1}
     \cC (F^{\pm}, G^{\pm}) (p) =  \nabla_p \cdot \int_{\bR^3}  \Phi (P_{\pm}, Q_{\pm}) \big(\nabla_p F^{\pm} (p) G^{\pm} (q) -   F^{\pm} (p) \nabla_q G^{\pm} (q)\big) \, dq
\end{align}
(see \eqref{eq1.8} - \eqref{eq1.11})  with the specular reflection boundary condition (SRBC) for charged particles
\begin{align*}
    F|_{\gamma_{-}} (t, x, p) = F|_{\gamma_{+}} (t, x, R_x p), \quad R_x p: = p - 2 (p \cdot n_x) n_x,
\end{align*}
where
\begin{align*}
     \gamma_{\pm}  = \{(x, p) \in \partial \Omega \times \bR^3: \pm p\cdot n_x  > 0\} \notag
\end{align*}
are the outgoing $(\gamma_{+})$ and the incoming $\gamma_{-}$ boundaries. Concerning nonlinear collisional kinetic models with self-consistent magnetic effects in the absence of spatial boundaries,  global well-posedness was first established in \cite{VMB_03} and \cite{GS_03} for the non-relativistic Vlasov-Maxwell-Boltzmann and RVML systems, respectively, under periodic boundary conditions. Further related studies can be found in  \cite{YY_12}, \cite{Du_14}, \cite{LZ_14}, and \cite{X_15} (see also references therein).

Spatial boundaries are natural in kinetic models, and understanding boundary value problems is one of the critical aspects of modern kinetic PDE theory. 
However, the investigation of hyperbolic kinetic models presents a significant challenge due to the intricate behavior near the grazing set $\gamma_0$ associated with the free streaming operator $\partial_t + p \cdot \nabla_x$. Close to this set, the solution's regularity diminishes, leading to mathematical complexities which cannot be addressed by the standard energy techniques.

More precisely, singularities arising from the grazing set in non-convex domains \cite{K_11} highlight an expected limitation in hyperbolic kinetic PDEs, where solutions may, at best, exhibit bounded variation \cite{GKTT_16} under the diffuse boundary condition. 
Additionally, the inclusion of \textit{self-consistent magnetic effects} can induce singular behavior, even in a half-space domain. 
The example of singularity arising in the $3D$ relativistic Vlasov-Maxwell (RVM) system under the perfect conductor boundary condition on a half-space (see \cite{G_94}--\cite{G_95}) shows the limited extent of our current understanding, as only the global existence of a weak solution has been established for the RVM system in a $3D$ bounded domain \cite{G_93}.

For Vlasov-type equations in \textit{convex domains},  it is known that the characteristics can be controlled near the grazing set via the so-called velocity lemma, leading to significant progress on regularity and well-posedness for Vlasov-Poisson, Vlasov-Maxwell, and Boltzmann equations in convex domains  (see  \cite{CKL_19},  \cite{CaoK_21},  \cite{CK_22},  \cite{CKL_20}, \cite{CK_24}, \cite{DW_19}, \cite{DKL_23}, \cite{EGKM_18}, \cite{G_94}, \cite{G_95}, \cite{G_10}, \cite{GKTT_17}, \cite{H_04}, \cite{HV_10} among others). Unfortunately, the inclusion of magnetic effects and a non-convex geometry (such as a torus for a tokamak) has remained elusive. On the other hand, \cite{CaoK_23} recently discovered a striking application of the Vlasov-Maxwell system for modeling the exospheric solar wind. The model incorporates crucial self-consistent magnetic effects along with essential external electromagnetic and gravitational forces that satisfy the renowned Pannekoek-Rosseland condition, resulting in particles' acceleration directed outward at the boundary, which differs significantly from behavior in a tokamak. By leveraging these external forces, the authors of \cite{CaoK_23} establish a novel variant of the velocity lemma, leading to local Lipschitz regularity and well-posedness.

In the context of the hard-sphere Boltzmann equation, an important problem in kinetic theory is understanding the well-posedness and regularity of solutions in the presence of specularly reflecting boundaries in general non-convex domains. The prominence of the specular reflection boundary condition in studies of the  hard-sphere Boltzmann equation stems from it being the only boundary problem that has been rigorously derived from  particle systems. In the absence of convexity, achieving  well-posedness is challenging due to the possibility of characteristics propagating inside the domain and infinite bouncing in finite time, creating difficulties in achieving the crucial $L_{\infty}$ control of the density function. Significant advancement has recently been made in several papers examining specific, physically important cases of non-convex geometries  \cite{AL_24}, \cite{KL_18ARMA}, \cite{KL_24},    \cite{KKL_23}.  Interestingly, the authors of \cite{KKL_23} and \cite{AL_24} demonstrate that, under certain geometric conditions, the characteristic flow exhibits H\"older-type continuity, enabling them to establish  H\"older regularity of the solution to the  Boltzmann equation.

In contrast to hyperbolic models, solutions to kinetic velocity diffusive PDEs are expected to exhibit higher regularity near the grazing set
due to a hypoelliptic gain \cite{S_22}. The specifics of this regularity depend on the boundary conditions imposed on the outgoing boundary $\gamma_{-}$. In particular, for a linear kinetic Fokker-Planck equation with the inflow (Dirichlet) boundary conditions, the solutions are merely H\"older continuous in both spatial and velocity variables \cite{HJL_14}.
Remarkably, in the presence of the specular reflection boundary condition, solutions have higher regularity, which is established by using a flattening and extension method combined with the $S_p$ estimates on the whole space (see \cite{GHJO_20}, \cite{DGO_22}, \cite{DGY_21}). The possibility of such an extension argument for other boundary conditions in kinetic theory remains unknown.

Recently, the $L_2$ to $L_{\infty}$ framework has been developed for nonlinear collisional kinetic models in bounded domains  \cite{G_10}, \cite{EGKM_13}, \cite{EGKM_18},  \cite{GKTT_17}, \cite{KL_18ARMA}, \cite{KL_18}, \cite{KGH_20}, \cite{GHJO_20}, \cite{DGO_22}. The approach is based on interpolating between the natural energy or entropy bound and a `higher regularity' estimate. Specifically, it employs the velocity averaging lemma for the Boltzmann equation and a hypoelliptic gain in the context of the Landau equation. Extending this method to the  RVML system poses a formidable challenge due to the anticipated derivative loss at the highest order caused by the perfect conductor boundary condition.
To overcome this difficulty, we devise an intricate scheme based on propagating temporal derivatives. By capitalizing on the rich structure of the RVML system, we precisely identify the aforementioned derivative loss, demonstrating that it affects only the electromagnetic field and the macroscopic density. For the closure of the energy estimate, we establish an \textit{unexpected} $W^1_3$ estimate of velocity averages. Finally, to conclude the argument, we use a delicate  \textit{descent} strategy by leveraging the $S_p$  and the div-curl estimates.

\section{Notation and conventions}
                                            \label{section 2}
\begin{itemize}
\item Geometric notation.
\begin{align}
 \label{eq1.1}
  &	 p_0^{\pm} = \sqrt{m_{\pm}^2 + |p|^2},  \quad  p_0 = \sqrt{ 1 + |p|^2},  \\
  & P_{\pm}=  (p_0^{\pm}, p), \quad  P_{\pm} \cdot Q_{\pm} = p_0^{\pm} q_0^{\pm} - p \cdot q, \notag\\
 &  B_r  (x_0) = \{x \in \bR^3: |x - x_0| < r\}, \notag \\
& \Sigma^T = (0, T)\times \Omega \times \bR^3,\quad \Sigma^T_{\pm} = (0,T) \times \gamma_{\pm}, \notag \\
&\mathcal{R} (\Omega) =\{\bm{u} = A x + \bm{b}: (\bm{u} \cdot n_x)|_{\partial \Omega} = 0, \, A - \text{skew symmetric matrix}, \bm{b} \in \bR^3\}, \notag\\
&R_k \, \text{is a basis vector of} \, \mathcal{R} (\Omega), k = 1, 2, \ldots. \notag
\end{align}

\item Matrix notation.
\begin{align*}
    \bm{1}_d = (\delta_{i j}, i, j = 1, \ldots d), \quad \bm{R} = \text{diag} (1, 1, -1).
\end{align*}

\item  We define the (global) J\"uttner's solution  as
\begin{equation}
    \label{juttner}
    J^{\pm} (p) = \big(4 \pi e_{\pm} m_{\pm}^2  k_b T K_2 (\frac{m_{\pm}}{k_b T})\big)^{-1} e^{- p_0^{\pm}/(k_b T)},
\end{equation}
where $T$ is the temperature, $k_b$ is the Boltzmann constant, and
$$
    K_2 (s) = \frac{s^2}{3} \int_1^{\infty} e^{-s t} (t^2-1)^{3/2} \, dt
$$
is the Bessel function (see \cite{GS_03}). 
Both $J^{\pm}$ are normalized so that
\begin{align}
    \label{eq1.1.20}
        e_{+} \int_{\bR^3} J^{+} \, dp = 1 =  e_{-} \int_{\bR^3} J^{-} \, dp.
\end{align}

\item Relativistic Landau-Belyaev-Budker kernel. Let $L_{+, -}$ be the Coulomb logarithm for the ion and electron scattering. We introduce
 \begin{align}
 \label{eq1.8}
&	\Lambda (P_{+}, Q_{-}) = \big(\frac{P_{+}}{ m_{+}}\cdot \frac{Q_{-}}{ m_{-} }\big)^2 \bigg(\big(\frac{P_{+}}{ m_{+} }\cdot \frac{Q_{-}}{ m_{-} }\big)^2 - 1\bigg)^{-3/2},\\
 \label{eq1.9}
	&S (P_{+}, Q_{-})  =  \bigg(\big(\frac{P_{+}}{ m_{ + }}\cdot \frac{Q_{-}}{ m_{-} }\big)^2  - 1\bigg) \bm{1}_3 \\
	&\quad \quad   \quad \quad - (\frac{p}{ m_{+}}-\frac{q}{ m_{-}  }) \otimes (\frac{p}{ m_{+} }-\frac{q}{ m_{-}  }) \notag\\
	& \quad \quad  \quad \quad+  (\frac{P_{+}}{m_{+}} \cdot \frac{Q_{-}}{m_{-}} - 1)(\frac{p}{ m_{+} } \otimes \frac{q}{ m_{-}  }  + \frac{q}{ m_{-}  } \otimes \frac{p}{ m_{+} }),\notag\\
&
\label{eq1.11}
	\Phi (P_{+}, Q_{-}) = 2 \pi e_{+} e_{-} L_{+, -} \Lambda (P_{+}, Q_{-})\frac{m_{+} }{p_0^{+}} \frac{m_{-}}{ q_0^{-}} S (P_{+}, Q_{-}).
\end{align}
The rest of the kernels  $\Phi (P_{-}, Q_{+})$, $\Phi (P_{+}, Q_{+})$, and $\Phi (P_{-}, Q_{-})$ are defined in the same way.

\item Function spaces.
\begin{itemize}[--]



\item[--] \textit{Anisotropic H\"older space.}
For  an open set $D \subset \bR^6$ and $\alpha \in (0, 1]$, 
by $C^{\alpha/3, \alpha}_{x, p} (D)$,
we denote the set of all bounded functions
$f = f (x, p)$ such
that
\begin{align*}
        &[f]_{C^{\alpha/3, \alpha}_{x, p} (D)  } \\
        &:= \sup_{ (x_i, p_i) \in \overline{D}:  (x_1, p_1) \neq (x_2, p_2) } \frac{|f (x_1, p_1)- f (x_2, p_2)|}{(|x_1-x_2|^{1/3}+|p_1 - p_2|)^{\alpha}} < \infty.
\end{align*}
Furthermore, the norm is given by
\begin{equation}
			\label{1.2.0}
    \|f\|_{C^{\alpha/3, \alpha}_{ x, p } (D)  } : = \|f\|_{  L_{\infty} (D) }  +  [f]_{C^{\alpha/3, \alpha}_{x, p} (D)  }.
\end{equation}


\item \textit{Traces}.
For a function $u$ such that
\begin{align}
    \label{eq1.2.25}
    u, (\partial_t + \frac{p}{ p_0^{+} } \cdot \nabla_x)  u \in L_r ((0, T) \times \Omega \times \bR^3), r \in [1, \infty),
\end{align}
 one can define traces of $u$. See the details in \cite{BP_87}  or \cite{U_86}. In particular, there exist functions  $(u (t, \cdot)$, $u (0, \cdot)$, $u|_{\gamma_{\pm}})$, which we call traces of $u$, such that a variant of Green's identity for the operator $(\partial_t +  \frac{p}{ p_0^{+} } \cdot \nabla_x)$ holds (see \eqref{eqE.A.1.1}).

\item
\textit{Weighted Lebesgue space.}
For $G \subset \bR^3_x \times \bR^3_p$,  $\theta \in \bR$, and $r \in [1, \infty]$, by $L_{r, \theta} (G)$
we denote the set of all  measurable functions $u$ such that
$$
    \|u\|_{ L_{r, \theta} (G) }:= \|p_0^{\theta} u\|_{L_r (G)} < \infty.
$$

\item \textit{Weighted Sobolev spaces.}
For $r \in [1, \infty]$, by $W^1_{r, \theta} (\bR^3)$ we denote the Banach space of functions  $u  \in L_{r, \theta} (\bR^3_p)$ with the norm
$$
    \|u\|_{ W^1_{r, \theta} (\bR^3)} :=   \||u|+|\nabla_p u|\|_{ L_{r, \theta} (\bR^3) } < \infty.
$$
For $\theta = 0$, we set $W^1_r (\bR^3): = W^1_{r, 0} (\bR^3)$.

\item \textit{Dual Sobolev space.} Let
$W^{-1}_r (\bR^3), r \in (1, \infty)$  be the space of all distributions  $u$
such that
$$
    u =  \partial_{p_i}  \eta_i + \xi
$$
for some $\xi, \eta_i \in L_r (\bR^3), i = 1, 2, 3$.  Furthermore, for $u \in W^{-1}_2 (\bR^3)$ and $f \in W^1_{2} (\bR^3)$, by
\begin{align}
      \label{eq1.2.16}
    \langle u, f \rangle =  \int_{\bR^3}  (- \eta_i  \partial_{p_i} f+ \xi f) \, dp,
\end{align}
we denote the duality pairing between $W^{-1}_2 (\bR^3)$ and $W^{1}_{2} (\bR^3)$, which is independent of the choice of $\eta_i$ and $\xi$.  We note that if $u, f \in L_2 (\bR^3)$ then the l.h.s. of \eqref{eq1.2.16} is the  $L_2$ interior product. 


\item \textit{Fractional Sobolev spaces.} For $r \in (1, \infty)$,  we set
\begin{itemize}
    \item[--] $H^s_r (\bR^d) = (1-\Delta)^{-s/2} L_r (\bR^d)$, $s \in \bR$, to be the Bessel potential space
    with the norm
    \begin{align}
        \label{eq1.2.20}
            \|u\|_{H^s_r (\bR^d)} = \|(1-\Delta)^{s/2} u\|_{L_r (\bR^d)},
        \end{align}
    \item[--] $W^s_r (\Omega), s \in (0, 1)$ to be the Sobolev-Slobodetskii space with the norm
    \begin{align}
            \label{eq1.2.21}
      &          \|u\|_{W^s_r (\bR^d)} = \|u\|_{L_r (\bR^d)} + [u]_{W^s_r (\bR^d)}, \\
     & [u]^r_{W^s_r (\bR^d)}: = \int_{\Omega} \int_{\Omega} \frac{ |u(x) -  u (y)|^r}{|x- y|^{d +  s r}} \, dx dy.  \notag
        \end{align}
\end{itemize}

\item \textit{Mixed-norm spaces.} For normed spaces $X$ and $Y$, we write $u = u (x, y) \in X Y$ if for each $x \in X$, we have $u (x, \cdot) \in Y$, and
$$
    \|u\|_{ X Y  }  :=   \| \|u (x, \cdot)\|_{ Y } \|_X  < \infty.
$$



\item
 \textit{Steady $S_r$ spaces.}   For $r \in (1, \infty)$, by $S_{ r, \theta} (\Omega \times \bR^3)$, we denote the set of all functions $u = (u^{+}, u^{-})$ on $\Omega \times \bR^3$ such that
\begin{equation}
                \label{eq1.2.7.0}
  \begin{aligned}
 & u, \frac{p}{p_0^{\pm}} \cdot  \nabla_x u^{\pm}, \nabla_p u, D^2_p u \in   L_{r, \theta} (\Omega \times \bR^3).
 \end{aligned}
 \end{equation}
 The norm is given by
\begin{equation}
			\label{eq1.2.7}
\begin{aligned}
	& \quad \quad \quad  \|u\|_{ S_{ r, \theta} (\Omega \times \bR^3) }  =  \||u| + |\frac{p}{p_0^{\pm}} \cdot \nabla_x u^{\pm}| + |\nabla_p u| + |D^2_p u|\|_{ L_{r, \theta} ( \Omega \times \bR^3 ) }.
\end{aligned}
\end{equation}
 We also define the steady Newtonian (non-relativistic) $S_r$ space  as
\begin{align}
       \label{eq14.C.20}
       S_r^N (\bR^{2d}) = \{ u, \nabla_v u, D^2_v u,  v \cdot \nabla_x u \in L_r (\bR^{2d})\}
\end{align}
with the norm 
$$
    \|u\|_{  S_r^N (\bR^{2d}) } : = \||u| + |v \cdot \nabla_x u| + |\nabla_v u| + |D^2_v u|\|_{ L_r (\bR^{2d}) }.
$$

\item 
 \textit{Unsteady $S_r$ spaces.} 
    \begin{align}
        \label{eq36.0.1}
         S_r (\Sigma^T) : = \{u: u, \nabla_p u, D^2_p u, (\partial_t +  \frac{p}{p_0^{\pm}} \cdot \nabla_x) u^{\pm} \in L_r (\Sigma^T)\}.
    \end{align}
\end{itemize}

\item \textit{Vector fields.} We use boldface letters to denote vector fields. We write $\bm{u} \in  X$, where $X$ is a functional space if each component of $\bm{u}$ belongs to $X$.

\item \textit{Stress tensor.}
We set
\begin{align}
    \label{eq1.2.23}
    S_{i j} (\bm{u}):= \frac{1}{2} (\partial_{x_i} \bm{u}_j +  \partial_{x_j} \bm{u}_i)
\end{align}
to be the stress tensor of $\bm{u}$.

\item
\textit{Conventions.}
\begin{itemize}
\item We assume the summation over repeated indices.

\item  By $N = N (\cdots)$, we denote a constant depending only on the parameters inside the
parentheses. The constants $N$ might change from line to line. Sometimes, when it is clear what parameters $N$  depends on, we omit them.

\item Whenever the relationships among physical constants are not relevant to the argument, we set all such constants to $1$ and drop the subscripts and superscripts $\pm$ in  $p_0^{\pm}$, $P_{\pm}$, $J^{\pm}$.
\end{itemize}
\end{itemize}

\section{Main results}
\label{section 3}
Let $f^{\pm}$ be the perturbations of $F^{\pm}$ near the J\"uttner's solution defined as
\begin{align}
    \label{eq36.20}
        F^{\pm} = J^{\pm} + \sqrt{J^{\pm}} f^{\pm}.
\end{align}
Then, the perturbation $f = (f^{+}, f^{-})$ satisfies the following system, which we also call the  RVML system (see \cite{GS_03}):
\begin{align}
\label{eq36.13}
& \partial_t f^{+}  +    \frac{p}{p_0^{+}} \cdot \nabla_x f^{+}  + e_{+} (\bE + \frac{p}{p_0^{+}} \times \bB) \cdot \nabla_p f^{+}
- \frac{e_{+}  }{k_b T} \frac{p}{p_0^{+}}  \cdot \mathbf{E} \sqrt{J^{+}} \\
&-  \frac{e_{+}  }{2 k_b T} \frac{p}{p_0^{+}}  \cdot \mathbf{E} f^{+}
+ L_{+} f = \Gamma_{+} (f, f), \notag \\
\label{eq36.14}
& \partial_t f^{-}  +  \frac{p}{p_0^{-}} \cdot \nabla_x f^{-} - e_{-} (\bE + \frac{p}{p_0^{-}} \times \bB) \cdot \nabla_p f^{-}
+ \frac{e_{-}  }{k_b T} \frac{p}{p_0^{-}}  \cdot \mathbf{E} \sqrt{J^{-}}\\
&+   \frac{e_{-}  }{2 k_b T} \frac{p}{p_0^{-}}  \cdot \mathbf{E} f^{-}
+ L_{-} f = \Gamma_{-} (f, f),  \notag \\
 \label{eq36.19}
  &   f_{-} (t, x, p) = f_{+} (t, x, R_x p), \quad f (0, \cdot) = f_0, \\
\label{eq36.16}
 &\partial_t \bE -  \nabla_x \times \bB = - 4 \pi \bm{j}: =  - 4 \pi \int   (e_{+}\frac{p}{p_0^{+}} f^{+} \sqrt{J^{+}} -  e_{-}\frac{p}{ p_0^{-} } f^{-}\sqrt{J^{-}} ) \, dp, \\
 &
 \label{eq36.16.1}
 \partial_t \bB  +  \nabla_x \times \bE = 0, \\
 &
 \label{eq36.17}
	\nabla_x \cdot \bE = 4 \pi \rho:=  4 \pi \int   (e_{+} f^{+}  \sqrt{J^{+}}  - e_{-} f^{-} \sqrt{J^{-}}) \, dp, \quad \nabla_x \cdot \bB = 0, \\
 &
  \label{eq36.18}
  (\bE \times n_x)|_{\partial \Omega} = 0, \quad (\bB \cdot n_x)|_{\partial \Omega} = 0,  \quad \bE (0, \cdot) = \bE_0 (\cdot), \, \, \bB (0, \cdot) = \bB_0 (\cdot),
\end{align}
where
\begin{align}
    \label{eq36.30}
   &  L_{\pm}  =A_{\pm} - K_{\pm}, \\
       \label{eq36.31}
   & A_{\pm} = (J^{\pm})^{-1/2} \cC (\sqrt{J^{\pm}} g^{\pm}, J^{+} + J^{-}), \\
       \label{eq36.32}
   & K_{\pm} g = (J^{\pm})^{-1/2} \cC (J^{\pm}, \sqrt{J^{+}} g^{+} + \sqrt{J^{-}} g^{-}), \\
       \label{eq36.33}
   & \Gamma_{\pm} (g, h) = (J^{\pm})^{-1/2} \cC (\sqrt{J^{\pm}} g^{\pm}, \sqrt{J^{+}} h^{+} + \sqrt{J^{-}} h^{-}), \\
       \label{eq36.34}
   & L = (L_{+}, L_{-}), \quad \Gamma (g, h) = (\Gamma_{+} (g, h), \Gamma_{-} (g, h)).
\end{align}

\textbf{Steady state solution.}
To guarantee that $F = (J^{+}, J^{-})$, $\bE  = 0 = \bB$ is a steady state of the RVML system \eqref{RVML}, we impose the \textbf{global neutrality condition}
\begin{equation}
		\label{eq0}
e_{+} \int_{\bR^3} J^{+} \, dp
	= e_{-} \int_{\bR^3} J^{-} \, dp
\end{equation}
(see \eqref{eq1.1.20}).
We denote
\begin{equation}
    \label{eq6.30}
        M_{\pm} = \int_{\bR^3} J^{\pm} \, dp.
\end{equation}
Note that, due to our choice of normalization in \eqref{eq1.1.20}, we have $M_{\pm} = e_{\pm}^{-1}$.

\textbf{Macro-micro decomposition.}
Recall that the linearized Landau operator $L$ has the following null space (see \cite{GS_03}):
$$
 \text{span} \,  \{(\sqrt{J^{+}}, 0), (0, \sqrt{J^{-}}), (p_i \sqrt{J^{+}}, p_i \sqrt{J^{-}}), (p_0^{+} \sqrt{J^{+}}, p_0^{-} \sqrt{J^{-}}), i = 1, 2, 3\}.
$$
  Its orthonormal basis can be chosen as follows:
\begin{align}
    \label{eq6.22}
& \chi_{1}= (M_{+})^{-1/2}(\sqrt{J^{+}}, 0), \quad \chi_{2} = (M_{-})^{-1/2}  (0, \sqrt{J^{-}}),\\
    \label{eq6.23}
& \chi_{i+2} =  \kappa_1 (p_i\sqrt{J^{+}}, p_i\sqrt{J^{-}}), i = 1, 2, 3, \\
    \label{eq6.24}
&  \chi_{6} = \kappa_3 \big((p_0^{+} - \kappa_2^{+})  \sqrt{J^{+}}, (p_0^{-}  - \kappa_2^{-})\sqrt{J^{-}}\big),
\end{align}
where
\begin{align}
\label{eq6.11}
&\kappa_1 =  \bigg(\int p_1^2 (J^{+} + J^{-}) \, dp\bigg)^{-1/2},\\
&\label{eq6.10}
 \kappa^{\pm}_2 = \frac{ \int J^{\pm} p_0^{\pm} \, dp }{  \int J^{\pm} \, dp},\\
 \label{eq6.12}
 & \kappa_3 = \bigg(\int |p_0^{+} - \kappa_2^{+}|^2 J^{+} \, dp + \int |p_0^{-}  - \kappa_2^{-}|^2 J^{-} \, dp\bigg)^{-1/2}.
\end{align}
By $\chi_i^{+}, \chi_i^{-}$, we denote the first and the second components of  $\chi_i$, respectively.
The  constants $\kappa_2^{\pm}$ were chosen  so that
$$
	\int J^{\pm} (p_0^{\pm}  - \kappa_2^{\pm}) \, dp = 0,
$$
which yields
\begin{equation*}
\langle \chi_6^{+}, \chi_1^{+} \rangle = 0
 = \langle \chi_6^{-}, \chi_2^{-} \rangle.
\end{equation*}

 The projection operator $P = (P^{+}, P^{-})$  onto  the kernel of $L$   is defined as follows (see p. 308 in \cite{GS_03}):
\begin{align}
 \label{eq6.16}
  &  P^{+} f
      = a^{+} \chi_1^{+}  + b_i \chi_{i+2}^{+} + c \chi_6^{+}\\
&
     =  [(M_{+})^{-1/2} a^{+}   + \kappa_1 b_i \cdot p_i + \kappa_3 c (p_0^{+} - \kappa_2^{+}) ] \sqrt{J^{+}},\notag \\
 \label{eq6.17}
&      P^{-} f
          = a^{-} \chi_2^{-}  + b_i \chi_{i+2}^{-} + c \chi_6^{-}\\
&
     =  [(M_{-})^{-1/2} a^{-}   + \kappa_1 b_i \cdot p_i + \kappa_3 c (p_0^{-} - \kappa_2^{-}) ] \sqrt{J^{-}},  \notag
     \end{align}
where
\begin{align}
 \label{eq6.18}
  &    a^{\pm} = (M_{\pm})^{-1/2} \int f^{\pm}  \sqrt{J^{\pm}} \, dp,\\
 \label{eq6.19}
  & b_i = \kappa_1 \int p_i (f^{+}\sqrt{J^{+} } + f^{-}\sqrt{J^{-}}) \, dp, i  = 1, 2, 3, \\
 \label{eq6.20}
  & c = \kappa_3 \int \big(p_0^{+} - \kappa_2^{+})  f^{+} \sqrt{J^{+}} +  (p_0^{-}  - \kappa_2^{-}) f^{-} \sqrt{J^{-}} \, dp.
\end{align}

\textbf{Initial data.}
\begin{definition} 
\label{definition 34}
\index{$[f_{0, k}, \bE_{0, k}, \bB_{0, k}]$} We set $[f_{0, 0}, \bE_{0, 0}, \bB_{0, 0}] = [f_0, \bE_0, \bB_0]$.
Furthermore, given $f_{0, j} (x, p)$, $\bE_{0, j} (x)$, $\bB_{0, j} (x), j = 0, \ldots, k$,    we set  
\begin{align}
  \label{eq33.2}
   &  f^{\pm}_{0, k+1} =   -  \frac{p}{p_0^{\pm}} \cdot \nabla_x f_{0, k}^{\pm} -  L_{\pm} f_{0, k}  \pm \frac{e_{\pm} }{k_b T} (\frac{p}{p_0^{\pm}} \cdot \bE_{0, k}) \sqrt{J^{\pm}}\\
   &+ \sum_{j =0}^k \binom{k}{j} \bigg(\mp e_{\pm} (\bE_{0, j} + \frac{p}{p_0^{\pm}} \cdot \bB_{0, j}) \cdot \nabla_p  f_{0, k-j}^{\pm}  \pm \frac{e_{\pm} }{2k_b T} (\frac{p}{p_0^{\pm}} \cdot \bE_{0, j}) f_{0, k-j}^{\pm} + \Gamma_{\pm} (f_{0, j}, f_{0, k-j})\bigg), \notag\\
   \label{eq33.3}
    &  \bE_{0, k+1}  :=  \nabla_x  \times \bB_{0, k}  - 4 \pi \int_{\bR^3} \big(e_{+}\frac{p}{p_0^{+}} f^{+}_{0, k} \sqrt{J^{+}} - e_{-}\frac{p}{p_0^{-}} f^{-}_{0, k}  \sqrt{J^{-}}\big) \, dp,\\
      \label{eq33.4}
    & \bB_{0, k+1} = -  \nabla_x \times \bE_{0, k}.
\end{align}
\end{definition}

\textbf{Controls.} Let  $m \ge 20$ be an even number, which is the maximal number of $t$-derivatives we control in our scheme.
We introduce the `natural' instant   energy and the dissipation
\begin{align}
    \label{eq14.0.10}
 &   \cI_{||}  (\tau) = \sum_{k = 0}^{m}  \big(\|\partial_t^k f (\tau, \cdot)\|^2_{ L_{2} (\Omega \times \bR^3) } + \|\partial_t^k [\bE, \bB] (\tau, \cdot)\|^2_{ L_{2} (\Omega) }\big), \\
    \label{eq14.0.11}
&
    \cD_{||} (\tau) = \sum_{k = 0}^{m} \|(1-P) \partial_t^k f (\tau, \cdot)\|^2_{ L_{2} (\Omega) W^1_2 (\bR^3) }.
\end{align}

\textit{Total energy functionals.}
Let $\Delta r \in (0, \frac{1}{42})$ be a constant,  and $r_1, \ldots, r_4$ be  numbers satisfying the  conditions
 \begin{equation}
     \label{eq14.0.3}
     \begin{aligned}
 &  r_1 = 2,  \quad    \frac{1}{r_{i}} = \frac{1}{r_{i-1}} -  \big(\frac{1}{6} - \Delta r\big), i = 2, 3, 4, \\
  & r_2 \in (2, 3), \, r_3 \in  (3, 6),  \, r_4 > 36.
  \end{aligned}
 \end{equation}
 Formally,
 $r_2 = 3-$, $r_3 =6-$, $r_4 =  36+$.

For $\theta > 0$,
  the  total instant
  functional  $\cI$ is defined as 
\begin{align}
    \label{eq14.0.5}
      & \cI  (\tau) =  \cI_{||} (\tau) + \sum_{k = 0}^{m-4} \|\partial_t^k f (\tau, \cdot)\|^2_{   L_{2, \theta/2^k} (\Omega \times \bR^3)  } 
     + \sum_{k=0}^{m-1} \|\partial_t^k  [\bE, \bB] (\tau, \cdot)\|^2_{ W^1_{ 2 } (\Omega) }  \\
      & + \sum_{i=1}^4 \sum_{k=0}^{  m-4-i } \|\partial_t^k f (\tau, \cdot)\|^2_{  S_{r_i, \theta/2^{ k + 2i   } } (\Omega \times \bR^3) }      
+ \sum_{i=2}^3 \sum_{k=0}^{ m  - 4 -  i} \|\partial_t^k  [\bE, \bB] (\tau, \cdot)\|^2_{ W^1_{r_i} (\Omega) }. \notag
 \end{align}
 Furthermore, we define $\cI_{||} (0), \cI (0)$ by replacing $\partial_t^k [f, \bE, \bB] (\tau)$ with $[f_{0, k}, \bE_{0, k}, \bB_{0, k}]$ in \eqref{eq14.0.10} and \eqref{eq14.0.5}, respectively.

Next, we define the total dissipation  as
 \begin{align}
&    \label{eq12.0.1}
    \cD (\tau)
= \cD_{||} (\tau) + \sum_{k = 0}^{m-2} \|\partial_t^k [a^{+}, a^{-}] (\tau, \cdot)\|^2_{L_2 (\Omega)} +   \sum_{k = 0}^m \|\partial_t^k  [b, c] (\tau, \cdot)\|^2_{ L_2 (\Omega) } \\
   &+  \sum_{k = 0}^{m-3}  \|\partial_t^k  \bB (\tau, \cdot)\|^2_{ L_2 (\Omega) }
   + \sum_{k = 0}^{m-4} \|\partial_t^k  \bE (\tau, \cdot)\|^2_{ L_2 (\Omega) }  \notag\\
& + \sum_{k = 0}^{m-4}  \|\partial_t^k f (\tau, \cdot)\|^2_{
 L_2 (\Omega) W^1_{2, \theta/2^k} (\bR^3) } \notag\\
 & + \sum_{i=1}^4 \sum_{k=0}^{ m-4-i }  \|\partial_t^k f (\tau, \cdot)\|^2_{ S_{r_i, \theta/2^{k+2i} } (\Omega \times \bR^3) }
 + \sum_{i=2}^3 \sum_{k=0}^{ m  - 4 -  i} \|\partial_t^k  [\bE, \bB] (\tau, \cdot)\|^2_{ W^1_{r_i} (\Omega) }. \notag
 \end{align}
\begin{remark}
            \label{remark 12.1}
We note that by the macro-micro decomposition  $P f + (1-P) f$ (cf. \eqref{eq6.16}--\eqref{eq6.17}),
\begin{align}
    \label{eq12.0.4}
    \sum_{k =0}^{m-2}    \|\partial_t^k f (\tau, \cdot)\|^2_{ L_2 (\Omega) W^1_2 (\bR^3) }  \lesssim \cD (\tau).
\end{align}
\end{remark}



\subsection{Finite energy and strong solutions }
\begin{definition}[finite energy solution]
    \label{definition 27.1}
    We say  that 
    \begin{align}
        \label{eq27.2.5}
    f \in C ([0, T]) L_2 (\Omega \times \bR^3) \cap L_2 ((0, T) \times \Omega) W^1_2 (\bR^3)
    \end{align}
    is a finite energy   solution to 
   the problem 
      \begin{align}
   			\label{eq1.0}
	&	(\partial_t +  \frac{p}{p_0^{\pm}} \cdot \nabla_x) f  - \nabla_p \cdot (A \nabla_p f) = \eta, \\
\label{eq1.0.0}
 &f (t, x, p) = f (t, x, R_x p), \, z \in \Sigma^T_{-}, \, \, f (0, \cdot) = f_{0} (\cdot), 
  \end{align}
    if for any test function $\phi$ satisfying
        \begin{align}
   \label{eq27.2.1}
    & 
       \phi  \in L_2 ((0, T) \times \Omega) W^1_2 (\bR^3),   (\partial_t +  \frac{p}{p_0^{\pm}} \cdot \nabla_x) \phi \in  L_2 (\Sigma^T), \\
       \label{eq27.2.2}
           & \phi \in  C ([0, T]) L_2 (\Omega \times \bR^3), \\
      \label{eq27.2.3}
    &   \phi (t, x, p) = \phi (t, x, R_x p), \, (t, x, p) \in \Sigma^T_{-} \, \, \text{(in the trace sense)},
        \end{align}
and all $t \in [0, T]$, one has
\begin{align}
    \label{eq27.2}
    &   \int_{\Omega \times \bR^3} (f \phi) (t, x, p) - f_0 (x, p) \phi (0, x, p) \, dx dp\\
   & + \int_{ (0, t)  \times \Omega \times \bR^3 }  \bigg(-f (\partial_t  \phi +  \frac{p}{p_0^{\pm}} \cdot \nabla_x \phi) +(\nabla_p \phi)^T   A \nabla_p f\bigg) 
   \, dz   = \int_{(0, t)  \times \Omega }  \langle \eta (\tau, x, \cdot), \phi (\tau, x, \cdot) \rangle \, dx d\tau,  \notag
\end{align}
where $\langle \cdot, \cdot \rangle$ is the duality pairing between $W^{-1}_2 (\bR^3)$  and $W^{1}_2 (\bR^3)$ (see \eqref{eq1.2.16}).

Furthermore, if $A$,  $f$, and $\eta$ do not depend on  $t$, we say that $f \in L_2 (\Omega) W^1_2 (\bR^3)$ is a finite energy  solution to the steady equation
\begin{equation}
\begin{aligned}
    \label{eq27.1.2}
    	& \frac{p}{p_0^{\pm}} \cdot \nabla_x f  - \nabla_p \cdot (A \nabla_p f)  = \eta, \\
 &f (x, p) = f (x, R_x p), \, z \in \gamma_{-}, 
 \end{aligned}
\end{equation}
if for any test function $\phi =  \phi (x, p)$ satisfying the conditions analogous to \eqref{eq27.2.1}--\eqref{eq27.2.3}, the  `steady' counterpart of the identity \eqref{eq27.2} holds.  
\end{definition}

\begin{definition}[strong solution]
    \label{definition 3.3}
    We say that $f \in S_2 (\Sigma^T)$ (see \eqref{eq36.0.1}) is a strong solution to \eqref{eq1.0}--\eqref{eq1.0.0} if the identity \eqref{eq1.0} holds a.e., and the initial and boundary conditions are understood in the sense of traces (see \eqref{eq1.2.25}). Furthermore, if $T = \infty$, we replace $S_2 (\Sigma^T)$ with
    $\cap_{\tau > 0} S_2 (\Sigma^{\tau})$ in the  above definition.
\end{definition}

\begin{remark}
The Landau equation can be rewritten as \eqref{eq1.0} with
 certain $A$ and $\eta$ depending on $f$. See the details in the proof of Proposition \ref{proposition 4.0}. 
\end{remark}

\begin{definition}[cf. Definition 3.1 in \cite{VML}]
                \label{definition 36.0}
We say that the RVML system \eqref{eq36.13}--\eqref{eq36.18} has a  strong solution  $[f^{\pm}, \bE, \bB]$ on  $[0, T]$ if
\begin{itemize}
   \item[--] $f = (f^{+}, f^{-})$ is a strong solution to the  Landau equations \eqref{eq36.13}--\eqref{eq36.19}.  
    \item[--]  $\bE, \bB \in C^1 \big([0, T], L_2 (\Omega))$,
  \item[--] for any $t \in [0, T]$,  $\bE (t, \cdot), \bB (t, \cdot) \in   W^1_2 (\Omega)$,  and
  $(\bE (t, \cdot) \times n_x)_{|\partial \Omega} \equiv 0$, $(\bB (t, \cdot) \cdot n_x)_{|\partial \Omega} \equiv 0$,
\item[--] the identities \eqref{eq36.16}--\eqref{eq36.17} hold in the $L_2 ((0, T) \times \Omega)$ sense.
\end{itemize}
\end{definition}

\subsection{Assumptions}


\begin{assumption}[Compatibility conditions]
                    \label{assumption 3.3}
We assume 
\begin{align}
  \label{eq33.7}
 & f_{0, k} \, \text{is a finite energy  solution to \eqref{eq33.2} with the SRBC}, \, k \le m-1,\\
  \label{eq33.9}
  &  f_{0, k} (x, p) = f_{0, k} (x, R_x p), \,  (x, p) \in \gamma_{-} \, \text{(in the trace sense)}, \, \, k \le m-8, \\
  \label{eq33.8}
  &  (\bE_{0, k} \times n_x)|_{\partial \Omega} \equiv 0, \, \, (\bB_{0, k} \cdot n_x)_{\partial \Omega} \equiv 0, \, k \le m-1,\\
          \label{eq33.10}
&\nabla \cdot \bB_{0, k} \equiv 0, \,   k \le m-1, \\
\label{eq33.10.1}
&\nabla \cdot \bE_{0, k} (x) =4 \pi \int   (e_{+} f^{+}_{0, k}  \sqrt{J^{+}}  - e_{-} f^{-}_{0, k} \sqrt{J^{-}}) \, dp, \,   k \le m-1,
\end{align}
where in \eqref{eq33.9}, we implicitly assume that $f_{0, k}, \displaystyle \frac{p}{p_0^{\pm}} \cdot \nabla_x f_{0, k}^{\pm} \in L_2 (\Omega \times \bR^3)$, so that the trace is well defined.
\end{assumption}

\begin{remark}
  Assumption \ref{assumption 3.3} is easily satisfied if $f_0, \bE_0$, and $\bB_0$ are smooth, compactly supported functions  away from $\partial \Omega$, $f_0$ decays sufficiently fast for large $p$, and \eqref{eq33.10}--\eqref{eq33.10.1} hold with $k = 0$ (see Remark 3.8 in \cite{VML}).
\end{remark}

\begin{assumption}
    \label{assumption 1.2}
   We assume that the initial densities $F^{\pm}_0$ have the same total mass as  J\"uttner's solution $J^{\pm}$ and that the initial data $[F^{\pm}_0, \bE_0, \bB_0]$ possess the same total energy as the steady state $F^{\pm} = J^{\pm}$, $\bE = 0 = \bB$. On the level of  initial perturbations $f^{\pm}_0$ (see \eqref{eq36.20}) and $[\bE_0, \bB_0]$, we formulate this condition  as follows:
       \begin{align}
            \label{eq6.1.0}
    &  \int_{ \Omega \times \bR^3} f^{\pm}_0 \sqrt{J^{\pm}} \, dx dp = 0,\\
   \label{eq6.1.2}
    & \int_{\Omega \times \bR^3}  (p_0^{+} f^{+}_0 \sqrt{J^{+}}  + p_0^{-} f^{-}_0 \sqrt{J^{-}}) \, dx dp
    +   \frac{1}{8 \pi} \int_{\Omega} (|\bE_0|^2 + |\bB_0|^2) \, dx  =  0.
   \end{align}
Furthermore, if $\Omega$ is an axisymmetric domain, we additionally assume that the total angular momentum of the initial data is the same as that of the steady state. In particular, if an  axis  of rotation contains $x_0$ and is parallel to $\omega$, we assume that
    \begin{align}
     \label{eq6.1.3}
        &        \int_{\Omega} \int_{\bR^3} p \cdot \big(\omega \times (x-x_0)\big) (f^{+}_0 \sqrt{J^{+}}  + f^{-}_0\sqrt{J^{-}}) \, dp dx  \\
        &  + \frac{1}{4 \pi } \int_{\Omega} \big(\omega \times (x-x_0)\big) \cdot (\bE_0 \times \bB_0) \, dx = 0. \notag
        \end{align}
\end{assumption}

\begin{assumption}[cf. Hypothesis 1.1 in \cite{AS_13}]
    \label{assumption E.1}
    We assume that $\partial \Omega$ is connected and that there exist open connected  surfaces $\Sigma_j, j = 1, \ldots, L$, which we call ``cuts", such that

    \begin{enumerate}
    \item[(i)] each $\Sigma_j$ is an open part of a smooth manifold $M_j$,

    \item[(ii)] $\partial \Sigma_j \subset \partial \Omega$ for each $j$,

    \item[(iii)] $\overline{\Sigma_i} \cap \overline{\Sigma_j} = \emptyset, i \neq j$,

    \item[(iv)]
    $$
        \tilde \Omega    =  \Omega \setminus \bigcup_{j = 1}^L \Sigma_j
    $$
    is a simply connected $C^{1, 1}$ domain.
    \end{enumerate}
\end{assumption}

\begin{remark}
    We note that a solid torus $B_1 \times S^1$  satisfies Assumption \ref{assumption E.1} because it requires only a single cut to obtain a simply connected $C^{ 1, 1}$ domain.      
\end{remark}

\begin{assumption}
                \label{assumption E.2}
We assume that $\bB_0$ satisfies the following `vanishing flux' condition
\begin{align}
    \label{eqE.2.52}
    \int_{\Sigma_j} \bB_0  \cdot n_x\, dS_x = 0, \, \,  j = 1, \ldots, L,
\end{align}
(see Assumption \ref{assumption E.1}).
\end{assumption}

\subsection{Main results}

\begin{theorem}[global  well-posedness]
    \label{theorem 3.1}
Let  $r_1, \ldots, r_4$ be numbers satisfying \eqref{eq14.0.3} and $m \ge  20$  be an even integer. We impose    Assumptions \ref{assumption 3.3}, \ref{assumption 1.2}--\ref{assumption E.2}.
  Then, there exist numbers 
  \begin{align*}
      \theta (m, r_1, \ldots, r_4) > 1,  \quad C_0 =  C_0  (m, r_1, \ldots, r_4, \Omega) \in (0, 1), 
          \end{align*}
  such that if $\cI (0) < \infty$ and 
  \begin{align}
      \label{eq14.0.12}
    &  I_0 := \sum_{k = 0}^{m}  \big(\|f_{0, k}\|^2_{ L_{2} (\Omega \times \bR^3) } + \|[\bE_{0, k}, \bB_{0, k}]\|^2_{ L_{2} (\Omega) }\big) \\
    & + \sum_{k = 0}^{m-4} \|f_{0, k}\|^2_{   L_{2, \theta/2^k} (\Omega \times \bR^3)  } < C_0 \notag
  \end{align}
  (see \eqref{eq33.2}), then, the following assertions hold. 

  $(i)$ The RVML system \eqref{eq36.13}--\eqref{eq36.18}   has a strong solution $[f, \bE, \bB]$   on $[0, \infty)$ (see Definition \ref{definition 36.0}) such that
  \begin{align}
        \label{eq14.0.13}
    \cI (t) + \int_0^t \cD (\tau) \, d\tau < C I_0, \, \, t >  0
  \end{align}
  (see \eqref{eq14.0.5}--\eqref{eq12.0.1}), where $C = C (m, r_1, \ldots, r_4, \Omega)$.

    $(ii)$ For $k \le m-5$, $\partial_t^k f$ is a strong   solution (see Definition \ref{definition 3.3})  to \eqref{eq36.13}--\eqref{eq36.14} differentiated $k$ times in $t$ with the initial data $\partial_t^k f (0, \cdot) = f_{0, k} (\cdot)$ and SRBC, while for  $m-4 \le k \le m$, $\partial_t^k f$ is a finite energy solution (see Definition \ref{definition 27.1}).
   
     $(iii)$ For $k \le m-1$, $\partial_t^k [\bE, \bB] \in C ([0, \infty)) L_2 (\Omega) \cap L_{\infty, \text{loc}} ((0, \infty)) W^1_2 (\Omega)$ is a  strong solution to Maxwell's equations \eqref{eq36.16}--\eqref{eq36.16.1} differentiated $k$ times with the initial data $[\bE_{0, k}, \bB_{0, k}]$ and the perfect conductor BC, whereas $\partial_t^m [\bE, \bB] \in C ([0,  \infty)) L_2 (\Omega)$ is a  weak solution (see \cite{DL_76}).
     In addition, the identities 
     $$\nabla_x \cdot \partial_t^k \bE = 4 \pi \partial_t^k  \rho, \quad \nabla_x \cdot \partial_t^k \bB = 0, \, k \le m$$ hold 
    due to the compatibility conditions \eqref{eq33.10}--\eqref{eq33.10.1} and the continuity equations
     $$
        \partial_t (\partial_t^k \rho) + \nabla_x \cdot \partial_t^k \bm{j} = 0, \, k \le m.
     $$

$(iv)$ If $[f_i, \bE_i, \bB_i], i = 1, 2,$ are   strong solutions to the RVML system on $[0, \infty)$ satisfying the bound \eqref{eq14.0.13}, then, we have $f_1 = f_2$ on $\Sigma^{\infty}$ and $\bE_{1} = \bE_{2}$, $\bB_{1} = \bB_{2}$ on $(0, \infty) \times \Omega$. 
\end{theorem}

\begin{remark}
    \label{remark 3.13}
This remark is to clarify the relation between Theorem \ref{theorem 3.1} and the local well-posedness (LWP) result in \cite{VML} (see Theorem 3.10 therein).
We note that 
$$
    \cI (\tau) = \cI_f (\tau), \tau \ge 0,  \quad I_0 = \cE_f (0),
$$
where 
\begin{itemize}
    \item $\cI$ and $\cI_f$ are the total instant functionals in the present work (see \eqref{eq14.0.5}) and in \cite{VML} (see the formulas $(3.31)$--$(3.32)$  therein), respectively,
    \item $I_0$   is the sum of the instant `baseline' and weighted energies at $t=0$  (see \eqref{eq14.0.12}), whereas $\cE_f (0)$ is the same object in \cite{VML} defined in $(3.37)$ therein.
\end{itemize}
Hence, the smallness condition on the initial data \eqref{eq14.0.12} coincides with that in the LWP theorem in \cite{VML} (see $(3.37)$ therein).

Thus, to prove the global existence part in Theorem \ref{theorem 3.1}, it suffices to show that for a strong solution to the RVML system on $[0, \tau]$, $\tau > 0,$ satisfying $(ii)$-$(iii)$, the estimate  \eqref{eq14.0.13} holds for all $t \in [0, \tau]$ with $C$ independent of $\tau$ provided that $\sup_{t \le \tau} \cI (t)$ and $I_0$ are sufficiently small.
\end{remark}


We note that the global estimate \eqref{eq14.0.13} gives
$$
   \sum_{k=0}^{m-4}   \int_0^{\infty} \|\partial_t^k [\bE, \bB] (t, \cdot)\|^2_{ L_2 (\Omega) } \, dt < \infty.
$$
 The next result establishes the pointwise temporal decay of the $t$-derivatives strictly below the $(m-4)$ order.
\begin{theorem}[temporal decay]
    \label{theorem 3.2}
Let $m \ge 24$, invoke the assumptions of Theorem   \ref{theorem 3.1}, and let $f$ be the global strong solution satisfying the properties $(i)$-$(iii)$ therein.
Furthermore, for any  integer $0 < n \le m-5$, we denote by $\cI_n$ and $\cD_n$ the expressions defined by \eqref{eq14.0.5}--\eqref{eq12.0.1} with $m$ replaced with $n$. 
Then,   we have  for all $t > 0$,
\begin{align}
    \label{eq132.3}
\cI_n (t) \lesssim_{n, m, \theta, \Omega, r_1, \ldots, r_4} I_0 (1+t)^{- \frac{m-4}{n} }.
\end{align}
\end{theorem}


\begin{remark}
Our scheme is designed to manage significant decay losses 
 in the momentum variable and enables seamless adaptation to the non-relativistic Vlasov-Maxwell-Landau system. 
However, the relativistic VML system proves more intricate, with the complexity of the relativistic Landau kernel introducing technical challenges, especially in establishing hypoelliptic smoothing near the spatial boundary (cf. \cite{VML}). 
\end{remark}

\begin{remark}
For the sake of convenience, in the sequel, we omit the dependence of constants on the r.h.s. of a priori estimates on the physical constants and the total number of $t$-derivatives $m$.
\end{remark}

\section{Method of the proof and the organization of the paper}
\label{section 4}

In this section, all the physical constants are set to $1$ unless stated otherwise. 

\subsection{Key difficulty}
\label{section 4.1}

It is well known that  controlling  $P f$ and $[\bE, \bB]$
    via the dissipation rate $\cD_{||}$ is crucial (see \cite{VMB_03}, \cite{GS_03}).
       Unfortunately, due to the hyperbolic nature of Maxwell's equations, the approach used to control the macroscopic densities $a^{\pm}$ and the electrostatic field $\bE$ of the Vlasov-Poisson-Landau system in Lemma 4.3 and Corollary 4.3.1 in  \cite{DGO_22}  fails entirely in the present setting. 

    \textit{Control of $a^{+}-a^{-}$.} 
    Due to the presence of the perfect conductor BC \eqref{eq1.5}, the argument developed to control $a^{+}-a^{-}$ in \cite{VMB_03} for the Vlasov-Maxwell-Boltzmann model on $\mathbb{T}^3$  cannot be applied.
   To elaborate, we recall  the macroscopic equations  (see \cite{GS_03})
    \begin{align}
\label{mac1}
  &  \partial_t c = l_{c}+ h_c, \\
  \label{mac2}
  & \partial_{x_i} c + \partial_t b_i = l_i + h_i, \\
    \label{mac3}
  &  2 (1-\delta_{i j}) S_{i j} (b)  = l_{i j} + h_{i j}, \\
      \label{mac4}
  & \partial_{x_i} [a^{\pm} + \rho \, c]  \mp \bE_i = l_{a i}^{\pm} + h_{a i}^{\pm}, \\
      \label{mac5}
  & \partial_t [a^{\pm} +  \rho \, c]  = l_a^{\pm} + h_a^{\pm},
\end{align}
where the $l$- and $h$-terms are certain weighted momentum averages of $(\partial_t + \frac{p}{p_0} \cdot \nabla_x + L_{\pm}) (1-P)$ and nonlinear terms, respectively.
   We first note the important relation 
   \begin{align}
    \label{mac6}
    \nabla_x (a^{+}-a^{-}) - 2 \bE  = (l\text{- and}\,  h\text{- terms}),
   \end{align}
   which follows from \eqref{mac4}.
   Furthermore, let $\phi$ be the solution to the Neumann problem
   \begin{align}
    \label{eq1.2.102}
        -\Delta_x \phi = a^{+}-a^{-}, \quad \frac{\partial \phi}{\partial n_x} = 0,
   \end{align}
   satisfying $\int_{\Omega} \phi \, dx = 0$.
   Then, formally taking the $L_2$ inner product of the l.h.s. of \eqref{mac6} with $\nabla_x \phi$ yields the terms
      \begin{align}
          \label{eq1.2.103}
         \|a^{+}-a^{-}\|^2_{L_2 (\Omega)} + 2 \|\nabla_x \phi\|^2_{L_2 (\Omega)},
   \end{align}
   thanks to the Gauss law $\nabla_x \cdot \bE = a^{+}-a^{-}$ and the Neumann BC.  However, due to the perfect conductor BC, 
 $\bE \times n_x = 0, \bB \cdot n_x = 0,$ such a procedure also produces the surface integral 
   \begin{align}
\label{eq1.2.103}
        \int_{\partial \Omega} (\bE \cdot n_x) \phi \, dS_x, 
   \end{align}
   which 
creates a major obstacle for the control of  $a^{+}-a^{-}$ at the top order, as explained in the next paragraph.

\textit{Control of the electromagnetic field.} It can be seen from Maxwell's equations
\begin{equation}
\begin{aligned}
     \label{eq1.2.51}
&
\begin{cases}
	  \nabla_x \times  \bE  = - \partial_t \bB, \\
	 \nabla_x \cdot  \bE = a^{+}-a^{-},\\
  (\bE \times n_x)|_{\partial \Omega} = 0,
 \end{cases} \\
& \begin{cases}
 	\nabla_x \times   \bB =  \partial_t \bE  + \int \frac{p}{p_0} (\sqrt{J}, \sqrt{J}) \cdot (1-P) f  \, dp,\\
  \nabla_x  \cdot \bB  = 0, \\
 (\bB \cdot n_x)|_{\partial \Omega} = 0,
\end{cases}
\end{aligned}
\end{equation}
and \eqref{mac6} that any direct attempt to estimate $[a^{+}-a^{-}, \bE, \bB]$  from solving both the Landau and Maxwell's equations leads to the derivative loss and non-closure.
The main mathematical achievement of this paper is to circumvent this key challenge and establish global well-posedness. 

\subsection{Top-order control of $a^{+}-a^{-}$ and $\bE$}
\label{section 4.10} 
To manage the derivative loss at the top order, we make use of the continuity equations
\begin{align}
    \label{eqE.2.50}
    \partial_t a^{\pm} + \nabla_x \cdot \bm{j}^{\pm} = 0, \quad  \bm{j}^{\pm} := \int_{\bR^3} \frac{p}{p_0}  f^{\pm}  \sqrt{J}\, dp,
\end{align}
and the  macro-micro decomposition to observe the bound 
 \begin{align}
    \label{eqE.2.15.1}
    &\|\partial_t^{m+1} a^\pm\|_{ W^{-1}_2 (\Omega) } \lesssim  \|\partial_{t}^{m} \bm{j}^{\pm}\|_{ L_2 (\Omega)}  \\
    &\lesssim \|\partial_t^m b\|_{ L_2 (\Omega) } + \|(1-P) \partial_t^m f\|_{  L_2 (\Omega \times \bR^3)} \lesssim \sqrt{\cD}.\notag
\end{align}
Despite the weak control of the $W^{-1}_2$ norm, the maximum derivative count in such a fundamental bound enables us to perform (repeated) integration by parts in time, which creates a sufficient gap in the number of temporal derivatives, compensating for the derivative loss and allowing us to close the $[a^{\pm}, \bE]$ estimate with lower derivative counts. 
This idea is also helpful in circumventing the derivative loss in the lower-order estimate of   $[a^{+}-a^{-}, \bE, \bB]$ via $\cD_{||}$, where it is used in combination with a duality argument and a weighted trace estimate (see Section \ref{section 4.11}).  

In the rest of this section, we demonstrate how the continuity equation \eqref{eqE.2.50} combined with a novel $W^1_3 (\Omega)$ velocity averaging estimate enables us to establish the top-order  energy estimate of $[a^{+}-a^{-}, \bE]$. For the remainder of Section \ref{section 4},  we assume that $0 \le s < t$ are arbitrary, and $f$ is a strong solution to the RVML system on $[s, t]$ satisfying properties $(ii)$--$(iii)$ in Theorem \ref{theorem 3.1} with a sufficiently small amplitude. Due to the aforementioned derivative loss,  the most difficult  terms in the energy estimate are the `cubic Lorentz' integrals, such as
\begin{align*}
    \int_s^t \int_{\Omega} (\partial^m_t \bE) \cdot \bm{j}^{\pm} \, (\partial_t^m a^{\pm}) \, dx d\tau,
\end{align*}
where $\bm{j}^{\pm}$ are defined in \eqref{eqE.2.50}. 
To start, a formal integration by parts   in $t$ gives
$$
    -\int_s^{t} \int_{\Omega } (\partial_{t}^{m-1} \bE) \cdot  \bm{j}^{\pm}  \,  (\partial_{t}^{m+1} a^{\pm}) \, dx d\tau + \text{good terms}.
$$
By using the continuity equation \eqref{eqE.2.50}, we may replace  $\partial_t^{m+1} a^{\pm}$ with $- \nabla_x \cdot \partial_t^{m} \bm{j}^{\pm}$.
Furthermore, due to the SRBC,
\begin{equation}
        \label{eqE.2.16}
       \bm{j}^{\pm} \cdot n_x = 0 \, \, \text{on} \, \, \partial \Omega,
\end{equation}
which,  combined with the fact that $(\partial_t^{m-1} \bE) \times n_x = 0$, gives
\begin{align}
        \label{eq12.2.19}
    (\partial_t^{m - 1} \bE) \cdot    \bm{j}^{\pm} = 0\, \, \text{on} \, \, \partial \Omega.
\end{align}
By integrating by parts in $x$, we reduce the problem to estimating  the integral
\begin{align*}
 & \mathfrak{I}: = \int_s^{t} \int_{\Omega } (\partial_{t}^{m-1} \bE_i) \,  (\partial_{x_l}  \bm{j}^{\pm}_i)  \,   (\partial_t^{m} \bm{j}^{\pm}_l)  \, dx d\tau.
\end{align*}
Applying   H\"older's inequality,  we obtain
\begin{align}
    \label{eq2.2.100}
  &  |\mathfrak{I}|  \le \|\partial_t^{m-1} \bE\|_{ L_{\infty} ((s, t)) L_6 (\Omega) }
  \|D_x \bm{j}^{\pm}\|_{ L_2 ((s, t)) L_3 (\Omega) }  \|\partial_{t}^{m} \bm{j}^{\pm}\|_{ L_2 ((s, t) \times \Omega)}.
\end{align}
Using  the  $W^{1}_2$ div-curl estimate (cf. \eqref{eq1.2.41}) and the Sobolev embedding $W^1_2  \subset L_6$, we conclude that the first factor on the r.h.s. is bounded by
$$\|\cI_{||}\|^{1/2}_{L_{\infty} ((s, t))}$$
(see \eqref{eq14.0.10}).
Furthermore, by \eqref{eqE.2.15.1},  the third factor in \eqref{eq2.2.100} is dominated by $\big(\int_s^t \cD \,  d\tau\big)^{1/2}$ (see \eqref{eq12.0.1}).
Then, for the closure, we need
\vskip0.1in
\noindent\fbox{\begin{minipage}{4.9in}
\paragraph{
\textbf{Gradient estimate of a velocity average} (see Proposition \ref{proposition 4.0}) }
\begin{align}
    \label{eq15.0}
       \int_s^t \|D_{x}  \bm{j}^{\pm}\|^2_{  L_3 (\Omega) } \,d\tau \lesssim
    \int_s^t \cD \,  d\tau.
\end{align}
\end{minipage}}
\vskip.1in

Thus, we reduced the task of estimating the top-order cubic Lorentz terms to establishing higher spatial regularity of lower-order terms.

\textit{Gradient estimate of a velocity average.} We list the key highlights of the proof of the crucial estimate \eqref{eq15.0}.

\begin{itemize}
    \item By using the mirror extension argument in \cite{VML}, near the boundary, one can locally extend the solution $f^{\pm}$ across the boundary and, after a change of variables, obtain a non-relativistic kinetic Fokker-Planck equation on the whole space
\begin{align}
                \label{eq1.2.11}
   &    v \cdot \nabla_y  \mathfrak{F} (t, y, v) - \nabla_v \cdot (\mathfrak{A} (t, y, v)  \nabla_v\mathfrak{F} (t, y, v))\\
   & =   \text{good terms} + \nabla_v \cdot (\mathfrak{B} (y, v)   \mathfrak{F} (t, y, v)). \notag
\end{align}

    \item The resulting equation has a ``geometric drift term" $\nabla_v \cdot (\mathfrak{B}   \mathfrak{F})$ with a discontinuous vector-valued coefficient $\mathfrak{B}$, which is a sum of terms related to the curvature of $\partial \Omega$.  \\

    \item Our key observation is that the discontinuity stems from the oddness of certain terms in $\mathfrak{B}$ in the normal direction $y_3$.\\

    \item Despite the discontinuity, it is well known that any odd function is in $W^{(1/p)-}_{p, \text{loc}}$ (see Lemma \ref{lemma F.2}), which provides a slight gain in spatial regularity of $\mathfrak{F}$ via a differentiation argument. It is quite striking that such a modest increase in regularity is crucial in achieving the desired $W^1_3$ estimate of the current densities $\bm{j}^{\pm}$. \\

    \item Applying formally $\nabla_x^{(1/3)-}$ to  \eqref{eq1.2.11} and using the steady $S_3^N (\bR^6)$ estimate \eqref{eq14.C.20}, we gain $\nabla_x^{2/3}$ via  hypoelliptic smoothing in addition to $\nabla_x^{(1/3)-}$, which  is due to the $W^{(1/3)-}_{3, \text{loc}}$ regularity of the geometric term. Thus,  we establish  $f^{\pm} \in L_2^t L_3 (\bR^3_p) W^{1-}_{3} (\Omega)$, which  is still short of the bound in \eqref{eq15.0}.\\

    \item To overcome the small gap in regularity, we apply a variant of the DiPerna-Lions-Meyer  $L_p$ velocity averaging lemma \cite{DLM_91}, which
    allows us to gain additional $\nabla_x^{(1/9)-}$ smoothness for the velocity average and deduce \eqref{eq15.0}.
\end{itemize}

\subsection{Lower-order control of $a^{+}-a^{-}$ and $\bE$} 
\label{section 4.11}
We delineate the key components of a multi-step argument,  where the derivative loss \eqref{eq1.2.51}  is managed via the continuity equation \eqref{eqE.2.50}. For the sake of simplicity,  terms involving small constants and quantities such as $[a^{+}+a^{-}, b, c, \cD_{||}]$ will be displayed as ``good terms". For the details of the estimate of $[a^{+}+a^{-}, b, c]$  in terms of $\cD_{||}$, see Section \ref{section 4.2}.
\begin{itemize}

    \item Step 1: \textit{creating a derivative gap between $a^{\pm}$ and $\bE$}. First, we  integrate by parts  twice in $t$ and use the continuity equation \eqref{eqE.2.50} to create a `\textit{derivative gap}' between $a^{\pm}$ and $\bE$ to decouple $a^{\pm}$ from $\bE$:
    \begin{align}
        \label{eq1.2.40}
                \int_s^{t} \|\partial_t^3 a^{\pm}\|^2_{L_2 (\Omega) } \, d\tau
        \lesssim \varepsilon_a \int_s^{t} \|\partial_t \bE\|^2_{L_2 (\Omega)} \, d\tau + \text{good terms},
        \end{align}
        where $\varepsilon_a \in (0, 1)$. \\

\item Step 2: \textit{electric field estimate}.  Applying the div-curl inequality  gives
    \begin{align}
        \label{eq1.2.41}
     \int_s^{t} \|\partial_t \bE\|^2_{L_2 (\Omega)} \, d\tau  \lesssim   \int_s^{t}  \|\partial_t^2 \bB\|^2_{L_2 (\Omega)}  \, d\tau
    +  \int_s^{t} \|\partial_t a^{\pm}\|^2_{L_2 (\Omega)} \, d\tau.
        \end{align}
        For additional details on the div-curl estimates, see Section \ref{section 4.4} and \eqref{eq4.11}--\eqref{eq4.12}. \\

\item Step 3: \textit{magnetic field estimate}. We derive an estimate of $\partial_t^2 \bB$ from the Landau equation via a duality argument.
In particular, due to the vanishing flux assumption \eqref{eqE.2.52}, there exists  a unique solution $w \in W^{1}_2 (\Omega)$ to
\begin{align*}
    \nabla_x \times w = \bB, \quad \nabla_x \cdot w = 0, \quad (w \times n_x)|_{\partial \Omega} = 0,
\end{align*}
which satisfies 
\begin{align}
    \label{eq1.2.50}
    \|w\|_{W^1_2 (\Omega) } \lesssim \|\bB\|_{L_2 (\Omega)}.
\end{align}
Then, integrating by parts in $x$ and using the Amper\'e-Maxwell law (see \eqref{eq36.16}), we get
\begin{align}
  & \int_s^t  \int_{\Omega} \|\partial_t^2 \bB\|^2_{L_2 (\Omega)} \, dx d\tau =   \int_s^t  \int_{\Omega} \partial_t^2 \bB \cdot \nabla_x \times \partial_t^2 w \, dx d\tau \notag\\
\label{eq1.2.100}
  & = \int_s^t \int_{\Omega} (\partial_t^3 \bE) \cdot (\partial_t^2 w) \, dx d\tau+  \int_s^t \int_{\Omega}  (\partial_t^2 \bm{j}) \cdot (\partial_t^2 w) \, dx d\tau.
\end{align}
Since the $L_2 (\Omega)$ norm of $\partial_t^2 \bm{j}$ is estimated via $\cD_{||}$ (see \eqref{eqE.1.18}--\eqref{eqE.1.19}), we only need to handle the first term in \eqref{eq1.2.100}. To this end, we rewrite the Landau equation for $f^{+}$ as
\begin{align*}
   \frac{p}{p_0} \cdot  \bE \sqrt{J} = \partial_t f^{+}  +   \frac{p}{p_0} \cdot  \nabla_x f^{+} + L_{+} f + \text{nonlinear terms},
\end{align*}
 differentiate it three times in $t$,     test the resulting identity with $\partial_t^2 w \sqrt{J}$, and use the estimate of the test function \eqref{eq1.2.50}. This approach, combined with the identity \eqref{eq1.2.100}, yields
\begin{align}
        \label{eq1.2.42}
  & \int_s^t \|\partial_t^2 \bB\|^2_{L_2 (\Omega)} \, d\tau \lesssim  \int_s^t \|\partial_t^3 a^{+}\|^2_{  L_2 (\Omega) } \, d\tau  \\
  & +  \int_s^t  \bigg\| \frac{|p \cdot n_x|}{p_0}  J^{1/4} \partial^{3}_t f^{+}\bigg\|^2_{  L_2 (\gamma_{+}) }   \, d\tau +  \text{good terms}.\notag
\end{align}


    \item  Step 4: \textit{weighted trace estimate}. 
   To estimate the integral over $\gamma_{+}$ on the r.h.s. of \eqref{eq1.2.42},  we test the Landau equation with a multiplier $\psi$ satisfying 
    $$
        \psi (t, x, p) = \frac{(p \cdot n_x) 1_{ p \cdot n_x > 0 }}{p_0} \sqrt{J} f^{+}   \,  \text{on} \, \partial \Omega
    $$
   and integrate by parts in $t$ in the term containing  $\frac{p}{p_0} \cdot \bE \sqrt{J}$ (cf. \eqref{eq1.2.40}). We obtain 
    \begin{align}
                \label{eq1.2.43}
      & \int_s^t  \bigg\| \frac{|p \cdot n_x|}{p_0}  J^{1/4} \partial^{3}_t f^{+}\bigg\|^2_{  L_2 (\gamma_{+}) }   \, d\tau  \\
        & \lesssim \int_s^t  \|\partial_t^{3} a^{\pm}\|^2_{L_2 (\Omega)} \, d\tau + \varepsilon_a \int_s^{t} \|\partial_t \bE\|^2_{L_2 (\Omega)} \, d\tau
         + \text{good terms}. \notag
      \end{align}
By following  this reasoning, we derive the desired estimates of all the derivative terms $\partial_t^k a^{\pm}, 1 \le k \le m-2$,
    $\partial_t^k \bE, 1 \le k \le m-4$,  and all the magnetic field  terms $\partial_t^k \bB, 0 \le k \le m-3$.\\

\item Step 5: \textit{Estimate of the non-derivative $a^{+}-a^{-}$ and $\bE$ terms.}
Unfortunately, the argument used in Steps 1--4  does not apply here because the aforementioned derivative gap (cf. \eqref{eq1.2.40}) cannot be created for the non-derivative terms $a^{\pm}$ and $\bE$. 
Instead, we refine the duality argument outlined in Section \ref{section 4.1} (see \eqref{mac6}--\eqref{eq1.2.103}). The key technical ingredient is a  Helmholtz-type decomposition of the electric field, which enables us to avoid the problematic surface integral \eqref{eq1.2.103} at the cost of a derivative loss. 
This approach yields 
\begin{align}
  \label{eq1.2.45}
 &   \int_s^t \big(\|a^{+} - a^{-}\|^2_{ L_2 (\Omega) } + \|\bE\|^2_{ L_2 (\Omega) }\big) \, d\tau \\
 &   \lesssim  \int_s^t \|\partial_t \bB\|^2_{ L_2 (\Omega) } \, d\tau
    + \int_s^t \|a^{+} + a^{-}\|^2_{ L_2 (\Omega) } \, d\tau + \text{good terms}. \notag
\end{align}
\end{itemize}

   Gathering \eqref{eq1.2.40}--\eqref{eq1.2.45}, we obtain 
    \vskip.1in

    \noindent\fbox{\begin{minipage}{4.9in}
\paragraph{
\textbf{Lower-order derivative control  of $[a^{\pm}, \bE]$} (see Proposition \ref{proposition E.4})
}
\begin{align}
    \label{eq1.2.46}
   &   \int_s^t   \|\text{lower-order} \,  [a^{\pm}, \bE, \bB]\|^2_{ L_2 (\Omega) }\, d\tau:=\sum_{k=0}^{m-2} \int_s^t \|\partial_t^k a^{\pm}\|^2_{ L_2 (\Omega) } \, d\tau\\
     &    + \sum_{k=0}^{m-3} \int_s^t \| \partial_t^k  \bB\|^2_{ L_2 (\Omega) } \, d\tau   + \sum_{k=0}^{m-4} \int_s^t \| \partial_t^k  \bE\|^2_{ L_2 (\Omega) } \, d\tau
         \lesssim \text{good terms}. \notag
\end{align}
\end{minipage}}

\subsection{Positivity of $L$ ($a, b, c$ estimates)}
                    \label{section 4.2}

It is well known that the coercivity estimate of the linearized collision operator $L$  plays a key role in establishing the asymptotic stability near a  Maxwellian.  


\vskip.1in

\noindent\fbox{\begin{minipage}{5.05in}
\paragraph{\textbf{Positivity estimate of $L$} (see  Proposition \ref{proposition 12.1}) }
There exists a small constant $\delta_0$ independent of $s$ and $t$ such that for any $\delta \in (0, \delta_0)$, we have
\begin{align}
    \label{eq122}
 & \sum_{k=0}^m \int_s^{t} \int_{  \Omega }  \langle L (\partial_t^k f), (\partial_t^k f)\rangle \,  dx   d\tau   + \delta \, (\text{good terms})\\
 & \ge \delta \int_s^t  \bigg(\sum_{k=0}^m \|\partial_t^k [b, c]\|^2_{ L_2 (\Omega)  } + \|\text{lower-order} \,  [a^{\pm}, \bE, \bB]\|^2_{ L_2 (\Omega) }\bigg) \, d\tau. \notag 
\end{align}
\end{minipage}}

\vskip.1in

By using the positivity estimate of $L$ \eqref{eq122} combined with the argument of Section \ref{section 4.10} and the $W^1_3$ velocity averaging estimate \eqref{eq15.0}, we establish the crucial energy estimate.

\vskip.1in

\noindent\fbox{\begin{minipage}{5.05in}
\paragraph{
\textbf{Top-order energy estimate} (see Proposition \ref{proposition 12.2})
}
    \begin{align}
            \label{eq12.2.1.0}
& \sum_{k=0}^m \bigg(\|\partial_t^k f (t, \cdot)\|^2_{L_2 (\Omega \times \bR^3)} - \|\partial_t^k f (s, \cdot)\|^2_{L_2 (\Omega \times \bR^3)} \\
& + \|\partial_t^k [\bE, \bB] (t, \cdot)\|^2_{ L_2 (\Omega)  } - \|\partial_t^k [\bE, \bB] (s, \cdot)\|^2_{ L_2 (\Omega)  }  \notag \\
&+  \int_s^t  \| (1-P) \partial_t^k  f\|^2_{ L_2 (\Omega) W^1_2 (\bR^3)}\, d\tau\bigg) \notag\\
& +   \int_s^t  \bigg(\sum_{k=0}^m \|\partial_t^k [b, c]\|^2_{ L_2 (\Omega)  } + \|\text{lower-order} \,  [a^{\pm}, \bE, \bB]\|^2_{ L_2 (\Omega)  }\bigg) \, d\tau  \notag\\
&\lesssim  
\text{good terms}.  \notag
\end{align}
\end{minipage}}
\vskip.1in

 The positivity estimate \eqref{eq122} is proved by estimating $P f$ and $[\bE, \bB]$ via $\cD_{||}$. 
 In addition to our major effort to control the key coupling of $[a^{+}-a^{-}, \bE]$, there are several other novelties in estimating of the rest of $P f$ and the electromagnetic field $[\bE, \bB]$  (see also Section \ref{section 4.4}).

\textit{Estimate of $b$, $c$, and $a^{+}+a^{-}$.}
By using a duality argument, we obtain the following bounds.

\vskip.1in

\noindent\fbox{\begin{minipage}{4.9in}
\paragraph{
\textbf{Preliminary estimates of $b$ and $c$}  (see Lemmas \ref{lemma 3.1} and \ref{lemma 7.1})
}
   There exist
   $\varepsilon_b \in (0, 1)$ independent of $s$ and $t$ such that 
\begin{align}
\label{eq3.1.0.1}
 &   \sum_{k = 0}^m \int_s^{t} \|\partial_{t}^k  b\|^2_{ L_2 (\Omega) } \, d\tau   \lesssim  \varepsilon_b \sum_{k = 0}^m \int_s^{t} \|\partial_{t}^k   c\|^2_{ L_2 (\Omega) } \, d\tau  \\
&  +   \varepsilon_b \sum_{k = 0}^m \int_s^{t} \|\partial_{t}^k   (a^{+} + a^{-})\|^2_{ L_2 (\Omega) } \, d\tau + \text{good terms},\notag \\
\label{eq3.1.0.10}
&   \sum_{k = 0}^m \int_s^{t} \|\partial_{t}^k  c\|^2_{ L_2 (\Omega) } \, d\tau  \lesssim  \sum_{k = 0}^m \int_s^{t} \|\partial_{t}^k  b\|^2_{ L_2 (\Omega) } \, d\tau +  \text{good terms}.
 \end{align}
\end{minipage}}

\vskip.1in

 Although the coupling between $[b, c]$  and $a^{+}+a^{-}$ is expected from the macroscopic equations \eqref{mac1}--\eqref{mac5}, it could lead to  non-closure if $a^{+}+a^{-}$ is coupled with $\bE$, for which the derivative loss is anticipated (see \eqref{eq1.2.51} and Section \ref{section 4.1}). 
Fortunately, 
by adding the macroscopic equations for $a^{\pm}$ (see \eqref{mac4}), one can cancel the problematic linear electric field term $\displaystyle \frac{p}{p_0}  \cdot \mathbf{E} \sqrt{J}$, thereby decoupling $a^{+}+a^{-}$  from $\bE$.
Interestingly,  this cancellation of the linear electric field term, given by 
$\displaystyle\frac{e_{\pm}  }{k_b T} \frac{p}{p_0^{\pm}}  \cdot \mathbf{E} \sqrt{J^{\pm}}$, holds for general physical constants, provided that the global neutrality condition \eqref{eq0} is imposed.
To justify the cancellation rigorously, we test the Landau equations with a special test function $\psi$ and obtain the following:

\vskip.1in

\noindent\fbox{\begin{minipage}{4.9in}
\paragraph{
\textbf{Estimate of $a^{+}+a^{-}$} (see Lemma \ref{lemma 3.2})
}
\begin{align}
	\label{eq3.2.0.1}
&    \sum_{k = 0}^m \int_s^{t}  \|\partial_{t}^k (a^{+} + a^{-})\|^2_{ L_2 (\Omega) } \, d\tau\\
& \lesssim  \sum_{k = 0}^m \int_s^{t} \|\partial_{t}^k [b, c]\|_{ L_2 (\Omega) }^2 \, d\tau + \text{good terms}. \notag
\end{align}
\end{minipage}}

\vskip.1in


Due to the presence of a small constant in \eqref{eq3.1.0.1}, we can decouple 
$b$, $c$, and $a^{+}+a^{-}$ from each other.

\vskip.1in

\noindent\fbox{\begin{minipage}{4.9in}
\paragraph{
\textbf{Final estimate of $b$ and $c$ and $a^{+}+a^{-}$} (cf. Lemmas \ref{lemma 3.3} and \ref{lemma 7.1})
}
\begin{align}
\label{eq3.3.0.1}
    \sum_{k = 0}^m \int_s^{t} \|\partial_{t}^k  [b, c, a^{+}+a^{-}]\|^2_{ L_2 (\Omega) } \, d\tau
    \lesssim \int_s^t \cD_{||} \, d\tau +  \text{good terms}.
\end{align}
\end{minipage}}

\vskip.1in




\vskip.1in

\subsection{Div-curl and $S_p$ estimates. Descent argument}
\label{section 4.4}

\textit{Div-curl estimates.} The crucial control of the $L_2^t L_r^x$ and $L_{\infty}^t L_r^x$ norms of the $t$-derivatives of $[\bE, \bB]$ is established by rewriting Maxwell's equations as two div-curl systems \eqref{eq1.2.51}, differentiating   these  equations with respect to $t$, and using a variant of the  $W^1_r (\Omega)$  div-curl estimate.

We emphasize that in the  $W^1_r$ div-curl estimate in a bounded domain, there is an additional $0$-order term, namely the $L_r$ norm of a solution to a div-curl system (see \cite{AS_13}). This particular term serves as a fundamental obstacle in establishing the temporal decay estimate for the electromagnetic field. In a general domain, the presence of such a norm is inevitable due to the existence of nontrivial divergence-free and curl-free vector fields  $\bm{u}$ satisfying $\bm{u} \times n_x = 0$ or $\bm{u} \cdot n_x = 0$  (see, for example,  Section 9  in \cite{CDG_02}) unless specific geometric conditions are imposed on both the domain and the initial data. 
To remove this challenging $0$-order term, we enforce Assumptions \ref{assumption E.1} and \ref{assumption E.2}.
We note the last condition is preserved in time, that is, for any $t > 0$, we have
\begin{align}
            \label{eq4.10}
        \int_{\Sigma_j} \bB (t, x) \cdot n_x \, dS_x = 0, \, \, \forall j = 1, \ldots, L,
\end{align}
which is proved by integrating Faraday's law \eqref{eq36.16.1} over $\Sigma_j$ and using the Stokes theorem combined with the boundary condition $\bE \times n_x = 0$.

Thus, under the aforementioned geometric conditions, for any $r \in (1, \infty)$, thanks to the results of \cite{AS_13} (see Corollaries 3.2 and 3.4 therein), we have
\begin{align}
    \label{eq4.11}
  &  \|\bE\|_{W^1_r (\Omega) } \lesssim_{r, \Omega}   \|\partial_t \bB\|_{ L_r (\Omega) } + \|\rho\|_{ L_r (\Omega) }, \\
      \label{eq4.12}
  & \|\bB\|_{W^1_r (\Omega) }  \lesssim_{r, \Omega}   \|\partial_t \bE\|_{ L_r (\Omega) } +  \|\bm{j}\|_{ L_r (\Omega) }.
\end{align}

\textit{$S_p$ estimate and lower-order energy.} 

For the closure of our argument, we need a higher regularity control of the lower-order $t$-derivatives of the velocity gradient, which is derived via the $S_p$ estimate in a domain. For the relativistic linear Landau equation, such a bound was established in \cite{VML} by using the mirror-extension method from \cite{GHJO_20}. 
Furthermore, the aforementioned $S_p$ estimate requires control of a weighted $L_2$-norm due to the presence of the spatial boundary and the relativistic transport operator $\frac{p}{p_0} \cdot \nabla_x$. See the details in \eqref{eq14.C.6.6} and in \cite{VML}. To handle such a momentum loss, we establish the following weighted estimate of the lower-order energy: 
    \begin{align}
        \label{eq14.1.1.0}
     &  \sum_{k = 0}^{m-4} \bigg(\|\partial_t^k f \|^2_{  L_{\infty} ((s, t))  L_{2, \theta/2^k} (\Omega \times \bR^3) }
     + \int_s^t \|\partial_t^k f\|^2_{ L_2 (\Omega)
 W^1_{2, \theta/2^k} (\bR^3) } \, d\tau \bigg) \\
   &  \lesssim  \sum_{k=0}^{m-4} \bigg(\|\partial_t^k f (s, \cdot)\|^2_{   L_{2, \theta/2^k} (\Omega \times \bR^3) } +  \int_s^t  \|\partial_t^k f \|^2_{   L_{2} (\Omega \times \bR^3) }    \, d\tau  \notag \\ 
   & +  \int_s^t \|\partial_t^k \bE\|^2_{ L_2 (\Omega) } \, d\tau\bigg)  +  \text{good terms}. \notag
           \end{align}

\textit{Descent argument.} Due to the $t$-derivative loss in the elliptic estimates \eqref{eq4.11}--\eqref{eq4.12},  it becomes essential to incorporate the aforementioned $S_p$ and div-curl bounds into a descent argument
(see the details in Proposition 6.3 in \cite{VML}). To illustrate our scheme,  we consider a  simplified version of the Landau equations \eqref{eq36.13}--\eqref{eq36.14}
\begin{align}
        \label{eq1.3.0}
         (\frac{p}{p_0} \cdot \nabla_x) f^{\pm} - \Delta_p f^{\pm} = - \partial_t f^{\pm} \pm   \frac{p}{p_0}  \cdot \mathbf{E} \sqrt{J},
\end{align}
where $\bE$ satisfies  Maxwell's  equations  \eqref{eq1.2.51}.
Here are the main highlights of the proof.
\begin{itemize}
    \item Given the $L_2^{t, x}$ and $L^{t, x, p}_2$-control of  the $t$-derivatives of $[f^{\pm}, \bE, \bB]$ up to the order $m-4$  as in the lower-order energy estimate \eqref{eq14.1.1.0}, we apply the steady $S_2$ estimate in \eqref{eq14.C.6.6} to Eq. \eqref{eq1.3.0} to gain the $L_2^t S_2$ control of $\partial_t^k f, k \le m-5$.\\
    \item We descend to the $(m-6)$-th level 
    and use the $W^1_2$ div-curl estimates \eqref{eq4.11}--\eqref{eq4.12} and the Sobolev embedding $W^1_2 \subset L_6$ to control the $L_2^t L_6^{x}$ norm of $\partial_t^k [\bE, \bB],  k \le m-5$.
    We then apply the steady $S_{r_2}$ estimate in \eqref{eq14.C.6.6}  to  Eq. \eqref{eq1.3.0}, where the constant $r_2$ (see \eqref{eq14.0.3}) is determined by the embedding $S_2 \subset L_{r_2}$ (see \eqref{eq14.C.6.7}).\\
    \item This process is repeated until we achieve    the $L_2^t L_{\infty}$ control of $f^{\pm}$ and $\bE, \bB$.
\end{itemize}
By using the above reasoning, we establish estimates of certain $L_{2}^t X$ and $L_{\infty}^t X$-norms of the lower-order $t$-derivatives of the solution provided that the weight parameter $\theta$ is sufficiently large.

\vskip.1in

\noindent\fbox{\begin{minipage}{4.9in}
\paragraph{
\textbf{Higher regularity norms in the dissipation functional} (see \eqref{eq17.9}) }
  \begin{align}
     \label{eq14.2.20}
   &    \int_s^t  |\text{$S_{r_i}$-norms in $\cD$ up to the order $m-4$}|^2  \, d\tau \\
      & + \int_s^t |\text{$W^1_{r_i}$-norms of $[\bE, \bB]$ in $\cD$ up to the order $m-4$}|^2 \, d\tau \notag  \\
  &   \lesssim  \text{(lower-order energy in \eqref{eq14.1.1.0})}  \notag  \\
  & +\sum_{k=0}^{m-4} \int_s^t \|\partial_t^k [\bE, \bB]\|^2_{ L_2 (\Omega) } \, d\tau+ \text{good terms}.  \notag 
    \end{align}
\end{minipage}}
\noindent\fbox{\begin{minipage}{4.9in}
\paragraph{
\textbf{Higher regularity norms in the instant functional}  (see \eqref{eq17.8}) }
  \begin{align}
    \label{eq14.2.1}
   &   \|\text{$S_{r_i}$-norms in $\cI$ up to the order $m-4$}\|^2_{L_{\infty} ((s, t)) } \\
   &+ \|\text{$W^1_{r_i}$-norms of $[\bE, \bB]$ in $\cI$ up to the order $m-4$}\|^2_{ L_{\infty} ((s, t)) } \notag \\
   &  \lesssim
     \text{(lower-order energy in \eqref{eq14.1.1.0})}  \notag \\
    & +\sum_{k=0}^{m-4} \|\partial_t^k [\bE, \bB]\|^2_{ L_{\infty} ((s, t)) L_2 (\Omega) }  + \text{good terms}, \notag \\
    \label{eq14.2.21}
    & \sum_{k=0}^{m-1} \|\partial_t^k [\bE, \bB] (\tau, \cdot)\|^2_{ W^1_2 (\Omega) } \lesssim \cI_{||} (\tau).
    \end{align}
\end{minipage}}

We note that the \eqref{eq14.2.21} is an immediate consequence of the div-curl estimates \eqref{eq4.11}--\eqref{eq4.12}.





 \subsection{Organization of the paper}
\begin{itemize}

 \item  In  Sections \ref{section 7}--\ref{section 9}, we establish the $L_2$-estimates of $a^{\pm}$, $b$, $c$, and $\bE, \bB$.

\item The gradient estimate of  velocity averages $\bm{j}^{\pm}$ (see \eqref{eq15.0}) is proved in Section \ref{section 10}.

 \item The proof of the positivity estimate of $L$ is given in Section \ref{section 15}.

 \item In Section \ref{section 12}, we verify   the top-order energy estimate.

 \item Finally, in Section \ref{section 14}, we prove the main results,  Theorems \ref{theorem 3.1} and  \ref{theorem 3.2}.

 \item We collect auxiliary results in Appendices \ref{appendix B}--\ref{appendix I}.
\end{itemize}

\section{Estimate of $b$}
    \label{section 7}
In this section, we rigorously state and justify the estimates \eqref{eq3.1.0.1} and \eqref{eq3.2.0.1}.

  \begin{assumption}
    \label{assumption 3.5}
        Invoke the assumptions of Theorem \ref{theorem 3.1}. Let $[f, \bE, \bB]$ be a strong solution  to the RVML system on $[s, t]$  satisfying  the assertions $(ii)$--$(iii)$  in Theorem \ref{theorem 3.1} such that
   \begin{align}
    \label{eq3.5.1}
     \cI (\tau) \le \varepsilon, \, \tau \in [s, t],   \quad
    \int_s^{t} \cD (\tau) \, d\tau < \infty,
    \end{align}
    where $\varepsilon \in (0, 1)$ is a small constant which we will choose later.
    \end{assumption}

Let $\eta$ be a function satisfying 
    \begin{align}
     \label{eq3.1.2}
   |\eta (\tau)| \lesssim_{\Omega,  \theta,  r_3, r_4} \cI_{||} (\tau), \, \tau \in [s, t].
   \end{align}
 The precise expression of $\eta$ is not important in our argument, and it might change from line to line.

\begin{lemma}[preliminary estimate of $b$, cf. \eqref{eq3.1.0.1}]
\label{lemma 3.1}
  Under Assumption \ref{assumption 3.5}, there exists a sufficiently small constant $\varepsilon_b =\varepsilon_b (\Omega,  r_3, r_4, \theta) > 0$  such that, one has
\begin{align}
\label{eq3.1.1}
 &   \sum_{k = 0}^m \int_s^{t} \|\partial_t^k  b\|^2_{ L_2 (\Omega) } \, d\tau
    \lesssim_{\Omega,   r_3, r_4, \theta} (\eta (t)  - \eta (s))   \\
    & +  \varepsilon_b \sum_{k = 0}^m \big(\int_s^{t} \|\partial_t^k   c\|^2_{ L_2 (\Omega) } \, d\tau
    + \notag  \int_s^{t} \|\partial_t^k   (\sqrt{M_{+}} a^{+} + \sqrt{M_{-}} a^{-})\|^2_{ L_2 (\Omega) } \, d\tau\big) \\
    &
 + \varepsilon_b^{-1} \big(\int_s^t \cD_{||} \, d\tau   + \varepsilon \int_s^t \cD \, d\tau\big). \notag
 \end{align}
\end{lemma}

\begin{proof}

\textbf{Step 1: an initial estimate of $b$.}
We employ a duality argument to estimate $\partial_t^k b$ from the integral formulation of the Landau equations.   
In particular, since  $\partial_t^k f$ is a finite energy solution to equations \eqref{eq36.13}--\eqref{eq36.19}  differentiated $k$ times with respect to $t$ (see Definition \ref{definition 27.1}), one can rewrite the equation for $\partial_t^k f$ as \eqref{eq1.0} with certain  $\eta \in L_2 ((0, T) \times \Omega) W^{-1}_2 (\bR^3)$ and $A$   (see the proof of Proposition \ref{proposition 4.0}).  
 Then, by the integral formulation \eqref{eq27.2},  for any test function $\psi = (\psi^{+}, \psi^{-})$ satisfying \eqref{eq27.2.1}--\eqref{eq27.2.3}, one has 
\begin{align}
		\label{eq2.3}
 &
 \underbrace{-  \int_{s}^t \int_{\Omega} \int_{\bR^3} \bigg(\frac{p}{p_0^{+}} \cdot   (\nabla_x \psi^{+})  (\partial_t^k f^{+})
 + \frac{p}{p_0^{-}} \cdot   (\nabla_x \psi^{-})  (\partial_t^k f^{-})\bigg)
 \, dz}_{ = I_1}\\
  &    =   \underbrace{ \int_{s}^t \int_{\Omega} \int_{\bR^3} (\partial_t   \psi) \cdot (\partial_t^{k} f) \, dz}_{I_2}\notag \\
 & \underbrace{ -  \int_{\Omega \times \bR^3} [(\psi \cdot \partial_t^{k} f) (t, x, p)- (\psi  \cdot \partial_t^{k} f) (s, x, p)] \, dx dp}_{= I_3 } \notag \\
 & \underbrace{ + \frac{ 1}{k_b T} \int_{s}^t \int_{\Omega} \int_{\bR^3} (\partial_t^k \mathbf{E}_i) 
  \big( e_{+}\frac{p_i}{p_0^{+}}   \sqrt{J^{+}} \psi^{+}
 - e_{-} \frac{p_i}{p_0^{-}} \sqrt{J^{-}} \psi^{-}\big) \, dz }_{= I_4 } \notag\\
 & \underbrace{  - \int_{s}^t \int_{\Omega}    \langle   L \partial_t^k  f, \psi \rangle \,  dx d\tau }_{ = I_5}
 +  \underbrace{    \int_{s}^t \int_{\Omega}     \langle \partial_t^k H, \psi \rangle \, dx d\tau}_{= I_6}, \notag
 \end{align}
 where 
\begin{align}
    \label{eq2.3.1}
  H^{\pm} =   \Gamma_{\pm} (f, g) \mp e_{\pm} (\bE  +  \frac{p}{p_0^{\pm}} \times \bB) \cdot \nabla_p f^{\pm}
\pm  \frac{e_{\pm}  }{2 k_b T} (\frac{p}{p_0^{\pm}}  \cdot \mathbf{E}) f^{\pm}.
\end{align}
We note that the integrals over the kinetic boundaries $\Sigma^T_{\pm}$ are absent in the above identity since $\psi$ satisfies the SRBC (see \eqref{eq27.2.3}). 

We will focus on the top derivative term $\partial_t^m b$ since the remaining ones are handled similarly.



\textit{Test function.} 
Let us consider  the Lam\'e system with the Navier boundary condition
\begin{equation}
			\label{lame}
\left\{\begin{aligned}
 &- \partial_{x_j} S_{i j} (\bm{\phi}) = b_i - \text{corrector},\\
 & (\bm{\phi} \cdot n_x)_{| \partial \Omega} = 0,\\
 & \big((S (\bm{\phi}) n_x) \times n_x\big)_{| \partial \Omega} = 0,
\end{aligned}
\right.
\end{equation}
where $S_{i j} (\bm{\phi})$ is the stress tensor defined in \eqref{eq1.2.23}.
 To ensure   the existence of a solution to \eqref{lame} (see Lemma \ref{lemma A.1}), we set
 the corrector to be the $L_2$-projection of $b$ onto the kernel of the operator $-\nabla_x \cdot S (u)$ acting on the space of vector fields $\bm{u} \in W^2_2 (\Omega)$ satisfying the Navier boundary condition.
 It is easily seen that this kernel is the subspace of infinitesimal rigid motions $\mathcal{R} (\Omega)$ (see \eqref{eq1.1} and Remark \ref{remark A.1}), which can be characterized as follows.
 \begin{itemize}
     \item If $\Omega$ is an irrotational domain, then $\mathcal{R} (\Omega) = \emptyset$.
     \item If $\Omega$ is an axisymmetric domain with a single axis directed along $\omega$ and passing through $x_0$, then  $\mathcal{R} (\Omega) = \text{span} \, \{\omega \times (x-x_0)\}$. 
     \item If $\Omega$ is a ball centered at $x_0$, then   $\mathcal{R} (\Omega) =\text{span} \, \{e_i \times (x-x_0), i = 1, 2, 3\}$.
 \end{itemize}
  The corrector term might obstruct the temporal decay of the perturbations $f^{\pm}$. Fortunately, by   the conservation of the angular momentum 
  \begin{align}
    \label{eq4.3.4}
 \int_{\Omega} \int_{\bR^3} R \cdot p (F^{+} + F^{-})  \, dp  dx + \frac{1}{4 \pi } \int_{\Omega}
 R \cdot (\bE \times \bB) \, dx = \text{const}, \, \, R \in \mathcal{R} (\Omega),
\end{align}
which is verified in Appendix \ref{appendix I}, 
   we have
  \begin{align}
    \label{proj}
           &    \text{corrector} =    \text{$L_2$-projection of} \, b \,  \, \text{onto $\mathcal{R} (\Omega)$}  \\
           & =  \varkappa \,  \sum \big(\int_{\Omega} R_i   \cdot (\bE  \times \bB) \, dx\big) R_i (x), \notag
     \end{align}
    where $\varkappa \in \bR,$ and $\{R_1, \ldots\}$ is  an orthonormal basis  of $\mathcal{R} (\Omega)$.  If $\Omega$ is an irrotational domain, then the corrector term vanishes. 
Hence, by Lemma \ref{lemma A.1}, and \eqref{proj}, there exists a unique strong solution $\bm{\phi} \in W^2_2 (\Omega)$ to \eqref{lame} 
satisfying
\begin{align}
    \label{eq3.1.32}
    \text{$\bm{\phi} \perp \mathcal{R} (\Omega)$ in  $L_2 (\Omega)$.}  
\end{align}
 In addition, by the estimate \eqref{eqA.1.10} in the same lemma, 
\begin{align}
    \label{eq3.1.5}
    \|\bm{\phi} \|_{ W^2_2 (\Omega) } \lesssim_{\Omega} \|b\|_{ L_2 (\Omega)  } + \sum_{i} \bigg|\int_{\Omega} R_i \cdot (\bE \times \bB) \, dx\bigg|.
\end{align}

Next, we set
$$
    \psi (t, x, p) = (B_{i j} (p), 0)  S_{i j} (\partial_t^m  \bm{\phi}) (t, x),
$$
where $B_{i j}$ is a Schwartz function satisfying the following conditions:
\begin{align}
  \label{eq3.1.6}
   &B_{i j} \perp \sqrt{J^{ + }}, p_k \sqrt{J^{ + }}, p_0^{+} \sqrt{J^{ + }},\\
  \label{eq3.1.7}
   &\frac{p_k}{ p_0^{ + }  } B_{i j} \perp \sqrt{J^{ + }}, p_0^{  + } \sqrt{J^{ + }},
   \\
     \label{eq3.1.8}
   &\sum_{i, j, k, l = 1}^3 \langle  \frac{p_k}{p_0^{+}} B_{i j}, \chi_{l+2}^{+} \rangle \big(\partial_{x_k} S_{i j} (\bm{\phi})\big) \xi_l
    = \xi \cdot \nabla_x \cdot S (\bm{\phi}) \,\, \forall \xi \in \bR^3,
\end{align}
where $\perp$ means orthogonality in $L_2 (\bR^3)$, and $\chi_{i+2}, i=1, 2, 3$ are defined in \eqref{eq6.23}. We set
\begin{align}
    \label{eq3.1.40}
 &   B_{i j} (p) :=
                   (p_i p_j    - \delta_{i j})  h (|p|),\\
     \label{eq3.1.41}
  &  h (|p|): = \mu (|p|)  (\sqrt{J^{+}})^{-1} p_0^{+}  (k_1 |p|^2 +  k_2 |p|^4 + k_3 |p|^6 + k_4 |p|^8),  \\
    \label{eq3.1.42}
  & \mu (r): = \frac{1}{\sqrt{2\pi}}  e^{-r^2/2},
\end{align}
where the constants $k_1, \ldots, k_4$ are chosen so that \eqref{eq3.1.6}--\eqref{eq3.1.8} are satisfied. See the details in Section \ref{section J}. We note that in the non-relativistic case, the construction of the function $B_{i j}$  requires less effort   (see \cite{CK_23}).


Let us check if $\psi$ satisfies the admissibility conditions \eqref{eq27.2.1}--\eqref{eq27.2.3} in the integral formulation of the Landau equations. 
By mollifying the Landau equation in $t$ with $\eta_{\delta} (t) = \delta^{-1} \eta (t/\delta)$, where $\eta \in C^{\infty}_0 ((-1, 0))$, we may assume that the   conditions \eqref{eq27.2.1}--\eqref{eq27.2.2} hold for $\psi$. This mollification argument can also be  used to justify calculations involving $t$-derivatives of order higher than  $m$. 
To verify \eqref{eq27.2.3}, for $x \in \partial \Omega$, we denote $p_{\perp} = p \cdot n_x$ and we set $P_{||}$ to be the projection operator onto the plane orthogonal to $n_x$.
 Then, by using the identity
 \begin{align}
    \label{eq3.1.4}
    p =  P_{||} p + p_{\perp} n_x
 \end{align}
 and the Navier boundary condition in \eqref{lame}, we conclude for  $x \in \partial \Omega$,
 \begin{align*}
     & \psi^{  + } (t, x, p) =   
        \bigg((P_{||} p)^T   S (\partial_t^m  \bm{\phi}) P_{||} p
        + p_{\perp}^2 n_x^T S (\partial_t^m  \bm{\phi}) n_x  -  (\text{tr}\, S))\bigg)  
        h (|p|).
 \end{align*}
Hence, by evenness and the spherical symmetry of $h (|p|)$, the function $\psi^{ + }$ satisfies the SRBC. Thus, $\psi$ is an admissible test function, and hence, the integral formulation \eqref{eq2.3} is valid. In the rest of the proof, we estimate the terms $I_1$--$I_6$ defined in \eqref{eq2.3}.

 \textit{Estimate of the key term.}
 The functions $B_{i j}$  were chosen so that 
 inner product between $\displaystyle \frac{p}{p_0^{  + }} \nabla_x \cdot \psi^{ + }$ and $P^{ +  } \partial_t^m f$ yields a function $(\partial_t^m b) \cdot \nabla_x \cdot S (\partial_t^m \phi)$ (see  \eqref{eq3.1.7}-\eqref{eq3.1.8}). Then, by using  Eq.  \eqref{lame}, the Cauchy-Schwarz inequality, and \eqref{eq3.1.5}, for $\varepsilon_b \in (0, 1)$, we obtain
 \begin{align}
        \label{eq3.1.15}
    I_1&  = \tilde I_{1, 1} + \tilde I_{1, 2},\\
    & \tilde I_1:= -     \sum_{i, j, k, l = 1}^3 \langle  \frac{p_k}{p_0^{+}} B_{i j}, \chi_{l+2}^{+} \rangle   \int_s^{t} \int_{\Omega} (\partial_t^m  b_l)  \partial_{x_k} S_{i j} (\partial_t^m  \bm{\phi}) \, dx d\tau  \notag\\
    &  = - \int_s^{t} \int_{\Omega} \partial_t^m b \cdot \nabla_x \cdot S  (\partial_t^m \bm{\phi}) \, dx d\tau \notag  \\
   &  = \int_s^{t} \|\partial_t^m b\|^2_{  L_2 (\Omega)} \, d\tau 
     - \varkappa  \sum_i \int_s^{t} \big(\int_{\Omega} \partial_t^m b \cdot R_i \, dx\big)  \big(\partial_{t}^m \int_{\Omega} R_i   \cdot (\bE  \times \bB) \, dx\big) d\tau \notag    \\
 &
\ge
 (1-\varepsilon_b) \int_s^{t} \|\partial_t^m  b\|^2_{  L_2 (\Omega)} \, d\tau  
 -    N \varepsilon_b^{-1} \sum_i \int_s^{t}  \bigg|\partial_{t}^m \int_{\Omega} R_i \cdot (\bE \times \bB) \, dx\bigg|^2 \, d\tau \notag,\\
&   \tilde I_2:=   -\sum_{i, j, k = 1}^3  \int_s^{t} \int_{\Omega}  \int_{\bR^3} (\partial_{x_k} S_{i j} (\partial_t^m  \bm{\phi}))  (\frac{p_k}{p_0^{+}} B_{i j}, 0)  \cdot  (1-P) (\partial_t^m f) \, dx dp d\tau \notag\\
&  \ge -  \varepsilon_b \int_s^{t} \| \partial_t^m  \bm{\phi}\|^2_{  W^2_2 (\Omega)} \, d\tau \notag  -  N \varepsilon_b^{-1}  \int_s^{t}   \| (1-P) \partial_t^m f\|^2_{ L_2 (\Omega \times \bR^3) } \, d\tau.
 \end{align}
By    \eqref{eq12.A.4.1} in Lemma \ref{lemma 12.A.4}, we have
\begin{align}
        \label{eq3.1.16}
  \int_s^{t}  \bigg|\partial_{t}^m \int_{\Omega} R_i \cdot (\bE \times \bB) \, dx\bigg|^2 \, d\tau   \lesssim  \varepsilon \int_s^t \cD \, d\tau.
\end{align}

 \textit{Estimate of the $t$-derivative term.}
Due to \eqref{eq3.1.6}, one has
\begin{align}
\label{eq3.1.17}
  &  I_2 =   \int_s^{t} \int_{\Omega}  \int_{\bR^3}  S_{i j} (\partial_t^{m+1} \bm{\phi}) B_{i j} \cdot  (1-P) (\partial_t^m f) \, dx dp d\tau\\
&  \leq \varepsilon_b \int_s^{t}  \| S (  \partial_t^{m+1}\bm{\phi})\|^2_{ L_2 (\Omega)} \, d\tau
+ N \varepsilon_b^{-1}  \int_s^{t}  \| (1-P) \partial_t^m f\|^2_{ L_2 (\Omega \times \bR^3) } \, d\tau. \notag
 \end{align}

 \textit{Estimate of the $t$-boundary term $I_3$.}
By the Cauchy-Schwarz inequality, the elliptic estimate \eqref{eq3.1.5} and \eqref{eq12.A.4.2} in Lemma \ref{lemma 12.A.4}, for any $\tau \in [s, t]$, one has
 \begin{align}
\label{eq3.1.20}
  & \bigg|\int_{\Omega \times \bR^3} S_{i j} (\partial_t^m \bm{\phi}) (B_{i j} \cdot \partial_t^{m} f^{ + }) (\tau, x, p) \, dx dp\bigg| \\
   & \lesssim_{\Omega} \|\partial_t^m f (\tau, \cdot)\|^2_{L_2 (\Omega \times \bR^3)}
    +   \sum_k\bigg|\partial_t^m \int_{\Omega} R_k \cdot (\bE \times \bB) (\tau, x) \, dx\bigg|^2 \notag \\
    &\lesssim_{\Omega} \cI_{||} (\tau) +  \cI^{2}_{||} (\tau) \lesssim \cI_{||} (\tau), \notag
 \end{align}
 where in the last inequality, we used the smallness assumption \eqref{eq3.5.1}.
 Hence, $I_3 = \eta (t) - \eta (s)$ with $\eta$ satisfying \eqref{eq3.1.2}.

\textit{Electric field term.} Due to \eqref{eq3.1.7},
\begin{equation}
    \label{eq3.1.9}
    I_4  = 0.
\end{equation}

\textit{Estimate of the linear collision term.} 
By using the fact that $L$ is a symmetric operator and the Cauchy-Schwarz inequality, we obtain
\begin{align}
		\label{eq3.1.14}
I_5 \le   \varepsilon_b \int_s^{t} \|\partial_t^m  \bm{\phi}\|^2_{  W^1_2 (\Omega)} \, d\tau +  N \varepsilon_b^{-1}  \int_s^{t} \| (1-P) \partial_t^m f\|^2_{ L_2 (\Omega \times \bR^3) } \, d\tau.
\end{align}

\textit{Estimate of the nonlinear term.}
By the Cauchy-Schwarz inequality,
\begin{align}
		\label{eq3.1.10}
    I_6  & \le
     \varepsilon_b \int_s^{t}   \|S (\partial_t^m \bm{\phi})\|^2_{  L_2 (\Omega)}  d\tau\\
     &+ N \varepsilon_b^{-1} \sum_{i, j = 1}^3 \int_s^{t} \int_{ \Omega}   |\langle  \partial_t^m H^{+}, B_{i j} \rangle|^2 \, dx   d\tau.\notag
\end{align}
By the definition of $H$ in \eqref{eq2.3.1} and the estimates   \eqref{eq12.A.2.3} in Lemma \ref{lemma 12.A.2}, and \eqref{eq12.A.3.2} in Lemma \ref{lemma 12.A.3}, 
\begin{align*}
    \text{the second term on the r.h.s. of  \eqref{eq3.1.10}} \le N \varepsilon_b^{-1}  \varepsilon \int_s^t \cD \, d\tau.
\end{align*}

\textit{Intermediate estimate of $b$.} Combining the identity \eqref{eq2.3} with the bounds \eqref{eq3.1.15}--\eqref{eq3.1.10}, we obtain
\begin{align*}
& \int_s^{t}  \|\partial_t^m  b\|^2_{ L_2 (\Omega)} \, d\tau   \le (\eta (t) - \eta (s))
 + N \varepsilon_b \int_s^{t} \| \partial_t^m  b\|^2_{  L_2 (\Omega)} \, d\tau
+  N \varepsilon_b \int_s^{t} \| \partial_t^m  \bm{\phi}\|^2_{  W^2_2 (\Omega)} \, d\tau \\
&
+ N \varepsilon_b  \int_s^t \| S (\partial_t^{m+1} \bm{\phi})\|^2_{ L_2 (\Omega) }  \, d\tau  
  + N \varepsilon_b^{-1}  \big(\int_s^t \cD_{||} \, d\tau    +  \varepsilon \int_s^t \cD \, d\tau\big).
\end{align*}
We note that by the elliptic estimate \eqref{eq3.1.5} and \eqref{eq3.1.16}, we may replace the third term on the r.h.s with
\begin{align*}
    N \big(\varepsilon_b \int_s^{t}  \|\partial_t^m  b\|^2_{ L_2 (\Omega)} \, d\tau
    + \varepsilon \int_s^t \cD \, d\tau\big).
\end{align*}
Choosing $\varepsilon_b$ sufficiently small, we may absorb the term containing $b$ on the r.h.s into the l.h.s..
Thus, to   obtain the desired estimate \eqref{eq3.1.1}, it suffices to prove that
\begin{align}
    \label{eq3.1.21}
 & \int_s^t \|S (\partial_t^{m+1} \bm{\phi})\|^2_{  L_2  (\Omega) } \, d\tau\\
 &  \lesssim_{\Omega, \theta, r_3, r_4} \int_s^t \| \sqrt{M_{+}} \partial_t^m a^{+} + \sqrt{M_{-}} \partial_t^m a^{-} \|^2_{ L_2 (\Omega) }  \, d\tau \notag\\
 & + \int_s^t \|\partial_t^m c \|^2_{ L_2 (\Omega) }  \, d\tau
  + \int_s^t \cD_{||}  \, d\tau +   \varepsilon \int_s^t \cD \, d\tau.\notag
\end{align}

\textbf{Step 2: estimate of $\partial_t^{m+1} \bm{\phi}$.}
The method is similar to that of Step 1. However, we rearrange the integrals as follows:
\begin{align}
    \label{eq3.73}
    -(I_2+I_3)=-I_1+ \sum_{j=4}^6 I_j,
\end{align}
so that the key term in this argument is  $-(I_2+I_3).$We use a different  test function given by
\begin{align}
    \label{eq3.72}
   \tilde  \psi (t, x, p) = \chi_{i+2} (p)  (\partial_t^{m+1}  \bm{\phi}_i) (t, x)
\end{align}
(see \eqref{eq6.23}).
Let us check that $\tilde \psi$ satisfies the admissibility condition \eqref{eq27.2.3}, which is the SRBC. We note that due to the boundary condition $\bm{\phi} \cdot n_x = 0$, one has $\bm{\phi} \cdot p  = \bm{\phi} \cdot P_{||} p$ (see \eqref{eq3.1.4}) for $x \in \partial\Omega$. Then, by the definition $\chi_{i+2}, i=1, 2, 3$ (see \eqref{eq6.23}), for $x \in \partial \Omega$, 
\begin{align*}
   \tilde \psi (t, x, p) =   \kappa_1  (\partial_t^{m+1} \bm{\phi}) \cdot (P_{||} p) \,    \big((p_0^{+})^{-1} \sqrt{J^{+}}, (p_0^{-})^{-1} \sqrt{J^{-}}\big),
\end{align*}
Hence, $\tilde \psi$ satisfies the SRBC, and we may use the integral formulation \eqref{eq2.3} of the Landau equation. Again, we will estimate $I_1$--$I_6$ in that integral identity.

 \textit{The key term.}
Integrating by parts in $t$ and using  the fact that $\bm{\phi}$ satisfies the Lam\'e system \eqref{lame} with the correction term given by \eqref{proj}, we obtain
\begin{align*}
  & -(I_2+I_3)  =  \int_s^t \int_{ \Omega  } (\partial_t^{m+1} \bm{\phi}) \cdot (\partial_t^{m+1} b) \,  dx d\tau \\
  & =  -\int_s^t  \int_{ \Omega  } (\partial_t^{m+1}  \bm{\phi}_i)  \, \partial_{x_j} S_{i j} (\partial_t^{m+1}  \bm{\phi})\,  dx d\tau \notag \\
 & + \varkappa  \sum_{k}   \int_s^t \bigg(\partial_{\tau}^{m+1} \int_{\Omega} R_k \cdot (\bE \times \bB)\, dx\bigg) \bigg(\int_{ \Omega  }  (\partial_t^{m+1}  \bm{\phi}) \cdot  R_k \, dx\bigg) \,  d\tau. \notag
\end{align*}

Next, by Green's formula for the deformation tensor (see \eqref{eqA.2.0}) and the Cauchy-Schwarz inequality, we get, for any $\tilde \varepsilon_b \in (0, 1)$,
\begin{align}
    \label{eq3.1.24}
 &  -(I_2+I_3)  \ge     \int_s^t \sum_{i, j = 1}^3 \int_{ \Omega  }  |S_{i j} (\partial_t^{m+1} \bm{\phi})|^2 \, dx d\tau   - \tilde \varepsilon_b \int_s^t \|\partial_t^{m+1} \bm{\phi}\|^2_{L_2 (\Omega)} \, d\tau \\
& - N \tilde \varepsilon_b^{-1} \sum_k \int_s^t \bigg|\partial_{\tau}^{m+1} \int_{\Omega} R_k \cdot (\bE \times \bB) \, dx\bigg|^2 \, d\tau. \notag
\end{align}
Despite the fact that the last term contains $m+1$ derivatives (above the highest order), thanks to \eqref{eq12.A.7.1} in Lemma \ref{lemma 12.A.7}, we can still replace it with
$$
    - N  \tilde \varepsilon_b^{-1}   \varepsilon \int_s^t \cD \, d\tau.
$$
We point out that the estimate \eqref{eq12.A.7.1} follows from the momentum identity for Maxwell's equations (see \eqref{eqH.4}).
 To justify integration by parts in $t$  rigorously, one can use a (forward-in-time) mollification argument by mollifying  the Landau equation in $t$ with $\eta_{\delta} (t) = \delta^{-1} \eta (t/\delta)$, $\eta \in C^{\infty}_0 ((-1, 0))$. 
To show that the temporal boundary terms in $I_3$ converge, we use the weak continuity of  $\partial_t^m f (t)$ in $L_2 (\Omega)$, which is true since the latter is a finite energy solution to the $k$ times differentiated in $t$ Landau equation (see \eqref{eq27.2.5} in Definition \ref{definition 27.1}). We will not mention this in the sequel.


\textit{The transport term.} 
By using the explicit form of the projection operator $P$ (see \eqref{eq6.16}--\eqref{eq6.17}) and that of functions $\chi_{i+2}, i=  1, 2, 3$ (see \eqref{eq6.23}), we obtain
\begin{align}
    \label{eq3.1.12}
   I_1 &= - \int_s^t \int_{\Omega \times \bR^3} \partial_{x_j}  (\partial_t^{m+1} \bm{ \phi}_i) \bigg(\frac{ p_j}{p_0^{+}} \chi_{i+2}^{+}       \partial_t^m f^{+}
  + \frac{ p_j}{p_0^{-}}  \chi_{i+2}^{-}   \partial_t^m f^{-}\bigg)
  \, dx dp d\tau\\
  &
  = -  \int_s^t  \int_{\Omega }  (\nabla_x \cdot  \partial_t^{m+1} \bm{\phi})   \big(\epsilon_{+} \partial_t^m a^{+}  + \epsilon_{-} \partial_t^m a^{-}\big)  \, dx d\tau \notag\\
  & -  \epsilon_1 \int_s^t  \int_{\Omega}  (\nabla_x \cdot \partial_t^{m+1} \bm{\phi})    \,  (\partial_t^m c)  \, dx d\tau \notag\\
  & -  \int_s^t  \int_{\Omega \times \bR^3} (\partial_{x_j} \partial_t^{m+1} \bm{ \phi}_i) \bigg((\frac{ p_j}{p_0^{+}} \chi_{i+2}^{+}, \frac{p_j}{p_0^{-}}  \chi_{i+2}^{-}) \cdot    (1-P)   (\partial_t^m f)\bigg) \, dx dp d\tau, \notag
\end{align}
where 
\begin{align}
    \label{eq3.1.50}
      \epsilon_{\pm} =  \kappa_1  \sqrt{M_{\pm}^{-1}}  \int \frac{p_{  1  }^2}{p_0^{\pm}} J^{\pm} \, dp, \quad \epsilon_1 = \langle \frac{p_{ 1 } }{p_0^{+}} \chi_{3}^{+}, \chi_6^{+} \rangle + \langle \frac{p_1 }{p_0^{-}} \chi_{3}^{-}, \chi_6^{-} \rangle.
\end{align}
Furthermore, by the definition of $J^{\pm}$ (see \ref{juttner}),  
\begin{align}
    \label{eq3.1.70}
       \frac{p}{p_0^{\pm}} J^{\pm} =  - k_b T \nabla_{p} J^{\pm}
\end{align}
and integration by parts, we get
\begin{equation}
            \label{eq3.1.13}
     \int \frac{p_i^2}{p_0^{\pm}} J^{\pm} \, dp = k_b T \int J^{\pm} \, dp  =  k_b T M_{\pm}
\end{equation}
(see \eqref{eq6.30}).
This enables us to simplify $\epsilon_{\pm}$:
\begin{equation}
            \label{eq3.1.22}
      \epsilon_{\pm} = \kappa_1 k_b T \sqrt{(M_{\pm})^{-1}}  \int  J^{\pm} \, dp    = \kappa_1 k_b T  \sqrt{M_{\pm}}.
\end{equation}

Therefore, for any $\tilde \varepsilon_b \in (0, 1)$, we have
\begin{align*}
 |I_1| \lesssim  &  \tilde \varepsilon_b \int_s^t \|\partial_t^{m+1} \bm{ \phi}\|^2_{ W^1_2 (\Omega)} \, d\tau + \tilde \varepsilon_b^{-1}  \bigg(\int_s^t  \|\partial_t^m c \|^2_{ L_2 (\Omega) } \, d\tau +  \int_s^t \cD_{||} \, d\tau\\
 &  + \int_s^t \|\sqrt{M_{+}} \partial_t^m a^{+} + \sqrt{M_{-}} \partial_t^m a^{-}\|^2_{ L_2 (\Omega) }  \, d\tau\bigg).
\end{align*}

\textit{Electric field term.}  By the definition of $\psi$ (see \eqref{eq3.72}) and $I_4$ (see \eqref{eq2.3},
\begin{align}
    \label{eq3.1.23}
I_4 & = \frac{  \kappa_1 }{k_b T}
 \bigg(\int_{ \bR^3 } e_{+}\frac{p_i^2}{p_0^{+}}   J^{+} - e_{-} \frac{p_i^2}{p_0^{-}} J^{-} \, dp\bigg)
 \int_s^t \int_{ \Omega  } (\partial_t^m \mathbf{E}_i)  (\partial_t^{m+1} \bm{\phi}_i)  \, dx d\tau.
\end{align}
Thanks to the identity \eqref{eq3.1.13}  and the neutrality condition \eqref{eq0}, the first integral in \eqref{eq3.1.23} vanishes, and thus, 
$$
    I_4=0.
$$

\textit{Remaining terms.} We observe that for  the integral involving $L$, we have 
\begin{align}
\label{eq3.1.25.1}
    I_5 = 0,
\end{align} since $\chi_{i+2}, i = 1, 2, 3,$ belong to the kernel of $L$. 
Furthermore, proceeding as in  \eqref{eq3.1.10}, we conclude
\begin{align}
\label{eq3.1.25}
   I_6 \lesssim  \tilde \varepsilon_b \int_s^t \|\partial_t^{m+1} \bm{ \phi}\|^2_{ L_2 (\Omega)} \, d\tau
    + \tilde \varepsilon_b^{-1}  \varepsilon \int_s^t \cD \, d\tau.
\end{align}

Finally, gathering \eqref{eq2.3} and \eqref{eq3.1.24}--\eqref{eq3.1.25}, we get
\begin{align}
    \label{eq3.1.30}
 & \int_s^t \|S (\partial_t^{m+1} \bm{\phi})\|^2_{  L_2  (\Omega) } \, d\tau\\
 &  \lesssim_{\Omega, \theta, r_3, r_4}   \tilde \varepsilon_b \int_s^t \|\partial_t^{m+1} \bm{ \phi}\|^2_{ W^1_2 (\Omega)} \, d\tau  +  \tilde \varepsilon_b^{-1} \bigg(\int_s^t \| \sqrt{M_{+}} \partial_t^m a^{+} + \sqrt{M_{-}} \partial_t^m a^{-} \|^2_{ L_2 (\Omega) }  \, d\tau  \notag \\
 & + \int_s^t \|\partial_t^m c \|^2_{ L_2 (\Omega) }  \, d\tau + \int_s^t \cD_{||} \, d\tau +    \varepsilon \int_s^t \cD \, d\tau\bigg). \notag
\end{align}
 Since $\bm{\phi} \perp \mathcal{R} (\Omega)$ in $L_2 (\Omega)$ (see \eqref{eq3.1.32}), we may apply a variant of Korn's inequality in
\eqref{eqA.2.1} in Lemma \ref{lemma A.2} and  replace the first term on the r.h.s. of \eqref{eq3.1.30} with
$$
    \tilde \varepsilon_b \times (\text{the l.h.s. of \eqref{eq3.1.30}}).
$$ Finally, choosing $\tilde \varepsilon_b$ sufficiently small, we obtain \eqref{eq3.1.21}, which finishes the proof of the desired estimate \eqref{eq3.1.1}.
\end{proof}

\begin{remark}
We point out that in the absence of the electromagnetic field, a similar argument was also recently carried out in   \cite{CK_23} in the non-relativistic case. We also mention that the connections between the Korn inequality, macroscopic velocity, and the trend to equilibrium were highlighted previously in \cite{DV_05}.
\end{remark}



\begin{lemma}[estimate of a weighted average of $a^{\pm}$, cf. \eqref{eq3.2.0.1}]
\label{lemma 3.2}
Under Assumption \ref{assumption 3.5}, we have
\begin{align}
	\label{eq3.2.1}
&    \sum_{k = 0}^m \int_s^{t}  \|\partial_{t}^k (\sqrt{M_{+}} a^{+} + \sqrt{M_{-}} a^{-})\|^2_{ L_2 (\Omega) } \, d\tau \lesssim_{\Omega,   \theta, r_3, r_4 }  (\eta (t)  - \eta (s)) \\
&+ \sum_{k = 0}^m \int_s^{t} \|\partial_{t}^k [b, c]\|_{ L_2 (\Omega) }^2 \, d\tau + \int_s^t \cD_{||}  \, d\tau   + \varepsilon \int_s^t \cD \, d\tau, \notag
\end{align}
where $\eta$ satisfies  \eqref{eq3.1.2}.
\end{lemma}

\begin{proof}
As in the proof of Lemma \ref{lemma 3.1}, we use the integral formulation of the Landau equation \eqref{eq2.3} and focus on the top-derivative term $\partial_t^m (\sqrt{M_{+}} a^{+} +  \sqrt{M_{-}} a^{-})$.

\textit{Test function.}
We consider the equation
\begin{equation}
    \label{eq3.2.2}
\left\{\begin{aligned}
 & -  \Delta_x \phi =  \sqrt{M_{+}} a^{+}
 +  \sqrt{M_{-}} a^{-},\\
 & \frac{\partial \phi }{\partial n_x} = 0 \, \, \text{on} \, \partial \Omega.
\end{aligned}
\right.
\end{equation}
Integrating the continuity equation 
\begin{align}
    \label{eqE.2.15}
        \sqrt{M_{\pm}} \partial_t a^{\pm} +  \nabla_x \cdot \bm{j}^{\pm} = 0, \quad \bm{j}^{\pm}: = \int_{\bR^3} \frac{p}{p_0^{\pm}} f^{\pm} \sqrt{J^{\pm}} \, dp
\end{align}
(cf. \eqref{eqE.2.50})
over $\Omega$ and using the SRBC, and the assumption on the initial data  \eqref{eq6.1.0}, we have
\begin{align}
    \label{eq3.1.51}
    \int_{\Omega \times \bR^3} f^{\pm} (t, x, p) \sqrt{J^{\pm}} \, dx dp = 0.
\end{align}
 By this and the definition of $a^{\pm}$ in \eqref{eq6.18}, one has,
$$
	\int_{\Omega} a^{\pm} \, dx = 0.
$$
Hence, the equation \eqref{eq3.2.2} has a unique strong solution $\phi \in W^2_2 (\Omega)$ satisfying
$$
    \int_{\Omega} \phi \, dx = 0,
$$
 and, furthermore,
\begin{equation}
				\label{eq3.2.3}
\|\phi\|_{ W^{ 2 }_2 (\Omega) } \lesssim_{\Omega}  \|\sqrt{M_{+}} a^{+}
 +  \sqrt{M_{-}} a^{-}\|_{  L_2 (\Omega) }.
\end{equation}

 We set
\begin{align}
    \label{eq3.2.10}
& \psi (t, x, p) = \chi_{i+2} (p) \partial_{x_i} \phi (t, x),\quad i=1,2,3
\end{align}
(see \eqref{eq6.23}). 
To check that $\psi$ satisfies the SRBC, we note that by the Neumann BC and the decomposition \eqref{eq3.1.4}, for any $x \in \partial \Omega$, 
\begin{align}
    \label{eq3.2.11}
    \psi (t, x, p) = \kappa_1 \nabla_x \partial_t^m \phi (t, x) \cdot (P_{||} p) \,  (\sqrt{J^{+}}, \sqrt{J^{-}}).
\end{align}
Thus, the SRBC holds and, hence, the integral formulation \eqref{eq2.3} is valid. Let $I_i, i = 1, \ldots, 6,$ be the integrals defined  therein.


\textit{The key term.} 
Inspecting   the calculations in \eqref{eq3.1.12}--\eqref{eq3.1.22}, we get
\begin{align}
        \label{eq3.2.4.0}
I_1 & =
- \kappa_1 k_b T \int_s^{t} \int_{\Omega} \big(\sqrt{M_{+}} \partial_t^m a^{+} + \sqrt{M_{-}}  \partial_t^m a^{-}\big)  \, (\Delta_x \partial_t^m \phi) \, dx d\tau
 \\
& - \epsilon_1 \int_s^{t} \int_{\Omega}
(\partial_t^m c) \, (\Delta_x \partial_t^m \phi) \, dx d\tau \notag\\
& - \sum_{i, j = 1}^3
\int_s^{t} \int_{\Omega \times \bR^3}  \big((1-P) \partial_t^m f\big) \cdot \big(\frac{p_j}{p_0^{+}} \chi_{i+2}^{+}, \frac{p_j}{p_0^{-}} \chi_{i+2}^{-}\big) \partial_{x_i x_j} \phi\, dx dp d\tau, \notag
\end{align}
where $\epsilon_1$ is defined in  \eqref{eq3.1.50}.


Due to our choice of $\phi$ (see \eqref{eq3.2.2})  and the elliptic estimate \eqref{eq3.2.3},  for any $\varepsilon_a \in (0, 1)$, we have
\begin{align}
    \label{eq3.2.4}
I_1 & \ge  (\kappa_1 k_b T -  \varepsilon_a) \int_s^{t}
\|\sqrt{M_{+}} \partial_t^m a^{+}
 +  \sqrt{M_{-}} \partial_t^m a^{-}\|^2_{ L_2 (\Omega) } \, d\tau\\
&- N \varepsilon^{-1}_a
\big(\int_s^{t}  \| \partial_t^m c\|^2_{ L_2 (\Omega) } \, d\tau + \int_s^{t} \cD_{||}  \, d \tau\big). \notag
\end{align}

 \textit{Estimate of the $t$-derivative term.}
By orthogonality, we have
\begin{align}
    \label{eq3.2.50}
	I_2 & =
	\int_s^{t} \int_{\Omega}
	 (\partial_t^{m} b) \cdot \nabla_{x} (\partial_t^{m+1} \phi) \, dx d\tau\\
  & \lesssim_{\Omega} \varepsilon_a \int_s^{t} \|\partial_t^{m+1} \nabla_x \phi\|_{ L_2 (\Omega) }^2 \, d\tau
+ \varepsilon_a^{-1} \int_s^{t} \|\partial_t^{m} b\|_{ L_2 (\Omega) }^2 \, d\tau. \notag
\end{align}
We note that due to \eqref{eqE.2.15}, $\partial_t^{m+1} \phi$ satisfies
     \begin{equation*}
\left\{\begin{aligned}
 & -  \Delta_x \partial_t^{m+1} \phi = -  \nabla_x \cdot(\partial_t^m \bm{j}^{+}+\partial_t^m\bm{j}^{-}),\\
 & \frac{\partial (\partial_t^{m+1} \phi) }{\partial n_x} = 0 \, \, \text{on} \, \partial \Omega.
\end{aligned}
\right.
\end{equation*}
We multiply the above equation by $\partial_t^{m+1} \phi$, integrate by parts, and use fact that $\partial_t^m \bm{j}^{\pm} \cdot n_x = 0$ on $\partial \Omega$ (see \eqref{eqE.2.16}).  Combining this argument with the estimate of $\bm{j}^{\pm}$ in \eqref{eqE.2.15.1}, we conclude that
\begin{align}
    \label{eq3.2.5.0}
  &  \|\partial_t^{m+1} \nabla_x \phi\|_{L_2 (\Omega)} \lesssim \|\partial_t^m [\bm{j}^{+}, \bm{j}^{-}]\|_{L_2 (\Omega)} \\
    & 
\lesssim \|\partial_t^m b\|_{ L_2 (\Omega) }
	+ \|(1-P) \partial_t^m f\|_{ L_2 (\Omega \times \bR^3) }. \notag
\end{align}
Hence,  we obtain
\begin{align}
    \label{eq3.2.5}
	I_2
\lesssim \varepsilon_a^{-1}
 \big(\int_s^t \|\partial_t^{m} b\|_{ L_2 (\Omega) }^2 \, dt  +  \int_s^t \cD_{||}  \,  d\tau\big).
\end{align}

\textit{The $t$-boundary term.} Proceeding as in  \eqref{eq3.1.20} and using the elliptic estimate \eqref{eq3.2.3}, we conclude that
\begin{align}
    \label{eq3.2.9}
    I_3  = \eta (t)- \eta (s)
\end{align}
with $\eta$ satisfying the bound  \eqref{eq3.1.2}.

\textit{The electric field term.} As in  \eqref{eq3.1.23}, we have
\begin{align}
    \label{eq3.2.7}
    I_4 = 0.
\end{align}

\textit{The remaining integrals.}
Repeating the argument  in  \eqref{eq3.1.25.1} and \eqref{eq3.1.10} and using the elliptic estimate \eqref{eq3.2.3}  yield
\begin{equation}
\begin{aligned}
        \label{eq3.2.8}
    &I_5 = 0, \\ 
    & I_6 \lesssim  \tilde \varepsilon_a \int_s^t \|\sqrt{M_{+}} \partial_t^m a^{+}
 +  \sqrt{M_{-}} \partial_t^m a^{-}\|^2_{ L_2 (\Omega)} \, d\tau  +  \tilde \varepsilon_a^{-1}  \varepsilon \int_s^t \cD \, d\tau. 
\end{aligned}
\end{equation}

Finally, combining the integral formulation \eqref{eq2.3} with the estimates  \eqref{eq3.2.4}--\eqref{eq3.2.8}, we obtain
\begin{align}
    \label{eq3.2.51}
&   \int_s^{t}  \|\sqrt{M_{+}} \partial_t^m a^{+}
 +  \sqrt{M_{-}} \partial_t^m a^{-}\|^2_{ L_2 (\Omega) } \, d\tau
 \lesssim_{\Omega, \theta, r_3, r_4 }  (\eta (t) -  \eta (s))\\
& +  \varepsilon_a   \int_s^{t}  \|\sqrt{M_{+}} \partial_t^m a^{+}
 +  \sqrt{M_{-}} \partial_t^m a^{-}\|^2_{ L_2 (\Omega) } \, d\tau \notag\\
& +  \varepsilon_a^{-1} \big(\int_s^{t} \|\partial_t^m [b, c]\|_{ L_2 (\Omega) }^2 \, d\tau + \int_s^{t} \cD_{||}  \, d\tau  +  \varepsilon \int_s^t \cD \, d\tau\big). \notag
\end{align}
 Choosing $\varepsilon_a$ sufficiently small,  we prove the desired bound \eqref{eq3.2.1}.
\end{proof}

Combining the estimates \eqref{eq3.1.1} and \eqref{eq3.2.1} in Lemmas \ref{lemma 3.1} and \ref{lemma 3.2}, respectively, we obtain the main result of this section.
\begin{lemma}[final estimate of $b$, cf. \eqref{eq3.3.0.1}]
\label{lemma 3.3}
     Under  Assumption  \ref{assumption 3.5},   there exists a sufficiently small constant $\varepsilon_b =\varepsilon_b (\Omega, r_3, r_4, \theta) > 0$  such that 
\begin{align}
\label{eq3.3.1}
 &   \sum_{k = 0}^m \int_s^{t} \|\partial_t^k  b\|^2_{ L_2 (\Omega) } \, d\tau
    \lesssim_{\Omega, \theta, r_3, r_4 } (\eta (t)  - \eta (s)) \\
    &+  \varepsilon_b \sum_{k = 0}^m \int_s^{t} \|\partial_t^k   c\|^2_{ L_2 (\Omega) } \, d\tau
 + \varepsilon_b^{-1} \big(\int_s^t \cD_{||}  \, d\tau   + \varepsilon \int_s^t \cD \, d\tau\big), \notag
 \end{align}
where $\eta$ is a function satisfying \eqref{eq3.1.2}.
\end{lemma}

\section{Estimate of $c$}
    \label{section 8}
The objective of this section is to derive an estimate of $c$, an important step in obtaining the positivity estimate for $L$ (see \eqref{eq122}). Given the derivative loss for $a^{\pm}$ at the highest order, it becomes essential to control the $L_2^{t, x}$ norms of $c$ up to the top order.
\begin{lemma}
    \label{lemma 7.1}
Under  Assumption  \ref{assumption 3.5},  there exists a function $\eta$ satisfying  \eqref{eq3.1.2}   such that
\begin{align}
\label{eq6.1.4}
 &   \sum_{k = 0}^m \int_s^{t} \|\partial_t^k  c\|^2_{ L_2 (\Omega) } \, d\tau
    \lesssim_{\Omega, \theta, r_3, r_4 } (\eta (t)  - \eta (s)) \\
 &  +   \sum_{k = 0}^m \int_s^{t} \|\partial_t^k  b\|^2_{ L_2 (\Omega) } \, d\tau  +        \int_s^{t} \cD_{||} (\tau) \, d\tau +  \varepsilon \int_s^t \cD \, d\tau. \notag
 \end{align}
\end{lemma}

\begin{proof}
\textbf{Step 1: a preliminary estimate of $c$.}

\textit{Test function.} As in the proof of Lemmas \ref{lemma 3.1}--\ref{lemma 3.2}, we focus on the top derivative term $\partial_t^m c$, and our proof involves the integral formulation \eqref{eq2.3} and a duality argument.
First, thanks to the elliptic regularity theory, the  Neumann problem
 \begin{equation}
			\label{eq7.1.1}
\left\{\begin{aligned}
 &- \Delta_x \phi = c - \int_{\Omega} c \, dx,\\
 & \frac{\partial \phi}{\partial  n_x} = 0 \, \, \text{on} \, \partial \Omega,
\end{aligned}
\right.
\end{equation}
has a unique strong solution $\phi \in W^2_2 (\Omega)$ satisfying
\begin{equation}
        \label{eq7.1.2}
    \int_{\Omega} \phi \, dx  = 0,
\end{equation}
 and, in addition,
\begin{equation}
			\label{eq7.1.3}
    \|\phi\|_{ W^{2}_2 (\Omega) } \lesssim_{\Omega} \|c  - \int_{\Omega} c \, dx\|_{ L_2 (\Omega) }.
\end{equation}
Furthermore,  to estimate the average of $c$, we recall the energy identity (cf. \cite{GS_03}) 
\begin{align*}
   &  \partial_t \big(\int_{\bR^3} (p_0^{+} F^{+} + p_0^{-} F^{-})  \, dp
    + \frac{1}{8 \pi} (|\bE|^2+|\bB|^2)\big) \\
    & +  \nabla_x \cdot \big(\int_{\bR^3}  p (F^{+} + F^{-}) \, dp  + \frac{1}{4 \pi} (\bE \times \bB)\big)   = 0.
    \end{align*}
Integrating the equality over $\Omega$ and  using the SRBC and the perfect conductor BC \eqref{eq1.5}, we obtain
\begin{align*}
 \int_{\Omega}  \int_{\bR^3} (p_0^{+} F^{+} + p_0^{-} F^{-})  \, dp \, dx
    + \frac{1}{8 \pi} \int_{\Omega}  (|\bE|^2+|\bB|^2)  \, dx  = \text{const}.
    \end{align*}
Then, by the assumption on the initial data $f_0$ (see \eqref{eq6.1.2}),
\begin{align*}
 \int_{\Omega}  \int_{\bR^3} (p_0^{+} f^{+} \sqrt{J^{+}} + p_0^{-} F^{-} \sqrt{J^{-}})  \, dp\, dx
    + \frac{1}{8 \pi} \int_{\Omega}  (|\bE|^2+|\bB|^2)  \, dx  = 0.
    \end{align*}
Thus,  by the definition of $c$ in \eqref{eq6.20} and \eqref{eq3.1.51}, we have
\begin{align}
    \label{eq7.1.3.1}
    \int_{\Omega} c \, dx & = \kappa_3 \int_{\Omega} \int_{  \bR^3 } (p_0^{+} f^{+} \sqrt{J^{+}} + p_0^{-} f^{-} \sqrt{J^{-}}) \, dx dp \\
   & = - \frac{\kappa_3 }{8 \pi} \int_{\Omega}  (|\bE|^2+|\bB|^2)  \, dx. \notag
\end{align}
Combining this with \eqref{eq7.1.3} gives
\begin{align}
\label{eq7.1.4}
 & \|\partial_{t}^m \phi (\tau, \cdot)\|_{ W^{2}_2 (\Omega) } \lesssim_{\Omega} \| \partial_{t}^m c (\tau, \cdot)\|_{ L_2 (\Omega) } + \bigg|\partial_{\tau}^m  \int_{\Omega}  (|\bE (\tau, x)|^2+|\bB (\tau, x)|^2)  \, dx\bigg|.
\end{align}
Furthermore, by the estimates \eqref{eq12.A.4.1}--\eqref{eq12.A.4.2} in Lemma \ref{lemma 12.A.4} and the smallness assumption \eqref{eq3.5.1},
\begin{align}
\label{eq7.1.5}
 &\int_s^t \|\partial_t^m \phi\|^2_{W^2_2 (\Omega) }  \, d\tau \lesssim \int_s^t  \| \partial_{t}^m c\|^2_{ L_2 (\Omega) }  \, d\tau
 + \varepsilon \int_s^t \cD \, d\tau,\\
 \label{eq7.1.6}
 & \|\partial_t^m \phi (\tau, \cdot)\|^2_{W^2_2 (\Omega) } 
 \lesssim   \cI_{||} (\tau) +  \cI^{2}_{||} (\tau) \lesssim \cI_{||} (\tau).
 \end{align}

Next, let $\rho^{\pm}_0$ be a number defined by
$$
    \int_{\bR^3} \frac{|p|^2}{p_0^{\pm}} (p_0^{\pm} - \rho^{\pm}_0) J^{\pm}\, dp = 0,
$$
so that by symmetry,
\begin{equation}
            \label{eq7.16}
            \int_{\bR^3} \frac{p_i^2}{p_0^{\pm}} (p_0^{\pm} - \rho^{\pm}_0) J^{\pm}\, dp = 0, i = 1, 2, 3.
\end{equation}
In addition, by \eqref{eq7.16}, for any number $r$, we also have
\begin{equation}
            \label{eq7.17}
\int_{\bR^3} \frac{p_i^2}{p_0^{\pm}} (p_0^{\pm} - \rho^{\pm}_0) (p_0^{\pm} - r)  J^{\pm} \,  dp = \int \frac{p_i^2}{p_0^{\pm}} (p_0^{\pm} - \rho^{\pm}_0)^2 J^{\pm} \, dp.
\end{equation}
We denote
\begin{align}
    \label{eq7.1.7}
      C_i =  (C_i^{+}, C_i^{-}): =  p_i \big((p_0^{+} - \rho^{+}_0) \sqrt{J^{+}}, (p_0^{-} - \rho^{-}_0) \sqrt{J^{-}}\big)
\end{align}
and note that by oddness, \eqref{eq7.16}--\eqref{eq7.17}, and the definition of $\chi_6$ in \eqref{eq6.24},
\begin{align}
    \label{eq7.18}
 &\langle \frac{p_j}{p_0^{\pm}} C_i^{\pm}, \sqrt{J^{\pm}} \rangle =   0, \,  i, j = 1, 2, 3, \\
    \label{eq7.19}
  &  \langle \frac{p_j}{p_0^{\pm}} C_i^{\pm}, \chi_6^{\pm} \rangle =  1_{ i = j} \kappa_3 \int \frac{p_i^2}{p_0^{\pm}} (p_0^{\pm} - \rho^{\pm}_0)^2 J^{\pm} \, dp \\
  & = 1_{ i = j} \frac{1}{3} \kappa_3 \int \frac{|p|^2}{p_0^{\pm}} (p_0^{\pm} - \rho^{\pm}_0)^2 J^{\pm} \, dp =: 1_{i = j} \rho_c^{\pm}, \, i  =1, 2, 3.  \notag
\end{align}
Let $\psi$ be a test function given by
$$
     \psi (t, x, p) =  C_i   (p)
    \, \partial_{x_i} \partial_{t}^m  \phi (t, x).
$$
Repeating the argument in \eqref{eq3.2.11}, we conclude that $\psi$ satisfies the SRBC. The remaining admissibility conditions are verified in the same way as in Lemma \ref{lemma 3.1}. Hence, as in the proof of Lemma \ref{lemma 3.1}, we may use the integral formulation of the Landau equation \eqref{eq2.3}.

 \textit{Estimate of the key term.}
 By using the macro-micro decomposition in the integral $I_1$ and noticing that the terms containing $b$ and $a^{\pm}$ vanish due to oddness and \eqref{eq7.18}, respectively, we get
  \begin{align}
  \label{eq7.3}
   I_1
  &  = -   (\rho_c^{+} + \rho_c^{-})
   \int_s^{t} \int_{\Omega}  (\partial_{t}^{m}  \Delta_x \phi)    \,  (\partial_{t}^m c)  \, dx d\tau\\
  & -   \int_s^{t} \int_{\Omega \times \bR^3} (\partial_{x_i x_j} \partial_{t}^{m} \phi) \, \big((1-P)   (\partial_{t}^m f) \cdot (C_i^{+} \frac{p_j}{p_0^{+}},  C_i^{-} \frac{p_j}{p_0^{-}})\big) \, dx dp d\tau. \notag
\end{align}
Hence, by using Eq. \eqref{eq7.1.1},  the Cauchy-Schwarz inequality, and the elliptic estimate \eqref{eq7.1.5}, we get,  for any $\varepsilon_c \in (0, 1)$, 
 \begin{align*}
I_1 &\ge
 (\rho_c^{+} + \rho_c^{-}) \bigg(\int_s^{t} \|\partial_{t}^m c\|^2_{ L_2 (\Omega) } \, d\tau -\int_s^t \bigg|\partial_{t}^m  \int_{\Omega} c  \, dx\bigg|^2 \, d\tau\bigg)\\
& - \varepsilon_c  \int_s^{t}  \|\partial_{t}^m \phi\|^2_{ W^2_2 (\Omega) } \, d\tau  - N \varepsilon_c^{-1}   \int_s^{t}  \cD_{||}  \, d\tau  \notag\\
     & \ge  (\rho_c^{+} + \rho_c^{-} - N_0 \varepsilon_c)  \int_s^{t} \|\partial_{t}^m c\|^2_{ L_2 (\Omega) } \, d\tau\\
    & -  N \varepsilon_c^{-1}  \big(\int_s^{t}  \cD_{||}  \, d\tau +  \varepsilon \int_s^t \cD \, d\tau\big),
 \end{align*}
 where $N_0 = N_0 (\Omega), N = N (\Omega, \theta, r_3, r_4)$.

 \textit{Estimate of the $t$-derivative term.}
We employ the macro-micro decomposition and observe that the terms involving $a^{\pm}$ and $c$ vanish due to the choice of the test function (see \eqref{eq7.1.7}) and oddness. Hence, by the Cauchy-Schwarz inequality,
 \begin{align}
     \label{eq7.5}
& I_2  \lesssim \varepsilon_c \int_s^{t}   \|   \partial_{t}^{m+1} \phi \|^2_{ W^1_2 (\Omega)} \, d\tau  \\
& +   \varepsilon_c^{-1}  \big(\int_s^{t}  \| \partial_{t}^m b\|^2_{ L_2 (\Omega) } \, d\tau
 +     \int_s^{t}  \| \partial_{t}^m (1-P) f\|^2_{ L_2 (\Omega \times \bR^3) } \, d\tau\big). \notag
 \end{align}
We will estimate the first term on the r.h.s. of \eqref{eq7.5} in Step 2.

\textit{The $t$-boundary term.} Here we show that $I_3$, as defined  in \eqref{eq2.3}, can be represented as $\eta (t) - \eta (s)$ with $\eta$ satisfying \eqref{eq3.1.2}. Indeed, by the Cauchy-Schwarz inequality and the bound  \eqref{eq7.1.6},  we have
\begin{align}
    \label{eq7.1.8}
 & \int_{\Omega \times \bR^3} |(\psi \cdot \partial_{t}^{m} f) (\tau, x, p)| \, dx dp \\
 & \lesssim \|\partial_{t}^m f (\tau, \cdot)\|^2_{ L_2 (\Omega \times \bR^3) }   + \|\partial_{t}^m \phi (\tau, \cdot)\|^2_{ W^{1}_2 (\Omega) } 
    \lesssim_{\Omega} \cI_{||} (\tau). \notag
\end{align}


\textit{The electric field term.}
Due to  the orthogonality property \eqref{eq7.18}  of $C_i$, we have
\begin{align}
          \label{eq7.7}
 &   I_4 =   \text{const}  \int_{\bR^3} \big(e_{+} \frac{p_j}{p_0^{+}} C_i \sqrt{J^{+}} -  e_{-} \frac{p_j}{p_0^{-}} C_i \sqrt{J^{-}}\big)\, dp \\
 & \times \int_s^{t}  \int_{\Omega} (\partial_{t}^m \bE_j) \,  (\partial_{x_i} \partial_{t}^m  \phi) \, dx d\tau = 0. \notag
\end{align}

\textit{Estimate of the remaining terms.}
Repeating the argument in \eqref{eq3.1.14}--\eqref{eq3.1.10}, we conclude
\begin{align}
          \label{eq7.8}
  I_5+I_6   \lesssim \varepsilon_c \int_s^{t} \|\partial_{t}^m \phi\|^2_{ W^1_2 (\Omega) } \, d\tau  +  \varepsilon_c^{-1} \big(\int_s^{t} \cD_{||}  \, d\tau + \varepsilon \int_s^t \cD \, d\tau\big).
\end{align}


\textit{Preliminary estimate of $c$.} Combining the identity \eqref{eq2.3} with the estimates \eqref{eq7.3}--\eqref{eq7.8} and using the bounds \eqref{eq7.1.4} to handle the terms containing the test function $\phi$, we obtain
\begin{align}
    \label{eq7.20}
  &  \int_s^{t} \|\partial_{t}^m c\|^2_{ L_2 (\Omega) } \, d\tau
    \lesssim_{\Omega, \theta, r_3, r_4} (\eta (t) - \eta (s)) \\
   &  + \varepsilon_c \int_s^{t} \|\partial_{t}^m c\|^2_{ L_2 (\Omega) } \, d\tau
     +  \varepsilon_c \int_s^{t}   \| \partial_{t}^{m+1} \phi \|^2_{ W^1_2 (\Omega)} \, d\tau  \notag\\
 & + \varepsilon_c^{-1}  \big(\int_s^t \|\partial_{t}^m b\|^2_{ L_2 (\Omega) }  \,d\tau
    + \int_s^t \cD_{||}\,d\tau   +   \varepsilon \int_s^t \cD \, d\tau\big). \notag
\end{align}
Taking  $\varepsilon_c$ sufficiently small, we may drop the second term on the r.h.s.. Thus, to finish the proof of the desired estimate \eqref{eq6.1.4}, it suffices to show that
\begin{align}
    \label{eq7.15}
 &  \int_s^{t} \| \partial_{t}^{m+1}  \phi\|^2_{ W^1_2 (\Omega) } \, d\tau
  \lesssim_{\Omega, \theta, r_3, r_4}   \int_s^{t} \| \partial_{t}^{m} b\|^2_{ L_2 (\Omega) } \, d\tau \\
 &   + \int_s^{t} \cD_{||}  \, d\tau +  \varepsilon \int_s^t \cD \, d\tau.  \notag
\end{align}

\textbf{Step 2: estimate of $\partial_{t}^{m+1} \phi$.}
We  use the rearrangement \eqref{eq3.73} of the integral formulation \eqref{eq2.3} with
$$
    \psi (t, x, p) = \chi_6 (p) (\partial_t^{m+1} \phi) (t, x),
$$
which satisfies the SRBC due to the spherical symmetry of $\chi_6$ (see \eqref{eq6.24}).

\textit{The key term.}
Integrating by parts in $t$ in the integral $I_2$ and using the orthogonality property of $\chi_6$, and Eq. \eqref{eq7.1.3},  we obtain
\begin{align*}
 -(I_2 + I_3)& =  \int_s^{t} \int_{\Omega}  (\partial_{t}^{m+1} c)  (\partial_{t}^{m+1} \phi)  \, dx d\tau \\
 & =  -\int_s^{t} \int_{\Omega}  (\Delta_x \partial_{t}^{m+1}  \phi)   \, (\partial_{t}^{m+1} \phi)  \, dx d\tau
  +  \int_s^{t} \bigg(\partial_{t}^{m+1} \int_{\Omega}   \phi \, dx\bigg) \,  \bigg(\partial_{t}^{m+1} \int_{\Omega}  c \, dx\bigg) d\tau.
\end{align*}
Integrating by parts in $x$ and using the Neumann boundary condition, and the zero-average property of $\phi$ in \eqref{eq7.1.2}, we conclude
\begin{equation}
  \label{eq7.11}
     -(I_2 + I_3)  =   \int_s^{t}  \| \partial_{t}^{m+1} \nabla_x  \phi\|^2_{ L_2 (\Omega) } \, d\tau.
\end{equation}

\textit{The transport term.} We note that by oddness and the  Cauchy-Schwarz inequality, for any $\tilde \varepsilon_c \in (0, 1)$, one has
\begin{align}
 \label{eq7.12}
|I_1| & \lesssim \tilde \varepsilon_c \int_s^{t} \|\nabla_x \partial_{t}^{m+1} \phi\|^2_{ L_2 (\Omega) } \, d\tau\\
& +  \tilde \varepsilon^{-1}_c \big(\int_s^{t} \| \partial_{t}^m b\|^2_{ L_2 (\Omega) } \, d\tau  +  \int_s^{t} \cD_{||} \, d\tau\big). \notag
\end{align}

\textit{The remaining terms.} We observe that for the linear electric field term, by oddness, we have
$$
    I_4 = 0.
$$
Next, since $\chi_6$ is in the kernel of $L$, the term $I_5$ also vanishes.
Furthermore, repeating the argument in \eqref{eq3.1.10}, we get
\begin{align}
      \label{eq7.14}
    I_6  & \lesssim
     \tilde \varepsilon_c \int_s^{t}   \|\partial_t^{m+1}  \phi \|^2_{ L_2 (\Omega) }  d\tau +    \tilde \varepsilon_c^{-1}  \varepsilon \int_s^t \cD \, d\tau.
\end{align}


Finally, combining the integral identity \eqref{eq2.3} with the estimates \eqref{eq7.11}--\eqref{eq7.14}, we obtain
\begin{align}
    \label{eq7.15.1}
  & \int_s^{t} \| \partial_t^{m+1} \nabla_x  \phi\|^2_{ L_2 (\Omega) } \, d\tau \lesssim_{\Omega} \tilde \varepsilon_c  \int_s^{t} \| \partial_t^{m+1} \phi\|^2_{ W^1_2 (\Omega) } \, d\tau \\
   & + \tilde \varepsilon_c^{-1} \big(\int_s^{t} \| \partial_t^{m} b\|^2_{ L_2 (\Omega) } \, d\tau 
   + \int_s^{t} \cD_{||}  \, d\tau  + \varepsilon \int_s^t \cD \, d\tau\big). \notag
\end{align}
Using the Poincar\'e inequality and the vanishing average property of $\phi$ (see \eqref{eq7.1.2}), and choosing $\tilde \varepsilon_c$ sufficiently small, we may absorb the first term on the r.h.s. of \eqref{eq7.15.1} into the l.h.s.. Thus, \eqref{eq7.15} holds, and the desired bound \eqref{eq6.1.4} is valid.
\end{proof}

\section{Estimates of the electromagnetic field and macroscopic densities}
    \label{section 9}

The goal of this section is to prove the crucial bound \eqref{eq1.2.46}, which is stated rigorously in Proposition \ref{proposition E.4}. We follow the ``multi-step" argument sketched in Section \ref{section 4.11} (see \eqref{eq1.2.40}--\eqref{eq1.2.45}).

In this section,  $\eta$ is a function satisfying the bound \eqref{eq3.1.2}.


\begin{lemma}[estimates of the lower-order $t$-derivatives of $a^{\pm}$, cf. \eqref{eq1.2.40}]
            \label{lemma E.7}
            Under  Assumption \ref{assumption 3.5},
 there exists  a sufficiently small constant $\varepsilon_a =\varepsilon_a (\Omega, r_3, r_4, \theta) > 0$  such that
 \begin{align}
\label{eqE.5.7}
 &   \sum_{k=1}^{m-2} \int_s^{t} \|\partial_t^k [a^{+}, a^{-}]\|^2_{ L_2 (\Omega) } \, d\tau\\
&
   \lesssim_{\Omega, \theta, r_3, r_4 }  (\eta (t)  - \eta (s)) +     \varepsilon_a  \sum_{k=1}^{m-   4  }  \int_s^t   \|\partial_t^{ k} \bE\|^2_{ L_2 (\Omega) } \, d\tau \notag \\
&  +   \varepsilon_a^{-1}  \bigg(\sum_{k=1}^{m} \int_s^{t} \| \partial_t^{k}  b\|^2_{ L_2 (\Omega) } \, d\tau  +      \sum_{k=1}^{m-2} \int_s^{t} \|\partial_t^k  c\|^2_{ L_2 (\Omega) } \, d\tau  +      \int_s^t \cD_{||} \, d\tau +  \varepsilon \int_s^t \cD \, d\tau\bigg). \notag
 \end{align}
\end{lemma}

\begin{proof}
We fix an arbitrary integer $1 \le k \le m-2$
and  follow the argument of Lemma \ref{lemma 3.2}   closely.

\textit{Test function.}
As in the proof of Lemma \ref{lemma 3.2},  the  Neumann problem
\begin{equation}
            \label{eq9.1.3}
\left\{\begin{aligned}
 & -  \Delta_x \phi =  \sqrt{M_{+}} a^{+}
 -  \sqrt{M_{-}} a^{-},\\
 & \frac{\partial \phi }{\partial n_x} = 0 \, \, \text{on} \, \partial \Omega,
\end{aligned}
\right.
\end{equation}
has a unique strong solution $\phi \in W^2_2 (\Omega)$ satisfying
$$
    \int_{\Omega} \phi \, dx = 0,
$$
and, in addition, the following elliptic estimate is valid: 
\begin{equation}
				\label{eq9.1.4}
\|\phi\|_{ W^{  2}_2 (\Omega) } \lesssim_{\Omega} \| \sqrt{M_{+}} a^{+} - \sqrt{M_{-}} a^{-} \|_{  L_2 (\Omega) }.
\end{equation}

Next, we set
\begin{align}
    \label{eq9.1.3.1}
 \psi (t, x, p) =    (\sqrt{J^{+}}, -\sqrt{J^{-}}) p_i  \partial_{x_i} \partial_t^k \phi (t, x)
\end{align}
and note that it satisfies the SRBC (cf. \eqref{eq3.2.11}).
 Let $I_1$--$I_6$ be the terms in the integral identity \eqref{eq2.3} with such $\psi$.


\textit{The key term.} Proceeding as  in  \eqref{eq3.1.12}--\eqref{eq3.1.22} and noticing that $\kappa_1$ is absent in the definition of $\psi$ (see \eqref{eq9.1.3.1}), we get for any $\varepsilon_a \in (0, 1)$ (cf. \eqref{eq3.2.4.0}--\eqref{eq3.2.4}),
   \begin{align}
        \label{eq9.1.5}
 I_1 & \geq  (k_b T  - \varepsilon_a )\int_s^{t}  \|\sqrt{M_{+}} \partial_t^k a^{+}  - \sqrt{M_{-}}\partial_t^k a^{-}\|^2_{ L_2 (\Omega) } \, d\tau \\
 & -  N \varepsilon_a^{-1} \bigg(\int_s^{t} \|\partial_t^k c \|^2_{ L_2 (\Omega) }  \, d\tau   +  \int_s^{t} \cD_{||}  \, d\tau\bigg).\notag
\end{align}

Estimates of $I_j, j\neq 1, 4$.
The integrals $I_2$, $I_3$, and  $I_6$  can be treated  as in the proof of Lemma \ref{lemma 3.2} (see \eqref{eq3.2.50}--\eqref{eq3.2.9} and \eqref{eq3.2.8}, respectively), whereas $I_5$ is estimated as in \eqref{eq3.1.14}. Furthermore, one minor difference with \eqref{eq3.2.50} is that by the oddness and macro-micro decomposition, there is an additional  term involving $(1-P) f$. Thus, we have 
\begin{align}
        \label{eq9.9}
&   I_2+I_3+I_5+I_6\\
& \lesssim_{\Omega, \theta, r_3, r_4} (\eta (t) - \eta (s))  +  \varepsilon_a  \int_s^{t}  \| \sqrt{M_{+}} \partial_t^k a^{+}  - \sqrt{M_{-}} \partial_t^k a^{-}\|^2_{ L_2 (\Omega) } \, d\tau \notag\\
& +  \varepsilon_a^{-1}   \bigg(\int_s^{t} \|\partial_t^k [b, c]\|_{ L_2 (\Omega) }^2 \, d\tau 
   + \int_s^{t} \cD_{||}\, d\tau   +   \varepsilon \int_s^t \cD \, d\tau\bigg). \notag
\end{align}

\textit{Estimate of the electric field term $I_4$.} For
\begin{itemize}
    \item $4 \le k \le m-2$, we set $j = 2$, 
    \item $k \in \{1, 2, 3\}$, we set $j = 0$.
\end{itemize}
Proceeding as in \eqref{eq3.1.23} and integrating by parts in the temporal variable $j$ times, we obtain
\begin{align}
            \label{eq9.6}
&   \text{(const)}   I_4   =  \int_s^{t} \int_{\Omega} (\partial_t^{k } \bE)  \cdot (\nabla_{x} \partial_t^{k} \phi)  \, dx d\tau =
\big((I_{4, 1} (t) - I_{4, 1} (s)) + I_{4, 2}\big), \\
&I_{4, 1} (\tau) = 1_{j > 0}  \sum_{l = 1}^j (-1)^{l-1} \int_{ \Omega} (\partial_t^{k-l} \bE) (\tau, x) \cdot \nabla_{x} \partial_t^{k+l-1} \phi (\tau, x)\, dx, \notag \\
& I_{4, 2}  = (-1)^j \int_s^{t} \int_{\Omega} (\partial_t^{k - j } \bE)  \cdot (\nabla_{x} \partial_t^{k+j} \phi)  \, dx d\tau.\notag
\end{align}
As in \eqref{eq3.1.20}, by the Cauchy-Schwarz inequality and the elliptic estimate \eqref{eq9.1.4}, we may replace $I_{4, 1} (\tau)$ with $\eta (\tau)$, where $\eta$ is a function satisfying  \eqref{eq3.1.2}.

Furthermore, by the Cauchy-Schwarz inequality and the argument of \eqref{eq3.2.5.0}, one has
\begin{align}
     \label{eq9.8}
    &  |I_{4, 2}| \lesssim   \varepsilon_a \int_s^{t}  \|\partial_t^{k-j} \bE\|^2_{ L_2 (\Omega)} \, d\tau  \\
    & +   \varepsilon_a^{-1} \int_s^{t} \big(\|\partial_t^{k+j-1} b\|^2_{ L_2 (\Omega)}  + \|\partial_t^{k+j-1} (1-P) f\|^2_{ L_2 (\Omega \times \bR^3) }\big) \, d\tau. \notag
 \end{align}
We note that in the case when $k = m-2$, we have $j=2$, and hence $k+j-1=m-1$. Thus, by using integration by parts in $t$  and the continuity equation \eqref{eqE.2.15}, we reduced    by $2$  the order of the $t$-derivatives of $\bE$  on the r.h.s. of the estimate of $\partial_t^k a^{\pm}$, while ensuring that the derivative terms $\partial_t^k [b, c, (1-P)f]$ remain below the highest order.

Finally, gathering \eqref{eq9.9}--\eqref{eq9.8} and choosing $\varepsilon_a$ sufficiently small, we obtain the desired estimate \eqref{eqE.5.7} with the l.h.s. replaced with
$$
    \sum_{k=1}^{m-2} \int_s^t \|\sqrt{M_{+}} \partial_t^k a^{+}  - \sqrt{M_{-}}\partial_t^k a^{-}\|^2_{ L_2 (\Omega) } \, d\tau
$$
(cf. \eqref{eq3.2.51}).
Combining this with the estimate of the weighted average with the $+$ sign in \eqref{eq3.2.1} in Lemma \ref{lemma 3.2}, we prove  \eqref{eqE.5.7}.
\end{proof}

\begin{lemma}[Estimate of $\bB$, cf. \eqref{eq1.2.42}]
			\label{lemma E.1}
Invoke Assumption \ref{assumption 3.5}. Then, for any  $k \le m-3$, we have 
\begin{align}
\label{eqE.1.0}
 &  \int_s^t \| \partial_t^k \bB\|^2_{ L_2 (\Omega) }   \, d\tau \lesssim_{\Omega, \theta, r_3, r_4 }   \int_s^t \|\partial_t^{k+1} [a^{+}, c]\|^2_{ L_2 (\Omega) }    \, d\tau
  +   \int_s^t \|\partial_t^{k+2} b\|^2_{ L_2 (\Omega) }   \, d\tau  \\
 &+  \int_s^t  \int_{ \gamma_{+}  } \frac{|p \cdot n_x|^2}{ |p_0^{+}|^2 }  |\partial_t^{k+1} f^{+}|^2 \sqrt{J^{+}}  \,   dS_x dp d\tau\notag \\
    & +  \int_s^t \cD_{||}  \, d\tau + \varepsilon \int_s^t \cD \, d\tau. \notag
\end{align}
\end{lemma}

\begin{proof}
\textbf{Step 1: duality argument.}
By the vanishing flux property in \eqref{eq4.10}, the function $\bB$ satisfies the assumption of Theorem 4.3 in \cite{AS_13} with $\bm{u} = \bB$, and, hence, the system
\begin{equation}
           \label{eqE.1.2}
\left\{\begin{aligned}
 & \nabla_x \times \bm{w} =   \bB,\\
 & \nabla_x \cdot \bm{w} = 0, \\
 &\bm{w} \times n_x = 0  \, \, \text{on} \, \partial \Omega,
\end{aligned}
\right.
\end{equation}
has a unique strong solution $\bm{w} \in W^1_2 (\Omega)$, and
\begin{equation}
        \label{eqE.1.3}
    \|\bm{w}\|_{W^1_2 (\Omega) } \lesssim_{\Omega}  \|\bB\|_{L_2 (\Omega) }.
\end{equation}
Then, by using integration by parts, the boundary condition $\bm{w} \times n_x = 0$, and the  Amper\'e-Maxwell law (see \eqref{eq36.16}), we conclude
\begin{align}
            \label{eqE.1.3.1}
    &\int_{\Omega} |\partial_t^k \bB|^2\, dx   = \int_{\Omega} (\partial_t^k \bB) \cdot (\nabla_x \times \partial_t^k \bm{w})  \, dx\\
   & =\int_{\Omega} (\nabla_x \times \partial_t^k \bB) \cdot (\partial_t^k \bm{w}) \, dx \notag\\
    &  =  \int_{\Omega} (\partial_t^{k+1} \bE) \cdot  (\partial_t^k \bm{w}) \, dx  + \int_{\Omega}  (\partial_t^{k} \bm{j}) \cdot (\partial_t^k \bm{w}) \, dx =: \cI_1 + \cI_2. \notag
\end{align}


\textbf{Step 2:  estimate of $\cI_1$.}  We  derive an estimate of $\partial_t^{k+1} \bE$ from the Landau equation  (cf. Lemma 9 in \cite{GS_03}).
We set 
$$
    \psi (x, p) = (\partial_t^k \bm{w} \cdot p \sqrt{J^{+}}, 0).
$$
We will apply a variant of Green's identity \eqref{eq30.1.0} in Lemma \ref{lemma 30.1}. Let us check the conditions of the lemma.  For the sake of simplicity, we temporarily set all the physical constants to $1$.
\begin{itemize}
\item Since $\partial_t^k f^{\pm}$ is a finite energy solution to the Landau equation differentiated $k$ times in $t$, one can rewrite the identity satisfied by $\partial_t^{k+1} f^{\pm}$  as a  kinetic Fokker-Planck equation \eqref{eq27.4.18} with 
 \begin{align}
 \label{eqE.1.20}
    g=(f^{+} + f^{-}) \sqrt{J}
 \end{align}
 and  $\eta \in L_2 ((0, T) \times \Omega) W^{-1}_2 (\bR^3)$.
See the details in the proof of Proposition \ref{proposition 4.0}.
\item Due to the assumption \eqref{eq3.5.1} and \eqref{eq14.4.7} in Lemma \eqref{eq14.4.7}, the regularity conditions 
\eqref{eq27.4.35} hold
 with $g$  given by \eqref{eqE.1.20}. 
\item The condition \eqref{eq27.4.16} is  valid due to the smallness of $\varepsilon$ in \eqref{eq3.5.1} and \eqref{eq14.4.7}. 

\item Since $\partial_t^l \bm{w} \in L_2 ((s, t)) W^1_2 (\Omega), l \le m-3,$ the conditions in \eqref{eq30.1.4} hold with $\psi$ in place of $\phi$. 
\end{itemize}
Next, applying  Green's formula in \eqref{eq30.1.0} and integrating by parts in $t$, we have   
\begin{align}
\label{eqE.1.5}
 &
 \underbrace{  \frac{1}{k_b T}  \int_s^t \int_{ \Omega \times \bR^3 } (\partial_t^{k+1} \mathbf{E})
  \big( e_{+}\frac{p}{p_0^{+}}   \sqrt{J^{+}} \psi^{+}
 - e_{-} \frac{p}{p_0^{-}} \sqrt{J^{-}} \psi^{-}\big) \, dx dp d\tau }_{= J_1}
\\
 &    =   \underbrace{ \int_s^t \int_{  \Omega \times \bR^3 }  \psi \cdot \partial_t^{k + 2} f \, dx dp d\tau}_{J_2}    \notag\\
  &
  -  \underbrace{\int_s^t \int_{ \Omega \times \bR^3 }
  \bigg(\frac{p}{p_0^{+}} \cdot   (\nabla_x \psi^{+})  (\partial_t^{k+1} f^{+})
  + \frac{p}{p_0^{-}} \cdot   (\nabla_x \psi^{-})  (\partial_t^{k+1} f^{-})\bigg)
  \, dx dp d\tau}_{ = J_3 } \notag \\
& \underbrace{+\int_s^t \int_{ \gamma_{+} \cup   \gamma_{-}  }
\bigg((\partial_t^{k+1} f^{+}) (\frac{p}{p_0^{+}} \cdot n_x)\psi^{+} +(\partial_t^{k+1} f^{-}) (\frac{p}{p_0^{-}} \cdot n_x) \psi^{-}\bigg)
\, dS_x dp d\tau }_{=  J_4 }  \notag \\
 & \underbrace{  +  \int_s^t\int_{  \Omega  }  \langle L \partial_t^{k+1} f, \psi  \rangle \, dx  d\tau}_{ = J_5}  \underbrace{ -  \int_s^t \int_{\Omega } \langle \partial_t^{k+1} H, \psi  \rangle \, dx  d\tau }_{= J_6},  \notag
 \end{align}
 where $H$ is defined in \eqref{eq2.3.1}.

\textit{The key term.} By symmetry, we have
$$
    J_1 = \text{const}  \int_s^t \int_{ \Omega } (\partial_t^{k+1} \mathbf{E}) \cdot (\partial_t^k \bm{w}) \, dx d\tau.
$$

\textit{Estimate of the $t$-derivative term.}
 Using the macro-micro decomposition and oddness, we get
\begin{equation}
\label{eqE.1.6}
        J_2 \lesssim (\|\partial_t^{k+2} b\|_{ L_2 ((s, t) \times \Omega)} + \| (1-P) \partial_t^{k+2} f\|_{ L_2 ((s, t) \times  \Omega \times \bR^3)}) \|\partial_t^k \bm{w}\|_{L_2 ((s, t) \times \Omega)}.
\end{equation}

\textit{Estimate of the transport term.} By the macro-micro decomposition and the Cauchy-Schwarz inequality,
\begin{align}
        \label{eqE.1.7}
   & J_3   \lesssim (\|\partial_t^{k+1} [a^{+}, c]\|_{ L_2 ((s, t) \times \Omega)} \\
   & +  \| (1-P) \partial_t^{k+1} f\|_{ L_2 ((s, t) \times \Omega \times \bR^3)}) \, \|\partial_t^k\bm{w}\|_{ L_2 (s, t)  W^1_2 (\Omega)}. \notag
\end{align}

\textit{Kinetic boundary term $J_4$.}
By the decomposition \eqref{eq3.1.4}
and  the boundary condition $\bm{w} \times n_x = 0$, we get 
\begin{equation}
        \label{eqE.1.8}
    J_4 = \int_s^t  \int_{\partial \Omega \times \bR^3} \frac{|p \cdot n_x|^2}{p_0^{+}} (\partial_t^{k+1} f^{+}) \, \sqrt{J^{+}} (\partial_t^k \bm{w} \cdot n_x) \, dS_x dp d\tau.
\end{equation}
Hence, by the Cauchy-Schwarz inequality, the trace theorem for $W^1_2 (\Omega)$ functions, and the SRBC, we have
\begin{equation}
            \label{eqE.1.9}
    J_4 \lesssim  \|\partial_t^k \bm{w}\|_{ L_2 (s, t) W^1_2 (\Omega)}  \bigg(\int_s^t \int_{ \gamma_{+}  }
  \frac{|p \cdot n_x|^2}{ |p_0^{+}|^2 }  |\partial_t^{k+1} f^{+}|^2  \sqrt{J^{+}}     dS_x dp d\tau\bigg)^{1/2}.
\end{equation}

Next, by the fact that  $L$ is a symmetric operator and the \cs, we get
\begin{align}
 \label{eqE.1.10}
    J_5  \lesssim \|(1-P) \partial_t^{k+1} f\|_{ L_2 ((s, t) \times \Omega \times \bR^3)} \|\partial_t^k \bm{w}\|_{L_2 ((s, t) \times \Omega)}.
\end{align}
Furthermore,  since $k + 1 < m$, by using the Cauchy-Schwarz inequality, the estimates \eqref{eq12.A.2.3} and \eqref{eq12.A.3.2} (cf. \eqref{eq3.1.10}), and the smallness assumption \eqref{eq3.5.1}, we obtain   
\begin{align}
  \label{eqE.1.11}
    J_6  \lesssim  \sqrt{\varepsilon} \big(\int_s^t \cD \, d\tau\big)^{1/2}  \|\partial_t^k \bm{w}\|_{L_2 ((s, t) \times \Omega)}.
\end{align}
Gathering \eqref{eqE.1.5}--\eqref{eqE.1.11}, we obtain
\begin{align}
\label{eqE.1.12}
    &\cI_1 =  \int_s^t \int_{\Omega}  (\partial_t^{k+1} \bE) \cdot (\partial_t^k \bm{w}) \, dx  d\tau \lesssim  \sqrt{\xi} \|\partial_t^k \bm{w}\|_{ L_2 ((s, t)) W^1_2 (\Omega) },   \\
    & 
    \xi:= \int_s^t \big(\|\partial_t^{k+1} [a^{+}, c]\|^2_{ L_2 (\Omega)}
    + \|\partial_t^{k+2} b\|^2_{  L_2 (\Omega)}\big) \, d\tau 
    + \int_s^t \cD_{||} \, d\tau  \notag \\
    & +\int_s^t \int_{ \gamma_{+}  }
  \frac{|p \cdot n_x|^2}{ |p_0^{+}|^2 }  |\partial_t^{k+1} f^{+}|^2 \sqrt{J^{+}}  \,   dS_x dp d\tau  + \varepsilon \int_s^t \cD \, d\tau.
   \notag
\end{align}

\textbf{Step 3:  estimate of $\cI_2$.}
By using the macro-micro decomposition and the definition of $b$ in \eqref{eq6.19}, we get
\begin{align}
    \label{eqE.1.18}
    \bm{j}_i & = \int   (e_{+} \frac{p_i}{p_0^{+}} f^{+} \sqrt{J^{+}} -  e_{-}\frac{p_i}{ p_0^{-} } \sqrt{J^{-}}) f^{-} \, dp\\
    & =   \lambda_i b_i
    + \int   (e_{+}\frac{p_i}{p_0^{+}}  \sqrt{J^{+}},  - e_{-}\frac{p_i}{ p_0^{-} } \sqrt{J^{-}}) \cdot (1-P) f \, dp.\notag
\end{align}
Furthermore, by the identity \eqref{eq3.1.13} and the neutrality condition \eqref{eq0},
\begin{align}
    \label{eqE.1.19}
   \lambda_i :& =  \kappa_1 \bigg(e_{+}   \int_{\bR^3}    \frac{p_i^2}{p_0^{+}}  J^{+} \, dp
    - e_{-}   \int_{\bR^3}    \frac{p_i^2}{p_0^{-}}  J^{-} \, dp\bigg)\\
&
     = \kappa_1 k_b T  (e_{+} \int_{\bR^3} J^{+} \, dp - e_{-} \int_{\bR^3} J^{-} \, dp) = 0. \notag
\end{align}
Therefore, by the Cauchy-Schwarz inequality,
\begin{equation}
            \label{eqE.1.13}
        \cI_2 \lesssim  \|(1-P) \partial_t^{k} f\|_{ L_2 ((s, t) \times \Omega \times \bR^3)} \|\partial_t^k \bm{w}\|_{ L_2 ((s, t) \times \Omega)}.
\end{equation}

Next, gathering  \eqref{eqE.1.3.1}, \eqref{eqE.1.12}--\eqref{eqE.1.13}, we obtain
\begin{align*}
  \int_s^t \| \partial_t^k \bB\|^2_{ L_2 (\Omega) } \, d\tau \lesssim_{\Omega} \sqrt{\xi}  \|\partial_t^k \bm{w}\|_{ L_2 ((s, t))  W^1_2 (\Omega)}.
\end{align*}
Due to the elliptic estimate \eqref{eqE.1.3}, we conclude
\begin{align*}
   \int_s^t  \| \partial_t^k \bB \|^2_{ L_2 (\Omega) } \, d\tau
   \lesssim_{\Omega}  \xi,
\end{align*}
which finishes the proof of the desired assertion \eqref{eqE.1.0}.
\end{proof}

\begin{lemma}[weighted trace estimate, cf. \eqref{eq1.2.43}]
        \label{lemma E.2}
 Under Assumption \ref{assumption 3.5},   for any $j\in \{2, \ldots, m/2\}$ and $k +1 \in \{j, m-j\}$,  any $\varepsilon_1 \in (0, 1)$, 
    \begin{align}
            \label{eqE.2.0}
&  \int_s^{t} \int_{ \gamma_{+}  }
   \frac{|p \cdot n_x|^2}{ |p_0^{\pm}|^2 }    |\partial_t^{k+1} f^{\pm}|^2 \sqrt{J^{\pm}}  \,   dS_x dp d\tau \\
  & \lesssim_{\Omega, \theta, r_3, r_4 } (\eta (t) - \eta (s))  +   \int_s^{t} \|\partial_t^{k+1} f\|^2_{ L_2 (\Omega \times \bR^3) } \, d\tau 
  +  \varepsilon_1 \int_s^{t} \|\partial_t^{k+1-j} \bE\|^2_{ W^1_2 (\Omega) } \, d\tau   \notag\\
& +  \varepsilon_1^{-1}  \bigg(\sum_{l = k+j}^{k+j+1} \int_s^{t} \|\partial_t^{l} b\|^2_{ L_2 (\Omega) } \, d\tau  + \int_s^{t} \|\partial_t^{k+j+1} c\|^2_{ L_2 (\Omega) } \, d\tau 
+ \int_s^{t} \cD_{||}  \, d\tau + \sqrt{\varepsilon} \int_s^t \cD \, d\tau\bigg).\notag
\end{align}
\end{lemma}

\begin{remark}
A similar estimate was established for the linear non-relativistic kinetic Fokker-Planck equation (see  Proposition 4.3 in \cite{S_22}). The novelty of Lemma \ref{lemma E.2} lies in handling the derivative loss in  $a^{\pm}$-terms via integration by parts in $t$ in the problematic integral involving both $\bE$ and a test function at the cost of losing derivatives in $b$, $c$, and $(1-P) f$ terms, which are `good' up to the highest order due to the estimates \eqref{eq3.3.1} and \eqref{eq6.1.4}.
\end{remark}



\begin{proof}
Since the values of physical constants do not play any role in the argument, we set all these constants to $1$.

 Let $\nu$  be a Lipschitz  vector field on $\bR^3$ such that $\nu (x) = n_x$ on $\partial \Omega$. For the construction, see the proof of Lemma \ref{lemma 30.1}. 

 Next, we introduce a function 
 \begin{equation}
    \label{eqE.2.14}
         \zeta (x, p) =  \sqrt{J} \frac{p \cdot \nu (x)}{ p_0 } 1_{ p \cdot \nu (x) > 0}.
  \end{equation}
We note that 
  \begin{equation}
            \label{eqE.2.2}
    |\zeta| +  |\nabla_{x, p} \zeta|  \lesssim_{\Omega} \sqrt{J}  \, \, \text{a.e.}.
\end{equation}

We apply the energy identity \eqref{eq30.1.7} in Lemma \ref{lemma 30.1} with $f$ and  $\phi$ replaced with $\partial_t^{k+1} f$ and $\zeta$, respectively. We already verified that $\partial_t^{k+1} f$ satisfies the conditions of Lemma \ref{lemma 30.1} (see p. \pageref{eqE.1.20} in the proof of Lemma \ref{lemma E.1}).
Then, by the aforementioned energy identity and  the fact that $\nu (x) = n_x$ on $\partial \Omega$, we have
\begin{align}
            \label{eqE.2.3}
&  \frac{1}{2} \int_s^{t} \int_{ \gamma_{+}  }
  \frac{|p \cdot n_x|^2}{ |p_0|^2 } |\partial_t^{k+1} f|^2 \sqrt{J}     dS_x dp d\tau \\
  &    = -  \frac{1}{2}\int_{\Omega \times \bR^3}  \big(|\partial_t^{k+1} f (t, x, p)|^2   - |\partial_t^{k+1} f (s, x, p)|^2 \big)   \zeta (x, p) \, dx dp \notag \\
   &
  \underbrace{+  \frac{1}{2} \int_{s}^t \int_{\Omega \times \bR^3}  \nabla_x \zeta \cdot  \frac{p}{p_0} |\partial_t^{k+1} f|^2   \, dx dp  d\tau}_{ = J_1} \notag\\
 & \underbrace{ +   \int_{s}^t \int_{\Omega \times \bR^3}  (\partial_t^{k+1} f^{+}  - \partial_t^{k+1} f^{-}) \sqrt{J} \frac{p}{p_0} \cdot (\partial_t^{k+1} \bE) \, \zeta \,  dx dp d\tau  }_{= J_2 } \notag\\
 & \underbrace{  -   \int_{s}^t \int_{\Omega } \langle L (1-P) \partial_t^{k+1} f, (\partial_t^{k+1} f)  \, \zeta \rangle \, dx  d\tau }_{ = J_3} \notag \\
& \underbrace{+  \int_{s}^t \int_{\Omega}  \langle \partial_t^{k+1} H,  \zeta \partial_t^{k+1} f   \rangle \,  dx d\tau}_{= J_4 }, \notag
  \end{align}
where $H$ is a function defined in \eqref{eq2.3.1}.

\textit{The $t$-boundary term.} Since $k + 1 < m$, by using  the estimate \eqref{eqE.2.2}, we may replace the  first term on the r.h.s. of \eqref{eqE.2.3} with $\eta (t) - \eta (s)$, where $\eta$ is a function satisfying  the bound  \eqref{eq3.1.2}.

\textit{Estimate of $J_1$.}
By \eqref{eqE.2.2}, we have
\begin{equation}
        \label{eqE.2.4}
    J_1 \lesssim_{\Omega} \int_s^{t} \|\partial_t^{k+1} f\|^2_{ L_2 (\Omega \times \bR^3) } \, d\tau.
\end{equation}

\textit{Estimate of $J_2$.}
Integrating by parts $j$ times in the $t$ variable and using the macro-micro decomposition, we get (cf. \eqref{eq9.6})
\begin{align}
        \label{eqE.2.5}
  J_2  =:(J_{2, 1} (t) - J_{2, 1} (s)) + (-1)^j (J_{2, 2} + J_{2, 3}),
\end{align}
where
\begin{align*}
& J_{2, 1} (\tau) =  \sum_{l=1}^j (-1)^{l-1} \int_{\Omega \times \bR^3}  \partial_t^{k+l} \big(f^{+} (\tau, x, p) - f^{-} (\tau, x, p)\big)  \, \sqrt{J} \frac{p}{p_0} \cdot (\partial_t^{k+1-l} \mathbf{E} (\tau, x, p)) \, \zeta \, dx dp, \notag \\
       & J_{2, 2} =     \int_s^t \int_{\Omega \times \bR^3}  \partial_t^{k+1+j}(a^{+}  \chi_1^{+} - a^{-} \chi_2^{-})   \sqrt{J}  \frac{p}{p_0}   \cdot (\partial_t^{k+1-j} \mathbf{E}) \, \zeta \, dx dp d\tau, \notag \\
       & J_{2, 3}=  \int_s^t \int_{\Omega \times \bR^3}  \partial_t^{k+1+j} \big(b_i \chi_{i+2}  + c \chi_6 + (1-P) f\big) \cdot  (1, - 1)  \, \sqrt J \frac{p}{p_0}  \cdot (\partial_t^{k+1-j} \mathbf{E}) \, \zeta\, dx dp d\tau.
\end{align*}

We note that, as in the treatment of the $t$-boundary term, by the Cauchy-Schwarz inequality and \eqref{eqE.2.2}, we may replace $J_{2, 1} (t) - J_{2, 1} (s)$ with $\eta (t)-\eta (s)$, where $\eta$ satisfies \eqref{eq3.1.2}.

To estimate $J_{2, 2}$, we note that by the continuity equation \eqref{eqE.2.15}, we may replace $\partial_t^{k+1+j} a^{\pm} $ with $- \nabla_x \cdot \partial_t^{k+j} \bm{j^{\pm}}$, and,  due to the boundary condition \eqref{eqE.2.16}, integrating by parts in $x$ gives
 \begin{align}
    \label{eqE.2.17}
 J_{2, 2} =  \int_s^t \int_{\Omega \times \bR^3}   \, \partial_t^{k+j} (\bm{j}^{+} \chi_1^{+} - \bm{j}^{-}\chi_1^{-}) \cdot \nabla_{x} \partial_t^{k+1-j} (p \cdot \bE \zeta ) \,  p_0^{-1}    \sqrt{J}  \, dx dp d\tau.
 \end{align}
Then,  by the Cauchy-Schwarz inequality,  \eqref{eqE.2.17}, 
\eqref{eqE.2.15.1},  and the bound of $\zeta$ in \eqref{eqE.2.2}, we get, for any $\varepsilon_1 \in (0, 1)$,
\begin{align}
        \label{eqE.2.6}
    J_{2, 2}  & \lesssim_{\Omega} \varepsilon_1 \int_s^{t} \|\nabla_x \partial_t^{k+1-j} \bE\|^2_{L_2 (\Omega) } \, d\tau \\
    & + \varepsilon^{-1}_1 \int_s^{t} \big(\|\partial_t^{k+j} b\|^2_{L_2 (\Omega) }  + \|(1-P) \partial_t^{k+j} f\|^2_{L_2 (\Omega \times \bR^3) }\big) \, d\tau. \notag
\end{align}

We estimate the last term in \eqref{eqE.2.5}, $J_{2, 3}$, via the Cauchy-Schwarz inequality:
\begin{align}
      \label{eqE.2.7}
      J_{2, 3}  & \lesssim_{\Omega} \varepsilon_1 \int_s^{t} \|\partial_t^{k+1-j} \bE\|^2_{L_2 (\Omega) } \, d\tau \\
    & +  \varepsilon^{-1}_1 \int_s^{t} \big(\|\partial_t^{k+1+j} [b, c]\|^2_{L_2 (\Omega) }  + \|(1-P) \partial_t^{k+1+j} f\|^2_{L_2 (\Omega \times \bR^3) }\big) \, d\tau. \notag
\end{align}

\textit{Estimate of $J_3$.} 
By Lemma 7 in \cite{GS_03}, the bound \eqref{eqE.2.2},  and the macro-micro decomposition, one has 
\begin{align}
\label{eqE.2.22}
 J_3 & \lesssim  \int_s^t \|\partial_t^{k+1} (1-P)  f\|^2_{ L_2 (\Omega) W^1_2 (\bR^3) } \, d\tau + \int_s^t \|\zeta \partial_t^{k+1} f \|^2_{ L_2 (\Omega) W^1_2 (\bR^3) } \, d\tau \\
&\lesssim \int_s^t \|\partial_t^{k+1} (1-P)  f\|^2_{ L_2 (\Omega) W^1_2 (\bR^3) } \, d\tau 
+  \int_s^t \|\partial_t^{k+1} [a^{\pm}, b, c]\|^2_{ L_2 (\Omega) } \, d\tau. \notag
\end{align}

\textit{Estimate of $J_4$.} Since $k+1 \le m-2$, using the bounds \eqref{eq12.A.2.1}  in Lemma \ref{lemma 12.A.2} and \eqref{eq12.A.3.1} in Lemma \ref{lemma 12.A.3}, we get
\begin{align}
    \label{eqE.2.21}
    J_4 \lesssim \sqrt{\varepsilon} \int_s^t \cD \, d\tau.
\end{align}

Finally, gathering \eqref{eqE.2.3}--\eqref{eqE.2.21}, we obtain the desired estimate \eqref{eqE.2.0}.
\end{proof}

\begin{lemma}[estimate of $a^{\pm}$ and $\bE$, cf. \eqref{eq1.2.45}]
  \label{lemma 9.2}
Under Assumption \ref{assumption 3.5},  we have
\begin{align}
\label{eq9.2.0}
 &    \int_s^{t} \|  \sqrt{M_{+}} a^{+} - \sqrt{M_{-}} a^{-}\|^2_{ L_2 (\Omega) } \, d\tau +  \int_s^{t} \|\bE\|^2_{ W^1_2 (\Omega) } \, d\tau
   \lesssim_{\Omega, \theta, r_3, r_4 }  (\eta (t)  - \eta (s))  \\
   & +  \int_s^{t} \|\sqrt{M_{+}} a^{+}  + \sqrt{M_{-}} a^{-}\|^2_{ L_2 (\Omega) } \, d\tau +   \int_s^{t}\|  [b, c]\|^2_{ L_2 (\Omega) } \, d\tau  \notag \\
    & +         \int_s^{t} \| (1-P)  f\|^2_{ L_2 (\Omega \times \bR^3) } \, d\tau +      \int_s^{t}\|\partial_t \bB\|^2_{ L_2 (\Omega) } \, d\tau + \varepsilon \int_s^t \cD \, d\tau.\notag
     \end{align}
 \end{lemma}

\begin{proof}
We follow the proof of Lemma \ref{lemma E.7} by using the integral identity \eqref{eq2.3} with the  test function defined by \eqref{eq9.1.3} and \eqref{eq9.1.3.1} with $k=0$. Due to the estimates \eqref{eq9.1.5}-\eqref{eq9.9}, we only need to handle the integral  $I_4$ (cf. the first equality in \eqref{eq9.6}). To this end, we use a different argument based on a  Helmholtz-type decomposition.

\textit{Helmholtz-type decomposition.}
First, we note that by the definition of $a^{\pm}$ (see \eqref{eq6.18}), the charge density $\rho$ (see \eqref{eq36.17}) satisfies  the identity
\begin{equation}
			\label{eq9.2.12}
    \rho  =   e^{+} \sqrt{M_{+}} a^{+}  - e^{-} \sqrt{M_{-}} a^{-}.
\end{equation}
Therefore, the electric field $\bE$ can be decomposed as  follows:
\begin{align}
    \label{hodge}
    \bE = \bE_1 + \nabla_x \xi,
\end{align}
where $\bE_1$ satisfies
\begin{equation}
			\label{eq9.2.2}
\left\{\begin{aligned}
 & \nabla_x \cdot \bE_1  = 0, \\
 & \nabla_x \times \bE_1 = -\partial_t \bB, \\
 & \bE_1 \times n_x = 0 \, \, \text{on} \, \partial \Omega,\\
\end{aligned}
\right.
\end{equation}
and $\xi \in W^2_2 (\Omega)$ is the strong solution to 
\begin{equation}
			\label{eq9.2.3}
\left\{\begin{aligned}
 & \Delta \xi =    e_{+} \sqrt{M_{+}} a_{+}  - e_{-} \sqrt{M_{-}} a_{-},\\
 & \xi = 0  \, \, \text{on} \, \partial \Omega,
\end{aligned}
\right.
\end{equation}

\textit{Estimates of $E_1$ and $\xi$.}
By the div-curl estimate for vector fields orthogonal to  the boundary $\partial \Omega$ (see \eqref{eq4.11}), we get
\begin{equation}
			\label{eq9.2.4}
   \|\bE_1\|_{W^1_2 (\Omega) } \lesssim_{\Omega} \|\partial_t \bB\|_{ L_2 (\Omega) },
\end{equation}
and by the standard elliptic estimate, 
\begin{equation}
			\label{eq9.2.5}
    \|\xi\|_{W^2_2 (\Omega) } \lesssim_{\Omega} \|   e_{+} \sqrt{M_{+}} a^{+}  - e_{-} \sqrt{M_{-}} a^{-}\|_{ L_2 (\Omega) }.
\end{equation}

\textit{Estimate of $I_4$} (see \eqref{eq9.6}). 
 By the decomposition \eqref{hodge}, we have
\begin{equation}
			\label{eq9.2.6}
(\text{const})\, I_4 =\int_s^{t} \int_{\Omega} \bE_1 \cdot \nabla_x \phi \, dx d\tau
    + \int_s^{t} \int_{\Omega}  \nabla_x \xi \cdot \nabla_x \phi \, dx d\tau  =: I_{4, 1} + I_{4, 2}.
\end{equation}
 First, by the Cauchy-Schwarz inequality and \eqref{eq9.2.4},   we get for any $\varepsilon_a \in (0, 1)$,
\begin{align}
        			\label{eq9.2.7}
    I_{4, 1}
& \lesssim_{\Omega} \varepsilon_a^{-1} \int_s^{t} \|\partial_t \bB\|^2_{ L_2 (\Omega) } \, d\tau  +    \varepsilon_a \int_s^{t} \|\nabla_x \phi\|^2_{ L_2 (\Omega) } \, d\tau.
\end{align}
Furthermore, by using integration by parts  and the equation \eqref{eq9.1.3},
\begin{align*}
    I_{4, 2} & = - \int_s^{t} \int_{\Omega}  \xi \Delta_x \phi \, dx d\tau \\
    &  =  \int_s^{t} \int_{\Omega}   \xi (\sqrt{M_{+}} a^{+}  - \sqrt{M_{-}} a^{-})  \, dx d\tau.
\end{align*}
 By the identity
\begin{align}
    \label{eq9.2.11}
  &  e_{+} \sqrt{M_{+}} a^{+}  - e_{-} \sqrt{M_{-}} a^{-} = \lambda_1 (\sqrt{M_{+}} a^{+}  - \sqrt{M_{-}} a^{-})\\
   & + \lambda_2 (\sqrt{M_{+}} a^{+}  + \sqrt{M_{-}} a^{-}), \quad 
    \lambda_1 = \frac{1}{2} (e_{+} + e_{-}),
    \quad  \lambda_2 = \frac{1}{2} (e_{+} - e_{-}), \notag
\end{align}
we have
\begin{align}
            			\label{eq9.2.8}
 &   I_{4, 2} = \frac{1}{\lambda_1} \int_s^{t} \int_{\Omega}  \xi  (e_{+} \sqrt{M_{+}} a^{+}  - e_{-} \sqrt{M_{-}} a^{-})     \, dx d\tau\\
&
   - \frac{\lambda_2}{\lambda_1} \int_s^{t} \int_{\Omega} \xi  (\sqrt{M_{+}} a^{+}  + \sqrt{M_{-}} a^{-})  \, dx d\tau =: I_{4, 2, 1} + I_{4, 2, 2}. \notag
\end{align}
By using the equation \eqref{eq9.2.3}, integration by parts, and the fact that $\lambda_1 > 0$, we obtain
\begin{align}
            			\label{eq9.2.9}
    I_{4, 2, 1} &= \frac{1}{\lambda_1}    \int_s^{t} \int_{\Omega}  \xi (\Delta_x \xi)   \, dx d\tau  = - \frac{1}{\lambda_1}  \int_s^{t} \|\nabla_x  \xi\|^2_{ L_2 (\Omega) } \, d\tau \le 0.
\end{align}
 Since $I_{4, 2, 1} \le 0$, we may drop this term from the r.h.s of the integral identity \eqref{eq2.3}.
Furthermore, by using the Cauchy-Schwarz inequality, the elliptic estimate \eqref{eq9.2.5} and the identity \eqref{eq9.2.11}, we conclude
\begin{align}
            			\label{eq9.2.10}
    I_{4, 2, 2}
  &  \lesssim    \varepsilon_a \int_s^{t} \|\sqrt{M_{+}} a^{+}  - \sqrt{M_{-}} a^{-}\|^2_{ L_2 (\Omega) } \, d\tau  \\
  & +  \varepsilon_a^{-1} \int_s^{t} \|\sqrt{M_{+}} a^{+}  + \sqrt{M_{-}} a^{-}\|^2_{ L_2 (\Omega) } \, d\tau. \notag  
\end{align}

Finally, gathering \eqref{eq9.1.5}--\eqref{eq9.9} with $k=0$ and  \eqref{eq9.2.6}--\eqref{eq9.2.10} and using the elliptic estimate \eqref{eq9.1.4} for the test function $\phi$, we obtain the desired estimate \eqref{eq9.2.0} for $a^{\pm}$ with the additional term on the r.h.s. given by
$$
   \varepsilon_a  \int_s^{t}  \| \sqrt{M_{+}}  a^{+}  - \sqrt{M_{-}}  a^{-}\|^2_{ L_2 (\Omega) } \, d\tau,
$$
 which is absorbed into the l.h.s. by choosing $\varepsilon_a$ is sufficiently small. 
 The estimate of $\bE$ follows from \eqref{hodge}, \eqref{eq9.2.4}--\eqref{eq9.2.5}, and the bound of $a^{\pm}$.
 \end{proof}

\begin{proposition}[final estimate of $a^{\pm}, \bE, \bB$, cf. \eqref{eq1.2.46}]
                \label{proposition E.4}
Under Assumption \ref{assumption 3.5}, we have
     \begin{align}
\label{eqE.4.0}
 &  \sum_{k=0}^{m-2}  \int_s^t \| \partial_t^k  [a^{+},  a^{-}]\|^2_{ L_2 (\Omega) } \, d\tau 
   + \sum_{k=0}^{m-3} \int_s^t \| \partial_t^k  \bB\|^2_{ L_2 (\Omega) } \, d\tau \\
 &+ \sum_{k=0}^{m-4} \int_s^t \| \partial_t^k  \bE\|^2_{ L_2 (\Omega) } \, d\tau
   \lesssim_{\Omega, \theta, r_3, r_4 }  (\eta (t)  - \eta (s)) \notag \\
&  +    \sum_{k=0}^{m} \int_s^t \| \partial_t^k  [b, c]\|^2_{ L_2 (\Omega) } \, d\tau   + \int_s^t \cD_{||} \, d\tau +  \sqrt{\varepsilon} \int_s^t \cD \, d\tau. \notag
 \end{align}
\end{proposition}

\begin{proof}

\textbf{Step 1: preliminary estimates of $\partial_t^k [\bE, \bB]$.}
First, by the div-curl estimate in \eqref{eq4.11} and the identity \eqref{eq9.2.12}, we have
\begin{align}
        \label{eqE.4.5}
     \sum_{k=1}^{m-4} \int_s^{t} \|\partial_t^{k} \bE\|^2_{ W^1_2 (\Omega) }  \, d\tau
   & \lesssim_{\Omega}   \sum_{k=2}^{m-3}  \int_s^{t}  \|\partial_t^{k} \bB\|^2_{ L_2 (\Omega) }  \, d\tau \\
   & + \sum_{k=1}^{m-4} \int_s^{t}  \|\partial_t^{k} [a^{+}, a^{-}]\|^2_{ L_2 (\Omega) }  \, d\tau. \notag
\end{align}
Furthermore, combining the above estimate \eqref{eqE.4.5} with 
\begin{itemize}
    \item \eqref{eqE.1.0} in Lemma \ref{lemma E.1} with $2 \le k \le m-3$,
    \item \eqref{eqE.5.7}  in Lemma \ref{lemma E.7}, 
    \item \eqref{eqE.2.0} in Lemma \ref{lemma E.2} with $j=2$ and $3 \le k+1 \le m-2$, 
\end{itemize}
we get
\begin{align}
\label{eqE.4.3}
 &    \sum_{k=1}^{m-4} \int_s^{t} \|\partial_t^{k} \bE\|^2_{ W^1_2 (\Omega) }  \, d\tau + \sum_{k=2}^{m-3}\int_0^{\tau} \| \partial_t^k \bB\|^2_{ L_2 (\Omega) } \, d\tau\\
&  \lesssim  (\eta (t) - \eta (s))
+  \varepsilon_1 \sum_{k=1}^{m-4}  \int_s^{t} \|\partial_t^{k} \bE\|^2_{ W^1_2 (\Omega) }  \, d\tau  +    \sum_{k=1}^{m-2}  \int_s^{t}  \|\partial_t^{k} [a^{+}, a^{-}]\|^2_{ L_2 (\Omega) }  \, d\tau  \notag\\
&  +    \varepsilon_1^{-1} \bigg(\sum_{k=1}^{m}  \int_s^{t}      \|\partial_t^{k} [b, c]\|^2_{ L_2 (\Omega) }  \, d\tau + \int_s^{t} \cD_{||}  \, d\tau +  \sqrt{\varepsilon} \int_s^t \cD \, d\tau\bigg). \notag
\end{align}
Then, for sufficiently small $\varepsilon_1$,  we may drop the term containing $\bE$
on the r.h.s. of \eqref{eqE.4.3}.

\textbf{Step 2: estimate of $\partial_t^k a^{\pm}, k = 1, \ldots, m-2$.}
  Combining  \eqref{eqE.5.7} in Lemma \ref{lemma E.7} with  \eqref{eqE.4.3} gives
   \begin{align}
\label{eqE.4.7}
 &   \sum_{k=1}^{m-2} \int_s^{t} \|\partial_t^k [a^{+}, a^{-}]\|^2_{ L_2 (\Omega) } \, d\tau\\
&
   \lesssim  (\eta (t)  -   \eta (s))   +   \varepsilon_a \sum_{k=1}^{m-2} \int_s^{t} \|\partial_t^k [a^{+}, a^{-}]\|^2_{ L_2 (\Omega) } \, d\tau  \notag\\
&  +  \varepsilon_a^{-1} \bigg(\sum_{k=1}^{m} \int_s^{t} \| \partial_t^{k}  [b, c]\|^2_{ L_2 (\Omega) } \, d\tau
 + \int_s^{t} \cD_{||} \, d\tau +  \sqrt{\varepsilon} \int_s^t \cD \, d\tau\bigg). \notag
 \end{align}
  Choosing $\varepsilon_a$ sufficiently small,  we absorb the sum containing $a^{\pm}$ into the l.h.s. and   obtain the desired estimate \eqref{eqE.4.0}  for the derivative terms $\partial_t^k a^{\pm}, k = 1, \ldots, m-2$.


\textbf{Step 3: estimates of  $\partial_t^k \bE, k = 1, \ldots, m-4$ and $\partial_t^k \bB, k = 0, \ldots, m-3$.} Combining \eqref{eqE.4.3} with \eqref{eqE.4.7}, we conclude
\begin{align}
\label{eqE.4.9}
 &   \sum_{k=2}^{m-3 }\int_s^{t} \| \partial_t^k \bB\|^2_{ L_2 (\Omega) } \, d\tau +  \sum_{k=1}^{m-4}\int_s^{t} \| \partial_t^k \bE\|^2_{ L_2 (\Omega) } \, d\tau
 \lesssim \text{r.h.s. of \eqref{eqE.4.0}}.
\end{align}

Furthermore, by the div-curl estimate in \eqref{eq4.12} and the fact that $\bm{j}$ is a certain velocity average of $(1-P) f$ (see \eqref{eqE.1.18}--\eqref{eqE.1.19}), we get 
\begin{equation}
        \label{eqE.4.10}
\| \partial_t^k \bB\|^2_{ L_2 (\Omega) }  \lesssim_{\Omega} \| \partial_t^{k+1} \bE\|^2_{ L_2 (\Omega) } +  \|(1-P) \partial_t^k f\|^2_{ L_2 (\Omega \times \bR^3) }, k = 0, \ldots, m-1.
\end{equation}
Combining  \eqref{eqE.4.9} with \eqref{eqE.4.10} with $k \in \{0, 1\}$, we prove   the desired estimate \eqref{eqE.4.0} for the \textit{full} sum involving $\bB$ and for all the $t$-derivative terms $\partial_t^k \bE, 1 \le k \le m-4$.

\textbf{Step 4: estimates of $a^{\pm}$ and $\bE$.}
First, gathering the estimates  \eqref{eq9.2.0} (see Lemma \ref{lemma 9.2}) and \eqref{eq3.2.1} (see Lemma \ref{lemma 3.2}), we obtain
\begin{align*}
 &    \int_s^{t} \|[a^{+}, a^{-}]\|^2_{ L_2 (\Omega) } \, d\tau +  \int_s^{t} \|\bE\|^2_{ W^1_2 (\Omega) } \, d\tau
   \lesssim   (\eta (t)  - \eta (s))  \\
&  +   \int_s^{t} \|[b, c]\|^2_{ L_2 (\Omega) } \, d\tau +  \int_s^{t} \| (1-P)  f\|^2_{ L_2 (\Omega \times \bR^3) } \, d\tau   \\
    & +      \int_s^{t} \|\partial_t \bB\|^2_{ L_2 (\Omega) } \, d\tau +   \varepsilon \int_s^t \cD \, d\tau. 
 \end{align*}
Estimating the term involving $\partial_t \bB$ via  \eqref{eqE.4.0} (see Step 3), we obtain the desired estimate \eqref{eqE.4.0} for the \textit{full} sums involving $a^{\pm}$ and $\bE$. Thus,  Proposition \ref{proposition E.4} is proved.
\end{proof}

\section{Gradient estimate of a velocity average}
    \label{section 10}
In this section, we prove the estimate \eqref{eq15.0} (see Proposition \ref{proposition 4.0}). For the sake of convenience, we set all the physical constants to $1$.

\begin{lemma}
\label{lemma 4.1}
    Let
\begin{itemize}
\item[--]   $L \ge 0$ be a nonnegative integer and
$\alpha \in (2/3, 1)$ be a constant,
\item[--] $g^l, l = 0, \ldots, L,$ be scalar functions such that
\begin{align}
    \label{eq4.1.31}
    &    \|g^0\|_{ L_{\infty} (\Omega \times \bR^3) } \le \delta, \\
\label{eq4.1.27}
 &  \sum_{l=0}^L   (\|[g^l, \nabla_p g^l]\|_{  C^{\alpha/3}_{x, p} (\Omega \times \bR^3) } )\le K,
\end{align}
where   $\delta \in (0, 1)$ and $K > 0$,

\item[--]   each $g^l$ satisfy the SRBC,

\item[--] $f^l \in S_3 (\Omega \times \bR^3)$, $l = 0, \ldots, L$,

\item[--] $\eta^l, f^l, \nabla_p f^l \in  C^{\alpha/3}_{x, p} (\Omega \times \bR^3)$, $l = 0, \ldots, L$,

\item[--]
\begin{align*}
&\sigma^0 (x, p) = \int_{\bR^3} \Phi (P, Q) (2 J + J^{1/2} (q) g^0 (x, q)) \, dq, \\
&   \sigma^l (x, p) = \int_{\bR^3} \Phi (P, Q) g^l (x, q) \, dq, \, \, l = 1, \ldots, L,
\end{align*}

\item[--] $f^0$ is a strong solution to the equation
\begin{align*}
    \frac{p}{p_0} \cdot \nabla_x f^0 - \nabla_p \cdot (\sigma^{0} \cdot \nabla_p f^0)   = \eta^0 \, \, \text{in} \, \Omega \times \bR^3,
\end{align*}
with the SRBC,

\item[--] for each $l  = 1, \ldots, k$, $f^l$ is a strong solution  to
\begin{align}
    \label{eq4.1.1}
   & \frac{p}{p_0} \cdot \nabla_x f^l - \nabla_p \cdot (\sigma^{0} \cdot \nabla_p f^l)  \\
   & - \sum_{l_1+l_2 = l, l_2 < l} c_{l_1, l_2, l} \nabla_p \cdot (\sigma^{l_1} \nabla_p f^{l_2})    = \eta^l \, \, \text{in} \, \Omega \times \bR^3, \notag
\end{align}
with the SRBC, where $c_{l_1, l_2, l}$ are certain constants,

\item[--]  $\zeta$   be a three times differentiable function  satisfying the estimate
\begin{equation}
        \label{eq1.1.0}
        \sum_{k=0}^3 |D^k  \zeta (p)|  \lesssim_{\beta}  p_0^{-\beta} \, \, \text{a.e.} \, p \in \bR^3, \, \forall \beta \ge 0.
\end{equation}
\end{itemize}

Then, if $\delta \in (0, 1)$ is sufficiently small, for
\begin{align}
    \label{eq4.1.0}
    \bar f^l  (x, p): = \int_{\bR^3} f^l (x, p) \zeta (p) \, dp,
\end{align}
we have $\nabla_x \bar f^l \in L_3 (\Omega)$, and
\begin{align}
    \label{eq4.1.2}
 &\sum_{l = 0}^L \|\nabla_x \bar f^l\|_{ L_3 (\Omega) } \\
 &\lesssim_{ \zeta, \alpha, L, \Omega, K }
  \sum_{l = 0}^L   \big(\|f^l \|_{ S_{3} (\Omega \times \bR^3) }
     +\|[\eta^l, f^l, \nabla_p f^l]\|_{  C^{\alpha/3}_{x, p} (\Omega \times \bR^3)  } \big). \notag
\end{align}
\end{lemma}

\begin{remark}
    The smallness of $\delta$ is needed to control the ellipticity of the leading coefficients in Eq. \eqref{eq4.1.1} in the boundary flattening and extension argument   (see \eqref{eq4.1.35}).
\end{remark}


\begin{proof}[Proof of  Lemma \ref{lemma 4.1}]
\textbf{Step 1: localization and change of variables.}
Let  $\xi_0$ and $\xi$ be radial nonnegative functions on $\bR^3$ such that $\xi$ is supported on $\{1 \le |p| \le 3\}$, and
$$
    \xi_0 (p) + \sum_{n = 1}^{\infty} \xi (2^{-n} p) = 1 \quad \forall p.
$$
We denote
\begin{align}
    \label{eq4.1.60}
    \xi_n (p)= \xi (2^{-n} p).
\end{align}
Furthermore, let $\chi_k, k = 1, \ldots, m$ be a partition of unity in $\Omega$
such that for 
\begin{itemize}
\item for $k \ge 2$, $\chi_k \in C^{\infty}_0 (B_{r_0/2} (\mathsf{x}_k))$, $\chi_k  = 1$ in $B_{r_0/4} (\mathsf{x}_k)$, where $\mathsf{x}_k \in \partial \Omega$,
\item $|\nabla_x \chi_k|\lesssim_{\Omega} r_0^{-1}, k = 1, \ldots, m$.
\end{itemize}

We set
\begin{equation}
       \label{eq4.1.25}
        f_{k, n}^l (x, p) = f^l (x, p) \chi_k (x) \xi_n (p)  \zeta (p)
\end{equation}
and note that for  $\bar f^l$ defined in \eqref{eq4.1.0}, we have
$$
    \bar f^l = \sum_{k, n}    \int f_{k, n}^l \, dp.
$$
Furthermore, $f_{k, n}^l$   satisfies the identity
\begin{align}
			\label{eq4.1.4}
&\frac{p}{p_0} \cdot \nabla_x f_{k, n}^l   - \nabla_p \cdot (\sigma_{g^0} \nabla_p f_{k, n}^l)\\
&-  1_{l > 0}\sum_{l_1+l_2 = l, l_2 < l} c_{l_1, l_2, l} \nabla_p \cdot (\sigma_{g^{l_1}} \nabla_p f^{l_2}_{k, n}) = \eta_{k, n}^l, \notag
\end{align}
where
\begin{equation}
			\label{eq4.1.5}
 \begin{aligned}
\eta_{k, n}^l =&   (\frac{p}{p_0} \cdot \nabla_x \chi_k) f^l  \xi_n \zeta+  \eta^l \chi_k   \xi_n \zeta \\
&
  + \sum_{l_1+l_2=l } \tilde c_{l_1, l_2, l}	\chi_k \bigg( (\partial_{p_i} \sigma^{i j}_{g^{l_1}})  (\partial_{p_j} (\xi_n \zeta)) f^{l_2}  \\
&
+ 2 \sigma_{g^{l_1}}^{i j} (\partial_{p_i} f^{l_2}) (\partial_{p_j} (\xi_n  \zeta)) + \sigma_{g^{l_1}}^{i j}   \partial_{ p_i p_j} (\xi_n  \zeta) f^{l_2} \bigg),
\end{aligned}
\end{equation}
where $\tilde c_{l_1, l_2, l}$ are certain numbers.
We  denote
\begin{align}
\label{eq4.1.5.1}
        U^l = f_{k, n}^l, \quad H^l = \eta^l_{k, n}.
\end{align}
We will focus on the case when $U^l$ is supported in a boundary chart of $\Omega \cap B_{r_0/2} (\mathsf{x_k})$, as the interior estimate is more straightforward. For the sake of convenience, we relabel $\mathsf{x_k}$ as $x_0$.

We will use the argument of Lemma 5.10 in \cite{VML} and make changes of variables to reduce \eqref{eq4.1.4} to a non-relativistic kinetic Fokker-Planck equation.
First, let $\psi: \Omega \cap B_{r_0} (x_0) \times \bR^3 \to \bR^3_{-} \times \bR^3$ be a special boundary flattening local diffeomorphism that sends a normal vector at $\partial \Omega$ to a normal vector of $\bR^3_{-}$ (see p. 6633  in \cite{VML}). 
Furthermore, we recall the following formulas related to the changes of variables in the proof of Lemma 5.10 in \cite{VML}: 
\begin{align}
&    y = \psi (x), \quad w = (D \psi (x)) p,  \notag\\
 \label{eq4.1.18}
  &  W  = \frac{w}{ \big(1 +  \big|\big(\frac{\partial x}{\partial y}\big) w\big|^2\big)^{1/2}}, \\
  & G \, \text{is the even extension of the domain} \,  \psi (\Omega \cap B_{r_0} (x_0)) \, \text{across the plane} \,  \{y_3 = 0\},  \notag\\
  & \bm{R} = \, \text{diag} (1, 1, -1),  \notag\\
\label{eq2.11.3}
&  \mathcal{W} (y, w) =  \frac{w}{(1+|M (y) w|^2)^{1/2}},  \, \text{where} \\
&
\label{eq2.11.3.1}
M (y) = \begin{cases}
\big(\frac{\partial x}{\partial y}\big) (y), \, \, y \in \overline{\psi (\Omega \cap B_{r_0} (x_0))}, \\
  \big(\frac{\partial x}{\partial y}\big) (\bm{R} y)\bm{R}, \,\,  y \in  G \cap \bR^3_{+},
\end{cases}
\\
\label{eq4.1.30}
&  	\Upsilon_n (y, w) = (y, \mathcal{W} (y, w)) : G \times \{|w| < 2^{n+2} \}  \to \bR^6, \quad 	v = \mathcal{W} (y, w).
\end{align}
We now define  $\mathcal{\dtilde U}^l$  in the same way as in the proof of Lemma 5.10  in \cite{VML}.
To that end, we introduce a sequence of functions  $\widehat U^l$,  $\mathcal{U}^l$,  $\mathcal{\hathat U}^l$, $\mathcal{\dtilde U}^l$.
In particular, loosely speaking (the exact formulas are presented below),
\begin{itemize}
\item[--] $\widehat U^l$ is $U^l$ in the coordinates $y, w$,
\item[--] $\widetilde U^l$ is $\widehat U^l$ multiplied by the Jacobian  determinant of the change of variables $(x, p) \to (y, w)$,
\item[--] $\mathcal{U}^l := \overline{U}^l$ is the `mirror extension' of $\widetilde U^l$,
\item[--] $\mathcal{\hathat U}^l$ is $\mathcal{U}^l$ in the coordinates $(y, v)$,
\item[--] $\mathcal{\dtilde U}^l$ is $\mathcal{\hathat U}^l$ multiplied by the Jacobian  determinant of the change of variables $w \to v$.
\end{itemize}
To be precise,
\begin{align}
    \label{eq14.C.10}
&\widehat U^l (y, w) = U^l (x (y), p (y, w)), \\
    \label{eq14.C.11}
& \widetilde U^l  (y, w) = \widehat U^l (y, w) \bigg|\text{det} \bigg(\frac{\partial x}{\partial y}\bigg)\bigg|^2, \\
    \label{eq14.C.12}
&  \overline{U}^l (y, w)  = \begin{cases} 	\widetilde  U^l (y,  w), \, \, (y, w) \in \bR^3_{-} \times \bR^3,\\
								\widetilde U^l (\bm{R} y, \bm{R} w),  \, \, (y, w)  \in  \bR^3_{+} \times \bR^3,
\end{cases}\\
    \label{eq14.C.13}
&\mathcal{U}^l = \overline{U}^l, \\
    \label{eq14.C.14}
&  \mathcal{\hathat U}^l (y, v) = \mathcal{U}^l (y, (\mathcal{W}_y)^{-1} (v)), \, \, \text{where} \, \,  \mathcal{W}_y (w) = \mathcal{W} (y, w), \\
    \label{eq14.C.15}
& \mathcal{\dtilde U}^l (y, v) =  \mathcal{\hathat U}^l (y, v) \mathsf{J}_{\mathcal{W}}, \quad \text{where} \,  \, \mathsf{J}_{\mathcal{W}} = \bigg|\text{det} \bigg(\frac{\partial w}{\partial v}\bigg)\bigg|.
\end{align}

 We now explain the relationship between $\dtilde U^l$ and the desired estimate \eqref{eq4.1.2}.
We fix a function $\phi \in C^{\infty}_0 (\Omega)$. By changing variables $x = x (y)$ and using the identity for $\widehat \phi (y) := \phi (x (y))$
$$
    (\partial_{x_i} \phi) (x (y)) = \frac{\partial y_j}{ \partial x_i}  \partial_{y_j} \widehat \phi (y),
$$
we get
\begin{align*}
&
   I_{n} (\phi) := \int_{\Omega \times \bR^3}    U^l (x, p)   \partial_{x_i} \phi (x) \,  dx dp \\
 &  =  \int_{ \psi (\Omega \cap B_{r_0} (x_0)) \times  \{|w| <  2^{n+2}\} } \widehat U^l (y, w)
   \bigg|\text{det} \bigg(\frac{\partial x}{\partial y}\bigg)\bigg|^2 \frac{\partial y_j}{ \partial x_i}  \partial_{y_j} \widehat \phi (y)  \,dy dw \\
&     = \int_{ \psi (\Omega \cap B_{r_0} (x_0)) \times  \{|w| < 2^{n+2}\} } \widetilde U^l (y, w)  \frac{\partial y_j}{ \partial x_i}  \partial_{y_j} \widehat \phi (y)  \,dy dw \\
&     = \int_{ \Upsilon_n \big(\psi (\Omega \cap B_{r_0} (x_0)) \times  \{|w| < 2^{n+2}\}\big) } \mathcal{\hathat U}^l (y, v) \bigg|\text{det} \bigg(\frac{\partial w}{\partial v}\bigg)\bigg|  \frac{\partial y_j}{ \partial x_i}  \partial_{y_j} \widehat \phi (y)  \,dy dv \\
&     = \int_{ \psi (\Omega \cap B_{r_0} (x_0)) } \bigg(\int_{|v| < 1} \mathcal{\dtilde U}^l (y, v) \, dv\bigg)   \frac{\partial y_j}{ \partial x_i}   \partial_{y_j} \widehat \phi (y)  \,dy.
\end{align*}
In the last identity, we used the fact that $\mathcal{\dtilde U}^l$ is supported in
$$\Upsilon_n \big(\psi (\Omega \cap B_{r_0} (x_0)) \times  \{|w| < 2^{n+2}\}\big) \subset \psi (\Omega \cap B_{r_0} (x_0)) \times \{|v|<1\}.$$
We claim that if
\begin{align}
        \label{eq4.1.6}
       \bigg\|\int_{|v| < 1} \mathcal{\dtilde U}^l (y, v) \, dv\bigg\|_{ W^1_3 (\bR^3) }
       \lesssim_{ \zeta, \alpha, L, K, \Omega}
      2^{-n}  (\text{the r.h.s. of \eqref{eq4.1.2}}),
\end{align}
then \eqref{eq4.1.2} is true.
Indeed,  since $\widehat \phi$ vanishes near the boundary of $\psi (\Omega \cap B_{r_0} (x_0))$, integrating by parts, we get
$$
    |I_n (\phi)| \lesssim   2^{-n}  (\text{r.h.s. of \eqref{eq4.1.2}})  \|\phi\|_{L_{3/2} (\Omega) }.
$$
Summing up the last inequality with respect to $n$ and $k$ gives
$$
    \bigg|\int_{\Omega} \bigg(\int_{\bR^3} f^l \psi \, dp\bigg) \partial_{x_i} \phi \, dx\bigg| \lesssim   (\text{r.h.s. of \eqref{eq4.1.2}}) \|\phi\|_{L_{3/2} (\Omega) },
$$
which implies the validity of the desired assertion \eqref{eq4.1.2} via a duality argument.  In the rest of the proof, we will show that \eqref{eq4.1.6} holds.

\textbf{Step 2: higher regularity of $\mathcal{\dtilde U}^l$ in the spatial variable.} In this step, we will, loosely speaking, show that
$$
    \mathcal{\dtilde U}^l \in L_3  (\bR^3_v) W^{1-}_3 (\bR^3_y).
$$
where $\mathcal{\dtilde U}^l$ is defined in  \eqref{eq14.C.15}.
This will be done via Lemma \ref{lemma 32.1}.
First, by the argument of the proof of Lemma 5.10 in \cite{VML}, we conclude that $\mathcal{\dtilde U}^l$ satisfies the identity (see the formula $(5.65)$ therein) 
\begin{align}
\label{eq4.1.8}
& v \cdot \nabla_y \, \mathcal{\dtilde U}^l - \nabla_v \cdot (\mathfrak{A}^0 \nabla_v \, \mathcal{\dtilde U}^l)\\
&   = \underbrace{ \mathcal{\hathat H}^l \mathsf{J}_{\mathcal{W}} }_{J_1^l} + \underbrace{ \sum_{l_1+l_2=l }  \lambda_{l_1, l_2, l}
\nabla_v \cdot \big(\mathfrak{A}^{l_1}  (\nabla_v \mathsf{J}_{\mathcal{W}}) \, \mathcal{\hathat U}^{l_2}\big) }_{ J_2^l }  \notag\\
&+ \underbrace{  \nabla_v \cdot (\mathbb{X} \, \mathcal{\hathat U}^l) }_{  J_3^l }  +  \underbrace{  \nabla_v \cdot (\mathbb{G} \, \mathcal{\hathat U}^l) }_{ J_4^l  }  \notag\\
&  \underbrace{  +  1_{ l > 0} \sum_{l_1+l_2=l, l_2 < l}  \nabla_v \cdot (\mathfrak{A}^{l_1} \nabla_v \, \mathcal{\dtilde U}^{l_2}) }_{  J_5^l  }=: \text{RHS}^l. \notag
\end{align}
Here, $\mathcal{\hathat H}^l$ is defined by replacing $U^l$ with $H^l = \eta_{n, k}^l$ (see \eqref{eq4.1.5.1}) in the definition of $\mathcal{\hathat U}^l$ (see \eqref{eq14.C.14}), and $\lambda_{l_1, l_2, l}$ are constants.
We first give informal definitions of all the coefficients  $\mathbb{X}, \mathbb{G}$, and $\mathfrak{A}^l$, and then give the exact formulas. To define $\mathbb{X}$, and $\mathfrak{A}^l$, one needs to introduce several `intermediate' functions $A^l$, $X$, $\mathcal{A}^l$, $\mathcal{X}$, $\mathcal{\hathat A}^l$, $\mathcal{\hathat X}$, $\mathbb{A}^l$, $\mathbb{X}$,  and $\mathfrak{A}^l$.
In particular,
\begin{itemize}
\item[--] $A^l$  is the  diffusion  matrix  obtained after the change of variables $(x, p) \to (y, w)$,
\item[--] $\nabla_w \cdot (X \widetilde U^l)$ is an additional (`geometric') term that is due to the change of variables $(x, p) \to (y, w)$,
\item[--] $\mathcal{A}^l$ and $\mathcal{X}$
are the  diffusion and `geometric' coefficients $A^l$ and $X$  `extended' across the boundary $\{y_3 = 0\}$,
\item[--] $\mathbb{A}^l$ and $\mathbb{X}$ are the diffusion and `geometric' coefficients  obtained after the change of variables $(y, w) \to (y, v)$,
\item[--] $\nabla_v \cdot (\mathbb{G} \mathcal{\hathat U}^l)$ is an additional  term (akin to the geometric one) that we obtain after the change of variables  $(y, w) \to (y, v)$,
\item[--] $\mathfrak{A}^l$ is an `extension' of $\mathbb{A}$  to the whole space $\bR^6$, which preserves the nondegeneracy of the matrix when $l  = 0$.
\end{itemize}
We list the relevant formulas:
\begin{align}
& A^l (y, w) = \bigg(\frac{\partial y}{\partial x}\bigg)  \sigma_{g^l} (x(y), p (y, w))  \bigg(\frac{\partial y}{\partial x}\bigg)^T, \notag \\
 \label{eq4.1.12}
      &  X (y, w) =  (X_1, X_2, X_3)^T =   \bigg(\frac{\partial y}{\partial x}\bigg) \bigg(\frac{\partial p}{\partial y}\bigg) W
        =  \bigg(\frac{\partial y}{\partial x}\bigg) \frac{\partial \big(\frac{\partial x}{\partial y} w\big)} {\partial y} W, \\
     \label{eq4.1.13}
              &   \mathcal{X} (y, w) = \begin{cases} X (y, w), \, \, (y, w) \in (y, w) \in \overline{\psi (B_{r_0} (x_0))} \times \bR^3,
      \\ \bm{R} \,  X (\bm{R} y, \bm{R} w), \,\,  (y, w) \in  (G \cap \bR^3_{+}) \times \bR^3, \end{cases} \\
     & \mathcal{A}^l (y, w)
         = \begin{cases} A^l (y, w), \quad (y, w) \in  \overline{\psi (B_{r_0} (x_0))} \times \bR^3,\\
            \bm{R} \, A^l (\bm{R} y, \bm{R} w)\,   \bm{R}, \, \, (y, w) \in  (G \cap \bR^3_{+}) \times \bR^3,
            \end{cases}  \notag\\
   \label{eq4.1.14}
      & \mathbb{X} (y, v) =   \bigg(\frac{\partial v}{\partial w}\bigg)\mathcal{ X}  (y, w (y, v)) 1_{ y \in G,  |w (y, v)| < 2^{n+2}  },\\
         \label{eq4.1.15}
& \mathbb{G} (y, v)  = \bigg(\frac{\partial v}{\partial w}\bigg)\bigg(\frac{\partial w}{\partial y}\bigg) v \,  1_{ y \in G,  |w (y, v)| < 2^{n+2}  },\\
    &  \mathbb{A}^l (y, v)
=  \bigg(\frac{\partial v}{\partial w}\bigg)    \mathcal{A}^l  (y, w (y, v)) \bigg(\frac{\partial v}{\partial w}\bigg)^T, \notag \\
& \mathfrak{A}^0  = \mathbb{A}^0  \zeta_n + (1 -  \zeta_n) \bm{1}_3, \quad  \mathfrak{A}^l  = \mathbb{A}^l \zeta_n, \notag
\end{align}
where $\zeta_n = \zeta_n (y, v)$ is a smooth cutoff function  such that $0 \le \zeta_n \le 1$ and 
\begin{align*}
	&	\zeta_n  = 1 \, \,  \text{on} \,\,   \Upsilon_n (G \times \{|w| < 2^{n+2}\}), \\
&	 |\nabla_{y, v} \zeta_n| \lesssim_{\Omega} 1.
\end{align*}

Next, we check the conditions of Lemma \ref{lemma 32.1}.
First, by the smallness assumption on $g^0$ (see \eqref{eq4.1.31})  and the argument in  Appendix C in \cite{VML} (see  formula (C.1) and the line below therein),  we have
\begin{align}
    \label{eq4.1.35}
     2^{-6 n} \bm{1}_3   \lesssim_{\Omega}  \mathfrak{A}^0 \lesssim_{\Omega} \bm{1}_3 ,
\end{align}
and hence, one can take $\delta = N (\Omega) 2^{-6 n}$ in Lemma  \ref{lemma 32.1}.
Furthermore, inspecting the argument  in  Appendix C in \cite{VML} (see (C.11) and the line below therein), we get
\begin{equation}
      \label{eq4.1.7}
      \sum_{l=0}^L (\|\mathfrak{A}^l\|_{   C^{\alpha/3 }_{x, v} (\bR^6)  }  + \|\nabla_v \mathfrak{A}^l\|_{   C^{\alpha/3 }_{x, v} (\bR^6)  })\lesssim_{\alpha, L,
         \Omega, K} 2^{n}.
\end{equation}
We now check that
\begin{equation*}
    \text{RHS}^l \,  (\text{see} \, \eqref{eq4.1.8}) \, \text{belongs to} \,  L_3 (\bR^3_v) H^{s}_3 (\bR^3_x) \, \, \forall s \in (0, \alpha/3).
\end{equation*}
We note that the term $J_5^l$ does not depend on  $f^l$ and, hence, can be handled by using an induction argument.
We split the terms $J_1^l$--$J_4^l$  into two groups:
\begin{enumerate}
    \item regular (H\"older continuous) terms $J_1^l$ and $J_2^l$,
    \item singular terms with a jump discontinuity $J_3^l$ and $J_4^l$.
\end{enumerate}
The key observation is that the terms $J_3^l$ and $J_4^l$ have a jump discontinuity because  their explicit expressions involve odd functions in the variable $y_3$. By using the fact that sufficiently regular odd functions belong to $W^{1/r-}_r (\bR^3_x)$ (see \eqref{eqF.4.1} in Lemma \ref{lemma F.4}), we will show that the same holds for $J_3^l$ and $J_4^l$.
In the sequel, $\beta$ is a constant independent of $n$, $\delta$, and $K$, which might change from line to line.

\textit{Regular terms.} By Lemmas (A.2)--(A.3) in \cite{VML} and  the argument in  Appendix C in \cite{VML} (see (C.3), (C.7), and (C.10) therein), we conclude that
\begin{align}
        \label{eq4.1.9}
      &       \||D_y^k D^j_v  \big(\frac{\partial w}{\partial v}\big)|+|D_y^k D^j_v  \big(\frac{\partial v}{\partial w}\big)| + |D_y^k D^j_v  \mathsf{J}_{\mathcal{W}}|\|_{ L_{\infty} \big(\Upsilon_n (G \times \{|w| < 2^{n+2}\})\big) } \\
      & \lesssim_{k, j}  2^{\beta n}, k \in  \{0, 1\}, \, j \in \{0, 1, \ldots\}, \notag
\end{align}
where $\beta=  \beta (k, j)$. 
Furthermore, the argument of Appendix C in \cite{VML} also shows that if $F \in  C^{\alpha/3}_{x, p} (\Omega \times \bR^3)$, then,  for $\hathat F$ defined in the same way as $\mathcal{\hathat U}^l$ in \eqref{eq14.C.14}, we have
\begin{equation}
        \label{eq4.1.10}
    \| \hathat F \|_{  C^{\alpha/3}_{y, v} \big(\Upsilon_n (G \times \{|w| < 2^{n+2}\})\big) } \lesssim_{\alpha, \Omega} 2^{\beta n}  \|F\|_{ C^{\alpha/3}_{x, p} (\Omega \times \bR^3) }.
\end{equation}
This is because the mirror extension (see \eqref{eq14.C.12}) preserves the continuity across $\{y_3 = 0\}$.
Then, by the definition of $J_1^l$ in \eqref{eq4.1.8},   the  product rule inequality in H\"older spaces, and \eqref{eq4.1.9}--\eqref{eq4.1.10}, we obtain
\begin{equation}
        \label{eq4.1.11}
    \|J_1^l\|_{   C^{\alpha/3}_{y, v} \big(\Upsilon_n (G \times \{|w| < 2^{n+2}\})\big) }  \lesssim  2^{\beta n} \|H^l\|_{ C^{\alpha/3}_{x, p} (\Omega \times \bR^3) },
\end{equation}
where $H^l$ is defined in \eqref{eq4.1.5.1}.
Since $J_1$ is compactly supported in $\Upsilon_n (G \times \{|w| < 2^{n+2}\}$, the above estimate \eqref{eq4.1.11} is valid on the whole space.

Next,  using \eqref{eq4.1.7}--\eqref{eq4.1.10},  we get
\begin{align}
        \label{eq4.1.11.1}
  &  \|J_2^l\|_{  C^{\alpha/3}_{y, v} \big(\Upsilon_n (G \times \{|w| < 2^{n+2}\})\big)} \\
  &\lesssim  (\|[\nabla_v \mathsf{J}_{\mathcal{W}}, D^2_v \mathsf{J}_{\mathcal{W}}]\|) \bigg(\sum_{l=0}^L \|[\mathfrak{A}^l, \nabla_v \mathfrak{A}^l]\|\bigg) \bigg(\sum_{l=0}^L  \|[\mathcal{\dtilde U}^l, \nabla_v \, \mathcal{\dtilde U}^l]\|\bigg) \notag\\
  & \lesssim_{\alpha, \Omega, K } 2^{\beta n}  \sum_{l=0}^L \|[U^l, \nabla_p U^l]\|_{ C^{\alpha/3}_{x, p} (\Omega \times \bR^3) }, \notag
\end{align}
where $\|\cdot\|$ is the $C^{\alpha/3}_{y, v} \big(\Upsilon_n (G \times \{|w| < 2^{n+2}\})\big)$-norm.
We note that by using the embedding into $W^s_p$ space, 
in the above estimate \eqref{eq4.1.11.1}, we may replace the H\"older norm on the l.h.s. with $$W^{s}_3 (\bR^6), s \in (0, \alpha/3).$$

\textit{Singular terms $J^l_3$ and $J^l_4$.} We start with $J_3^l$ (see \eqref{eq4.1.8}). To estimate this term, we first show that
$$
    \mathbb{X}, D_v  \mathbb{X} \in W^s_3 (\bR^6), s \in (0, \alpha/3),
$$
where $\mathbb{X}$ is given by \eqref{eq4.1.14}  inside   $\Upsilon_n (G \times \{|w| < 2^{n+2}\})$ and extended by $0$ outside that region.
First, it  follows from the definitions of $X$, $\mathcal{X}$, and $W$ (see \eqref{eq4.1.12}--\eqref{eq4.1.13} and \eqref{eq4.1.18}) that $\mathcal{X}$  is a  linear combination of terms
\begin{align}
    \label{eq4.1.41}
    h (y) w_i w_j \big(1 +  \big|M (y) w\big|^2\big)^{-1/2}, 
\end{align}
where 
\begin{itemize}
\item[--] $M$ is defined in \eqref{eq2.11.3.1},
    \item[--] $h$ is either even or odd in $y_3$ and is Lipschitz continuous in $y$ up to the boundary of the lower half of its domain $G$, that is, $\psi (\Omega \cap B_{r_0} (x_0))$.
\end{itemize}
Furthermore, due to the identity
\begin{align*}
   \bigg((\frac{\partial x}{\partial y})^T (\frac{\partial x}{\partial y})\bigg)_{i 3} = 0,  \, \, i = 1, 2, \quad \text{when $y_3 = 0$},
\end{align*}
(see the formula (A.1) in \cite{VML}), for the function $M$ in \eqref{eq2.11.3.1},  we have
\begin{align*}
   M^T M =   (\frac{\partial x}{\partial y})^T (\frac{\partial x}{\partial y}) \quad \text{when $y_3 = 0$}.
\end{align*}
Hence,  
\begin{align}
    \label{eq4.1.50}
    \text{$M (y)$  is Lipschitz continuous  across the plane $\{y_3=0\}$.}
\end{align}
By using the identity
$$
    \partial_{w_j} |M w|^2 = 2 (M^T M)_{j k} w_k,
$$
we conclude that $\partial_{w_j} \mathcal{X}$ is also a linear combination of terms 
$$
    h (y) P (w) \big(1 +  \big|M (y) w\big|^2\big)^{-r/2}, \, r \in \{1, 3\},
$$
where $P (w)$ is a monomial.
Then, 
by \eqref{eqF.1.1} and \eqref{eqF.4.1} in Lemmas \ref{lemma F.1} and \ref{lemma F.4}, respectively, for the extended function $\mathcal{X}$, we have
$$
    \mathcal{X}, D_w \mathcal{X} \in W^{(1/3)-}_3 (G  \times \{|w| < 2^{n+2}\}).
$$
Next, we  recall that $\Upsilon_n$ (see \eqref{eq4.1.30}) is a  bi-Lipschitz homeomorphism onto its image with the Lipschitz constant of order $2^{\beta n}$ (see Lemma A.3 in \cite{VML} and  \eqref{eq4.1.9}). By this, the definition of $\mathbb{X}$ in \eqref{eq4.1.14}, and the bound \eqref{eq4.1.9},
we find
\begin{align}
    \label{eq4.1.40}
    \|[\mathbb{X}, D_v \mathbb{X}]\|_{ W^s_3 \big(\Upsilon_n (G \times \{|w| < 2^{n+2}\})\big) } \lesssim_{\Omega, s} 2^{ \beta n}, \, s \in (0, 1/3).
\end{align}
Hence, extending $\mathbb{X}$ by $0$ outside the region in \eqref{eq4.1.40} and using  Lemma \ref{lemma F.3},  we have
\begin{equation*}
        \|[\mathbb{X},  D_v \mathbb{X}]\|_{ W^s_3 (\bR^6) }  \lesssim_{\Omega, s} 2^{ \beta n}, s \in (0, 1/3).
\end{equation*}
Thus, combining the last inequality with  a simple bound
$$
    \|u v\|_{ W^s_3 } \lesssim_{s, s_1}  \|u\|_{ W^s_3 }  \|v\|_{ C^{s_1} }, s_1 \in (s, 1],
$$
 and using  the estimate \eqref{eq4.1.10}, we conclude that for $J^l_3$, defined in \eqref{eq4.1.8}, and any $s \in (0, \alpha/3)$,
\begin{align}
   \label{eq4.1.16}
    \|J_3^l\|_{  W^{s}_3 (\bR^6)}  &   \lesssim_{s, \alpha, \Omega}  \|[\mathbb{X},  D_v \mathbb{X}]\|_{ W^s_3 (\bR^6) } \|[\mathcal{\hathat U}^l, \nabla_v \, \mathcal{\hathat U}^l]\|_{ C^{\alpha/3}_{y, v} (\bR^6) } \\
    & \lesssim 2^{\beta n}   \|[U^l, \nabla_p U^l]\|_{C^{\alpha/3}_{x, p} (\Omega \times \bR^3)}. \notag
\end{align}

Next, we estimate $J_4^l$. We invoke the definition of $\mathbb{G}$ in \eqref{eq4.1.15}. The argument is similar to the one in the previous paragraph. We claim that the discontinuity comes from the spatial Jacobian of $w (y, v)$. In particular,  by explicit calculations (see  the proof of  Lemma A.3  in \cite{VML}),
$$
    \frac{\partial w_i}{\partial y_{r}} =   \frac{ (\partial_{y_r} c_{j j'}) v_j v_{j'} v_i }{(1-|M v|^2)^{1/2}}
     = (\partial_{y_r} c_{j j'}) v_j v_{j'} v_i (1+|M w|^2)^{1/2},
$$
where
$(c_{i j}, i, j  = 1, 2, 3) := M^T M$, and $M$ is defined in \eqref{eq2.11.3.1}.
Due to \eqref{eq4.1.50}, $\partial_{y_r} c_{j j'}$ is well defined and is either an even or an odd function in $y_3$. Hence,   by \eqref{eqF.1.1} and \eqref{eqF.4.1} in Lemmas \ref{lemma F.1} and \ref{lemma F.4},
$$
   \|\big(\frac{\partial w}{\partial y}\big)\|_{ W^s_3 \big(\Upsilon_n (G \times \{|w| < 2^{n+2}\})\big) }   \lesssim_{s, \Omega} 2^{\beta n}, \, s \in (0, 1/3),
$$
and a similar estimate holds for
$$
    D_v \big(\frac{\partial w}{\partial y}\big).
$$
Then, proceeding as in \eqref{eq4.1.16}, we obtain for $s \in (0, \alpha/3)$,
\begin{equation}
   \label{eq4.1.17}
    \|J_4^l\|_{  W^s_3  (\bR^6)  }
    \lesssim_{s, \alpha, \Omega} 2^{ \beta n} \|[U^l, \nabla_p U^l]\|_{C^{\alpha/3}_{x, p} (\Omega \times \bR^3)}.
\end{equation}
Thus, gathering \eqref{eq4.1.11}--\eqref{eq4.1.17} and using the fact that
\begin{equation}
        \label{eq4.1.29}
    W^s_3 (\bR^6) \, \text{is embedded into} \, H^{ s- }_3 (\bR^6),
\end{equation} we conclude that for any $s \in (0, \alpha/3)$,
\begin{align}
\label{eq4.1.20}
  &  \sum_{l = 0}^L \sum_{i = 1}^4 \|J^l_i\|_{  H^{ s }_3 (\bR^6)  } \\
  & \lesssim_{ s,  \alpha, \Omega, K }  2^{\beta n}   \sum_{l = 0}^L  \|[H^l, U^l, \nabla_p U^l]\|_{C^{\alpha/3}_{x, p} (\Omega \times \bR^3)} \notag
\end{align}
(see \eqref{eq4.1.5.1}).

We now use an induction argument.
\textit{Case  $l   = 0$.} We recall the definition of the steady non-relativistic kinetic Sobolev space in \eqref{eq14.C.20}.
Since  $J_5^0 = 0$, by the estimate \eqref{eq32.1.1} in Lemma \ref{lemma 32.1} with $s \in (0, \alpha/3)$ applied to Eq. \eqref{eq4.1.8} and \eqref{eq4.1.20}, one has
\begin{align}
\label{eq4.1.21}
&    \| \mathcal{\dtilde U}^0 \|_{ L_3 (\bR^3_v)  H^{ \frac 2 3 +  s}_3 (\bR^3_y) } + \|(1-\Delta_y)^{\frac{ s }{ 2} } \mathcal{\dtilde U}^0 \|_{ S_3^N (\bR^6) } \\
& \lesssim_{  \alpha, s, K, \Omega  } 2^{\beta n}  \big(\sum_{i = 1}^4 \|J^0_i\|_{  L_3 (\bR^3_v) H^{ s }_3 (\bR^3_y)  } + \|\mathcal{\dtilde U}^0\|_{ S_3^N (\bR^6) }\big). \notag
\end{align}

\textit{Induction step.}
For the induction step, we  estimate $J_5^l$ (see \eqref{eq4.1.8}) and apply the bound \eqref{eq32.1.1} in Lemma \ref{lemma 32.1}.
Let us consider the case when $l = 1$ for the sake of simplicity.
Then, by a variant of the product rule inequality in Bessel potential spaces (see \eqref{eq22.3.1}),
 \eqref{eq4.1.7}, and \eqref{eq4.1.21}, we have
\begin{align}
    \label{eq14.1.32}
   & \|J_5^1\|_{  L_3 (\bR^3_v) H^{s}_3 (\bR^3_y) }   \\
   &\lesssim  \|[\mathfrak{A}^1, \nabla_v \, \mathfrak{A}^1]\|_{ L_{\infty} (\bR^3_v) C^{\alpha/3} (\bR^3_y) }  \, \|[\nabla_v \, \mathcal{\dtilde U}^0, D^2_v \,  \mathcal{\dtilde U}^0]\|_{  L_3 (\bR^3_v) H^{s}_3 (\bR^3_y) }  \notag\\
& \lesssim 2^{\beta n}  (\sum_{i = 1}^4 \|J^0_i\|_{ L_3 (\bR^3_v) H^{ s }_3 (\bR^3_y)  } + \|\mathcal{\dtilde U}^0\|_{ S_3^N (\bR^6) }).  \notag
\end{align}
Hence, by the estimate \eqref{eq32.1.1} in Lemma \ref{lemma 32.1},  we get
\begin{align*}
 &\| \mathcal{\dtilde U}^1 \|_{ L_3 (\bR^3_v)  H^{s + \frac 2 3 }_3 (\bR^3_y) } + \|(1-\Delta_y)^{\frac{s}{ 2}} \mathcal{\dtilde U}^1 \|_{ S_3^N (\bR^6) } \\
  & \lesssim  2^{\beta n}   \big(\sum_{i=1}^5 \|J_i^1\|_{ L_3 (\bR^3_v) H^{ s }_3 (\bR^3_y) } + \|\mathcal{\dtilde U}^1\|_{ S_3^N (\bR^6) }\big) \\
   & \lesssim  2^{\beta n}   \sum_{j=0}^1  (\sum_{i=1}^4 \|J^j_i \|_{  H^{ s }_3 (\bR^6)  } + \|\mathcal{\dtilde U}^j\|_{ S_3^N (\bR^6) }).
\end{align*}
Thus, by an induction argument, we conclude that for any $s \in (0, \alpha/3)$,
\begin{align}
\label{eq4.1.22}
&   \sum_{l = 0}^L \big(\| \mathcal{\dtilde U}^l \|_{ L_3 (\bR^3_v)  H^{\frac 2 3 + s}_3 (\bR^3_y) } + \|(1-\Delta_y)^{\frac{s}{ 2}} \mathcal{\dtilde U}^l \|_{ S_3^N (\bR^6) } \\
& + \|\text{RHS}^l\|_{ L_3 (\bR^3_v)  H^{s}_3 (\bR^3_y) } \big) 
\lesssim 2^{\beta n}  \sum_{l = 0}^L \big(\sum_{i = 1}^4 \|J_i^l\|_{   H^{ s  }_3 (\bR^6)  } + \|\mathcal{\dtilde U}^l\|_{ S_3^N (\bR^6) }\big), \notag
\end{align}
where  $\text{RHS}^l$ is defined in \eqref{eq4.1.8} and the first term on the r.h.s. is estimated in \eqref{eq4.1.20}.

\textbf{Step 3: regularity of a velocity average.} We fix  $l \in \{0, \ldots, L\}$ and  denote
\begin{align*}
   & \mathsf{f} = (1-\Delta_y)^{ \frac{s}{2}  + \frac 1 3  } \mathcal{\dtilde U}^l, \\
   &\mathsf{g}  = (1-\Delta_y)^{\frac{s}{2}} \big(\nabla_v \cdot (\mathfrak{A}^0 \nabla_v \, \mathcal{\dtilde U}^l)\big)
    + (1-\Delta_y)^{\frac{s}{2}} \text{RHS}^l\big)
\end{align*}
and note that by \eqref{eq4.1.8},
$$
    v \cdot \nabla_y \,  \mathsf{f}=    (1-\Delta_y)^{ \frac 1 3 } \mathsf{g}.
$$
By a variant of the velocity averaging lemma (see \eqref{eq14.1.0} in Lemma \ref{lemma 14.1})  with $2/3$ and $3$ in place of $\alpha$ and $p$, respectively,  and \eqref{eq4.1.29},
  for any $\gamma \in (0, 1/9)$, we have
\begin{align*}
    & \bigg\| \int_{|v|<1} \mathsf{f}  \, dv \bigg\|_{ H^{\gamma}_3 (\bR^3) } \lesssim_{\gamma}  \|\mathsf{f} \|_{ L_3 (\bR^6)} + \|\mathsf{g} \|_{ L_3 (\bR^6)}\\
    & \lesssim     \| (1-\Delta_y)^{\frac{s}{2} + \frac 1 3  } \mathcal{\dtilde U}^l \|_{ L_3 (\bR^6)} \\
    & +  \|(1-\Delta_y)^{\frac{s}{2}} \big(\nabla_v \cdot (\mathfrak{A}^0 \nabla_v \, \mathcal{\dtilde U}^l))\big\|_{ L_3 (\bR^6)}  +  \|(1-\Delta_y)^{\frac{s}{2}} \text{RHS}^l \|_{ L_3 (\bR^6) }.
    \end{align*}
By using \eqref{eq4.1.22} to bound the first  and the third terms on the r.h.s. in the above inequality and  estimating the
second one as in \eqref{eq14.1.32},
we obtain
\begin{align}
    \label{eq4.1.23}
 & \sum_{l=0}^L   \bigg\| \int_{|v|<1} \mathcal{\dtilde U}^l \, dv \bigg\|_{ H^{\frac 2 3 +s +\gamma}_3 (\bR^3) }  \\
 &\lesssim 2^{\beta n}  \sum_{l = 0}^L \big(\sum_{i = 1}^4 \|J_i^l\|_{   H^{ s }_3 (\bR^6)  } + \|\mathcal{\dtilde U}^l\|_{ S_3^N (\bR^6) }\big). \notag
\end{align}
Taking $s$ and $\gamma$ close to $\alpha/3$ and $1/9$, respectively, and using the fact that $\alpha \in  (2/3, 1)$, we conclude that $\frac 2 3 +s +\gamma > 1$, and hence,
we may replace the $H^{\frac 2 3 +s +\gamma}_3 (\bR^3)$-norm on the l.h.s. with the $W^1_3 (\bR^3)$-norm.

Next, inspecting the argument of Step 6 in the proof of Lemma 5.10 in \cite{VML},  
we have
$$
     \|\mathcal{\dtilde U}^l\|_{ S_3^N (\bR^6) }  \lesssim_{\Omega} 2^{\beta n}  \|U^l\|_{ S_{3} (\Omega \times \bR^3_p) }.
$$
By combining this with \eqref{eq4.1.20} and recalling that $H^l = \eta^l_{k, n}, U^l = f^l_{k, n}$ (see \eqref{eq4.1.5.1}), we conclude that
\begin{align}
\label{eq4.1.24}
&  \sum_{l=0}^L   \bigg\| \int_{|v|<1} \mathcal{\dtilde U}^l \, dv \bigg\|_{ W^1_3 (\bR^3) }  \\
& \lesssim  2^{\beta n} \sum_{l = 0}^L  \big(\|[\eta^l_{k, n}, f^l_{k, n}, \nabla_p f^l_{k, n}]\|_{C^{\alpha/3}_{x, p} (\Omega \times \bR^3)} + \|f^l_{k, n}\|_{ S_{3} (\Omega \times \bR^3) }\big). \notag
\end{align}
We note  that by the product rule inequality in H\"older spaces,  the fast decay of $\psi$ (see the assumption \eqref{eq1.1.0}),
and the definition of $\xi_n$ in \eqref{eq4.1.60}, for  $Z  = f^{l}, \nabla_p f^l, \eta^l$ and any $\beta_1 > 0$, we have
\begin{align}
\label{eq4.1.26}
   \|Z \xi_n  \zeta\|_{ C^{\alpha/3}_{x, p} (\Omega \times \bR^3) }
       \lesssim_{\xi, \xi_0, \alpha, \beta_1,  \zeta }  2^{-\beta_1 n} \|Z\|_{ C^{\alpha/3}_{x, p} (\Omega \times \bR^3) }.
       \end{align}
Similarly, 
\begin{align}
\label{eq4.1.33}
   \|\xi_n f^l  \zeta\|_{ S_{3} (\Omega \times \bR^3) }
       \lesssim_{\xi, \xi_0, \alpha, \beta_1,  \zeta}  2^{-\beta_1 n} \|f^l\|_{ S_3 (\Omega \times \bR^3) }.
       \end{align}
Due to \eqref{eq4.1.26}--\eqref{eq4.1.33}, we may replace the r.h.s. of \eqref{eq4.1.24} with
$$
     2^{- n}  \sum_{l = 0}^L   \big(\|f^l \|_{ S_{3} (\Omega \times \bR^3) }  + \|[\eta^l,  f^l, \nabla_p f^l]\|_{  C^{\alpha/3}_{x, p} (\Omega \times \bR^3)  } \big),
$$
which gives \eqref{eq4.1.6}. Thus, the desired estimate \eqref{eq4.1.2} is valid.
\end{proof}



\begin{proposition}
    \label{proposition 4.0}
 Let $\zeta = (\zeta^{+}, \zeta^{-})$ be a three times differentiable function satisfying \eqref{eq1.1.0} and denote
 \begin{align*}
  \bar f (t, x) = \int_{\bR^3} f (t, x, p)  \cdot \zeta (p) \, dp.
     \end{align*}
Then,    under Assumption \ref{assumption 3.5},  we have
    \begin{equation}
    \label{eq15.1}
  \sum_{k=0}^{  m-9 }  \int_s^t \|D_{x}  \partial_t^k  \bar f\|^2_{  L_3 (\Omega) } \, d\tau
  \lesssim_{\Omega, \theta, r_3, r_4, \zeta}
     \int_s^t \cD\,  d\tau.
\end{equation}
\end{proposition}

\begin{remark}
By inspecting the argument of the proof, we also obtain
\begin{align*}
     \sum_{k=0}^{m- 9  }  \int_s^t \|\nabla_x \partial_t^k f\|^2_{  L_2 (\Omega \times \bR^3) } \, d\tau \lesssim_{\Omega, \theta, r_3, r_4}  \int_s^t \cD\,  d\tau.
\end{align*}
\end{remark}

\begin{proof}[Proof of  Proposition \ref{proposition 4.0}]
We set
$$
      g^l = \partial_t^l f \cdot (1, 1), \,  l = 0, \ldots, m-9,
$$
so that
\begin{align*}
  &  \sigma^0 = \int_{\bR^3} \Phi (P, Q) \big(J (q) + J^{1/2} (q) f (t, x, q) \cdot (1, 1)\big) \, dq, \\
  & \sigma^l (t, x, p) = \int_{\bR^3} \Phi (P, Q) J^{1/2} (q) \partial_t^l f (t, x, q) \cdot (1, 1) \, dq  =  \partial_t^l \sigma^0 (t, x, p), \, l = 1, \ldots, L.
\end{align*}
Then, by Assumption \ref{assumption 3.5}, the function $u = \partial_t^l f^{\pm}, l = 0, \ldots, m-5,$ is a strong solution (see Definition \ref{definition 3.3}) to the equation (cf. (6.17) in \cite{VML})
\begin{align}
\label{eq1.1.3}
    \frac{p}{p_0} \cdot \nabla_x  u  -  \nabla_p \cdot (\sigma^0   \nabla_p u) - 1_{l > 0} c_{l_1, l_2, l} \sum_{l_1 + l_2 = l, l_2 < l}  \nabla_p \cdot ( \sigma^l \nabla_p \partial_t^{l_2} f) = \eta^l
\end{align}
with the SRBC, where (cf. (6.1)--(6.4) in \cite{VML})
\begin{align}
    &
      \eta^l =  - \partial_t^{l+1} f^{\pm}   \pm (\frac{p}{p_0} \cdot \partial_t^l \bE) \sqrt{J} + \eta^l_1 + \eta^l_2  + \eta^l_3 ,  \notag\\
       & \eta^l_1 =  \sum_{l_1 + l_2 = l} \bigg(\mp(\partial_t^{l_1} \bE + \frac{p}{p_0} \times \partial_t^{l_1} \bB) \cdot (\nabla_p \partial_t^{l_2} f)  
       \pm \frac{1}{2} (\frac{p}{p_0} \cdot \partial_t^{l_1} \bE) \partial_t^{l_2} f\bigg),  \notag\\
     \label{eq1.1.4}
     & \eta^l_2 = - \sum_{l_1 + l_2 = l}  \big((\partial_t^{l_1} C_f) (\partial_t^{l_2} f^{\pm}) - (\partial_t^{l_1} a_f) \cdot (\nabla_p \partial_t^{l_2} f^{\pm})\big), \\
          &\eta^l_3 =  K_{\pm} (\partial_t^l f),  \notag \\
    \label{eq1.1.5}
& a_f^i (t, x, p) =  -  \int \Phi^{ i j } (P, Q)   J^{1/2} (q)    \big(\frac{p_i}{2 p_0}  f (t, x, q)
 	+ \partial_{q_j}  f (t, x, q)\big) \cdot (1, 1)\, dq,\\
     \label{eq1.1.6}
 & C_f (t, x, p) =
	  -  \frac{1}{2} \sigma^{ i j} \frac{ p_i}{  p_{ 0 } } \frac{ p_j}{ p_0} + \partial_{p_i} \big(\sigma^{ i j}  \frac{ p_j}{ p_0}\big) \\
&	-  \int  \big(\partial_{p_i} 	- \frac{ p_i}{2 p_0}\big)   \Phi^{ i j } (P, Q)  J^{1/2} (q)  \partial_{q_j}  f (t, x, q) \cdot (1, 1) \, dq\notag,\\
     \label{eq1.1.7}
& K_{\pm} f (t, x, p) = - J^{-1/2} (p) \partial_{p_i}  \bigg(J (p) \int \Phi^{ i j } (P, Q)  J^{1/2} (q) \big(\partial_{q_j} f (t, x, q)  \\
&+ \frac{ q_j}{2 q_0} f (t, x, q)\big) \cdot  (1, 1)\, dq\bigg). \notag
\end{align}

%

We note that by Assumption \ref{assumption 3.5}  and the estimate  \eqref{eq14.4.7} in Lemma \ref{lemma 5.1}, the conditions \eqref{eq4.1.31}--\eqref{eq4.1.27} of Lemma \ref{lemma 4.1} hold with $\alpha$ satisfying
$$
    2/3 < 1-12/r_4 < \alpha < 1,
$$
and $\delta$ and $K$ independent of $s, t$, and $\varepsilon$.
Applying the estimate \eqref{eq4.1.2} in Lemma \ref{lemma 4.1} for a.e. $\tau \in (s, t)$, raising  the resulting inequality to the power $2$, and integrating over the interval $(s, t)$, we get
\begin{align}
        \label{eq1.1.12}
& \sum_{l = 0}^{  m-9 }  \int_s^{t} \|D_x \partial^l_t  \bar f\|^2_{  L_3 (\Omega) } \, d\tau \\
& \lesssim_{\alpha, \Omega} \sum_{l = 0}^{  m-9 }   \int_s^{t} \|\partial_t^l f  \|^2_{ S_3 (\Omega \times \bR^3) } \, d\tau  \notag\\
&   +  \sum_{l = 1}^{   m-8 }  \int_s^{t} \|\partial^l_t f \|^2_{  C^{\alpha/3}_{x, p} (\Omega \times \bR^3) } \, d\tau +    \sum_{l = 0}^{  m-9 } \int_s^{t} \|\partial^l_t \bE\|^2_{  C^{\alpha/3} (\Omega) } \, d\tau  \notag\\
& +    \sum_{l = 0}^{  m-9 } \sum_{j = 1}^3 \int_s^{t}   \|\eta^l_j\|^2_{  C^{\alpha/3}_{x, p} (\Omega \times \bR^3) } \, d\tau.\notag
\end{align}
We note that by the definition of $\cD$ (see \eqref{eq14.0.3} and \eqref{eq12.0.1}) and the estimates \eqref{eq14.4.7}-\eqref{eq14.4.6} in Lemma \ref{lemma 5.1}, the first three terms on the r.h.s. of \eqref{eq1.1.12} can be replaced with $\int_s^t \cD \, d\tau$.


\textit{Estimate of $\eta^l_1$.}
We note that by the product rule inequality in H\"older spaces,  the smallness assumption \eqref{eq3.5.1}, and the bound \eqref{eq14.4.7}, 
  we have
\begin{align}
    \label{eq1.1.13}
   & \sum_{l=0}^{  m-9 }\int_s^{t} \|\eta^l_1\|^2_{ C^{\alpha/3}_{x, p} (\Omega \times  \bR^3)} \, d\tau \\
 &  \le   \bigg(\sum_{l=0}^{ m-9 }   \|\partial_t^l[\bE, \bB] (\tau, \cdot)\|^2_{ L_{\infty} ((s, t)) C^{\alpha/3} (\Omega) }\bigg)  \bigg(\int_s^{t} \sum_{l=0}^{  m-9 } \|[\partial_t^l f, \nabla_p \partial_t^l f]\|^2_{ C^{\alpha/3}_{x, p} (\Omega \times  \bR^3)} \, d\tau\bigg) \notag \\
 &  \lesssim  \varepsilon \int_s^{t}  \cD \, d\tau. \notag
\end{align}

\textit{Estimate of $\eta^l_2$.}
 By using 
the $L_{\infty}$ estimates  of $a_f$ and $C_f$ in  \eqref{eqB.1.5}--\eqref{eqB.1.3},  the  smallness assumption \eqref{eq3.5.1}, and the estimate \eqref{eq14.4.7}, we find for $l = 0, \ldots, 4$, and $\tau \in [s, t]$, 
\begin{align}
\label{eq1.1.8}
        &  \|\partial_t^l a_f (\tau, \cdot)\|^2_{ L_{\infty} (\bR^3_p)  C^{\alpha/3} (\Omega) }  + \|\partial_t^l C_f (t, \cdot)\|^2_{ L_{\infty} (\bR^3_p)  C^{\alpha/3} (\Omega) }  \\
        & \lesssim   1 + \|[\partial_t^l f (\tau, \cdot), \nabla_p \partial_t^l  f (\tau, \cdot)]\|^2_{   C^{\alpha/3}_{ x, p } (\Omega \times \bR^3) }  \lesssim 1. \notag
\end{align}
Furthermore, by the H\"older estimates of $a_f$ and $C_f$ in  \eqref{eqB.1.6}--\eqref{eqB.1.4},  the assumption \eqref{eq3.5.1} and the bound \eqref{eq14.4.7},  for the  same  $l$ and $\tau$, we have
\begin{align}
\label{eq1.1.9}
            &  \|\partial_t^l a_f (\tau, \cdot)\|^2_{ L_{\infty} (\Omega)  C^{\alpha/3} (\bR^3_p) }  + \|\partial_t^l C_f (\tau, \cdot)\|^2_{ L_{\infty} (\Omega)  C^{\alpha/3} (\bR^3_p)  }   \\
           & \lesssim 1 + \|\partial_t^l f (\tau, \cdot)\|^2_{ L_{\infty} (\Omega) W^1_{\infty}  (\bR^3) }
           \lesssim 1. \notag
\end{align}
Hence, by using the  product rule inequality and the bounds \eqref{eq1.1.8}--\eqref{eq1.1.9},  we obtain (cf. \eqref{eq1.1.13})
\begin{align}
\label{eq1.1.10}
     &\sum_{l=0}^{ m-9 }  \int_s^{t} \|\eta^l_2\|^2_{ C^{\alpha/3}_{x, p} (\Omega \times  \bR^3)} \, d\tau 
     \lesssim  \int_s^{t} \cD\, d\tau. 
\end{align}

\textit{Estimate of $\eta^l_3$.}
By  the estimates of $K$ in \eqref{eqB.1.1} and \eqref{eqB.1.2}  
and the bound \eqref{eq14.4.7}, we
have
\begin{align}
\label{eq1.1.14}
&
  \sum_{l=0}^{ m-9 }  \int_s^{t}  \| K \partial_t^l f\|^2_{ C^{\alpha/3}_{x, p} (\Omega \times \bR^3)}  \, d\tau
     \\
  & \lesssim \sum_{l=0}^{ m-9 } \int_s^{t} \|[\partial_t^l f, \nabla_p \partial_t^l f]\|^2_{ C^{\alpha/3}_{x, p} (\Omega \times \bR^3)} \, d\tau
  \lesssim \int_s^{t}  \cD \, d\tau.    \notag
\end{align}
Finally, gathering \eqref{eq1.1.12}--\eqref{eq1.1.13} and \eqref{eq1.1.10}--\eqref{eq1.1.14}, we prove the desired estimate \eqref{eq15.1}.
\end{proof}

\section{Positivity estimate of $L$}

\label{section 15}


\begin{proposition}[cf. \eqref{eq122}]
    \label{proposition 12.1}
  Under Assumption \ref{assumption 3.5},  there exists a constant $\delta_0  = \delta_0 (\theta, \Omega, r_3, r_4) > 0$ such that for any $\delta \in (0, \delta_0)$,  one has
    \begin{align}
    \label{eq12.1.1}
&\sum_{k=0}^m \int_s^{t} \int_{  \Omega  }  \langle L (\partial_t^k f), (\partial_t^k f)\rangle \, dx  d\tau\\
& \ge \delta \bigg(\sum_{k = 0}^{m-2} \int_s^{t} \|\partial_t^k [a^{+}, a^{-}]\|^2_{L_2 (\Omega)} \, d\tau +   \sum_{k = 0}^m \int_s^{t} \||\partial_t^k  [b,  c]\|^2_{ L_2 (\Omega) } \, d\tau \notag \\
& +  \sum_{k = 0}^{m-4} \int_s^{t} \|\partial_t^k  \bE\|^2_{ L_2 (\Omega) } \, d\tau
+  \sum_{k = 0}^{m-3} \int_s^{t} \|\partial_t^k  \bB\|^2_{ L_2 (\Omega) } \, d\tau \notag  \ - (\eta (t) - \eta (s)) - \sqrt{\varepsilon} \int_s^t \cD \, d\tau\bigg), \notag
\end{align}
where $\eta$ is a function satisfying the  bound  \eqref{eq3.1.2}.
\end{proposition}

\begin{proof}
    First, by the semipositivity estimate (see Lemma 8 in \cite{GS_03}), there exists a constant $\delta_{\star} \in (0, 1)$ such that for any $u  = (u^{+}, u^{-}) \in W^1_2 (\bR^3)$,
\begin{align}
    \label{eq12.1.3.1}
     \langle L  u,  u \rangle   \ge \delta_{\star} \|(1-P) u\|^2_{W^1_2 (\bR^3)}.
\end{align}
Hence, to prove \eqref{eq12.1.1}, it suffices to show that
\begin{align}
        \label{eq12.1.3}
&\sum_{k = 0}^{m-2} \int_s^{t} \|\partial_t^k [a^{+}, a^{-}]\|^2_{L_2 (\Omega)} \, d\tau +   \sum_{k = 0}^m \int_s^{t} \|\partial_t^k  [b, c]\|^2_{ L_2 (\Omega) } \, d\tau\\
& +  \sum_{k = 0}^{m-4} \int_s^{t} \|\partial_t^k  \bE\|^2_{ L_2 (\Omega) } \, d\tau
+  \sum_{k = 0}^{m-3}  \int_s^{t} \|\partial_t^k  \bB\|^2_{ L_2 (\Omega) } \, d\tau \notag\\
& \lesssim_{\Omega, \theta, r_3,  r_4 } (\eta (t) - \eta (s)) + \int_s^t \cD_{||}  \, d\tau + \sqrt{\varepsilon} \int_s^t \cD \, d\tau. \notag
\end{align}

\textbf{Step 1: estimates of $b$ and $c$.}
First, by  \eqref{eq3.3.1} in Lemma \ref{lemma 3.3}, for sufficiently small $\varepsilon_b \in (0, 1)$, we have
\begin{align*}
&   \sum_{k = 0}^m  \int_s^{t}  \|\partial_t^k  b\|^2_{ L_2 (\Omega) } \, d\tau
 \lesssim (\eta (t) - \eta (s))
 +  \varepsilon_b \sum_{k = 0}^m  \int_s^{t}  \|\partial_t^k  c\|^2_{ L_2 (\Omega) } \, d\tau\\
&  +  \varepsilon_b^{-1} \big(\int_s^{t} \cD_{||}  \, d\tau +   \varepsilon \int_s^t \cD \, d\tau\big).
\end{align*}
Furthermore,  by \eqref{eq6.1.4} in Lemma \ref{lemma 7.1}, we have
\begin{align}
      \label{eq12.1.4}
 &   \sum_{k = 0}^m \int_s^{t} \|\partial_t^k  c\|^2_{ L_2 (\Omega) } \, d\tau
    \lesssim (\eta (t)  - \eta (s))  \\
 &  +   \sum_{k = 0}^m \int_s^{t} \|\partial_t^k  b\|^2_{ L_2 (\Omega) } \, d\tau
 +   \int_s^{t} \cD_{||}  \, d\tau  +  \varepsilon \int_s^t \cD \, d\tau.  \notag
 \end{align}
Combining the above estimates of $b$ and $c$, we obtain
\begin{align}
        \label{eq12.1.8}
  \sum_{k = 0}^m \int_s^{t} \|\partial_t^k  b\|^2_{ L_2 (\Omega) } \, d\tau
 & \lesssim (\eta (t) - \eta (s))
 +  \varepsilon_b \sum_{k = 0}^m \int_s^{t} \|\partial_t^k  b\|^2_{ L_2 (\Omega) } \, d\tau \\
&  +  \varepsilon_b^{-1}  \big(\int_s^{t} \cD_{||}  \, d\tau +   \varepsilon \int_s^t \cD \, d\tau\big).  \notag
\end{align}
Thus, by taking $\varepsilon_b$ sufficiently small, we may drop the term involving $b$ on the r.h.s. of \eqref{eq12.1.8} and conclude that \eqref{eq12.1.3} holds for the sum involving $b$.
Hence, we may also drop the term involving $b$ on the r.h.s. of \eqref{eq12.1.4}. Thus, the desired estimate  \eqref{eq12.1.3} also holds for the sum involving $c$.

\textbf{Step 2: estimates of $a^{\pm}, \bE$, and $\bB$.} By the conclusion of Step 1, we may drop the terms involving $b$ and $c$ in the estimate of $[a^{\pm}, \bE, \bB]$ in \eqref{eqE.4.0} and obtain  \eqref{eq12.1.3}. Thus, the proposition is proved.
\end{proof}


\section{Top-order energy estimate}
    \label{section 12}

\begin{proposition}[cf. \eqref{eq12.2.1.0}]
    \label{proposition 12.2}
 Under Assumption \ref{assumption 3.5}, we have  
    \begin{align}
          \label{eq12.2.1}
& \cI_{||} (t)  +  \bigg(\int_s^t \cD_{||} \, d\tau + \sum_{k = 0}^{m-2} \int_s^{t} \|\partial_t^k [a^{+}, a^{-}]\|^2_{L_2 (\Omega)} \, d\tau \\
& +   \sum_{k = 0}^m \int_s^{t}  \|\partial_t^k  [b, c]\|^2_{ L_2 (\Omega) } \, d\tau  
 +  \sum_{k = 0}^{m-4}  \int_s^{t}  \|\partial_t^k  \bE\|^2_{ L_2 (\Omega) } \, d\tau
+  \sum_{k = 0}^{m-3}  \int_s^{t}  \|\partial_t^k  \bB\|^2_{ L_2 (\Omega) } \, d\tau\bigg) \notag \\
& \lesssim_{\Omega, r_3,  r_4, \theta } 
 \cI_{||}(s)+\sqrt{\varepsilon}\big(\|\cI_{||}\|_{ L_{\infty} ((s, t)) } + \int_s^t \cD \, d\tau\big).  \notag
\end{align}
\end{proposition}

We will need the following lemma.

\begin{lemma}
    \label{lemma 12.4}
For any $k \in \{0, \ldots, m-1\}$, and any $\tau \in [s, t]$,
\begin{equation}
    \label{eq12.4.1}
   \|\partial_t^k \bE (\tau, \cdot)\|_{  W^1_2 (\Omega) } + \|\partial_t^k \bE (\tau, \cdot)\|_{  L_6 (\Omega) }  \lesssim_{\Omega} \cI_{||}^{1/2} (\tau).
\end{equation}
\end{lemma}

\begin{proof}
    By the div-curl estimate \eqref{eq4.11}, for fixed $\tau$, we have
    \begin{align*}
        \|\partial_t^k \bE (\tau, \cdot) \|_{ W^1_2 (\Omega) }  \lesssim_{\Omega} \|\partial_t^k f (\tau, \cdot)\|_{ L_2 (\Omega) }
        + \|\partial_t^{k+1} \bB (\tau, \cdot) \|_{ L_2 (\Omega) } \lesssim \cI_{||}^{1/2} (\tau).
    \end{align*}
    The estimate of the $L_6^x$ norm follows from the last inequality and the Sobolev embedding theorem.
\end{proof}

\begin{proof}[Proof of Proposition \ref{proposition 12.2}]
    In this proof, $N= N (\Omega, \theta,  r_3,  r_4)$.

\textbf{Energy inequality.}
First, by the energy identity in \eqref{eq27.4.12} (see Lemma \ref{lemma 27.4}),   for any $k \in \{0, \ldots, m\}$,
\begin{align}
    \label{eq12.2.2}
& \frac 1 2 \big(\|\partial_t^k f (t, \cdot)\|^2_{ L_2 (\Omega \times \bR^3) } - \|\partial_t^k f (s, \cdot)\|^2_{ L_2 (\Omega \times \bR^3) }\big)  \\
& - \frac{1 }{k_b T} \int_s^{t} \int_{\Omega} (\partial_t^k \bE) \cdot \bigg(\int_{\bR^3} e_{+}  \frac{p}{p_0^{+}} \sqrt{J^{+}} (\partial_t^k f^{+})  - e_{+}  \frac{p}{p_0^{-}} \sqrt{J^{-}} (\partial_t^k f^{-}) \, dp\bigg) dx d\tau \notag \\
& + \int_s^{t} \int_{ \Omega  }  \langle L \partial_t^k f, \partial_t^k f \rangle \, dx d\tau \notag \\
& =  \int_s^{t} \int_{ \Omega  } \langle \partial_t^k (\Gamma (f, f)),    \partial_t^k f\rangle \,  dx d\tau \notag\\
& +  \sum_{k_1+k_2=k} \binom{k}{k_1} \bigg(\frac{ 1 }{2 k_b T}\int_s^{t} \int_{ \Omega \times \bR^3 }  \big(e_{+}  \frac{p}{p_0^{+}}  \cdot (\partial_t^{k_1} \bE) (\partial_t^{k_2} f^{+}) (\partial_t^k f^{+}) - e_{-}  \frac{p}{p_0^{-}}  \cdot (\partial_t^{k_1} \mathbf{E}) (\partial_t^{k_2} f^{-}) (\partial_t^k f^{-})\big) \, dz \notag \\
& -  \int_s^{t} \int_{ \Omega \times \bR^3 } (\partial_t^{k_1} \bE)  \cdot \big(e_{+} (\nabla_p \partial_t^{k_2}  f^{+}) (\partial_t^k f^{+}) - e_{-} (\nabla_p \partial_t^{k_2}  f^{-}) (\partial_t^k f^{-})\big) \, dz \notag\\
& -  \int_s^{t} \int_{ \Omega \times \bR^3 }  p \times  (\partial_t^{k_1}  \bB)  \cdot \big(e_{+} (p_0^{+})^{-1} (\nabla_p \partial_t^{k_2}  f^{+}) (\partial_t^k f^{+}) - e_{-} (p_0^{-})^{-1} (\nabla_p \partial_t^{k_2}  f^{-}) (\partial_t^k f^{+})\big) \, dz\bigg). \notag
\end{align}
We note that by the definition of $\bm{j}$ in  \eqref{eq36.16} and 
the energy identity for Maxwell's equations, the second term on the l.h.s.  of \eqref{eq12.2.2} (the first integral term therein) equals
\begin{align*}
&  -   \frac{1 }{k_b T} \int_s^{t}  \int_{\Omega} (\partial_t^k \bE) \cdot (\partial_t^k \bm{j}) \, dx d\tau \\
& =  \frac{ 2 \pi  }{ k_b T}  (\|\partial_t^k \bE (t, \cdot)\|^2_{ L_2 (\Omega) } + \|\partial_t^k \bB (t, \cdot)\|^2_{ L_2 (\Omega) }) - (\|\partial_t^k \bE (s, \cdot)\|^2_{ L_2 (\Omega) } + \|\partial_t^k \bB (s, \cdot)\|^2_{ L_2 (\Omega) }).
\end{align*}
Summing up the inequalities with respect to $k \in \{0, \ldots, m\}$, using the  semipositivity estimate in \eqref{eq12.1.3.1}, and the positivity estimate of $L$ \eqref{eq12.1.1} in Proposition \ref{proposition 12.1},  we obtain
\begin{align*}
& \cI_{||} (t)  - \cI_{||} (s)   +   \delta (\eta (t) - \eta (s))  \\
& + \delta \bigg(\int_s^t \cD_{||} \, d\tau \notag  + \sum_{k = 0}^{m-2} \int_s^{t} \|\partial_t^k [a^{+}, a^{-}]\|^2_{L_2 (\Omega)} \, d\tau +   \sum_{k = 0}^m \int_s^{t}  \|\partial_t^k  [b, c]\|^2_{ L_2 (\Omega) } \, d\tau  \notag\\
& +  \sum_{k = 0}^{m-4}  \int_s^{t}  \|\partial_t^k  \bE\|^2_{ L_2 (\Omega) } \, d\tau
+  \sum_{k = 0}^{m-3}  \int_s^{t}  \|\partial_t^k  \bB\|^2_{ L_2 (\Omega) } \, d\tau\bigg) \notag\\
&\lesssim \sqrt{\varepsilon} \int_s^t \cD \, d\tau +  |\text{r.h.s of \eqref{eq12.2.2}}|,
\end{align*}
where $\delta \in (0, \delta_0)$ and $\delta_0 =  \delta_0 (\Omega, r_3, r_4, \theta)  \in (0, 1)$, and 
$\eta$ is a function satisfying the bound (see \eqref{eq3.1.2}) 
$$
    |\eta (\tau)| \le N (\Omega,   r_3,  r_4, \theta) \cI_{||} (\tau), \, \tau \in [s, t].
$$
Hence, taking $\delta < 1/(2N)$, we may absorb the term $\delta \eta (t)$ into $\cI_{||} (t)$ and replace the term $\cI_{||} (s)+\delta \eta (s)$  with $ N_1 \cI_{||} (s)$.


\textbf{Collision term.}
By \eqref{eq12.A.2.2} in Lemma \ref{lemma 12.A.2} and the smallness assumption \eqref{eq3.5.1}, 
$$
   |\text{the first  term on the r.h.s of \eqref{eq12.2.2}}|  \lesssim \sqrt{\varepsilon} \int_s^t \cD \, d\tau.\notag
$$
Thus, to finish the proof of the  desired estimate \eqref{eq12.2.1}, it suffices to show that
\begin{align}
        \label{eq12.2.3}
   & |\text{the last three terms on the r.h.s. of \eqref{eq12.2.2}}| \\
   & \lesssim \sqrt{\varepsilon} \big( \|\cI_{||}\|_{ L_{\infty} ((s, t)) } + \int_s^t \cD \, d\tau\big). \notag
\end{align}

\textbf{Proof of the claim \eqref{eq12.2.3}.}
First,  by using the bound \eqref{eq12.A.3.1} in Lemma \ref{lemma 12.A.3} with $\zeta \equiv 1$, we conclude that, in the case when $k \in \{0, \ldots, m-2\}$, the desired claim \eqref{eq12.2.3} is valid.
Hence, we may assume that
$
    k \in \{m-1, m\}.
$

We split  each integral in each sum  into two terms as follows:
\begin{align}
        \label{eq12.2.15}
  & 2 k_b T I^1_{k_1, k_2} (\bE)
      :=  e_{+}  \int_s^t \int_{ \Omega  \times \bR^3} \frac{p_i}{p_0^{+}}   (\partial_t^{k_1} \bE_i) (\partial_t^{k_2} f^{+}) \partial_t^k (P^{+} f) \, dz \\
      & - e_{-}   \int_s^t \int_{ \Omega  \times \bR^3} \frac{p_i}{p_0^{-}}  (\partial_t^{k_1} \bE_i) (\partial_t^{k_2} f^{-}) \partial_t^k  (P^{-} f) \, dz  =: e_{+ } I^{1, +}_{k_1, k_2} (\bE) - e_{-} I^{1, -}_{k_1, k_2} (\bE), \notag \\
    &2 k_b T  J^1_{k_1, k_2} (\bE)  :=   \int_s^t \int_{ \Omega  \times \bR^3} p_i   (\partial_t^{k_1} \bE_i)  (\partial_t^k (1-P) f) \cdot (\frac{e_{+}}{p_0^{+}} \partial_t^{k_2} f^{+},  - \frac{e_{-}}{p_0^{-}} \partial_t^{k_2} f^{-}) \, dz,  \notag \\
    \label{eq12.2.13}
 &  I^2_{k_1, k_2} (\bE) := - e_{+}   \int_s^t \int_{ \Omega  \times \bR^3} (\partial_t^{k_1} \bE_i)  \,  (\partial_{p_i}  \partial_t^{k} P^{+}  f) (\partial_t^{k_2} f^{+}) \, dz \\
 & + e_{-}  \int_s^t \int_{ \Omega  \times \bR^3} (\partial_t^{k_1} \bE_i) \,   (\partial_{p_i} \partial_t^k P^{-}  f) (\partial_t^{k_2} f^{-}) \, dz
    =: - e_{+} I^{2, +}_{k_1, k_2} (\bE)  + e_{-} I^{2, -}_{k_1, k_2} (\bE), \notag \\
 &    J^2_{k_1, k_2} (\bE) :=
 -   \int_s^t \int_{ \Omega  \times \bR^3} (\partial_t^{k_1} \bE_i) \, (\partial_{p_i} \partial_t^k (1 - P) f) \cdot (e_{+} \partial_t^{k_2} f^{+}, - e_{-} \partial_t^{k_2} f^{-}) \, dz, \notag\\
\label{eq12.2.14}
 &  I_{k_1, k_2} (\bB) := - e_{+}  \int_s^t \int_{ \Omega  \times \bR^3}  \frac{p}{p_0^{+}} \times (\partial_t^{k_1}  \bB)   \cdot (\nabla_{p}  \partial_t^{k} P^{+}  f) (\partial_t^{k_2} f^{+}) \, dz \\
 & + e_{-}  \int_s^t \int_{ \Omega  \times \bR^3}  \frac{p}{p_0^{-}} \times (\partial_t^{k_1}  \bB)  \cdot  (\nabla_{p} \partial_t^k P^{-}  f) (\partial_t^{k_2} f^{-}) \, dz , \notag\\
 &    J_{k_1, k_2} (\bB) :=  -  \int_s^t \int_{ \Omega  \times \bR^3} \partial_t^{k_1} (p \times \bB)_i \,  (\partial_{p_i} \partial_t^k (1- P)  f) \cdot  (\frac{e_{+}}{p_0^{+}} \partial_t^{k_2} f^{+}, -\frac{e_{-}}{p_0^{-}} \partial_t^{k_2} f^{-})\, dz.\notag
 \end{align}

 \textit{Estimate of the  $J$-terms.}
By applying the Cauchy-Schwarz inequality in the $p$ variable and the $L_{\infty}^t L_2^x$-$L_{2}^t L_{\infty}^x$-$L_2^{t, x}$
H\"older's inequality, we get
 \begin{align}
    \label{eq12.2.61}
  &   |J^1_{k_1, k_2 } (\bE)| + |J^2_{k_1, k_2 } (\bE)|  +  |J_{k_1, k_2 } (\bB)|  \\
 & \lesssim \|(1-P) \partial_t^k f\|_{L_2 ((s, t) \times \Omega) W^1_2 (\bR^3)  } M_{k_1, k_2}, \notag \\
  & M_{k_1, k_2}:= \|\partial_t^{k_1} [\bE, \bB]\|_{L_{\infty} ((s, t)) L_2 (\Omega)}
 \big(1_{ k_2 \le m/2} \|\partial_t^{k_2} f\|_{ L_2 ((s, t)) L_{\infty} (\Omega) L_2 (\bR^3) }) \notag  \\
   &  + 1_{ k_1 \le m/2}  \|\partial_t^{k_1} [\bE, \bB]\|_{ L_2 ((s, t)) L_{\infty} (\Omega) }  \|\partial_t^{k_2} f\|_{L_{\infty} ((s, t)) L_2 (\Omega \times \bR^3)  }.\notag
 \end{align}
 We note that
 \begin{itemize}
     \item[--] the first factor in \eqref{eq12.2.61} is bounded by $(\int_s^t \cD_{||} \, d\tau)^{1/2}$,
     \item[--] $M_{k_1, k_2}$ is the same as in \eqref{eq12.A.3.6}, and then, by \eqref{eq12.A.3.7} and the smallness assumption \eqref{eq3.5.1},  we have
     \begin{align}
        \label{eq12.2.63}
       M_{k_1, k_2} \lesssim_{\theta, \Omega, r_3, r_4} \|\cI_{||}\|_{ L_{\infty} ((s, t)) }^{  1/2  }  \big(\int_s^t \cD\,  d\tau\big)^{1/2} \lesssim \sqrt{\varepsilon} \big(\int_s^t \cD\,  d\tau\big)^{1/2}.
     \end{align}
     
 \end{itemize}
 Hence, we conclude that
 \begin{align}
     \label{eq12.2.4}
|\text{all the  $J$-terms}|   \lesssim \sqrt{\varepsilon} \int_s^t \cD \, d\tau.
 \end{align}

 \textit{Estimate of the $I$-terms.}  We start with the explicit computation of the $I$-terms. By the definition of the projection operator $P$ in \eqref{eq6.16}--\eqref{eq6.17} and the identity \eqref{eq3.1.70},
 we have 
\begin{align}
\label{eq12.9.1}
       & \partial_{p_i} P^{\pm} f \\
        &=  - \frac{1}{ k_b T} \big(\sqrt{M_{\pm}^{-1}} a^{\pm}   + \kappa_1 p \cdot b   +    \kappa_3 c (p_0^{\pm} - \kappa_2^{\pm})\big) \frac{p_i}{p_0^{\pm}}  \sqrt{J^{\pm}} + \kappa_1 b_i  \sqrt{J^{\pm}} +\kappa_3  c \frac{p_i}{p_0^{\pm}} \sqrt{J^{\pm}}, \notag
\end{align}
which yields
\begin{equation}
\label{eq12.B.5}
    \nabla_p P^{\pm} f + \frac{1}{  k_b T} \frac{p}{p_0^{\pm}} P^{\pm} f  = \kappa_1 b  \sqrt{J^{\pm}} +\kappa_3 c \frac{p}{p_0^{\pm}} \sqrt{J^{\pm}}.
\end{equation}

\textit{Magnetic field term \eqref{eq12.2.14}.}
 By \eqref{eq12.B.5},
$$
    \frac{p}{p_0^{\pm}} \times (\partial_t^{k_1} \bB) \cdot  \nabla_p   P^{\pm} (\partial_t^{k} f) =  \kappa_1  (\frac{p}{p_0^{\pm}} \times \partial_t^{k_1} \bB)  \cdot \partial_t^{k} b  \sqrt{J^{\pm}},
$$
 and hence, by \eqref{eq12.2.14} and the definition of $\bm{j}$ in  \eqref{eq36.16}, we have
\begin{align}
   & I_{k_1, k_2} (\bB)  = - \kappa_1
    \int_{s}^t \int_{\Omega \times \bR^3} \partial_t^{k_2} (e_{+}\frac{p}{p_0^{+}} \sqrt{J^{+}} f^{+} -  e_{-}\frac{p}{p_0^{-}} \sqrt{J^{-}} f^{-}) \times  (\partial_t^{k_1} \bB)  \cdot (\partial_t^k b) \, \, dz \notag\\
   \label{eq12.B.2}
    & = -\kappa_1  \int_{s}^t \int_{\Omega } (\partial_t^{k_2} \bm{j}) \times (\partial_t^{k_1} \bB)  \cdot (\partial_t^k b)  \, dx d\tau.
\end{align}
Then, proceeding as in \eqref{eq12.2.61}, invoking the  definition of $\cD$ in \eqref{eq12.0.1}, and using the bound \eqref{eq12.2.63}, we obtain
\begin{align}
   \label{eq12.2.4.1}
  &  |I_{k_1, k_2} (\bB)|  
     \lesssim  \|\partial_t^k b\|_{ L_{2} ((s, t) \times \Omega) }  M_{k_1, k_2} 
     \lesssim \sqrt{\varepsilon} \int_s^t \cD \, d\tau. 
\end{align}

\textit{Electric field terms  \eqref{eq12.2.15} and \eqref{eq12.2.13}.}
We first consider $I_{k_1, k_2}^{2, \pm} (\bE)$ defined in \eqref{eq12.2.13}.
By the identity \eqref{eq12.9.1}, we have
\begin{align*}
 &   I_{k_1, k_2}^{2, \pm} (\bE)
 =  - \frac{1}{k_b T} \int_s^t \int_{ \Omega \times \bR^3}  (\partial_t^{k_1} \bE_i) \, \frac{p_i}{p_0^{\pm}} (\partial_t^{k_2} f^{\pm}) \sqrt{J^{\pm}} \\
& \times \partial_t^{k} \bigg(\sqrt{M_{\pm}^{-1}}  a^{\pm}   + \kappa_1 p_l   b_l   +    \kappa_3  (p_0^{\pm} - \kappa_2^{\pm} -  k_b T)  c\bigg)  \, dz \\
 &
 + \kappa_1\int_s^t \int_{ \Omega \times \bR^3}  (\partial_t^{k_1} \bE_i) \,  (\partial_t^{k_2} f) \sqrt{J^{\pm}} \,
  (\partial_t^k b_i)  \, dz. 
\end{align*}
Inspecting the above expression, we conclude that
\begin{equation}
    \label{eq12.B.4}
        I^{2, \pm}_{k_1, k_2} (\bE) \, \text{is a linear combination of terms of Type I and Type II},
\end{equation}
where
\begin{align*}
 & \text{Type I} =  \int_s^{t} \int_{\Omega} (\partial_t^{k_1} \bE_i) \,  (\partial_t^{k_2} \bm{j}^{\pm}_i)
 \,  (\partial_t^k a^{\pm}) \, dx d\tau,
   \quad  \bm{j}_i^{\pm} =\int_{\bR^3} \frac{ p_i }{p_0^{\pm}} \sqrt{J^{\pm}} f^{\pm}  \, dp,   \\
&
     \text{Type II} =   \int_s^{t} \int_{\Omega} (\partial_t^{k_1} \bE_i) \,  (\partial_t^{k_2} \bar f) \,  (\partial_t^k h)  \, dx d\tau, \, \,
     \quad  h = b_j \, \, \text{or}\,\,  c,
\end{align*}
and 
$$
    \bar f (t, x)=  \int_{\bR^3}  p_i^{n_1}  p_l^{n_2} (p_0^{\pm})^{-n_3} f^{\pm}  (t, x, p) \sqrt{J^{\pm}} \, dp, \, \, n_{j} \in \{0, 1\}, j = 1, 2, 3.
$$
The same conclusion also holds for $I^{1, \pm}_{k_1, k_2} (\bE)$.
Thus, to finish the proof of the claim \eqref{eq12.2.3}, it suffices to estimate terms of Type I and II.

\textit{Type II term.}
We observe that a term of Type II is similar to the integral $I_{k_1, k_2} (\bB)$ (see \eqref{eq12.B.2}). Then, proceeding as in \eqref{eq12.2.4.1}, we obtain
\begin{align*}
 & |\text{A term of Type II}|  \lesssim  \|\partial_t^k [b,  c]\|_{ L_{2} ((s, t) \times \Omega) }  M_{k_1, k_2}  \\
  & \lesssim   
  \sqrt{\varepsilon}  \int_s^t \cD \, d\tau.
\end{align*}

\textit{Type I term.}
To estimate a term of Type I,  we consider two cases separately: $k_1 \le m-4$ and $m-3 \le k_1 \le m$. 

\textit{Case $k_1 \le m-4$}.
By the $L_{\infty}^t L_2^x$-$L_{2}^t L_{\infty}^x$-$L_2^{t, x}$ H\"older's inequality, we have
\begin{align*}
 & |\text{A term of type I}| \lesssim  \|\partial_t^k f \|_{ L_{\infty} ((s, t)) L_2 (\Omega \times \bR^3) } \tilde M,  \\
 & \tilde M = (1_{ k_1 \le m/2}  \|\partial_t^{k_1} \bE\|_{ L_{2} ((s, t)) L_{\infty} (\Omega)})  \|\partial_t^{k_2} \bm{j}^{\pm}\|_{ L_{2} ((s, t) \times \Omega) }  \notag\\
&  +  (1_{k_2 \le m/2} \|\partial_t^{k_2} f\|_{ L_2 ((s, t)) L_{\infty} (\Omega) L_2 (\bR^3) })  \|\partial_t^{k_1} \bE\|_{ L_2 ((s, t) \times \Omega) }=: \tilde M_1 \tilde M_2 + \tilde M_3 \tilde M_4. \notag  
\end{align*}
Furthermore,
\begin{itemize}
\item[--] by \eqref{eq14.4.7}--\eqref{eq14.4.6}, 
\begin{align*}
    \tilde M_1+\tilde M_3 \lesssim (\int_s^t \cD \, d\tau)^{1/2},
\end{align*}
\item[--] by \eqref{eqE.2.15.1}, 
\begin{align*}
    \tilde  M_2 \lesssim (\int_s^t \cD \, d\tau)^{1/2},
\end{align*}
    \item[--] for $k_1 \le m-4$, one has  $\|\partial_t^{k_1} \bE\|^2_{ L_2 (\Omega) }$
is in $\cD$ (see \eqref{eq12.0.1}, and, hence, 
\begin{align*}
    \tilde  M_4 \lesssim (\int_s^t \cD \, d\tau)^{1/2}.
\end{align*}
\end{itemize}
Thus, we conclude
\begin{align}
    \label{eq12.2.16}
    |\text{A term of type I}| \lesssim
     \|\cI_{||}\|_{ L_{\infty} ((s, t)) }^{1/2}  \int_s^t \cD \, d\tau  \lesssim \sqrt{\varepsilon}  \int_s^t \cD \, d\tau.
\end{align}


\textit{Type I term: case $k_1 \ge  m-3$}.
We  consider an integral of Type I
\begin{align}
    \label{eq12.2.9.1}
    \mathfrak{I}_0  : =  \int_s^{t} \int_{\Omega } (\partial_t^{k_1} \bE_i) \,
 (\partial_t^{k_2} \bm{j}^{\pm}_i) \,    (\partial_t^k a^{\pm}) \, dx d\tau.
\end{align}
Formally integrating by parts in the $t$ variable gives
\begin{align}
    \label{eq12.2.9}
    \mathfrak{I}_0 & = \tilde \eta (t) - \tilde \eta (s)  + \mathfrak{I},
\end{align}
where
\begin{align}
 &\tilde  \eta (\tau) :=  \int_{\Omega } (\partial_t^{k_1-1} \bE_i (\tau, x))  \,  (\partial_t^{k_2} \bm{j}^{\pm}_i (\tau, x))  \,  (\partial_t^{k}  a^{\pm} (\tau, x))  \, dx, \notag\\
\label{eq12.2.9.2}
 & \mathfrak{I} : = - \int_s^{t} \int_{\Omega } (\partial_t^{k_1-1} \bE_i) \,  \big((\partial_t^{k_2} \bm{j}^{\pm}_i)  \,  (\partial_t^{k+1} a^{\pm}) + (\partial_t^{k_2+1} \bm{j}_i)  \,   (\partial_t^{k} a^{\pm})\big)  \, dx d\tau.
\end{align}
For the temporal boundary term, by the $L_2$-$L_2$-$L_{\infty}$ H\"older's inequality, the fact that $k_2 \le 3$,  the $L_{\infty}^{x, p}$ estimate of $\partial_t^{l} f, l \le m-8$ in \eqref{eq14.4.7} in Lemma \ref{lemma 5.1}, and the smallness assumption \eqref{eq3.5.1}, we have for $\tau \in [s, t]$,
\begin{align}
    \label{eq12.2.5}
    |\tilde \eta (\tau)|  &\leq  \|\partial_t^{k}  f^{\pm} (\tau, \cdot)\|_{ L_2 (\Omega \times \bR^3) }
  \|\partial_t^{k_1-1} \bE (\tau, \cdot)\|_{ L_2 (\Omega) }  \|\partial_t^{k_2} f (\tau, \cdot)\|_{L_{\infty} (\Omega \times \bR^3)  }   \\
  & \lesssim
    \|\cI\|^{1/2}_{ L_{\infty} ((s, t)) }   \cI_{||} (\tau) \lesssim  \sqrt{\varepsilon}  \cI_{||} (\tau). \notag
\end{align}

By \eqref{eqE.2.15},  we may
replace $\partial_t^{l} a^{\pm}$ with $\nabla_x \cdot \partial_t^{l-1} \bm{j}^{\pm}$ in the integral term $\mathfrak{I}$.
Furthermore, integrating by parts in $x$ and using \eqref{eqE.2.16}--\eqref{eq12.2.19}, we get
\begin{align}
\label{eq12.2.12}
& \mathfrak{I}  = (\text{const}) \, (\mathfrak{I}_1  + \mathfrak{I}_2), \\
 & \mathfrak{I}_1  :=  \int_s^{t} \int_{\Omega } (\partial_{x_l} \partial_t^{k_1-1} \bE_i) \,  \big((\partial_t^{k_2} \bm{j}^{\pm}_i)  \,   (\partial_t^{k} \bm{j}^{\pm}_l)  + (\partial_t^{k_2+1} \bm{j}^{\pm}_i) \, (\partial_t^{k-1} \bm{j}^{\pm}_l) \big)  \, dx d\tau, \notag \\
 & \mathfrak{I}_2  : =   \int_s^{t} \int_{\Omega }  (\partial_t^{k_1-1} \bE_i) \,
 \big((\partial_{x_l} \partial_t^{k_2+1}  \bm{j}^{\pm}_i\big)  (\partial_t^{k-1}  \bm{j}^{\pm}_l) + \big(\partial_{x_l} \, \partial_t^{k_2} \bm{j}^{\pm}_i) (\partial_t^{k} \bm{j}^{\pm}_l)\big) \, dx d\tau. \notag
\end{align}
Then, by the $L_{\infty}^t L_2^x$-$L_2^t L^x_{\infty}$-$L_2^{t, x}$ H\"older's  inequality, 
we have
\begin{align}
    \label{eq12.2.6}
   & |\mathfrak{I}_1|   \lesssim  \|D_x \partial_t^{k_1-1} \bE\|_{ L_{\infty} ((s, t)) L_2 (\Omega) }\\
  & \times \big( \sum_{r=k_2}^{k_2+1}  \|\partial_t^{r}  f\|_{ L_{2} ((s, t)) L_{\infty} (\Omega \times \bR^3) }\big)
\big( \sum_{r=k-1}^{k}\|\partial_t^{r} \bm{j}^{\pm}|\|_{  L_2 ((s, t) \times \Omega) }\big). \notag
\end{align}
We note that
\begin{itemize}
    \item[--] by \eqref{eq12.4.1}, the first factor on the right-hand side of \eqref{eq12.2.6} is bounded by $N (\Omega)  \|\cI_{||}\|_{ L_{\infty} ((s, t)) }^{1/2}$,
    \item[--] by \eqref{eq14.4.7}  and the fact that $k_2 + 1 \le 4 < m-8$, the second factor is bounded by $N  (\int_s^t \cD \, d\tau)^{1/2}$,
    \item[--] by \eqref{eqE.2.15.1}, 
    we may replace the third factor with $(\int_s^t \cD \, d\tau)^{1/2}$.
\end{itemize}
Hence, by combining these estimates, we obtain
\begin{align}
    \label{eq12.2.7}
    \mathfrak{I}_1  \lesssim  \|\cI_{||}\|_{ L_{\infty} ((s, t)) }^{1/2}   \int_s^t \cD \, d\tau \lesssim \sqrt{\varepsilon} \int_s^t \cD \, d\tau.
\end{align}

Next, by the $L_{\infty}^t L_6 (\Omega)$-$L_2^t L_3 (\Omega)$-$L_2^{t, x}$ H\"older's inequality, we have
\begin{align}
    \label{eq12.2.20}
    \mathfrak{I}_2  &  \lesssim  \|\partial_t^{k_1-1} \bE\|_{ L_{\infty} ((s, t)) L_6 (\Omega) }    \big(\sum_{r=k_2}^{k_2+1}\|D_x \partial_t^r \bm{j}^{\pm}\|_{ L_2 ((s, t)) L_3 (\Omega) }\big)  \\
 &  \times \big(\sum_{r=k-1}^k \|\partial_t^{r} \bm{j}^{\pm}\|_{  L_2 ((s, t) \times \Omega) }\big). \notag
\end{align}
We note that
\begin{itemize}
   \item[--] by \eqref{eq12.4.1}, the first factor on the right-hand side of \eqref{eq12.2.20} is bounded by $N \|\cI_{||}\|_{ L_{\infty} ((s, t)) }^{1/2}$,
   \item[--] since $k_2 +1 \le  4 < m-9$, by \eqref{eq15.1}, the second factor is dominated by
   $$
        N    \big(\int_s^t \cD \, d\tau\big)^{1/2}.
   $$
\end{itemize}
Thus, by this and  \eqref{eqE.2.15.1},
$$
    \mathfrak{I}_2
    \lesssim   \|\cI_{||}\|^{1/2}_{ L_{\infty} ((s, t)) } \int_s^t \cD \, d\tau  \lesssim \sqrt{\varepsilon}  \int_s^t \cD \, d\tau.
$$
Gathering  \eqref{eq12.2.9}--\eqref{eq12.2.5} and \eqref{eq12.2.7}--\eqref{eq12.2.20},   we obtain
\begin{align}
    \label{eq12.2.10}
 & \mathfrak{I}_0 =  \int_s^{t} \int_{\Omega } (\partial_t^{k_1} \bE) \,  \cdot (\partial_t^{k_2} \bm{j}^{\pm})  \,   (\partial_t^k a^{\pm})\, dx d\tau \\
 & \lesssim   \sqrt{\varepsilon} \big(\|\cI_{||}\|_{ L_{\infty} ((s, t)) } + \int_s^t \cD \, d\tau\big). \notag
 \end{align}
Thus, the bound \eqref{eq12.2.3} is valid, and this concludes the proof of the desired estimate \eqref{eq12.2.1}.
\end{proof}

\section{Proof of main results}

\begin{proof}[Proof of Theorem \ref{theorem 3.1}]
\label{section 14}
\textbf{Step 1: a priori estimate.}

First, we impose Assumption \ref{assumption 3.5}. As explained in Section \ref{section 4}, the main ingredients are 
\begin{itemize}
    \item the top-order energy estimate 
    \eqref{eq12.2.1.0}, 
    \item the lower-order weighted energy estimate \eqref{eq14.1.1.0},
    \item  the steady $S_p$ estimates and the div-curl bounds of the electromagnetic field  \eqref{eq14.2.20}--\eqref{eq14.2.21}.
\end{itemize}

\textit{The lower-order weighted energy estimate.} The desired estimate \eqref{eq14.1.1.0} is contained in Step 1 of the proof of  Proposition 3.11 in \cite{VML} (see p. 6654--6656 therein). We note that the  Landau equation \eqref{eq36.13}--\eqref{eq36.14} differs from its Picard approximation   $(3.46)$ considered in Proposition 3.11 of \cite{VML} in the nonlinear terms. In the equation (3.46) in \cite{VML}, these nonlinear terms are linearized by replacing $f$ with the previous iteration $g$. Due to the smallness assumption \eqref{eq3.5.1} in Assumption \ref{assumption 3.5} and the fact that the instant functional $\cI$ (see \eqref{eq14.0.5}) in the present paper coincides with that in \cite{VML} (see the definition of $\cI_f$ in the formula $(3.31)$ therein), the nonlinear terms in the Landau equation \eqref{eq36.13}--\eqref{eq36.14} can be treated as in the argument on p. 6655--6656 in \cite{VML}. Hence, by repeating the argument on p.  6654--6656 in \cite{VML}, we obtain 
\begin{align}
    \label{eq17.7}
& \sum_{k=0}^{ m - 4 } \bigg(\|\partial_t^k  f (t, \cdot)\|^2_{ L_{2, \theta/2^{k} } (\Omega \times \bR^3) }
+ \int_s^t  \|\partial_t^k  f (\tau, \cdot)\|^2_{ L_{2} (\Omega) W^1_{2, \theta/2^{k} } (\bR^3) } \, d\tau\bigg)\\
& \lesssim_{\Omega, \theta,  r_3, r_4} \sum_{k=0}^{ m - 4 } \bigg(\|\partial_t^k  f (s, \cdot)\|^2_{ L_{2, \theta/2^{k} } (\Omega \times \bR^3) }
+  \int_s^t \|\partial_t^k  f\|^2_{ L_{2 } (\Omega \times \bR^3) }  \, d\tau \notag \\
&+ \int_s^t \|\partial_t^k  \bE\|^2_{ L_2 (\Omega)  }  \, d\tau\bigg) + \sqrt{\varepsilon} \int_s^t \cD \, d\tau. \notag
\end{align}

\textit{Steady $S_p$ and div-curl estimates.} For the brief exposition of the argument, see Section \ref{section 4.4}. We inspect the proof of Proposition 6.3 in \cite{VML} (see p. 6649--6654 therein).
As we explained in the previous paragraph, the smallness assumption \eqref{eq3.5.1} enables us to repeat the argument of the aforementioned proposition, even though the equation considered therein is linear. Thus,  for any
$\tau \in [s, t]$, we have
\begin{align}
\label{eq17.8}
 &   \sum_{i=1}^4 \sum_{k=0}^{  m-4-i } \|\partial_t^k f (\tau, \cdot) \|^2_{  S_{r_i, \theta/2^{ k + 2i   } } (\Omega \times \bR^3) }  \\
& 
+ \sum_{i=2}^3 \sum_{k=0}^{ m  - 4 -  i} \|\partial_t^k  [\bE, \bB] (\tau, \cdot) \|^2_{  W^1_{r_i} (\Omega) }  + \sum_{k=0}^{m-1} \|\partial_t^k  [\bE_f, \bB_f] (\tau, \cdot) \|^2_{  W^1_{ 2 } (\Omega) }  \notag  \\
&  \lesssim_{\Omega, \theta, r_1, \ldots, r_4} \|\cI_{||}\|_{ L_{\infty} ((s, t)) } +  \varepsilon  \|\cI\|_{ L_{\infty} ((s, \tau)) }  \notag  \\
&
  +   \sum_{k=0}^{m-4} \|\partial_t^k f\|^2_{  L_{\infty} ((s, \tau)) L_{2, \theta/2^k} (\Omega \times \bR^3) }. \notag 
\end{align}
We note that the first and the second term on the l.h.s of \eqref{eq17.8} are contained in $\cD$. Hence, repeating the descent argument in Proposition 6.3 in \cite{VML} and using the smallness condition \eqref{eq3.5.1}, we obtain
\begin{align}
\label{eq17.9}
 &   \sum_{i=1}^4 \sum_{k=0}^{  m-4-i } \int_s^t \|\partial_t^k f  \|^2_{  S_{r_i, \theta/2^{ k + 2i   } } (\Omega \times \bR^3) }  \, d\tau \\
& 
+ \sum_{i=2}^3 \sum_{k=0}^{ m  - 4 -  i} \int_s^t \|\partial_t^k  [\bE, \bB]  \|^2_{  W^1_{r_i} (\Omega) } \, d\tau  \notag  \\
&  \lesssim_{\Omega, \theta, r_1, \ldots, r_4}  \varepsilon  \int_s^{t} \cD \, d\tau  
  +   \sum_{k=0}^{m-4} \bigg(\int_s^t \|\partial_t^k f\|^2_{   L_{2, \theta/2^k} (\Omega \times \bR^3) }  \, d\tau
 +  \int_s^t \|\partial_t^k [\bE, \bB] \|^2_{  L_{ 2 } (\Omega) }  \, d\tau\bigg), \notag 
\end{align}

Gathering \eqref{eq12.2.1}, \eqref{eq17.7}--\eqref{eq17.9},  we find
\begin{align}
    \label{eq17.1}
  & \|\cI\|_{ L_{\infty} ((s, t)) }  +  \int_s^t \cD \, d\tau \\
  &   \lesssim_{\theta, r_1, \ldots, r_4,  \Omega} \cI_{||} (s)  +   \sum_{k=0}^{ m - 4 } \|\partial_t^k  f (s, \cdot)\|^2_{ L_{2, \theta/2^{k} } (\Omega \times \bR^3) }   +    \sqrt{\varepsilon}  \big(\|\cI\|_{ L_{\infty} ((s, t)) } + \int_s^t \cD \, d\tau\big). \notag 
\end{align}

By taking $\varepsilon = \varepsilon (\Omega, \theta, r_1, \ldots, r_4)$ sufficiently small and plugging $s=0$, we obtain  the desired estimate \eqref{eq14.0.13}.

\textbf{Step 2: continuity argument.} Given the local existence and uniqueness result established in \cite{VML} (see Theorem 3.10 therein) and the global estimate \eqref{eq14.0.13}, the global existence follows from the standard continuity argument (see, for example, the proof of Theorem 1 in \cite{VMB_03}). We emphasize that to implement the argument, we use the fact that
  the total instant functional $\cI$ (see \eqref{eq14.0.5}) coincides with that in \cite{VML}, which is denoted by $\cI_f$ (see the formula $(3.31)$ p. 6622 therein). In Remark \ref{remark 3.13}, we elaborate on the relationship between the global estimate \eqref{eq14.0.13} and the local well-posedness result in \cite{VML}.
\end{proof}

\begin{proof}[Proof of Theorem \ref{theorem 3.2}]
 We first delineate the argument in Section 2 of \cite{GS_06}, which gives a polynomial temporal decay rate of the lower-order instant energy for the RVML system on $\mathbb{T}^3$.  Given a Lyapunov-type inequality
\begin{align}
        \label{eq15.2}
        \cI' + \cD \le 0,
\end{align}
one can derive an upper bound of a lower-order instant energy $\cI_{low}$ in terms of a lower-order dissipation $\cD_{low}$:
 \begin{align}
        \label{eq15.5}
    \cI_{low} \le N \cD_{low}^{1-}.
 \end{align}
 This inequality combined with the global estimate \eqref{eq15.2} with $\cI$ and $\cD$ replaced with $\cI_{low}$ and $\cD_{low}$
gives
$$
     \cI_{low}' + N \cI_{low}^{1+} \le 0,
$$
which implies a `fast' polynomial decay of $\cI_{low}$.
We point out the major differences with our setup.
\begin{itemize}
    \item Our global estimate \eqref{eq17.1} is weaker than \eqref{eq15.2}.
    \item  The argument of \cite{GS_06} involves interpolation between  Sobolev spaces with many \textit{spatial derivatives}. We stress that in our problem, the solution $f^{\pm}, \bE, \bB$ has a limited regularity in the spatial variable.
\end{itemize}
To overcome these issues, we establish an integral inequality for lower-order instant energies and dissipations on an arbitrary interval and interpolate between temporal Sobolev spaces.

Next, for $n \in \{20, \ldots m-4\}$, let $\cI_{||, n},$ $\cI_{n}$, and $\cD_{n}$ be  given by \eqref{eq14.0.10}, \eqref{eq14.0.5} and \eqref{eq12.0.1},   respectively, with $m$  replaced with $n$.
Furthermore, by $I_{0, n}$, we denote the sum of the total instant energy up to the order $n$ and the weighted instant energy  up to the order  $n-4$ (cf. \eqref{eq14.0.5}):
\begin{align}
    \label{eq15.20}
I_{0, n} (\tau) := 
\cI_{||, n} (\tau) +  \sum_{k=0}^{n-4} \|\partial_t^k f  (\tau, \cdot)\|^2_{ L_{2, \theta/2^k} (\Omega \times \bR^3) }. 
\end{align}
We observe that due to \eqref{eq14.0.12}, 
\begin{align}
    \label{eq15.21}
I_{0, m} (0)=I_0.
\end{align}

 Then, as in  \eqref{eq17.1}, by taking $\varepsilon = \varepsilon (n, \theta, \Omega, r_1, \ldots, r_4)$ sufficiently small, we obtain for any $0 \le s < t$,
\begin{align}
   \label{eq15.8}
 & \|\cI_n\|_{ L_{\infty} ((s, t)) } + \int_s^t \cD_n \, d\tau \le N I_{0, n} (s),
\end{align}
where $N = N (\Omega, n, \theta, r_1, \ldots, r_4).$
In the sequel, we will estimate $\int_s^t \cI_{n} \, d\tau$ in terms of $\int_s^t \cD_{n} \, d\tau$, which is analogous to \eqref{eq15.5}.

\textit{Estimate of $\int_s^t \cI_{||, n} \, d\tau$.}
First, we   note that for any $0 \le s < t$ such that $t-s \ge 1$, the constant in the interpolation inequality for  Sobolev spaces $W^k_2 ((s, t)), k \in \{0, 1, \ldots\}$, is independent of $s$ and $t$. Then, denoting
$$
    \gamma = \frac{n}{m-4}
$$ 
 and using the interpolation and H\"older's inequalities, we have
\begin{align}
\label{eq15.6}
 &    \sum_{k=0}^{n}   \int_{\Omega}  \int_s^t |\partial_t^k \bE|^2  \,d\tau \, dx
    \lesssim_{\gamma}  \int_{\Omega}  \|\bE\|^{2(1-\gamma)}_{ L_2 ((s, t)) }   \|\bE\|^{ 2 \gamma}_{ W^{m-4}_2 ((s, t)) } \, dx \\
& \lesssim_{\gamma} \|\bE\|^{ 2 (1-\gamma) }_{ L_2 ((s, t) \times \Omega) }  \bigg(\sum_{k=0}^{m-4} \|\partial_t^k \bE\|_{ L_2 ((s, t) \times \Omega) }^{2} \bigg)^{ \gamma}.\notag
\end{align}
We recall that $\|\bE\|_{L_2 (\Omega)}^2$ is in the dissipation $\cD_{n}$. Furthermore, by the global estimate \eqref{eq15.8} with $s=0$ and $m$ in place of $n$, and \eqref{eq15.21}, the last factor on the r.h.s. of the second inequality in \eqref{eq15.6} is bounded by $N I_{0}^{\gamma}$, which gives
\begin{align}
    \label{eq15.9}
    \sum_{k=0}^{n}   \int_s^t \|\partial_t^k \bE\|_{  L_2 (\Omega) }^2  \,d\tau \le N I_{0}^{ \gamma} \bigg(\int_s^t \cD_{n} \, d\tau\bigg)^{ (1-\gamma)}.
\end{align}
Similarly, one has
\begin{align}
    \label{eq15.22}
     \sum_{k=0}^{n}  \|\partial_t^k \bB\|^2_{  L_2 ((s, t) \times \Omega) }, \quad
      \sum_{k=0}^{n}  \|\partial_t^k f\|^2_{  L_2 ((s, t) \times \Omega \times \bR^3) }
\le N (\text{r.h.s. of} \, \eqref{eq15.9}).
\end{align}

\textit{$L_2 ((s, t)) W^1_2 (\Omega)$-norms of the electromagnetic field.}
We note that by the div-curl estimates \eqref{eq4.11}-\eqref{eq4.12} and \eqref{eq15.9}-\eqref{eq15.22}, 
\begin{align} 
    \label{eq15.23}
&\sum_{k=0}^{n-1} \int_s^t  \|\partial_t^k  [\bE, \bB]\|^2_{ W^1_{ 2 } (\Omega) } \, d\tau \lesssim_{\Omega}  \sum_{k=0}^{n} \int_s^t  \|\partial_t^k  [\bE, \bB]\|^2_{ L_{ 2 } (\Omega) } \, d\tau \\
& + \sum_{k=0}^{n-1} \int_s^t \|\partial_t^k f\|^2_{ L_2 (\Omega \times \bR^3) }\, d\tau   \le N (\text{r.h.s. of} \, \eqref{eq15.9}). \notag
\end{align}

\textit{Estimate of the remaining terms in $\int_s^t \cI_n \, d\tau$.} We note that the $L_{2, \theta/2^k}$-norms of the $t$-derivatives of $f$, the $W^1_{r_i} (\Omega), i \in \{2, 3, 4\}$ norms of $\partial_t^k [\bE, \bB]$, and the steady $S_{r_i}$-norms in $\cI_n$ (cf. \eqref{eq14.0.5}) are also present in the dissipation $\cD_n$ (cf. \eqref{eq12.0.1}). Hence, by 
this and the global estimate \eqref{eq15.8}, for each such term $h$, we have
\begin{align}
    \label{eq15.24}
 \int_s^t h \, d\tau = \big(\int_s^t h \, d\tau\big)^{\gamma} \big(\int_s^t h \, d\tau\big)^{1-\gamma}  \le N (\text{r.h.s. of} \, \eqref{eq15.9}). 
\end{align}

Thus, combining \eqref{eq15.9}--\eqref{eq15.24}, we have
\begin{align}
    \label{eq15.25}
 \int_s^t \cI_{n} \,d\tau \le N  I_{0}^{\gamma} \bigg(\int_s^t \cD_{n }  \, d\tau\bigg)^{1-\gamma},
\end{align}
which implies
\begin{align*}
  (I_{0})^{-\gamma/(1-\gamma)} \bigg(\int_s^t \cI_{n} \,d\tau\bigg)^{1/(1-\gamma)} \le N \int_s^t \cD_{n} \, d\tau.
\end{align*}
By this and \eqref{eq15.8}, we obtain
$$
      \|\cI_n\|_{ L_{\infty} ((s, t)) } +  N_1 (I_{0})^{-\gamma/(1-\gamma)} \bigg(\int_s^t \cI_{n } \, d\tau\bigg)^{1/(1-\gamma)} \le N_2  I_{0, n} (s),
$$
where
$
     N_i = N_i (n, m, \theta, \Omega, r_1, \ldots, r_4), i = 1, 2.
$
Since $t>1$ is arbitrary, we have
$$
    (I_{0})^{-\gamma/(1-\gamma)} (\mathcal{Z} (s))^{1/(1-\gamma)} \le N \, \cI_{n } (s),
$$
where
$$
    \mathcal{Z} (s)   = \int_s^{\infty} \cI_{n} (\tau) \, d\tau.
$$
Then,  
$$
    \mathcal{Z}' (s)  = -    \cI_{n } (s)  \le - N  \, (I_{0})^{-\gamma/(1-\gamma)} \mathcal{Z}^{1/(1-\gamma)} (s).
$$
Furthermore, for the sake of convenience, we denote  $r = \frac{\gamma}{1-\gamma}$, so that $1/(1-\gamma) = r+1$.
Then, dividing both sides by $\mathcal{Z}^{r+1} (s)$ gives
$$
    -\frac{d}{ds}   (\mathcal{Z} (s))^{-r} \le -  N  (I_{0})^{-r}.
$$
 Integrating, we get
$$
    (\mathcal{Z} (0))^{-r} + N  (I_{0})^{-r} s \le  (\mathcal{Z} (s))^{-r}.
$$
By \eqref{eq15.25} and the global estimate \eqref{eq15.8},
$$
  \mathcal{Z} (0)  \le N I_0.
$$
 Hence,  one has
$$
    \mathcal{Z} (s) \le  N  I_{0} (1+s)^{-1/r}.
$$
Furthermore, for any $s > 1$, applying the estimate \eqref{eq15.8} with  $\tau \in [s, 2s]$ and $3 s$ in place of $s$ and $t$, respectively, gives
$$
    \mathcal{Z} (s) \ge  s \inf_{ s \le \tau \le 2 s } \cI_{n} (\tau)  \ge  N \, s  \, \cI_{n } (3s).
$$
Using this and the fact that $1+ 1/r =  1/\gamma$, we obtain
$$
    \cI_{n } (s) \le   N I_0 (1+s)^{-1/\gamma },
$$
as claimed in \eqref{eq132.3}.
\end{proof}


\appendix



\section{Estimates of nonlinear terms}

\label{appendix B}

\begin{lemma}[$L_{\infty}$ and $C^{\alpha}$ estimates]
    \label{lemma 5.1}
    Let $f$ be a function such that $\cI (\tau), \cD (\tau)$ (see \eqref{eq14.0.5}-\eqref{eq12.0.1}) are finite for each $\tau \in [0, T],$ for some $T$, and $\partial_t^k f, k \le m-8,$ satisfy the SRBC.
Then, for any  $\alpha \in (0,  1 - \frac{12}{r_4})$,   we have
\begin{align}
    \label{eq14.4.7}
& \sum_{k=0}^{m-8} \big(\sum_{r \in \{2, \infty\} }\|\partial_t^k [f, \nabla_p f] (\tau, \cdot)\|_{  L_{\infty} (\Omega) W^1_{r, \theta/2^{k+9}} (\bR^3)  } \\
& + \|\partial_t^k f (\tau, \cdot)\|_{  C^{\alpha/3,  \alpha}_{x, p} (\Omega \times \bR^3)  }  \big)\lesssim_{\Omega, \theta, \alpha,  r_4} \cD (\tau), \cI (\tau), \notag
\end{align}
provided that $\theta$ is sufficiently large.
Furthermore, for any   $\beta \in (0, 1-\frac{3}{ r_4 }) \supset (0, \frac{11}{12})$,
\begin{align}
    \label{eq14.4.6}
  &  \sum_{k=0}^{m-7} \|\partial_t^k [\bE, \bB]  (\tau, \cdot)\|_{   L_{\infty} (\Omega) }
   + \sum_{k=0}^{m-8} \|\partial_t^k [\bE, \bB]  (\tau, \cdot)\|_{ C^{\beta} (\Omega)  } 
\lesssim_{\Omega,  \beta, \theta, r_3, r_4}  \cD (\tau), \cI (\tau). 
\end{align}
\end{lemma}

\begin{proof}
We note that \eqref{eq14.4.7} is a direct corollary of the embedding result for functions of class $S_{r_4  }$  satisfying the SRBC (see \eqref{eq14.C.8.20} in Corollary \ref{corollary 14.C.8}). Furthermore, the estimate \eqref{eq14.4.6} follows from the Sobolev embedding theorem  $W^{1}_{r_3} \subset L_{\infty}$ and $W^1_{r_4} \subset C^{\beta}$ and the fact that $r_4 > 36$.
\end{proof}

For the proof of the following two lemmas, see Lemmas B.3 and B.8 in \cite{VML}.

\begin{lemma}
        \label{lemma B.2}
Let $k \ge 0$ be an integer, $r \in (3/2, \infty]$, and  $g \in W^k_r (\bR^3)$.
Then, for
\begin{equation}
        \label{eqB.2.0}
  I (p) =  \int  \Phi^{i j} (P, Q) J^{1/2} (q)   g (q)  \, dq,
\end{equation}
we have
\begin{equation}
            \label{eqB.2.1}
    \|D^k_p I\|_{ L_{\infty} (\bR^3) } \lesssim_{k, r}  \|g\|_{W^k_r (\bR^3) }.
\end{equation}
\end{lemma}


\begin{lemma}
            \label{lemma 12.A.1}
For  sufficiently regular functions $f_j = (f^{+}_j, f^{-}_j)$, $j = 1, 2, 3,$ on $\bR^3$ and any  $r \in (3/2, \infty]$, and $\theta \ge 0$, we have
\begin{align}
    \label{eq12.A.1.1}
    &  \big|\langle   \Gamma (f_1, f_2),  f_3  p_0^{2 \theta} \rangle\big| \\
   & \lesssim_{\theta} \big(\|\nabla_p f_1 \|_{  L_{2, \theta} (\bR^3) } \|f_2\|_{  L_{r} (\bR^3) }   + \|f_1 \|_{  L_{2, \theta} (\bR^3) } \|\nabla_p f_2\|_{  L_{r} (\bR^3) }\big) \| f_3 \|_{   W^1_{2, \theta} (\bR^3) }. \notag
\end{align}
\end{lemma}

\begin{lemma}
        \label{lemma 12.A.2}
Let $f$ be a function such that $\sup_{s \le \tau \le t}\cI (\tau),  \int_s^t \cD \, d\tau$ are finite, and $\partial_t^k f, k \le m-8,$ satisfy the SRBC.
Then, the following assertions hold. 

$(i)$ For any  $\xi  = (\xi^{+}, \xi^{-}) \in  W^2_{2, 1} (\bR^3)$, we have
\begin{align}
    \label{eq12.A.2.3}
          &\sum_{k = 0}^m \int_s^{t}  \int_{ \Omega}  |\langle \partial_t^k \Gamma (f, f),  \xi\rangle|^2 \, dx d\tau \\
          &\lesssim_{\xi, \Omega, \theta, r_4} \|\cI\|_{ L_{\infty} ((s, t)) }  \int_s^t \cD  \, d\tau.  \notag
\end{align}

$(ii)$ Let $\zeta  = \zeta (x, p) \in L_{\infty} (\Omega \times \bR^3)$ be a function such that
$
    \nabla_p \zeta \in L_{\infty} (\Omega \times \bR^3).
$
Then, one has
\begin{align}
    \label{eq12.A.2.1}
& \sum_{k=0}^{m-2}         \bigg|\int_s^{t}  \int_{ \Omega } \langle \partial_t^k \Gamma (f, f), \partial_t^k f \zeta  \rangle   \,  dx d\tau \bigg| \\
  &   \lesssim_{\zeta, \Omega, \theta, r_4}  \|\cI\|_{ L_{\infty} ((s, t)) }^{1/2} \int_s^t \cD \, d\tau. \notag
\end{align}

$(iii)$ For any $k \le m$, 
  \begin{equation}
    \label{eq12.A.2.2}
    \int_s^{t}  \int_{ \Omega } \langle \partial_t^k \Gamma (f, f), \partial_t^k f \rangle  \, dx d\tau
    \lesssim_{ \Omega, \theta, r_4 }   \|\cI\|_{ L_{\infty} ((s, t)) }^{1/2}   \int_s^t \cD \, d\tau.
  \end{equation}
\end{lemma}

\begin{proof}

$(i)$
First, we claim that, to prove \eqref{eq12.A.2.3}, it suffices to show that
\begin{align}
    \label{eq12.A.2.4}
      &   |\langle \partial_t^k \Gamma (f, f) (\tau, x, \cdot), \xi (\cdot)\rangle| \\
      & \lesssim_{\xi} \sum_{l=0}^{m/2} \|\partial_t^{l} f (\tau, x, \cdot)\|_{ W^1_2 (\bR^3) }  \sum_{l=m/2}^{m} \|\partial_t^{l} f (\tau, x, \cdot)\|_{ L_2 (\bR^3) }.\notag
\end{align}
Indeed, if this is true, then by the $L_2^t L_{\infty}^x$-$L_{\infty}^t L_2^x$ H\"older's inequality and the estimate \eqref{eq14.4.7} in Lemma \ref{lemma 5.1},
the left-hand side of \eqref{eq12.A.2.3} is dominated by
\begin{align*}
  &  \sum_{l=0}^{m/2}  \|\partial_t^{l} f\|^2_{ L_2 ((s, t)) L_{\infty} (\Omega) W^1_2 (\bR^3) }  \sum_{l=m/2}^{m} \|\partial_t^{l} f\|^2_{ L_{\infty} ((s, t))  L_2 (\Omega \times \bR^3) }  \\
& \lesssim  \|\cI\|_{ L_{\infty} ((s, t)) }    \int_s^t \cD  \, d\tau, \notag
\end{align*}
as desired.

Furthermore, for the sake of simplicity,  we assume that $f$ and $\xi$ are scalar functions and  replace the integral with a simplified expression 
(cf. formula $(68)$ on p. 290 in \cite{GS_03}): 
\begin{align}
    \label{eq12.A.1.4}
&     I =    \langle  \big(\partial_{p_i} - \frac{ p_i}{2 p_0}\big)    \int \Phi^{i j} (P, Q) J^{1/2} (q) (\partial_{p_j} f (p))  f (q)  \, dq, \xi  \rangle\\
& - \langle  \big(\partial_{p_i} - \frac{ p_i}{2 p_0}\big)    \int \Phi^{i j} (P, Q) J^{1/2} (q) f (p)  (\partial_{q_j} f (q))  \, dq, \xi \rangle. \notag
\end{align}

Next, we fix nonnegative integers $k_1+k_2 = k$.
Integrating by parts in $p$ in \eqref{eq12.A.1.4} gives
\begin{align}
    \label{eq12.A.1.2}
  &  I =    - \langle \partial_{p_j} f    \int \Phi^{i j} (P, Q) J^{1/2} (q)  f (q)  \, dq, (\partial_{p_i} + \frac{ p_i}{2 p_0}) \xi \rangle \\
 & + \langle  f    \int \Phi^{i j} (P, Q) J^{1/2} (q)   \partial_{q_j} f (q)  \, dq, (\partial_{p_i}  + \frac{ p_i}{2 p_0}) \xi\rangle
  =: I_1+I_2. \notag
\end{align}
Then, to prove \eqref{eq12.A.2.4},  it suffices to estimate two types of terms:
\begin{align}
  &  I_1 = - \int (\partial_{p_j} \partial_t^{k_1} f)    \bigg(\int \Phi^{i j} (P, Q) J^{1/2} (q) (\partial_t^{k_2} f)  \, dq\bigg) \, (\partial_{p_i} + \frac{ p_i}{2 p_0}) \xi\, dp \notag \\
  \label{eq12.A.2.17}
 & I_2 =  \int (\partial_t^{k_1} f)   \bigg(\int \Phi^{i j} (P, Q) J^{1/2} (q)   \partial_{q_j} (\partial_t^{k_2} f)  \, dq\bigg)\, (\partial_{p_i}  + \frac{ p_i}{2 p_0}) \xi \, dp. 
\end{align}

\textit{Estimate of $I_1$.}
In the case when $k_1 \le m/2$, we use \eqref{eqB.2.1} in Lemma \ref{lemma B.2} with $k=0$ and obtain
\begin{align*}
    I_1 \lesssim \|\xi\|_{W^1_2 (\bR^3)} \|\nabla_p \partial_t^{k_1} f (t, x, \cdot)\|_{ L_2 (\bR^3) }  \|\partial_t^{k_2} f (t, x, \cdot)\|_{ L_2 (\bR^3) },
\end{align*}
where the right-hand side is less than that of \eqref{eq12.A.2.4}.
In the case when $k_1 > m/2$, integrating by parts in the $p_j$ variable gives
\begin{align*}
&    I_1 = \int  \partial_t^{k_1} f    \bigg(\int \Phi^{i j} (P, Q) J^{1/2} (q) \partial_t^{k_2} f  \, dq\bigg) \, \partial_{p_j} (\partial_{p_i} + \frac{ p_i}{2 p_0}) \xi\, dp \\
    & +  \int  \partial_t^{k_1} f    \bigg(\partial_{p_j} \int \Phi^{i j} (P, Q) J^{1/2} (q) \partial_t^{k_2} f   \, dq\bigg) \, (\partial_{p_i} + \frac{ p_i}{2 p_0}) \xi\, dp.
\end{align*}
By using the Cauchy-Schwarz inequality and the estimate \eqref{eqB.2.1} in Lemma \ref{lemma B.2} with $k \in \{0, 1\}$, we get
\begin{align}
    \label{eq12.A.2.5}
  I_1    \lesssim \|\xi\|_{W^2_2 (\bR^3)} \|\partial_t^{k_1} f (t, x, \cdot)\|_{ L_2 (\bR^3) }  \|\partial_t^{k_2} f (t, x, \cdot)\|_{ W^1_2 (\bR^3) },
\end{align}
and the right-hand side is less than that  of \eqref{eq12.A.2.4} since $k_2 \le m/2$.

\textit{Estimate of $I_2$.} We only need to consider the case when $k_2 > m/2$ as the remaining case is handled as in \eqref{eq12.A.2.5}.
We first state the key idea formally. One can rewrite the integral with respect to $q$ in $I_2$  as
$$
    p_0 \partial_{p_j} \bigg(\int \Phi^{i j} (P, Q) J^{1/2} (q) \partial_t^{k_2} f  \, dq\bigg) + \text{`zero-order' terms}.
$$
Then, integrating by parts in $p_j$, we can move the derivative to the factors $\partial_t^{k_1} f$ or $(\partial_{p_i}  + \frac{ p_i}{2 p_0})  \xi$, which are `good'.
To justify this argument rigorously, we first recall the following identity on p. 281  in the proof of Theorem 3  in \cite{GS_03}:
\begin{align}
    \label{eq12.A.2.15}
	& \partial_{p_j} \int_{ \bR^3 } \Phi^{i j} (P, Q) J^{1/2} (q) h (q) \, dq\\
&	=  \int   \Phi^{i j} (P, Q) J^{1/2} (q) \frac{q_0}{p_0}\partial_{q_j} h (q) \, dq \notag \\
&\quad	+ \int   \Phi^{i j} (P, Q) J^{1/2} (q) \big(\frac{q_j}{q_0 p_0}  - \frac{q_j}{2 p_0}\big) h (q) \, dq \notag\\
 &\quad	+  \int  (\partial_{p_j} + \frac{q_0}{p_0} \partial_{q_j}) \Phi^{i j} (P, Q) J^{1/2} (q) h (q) \, dq. \notag
 \end{align}
Multiplying the last identity by $p_0$ and  replacing $h (q)$ with $\frac{1}{q_0} f (q)$, we get
 \begin{align*}
    & \int   \Phi^{i j} (P, Q) J^{1/2} (q) \partial_{q_j} f (q) \, dq = p_0  \partial_{p_j} \int   \Phi^{i j} (P, Q) J^{1/2} (q) \frac{1}{q_0} f (q) \, dq  \\
    & - \int   \Phi^{i j} (P, Q) J^{1/2} (q)  \frac{q_j}{q_0}  f (q) \, dq \\
    &  + \int   \Phi^{i j} (P, Q) J^{1/2} (q) \big(\frac{q_j}{q_0^2 }  + \frac{q_j}{2 q_0}\big)  f (q) \, dq \\
    &  - p_0 \int  \big((\partial_{p_j} + \frac{q_0}{p_0} \partial_{q_j}) \Phi^{i j} (P, Q)\big) J^{1/2} (q) \frac{1}{q_0} f (q) \, dq.
 \end{align*}
Then, by the definition of $I_2$ in \eqref{eq12.A.2.17},  the last identity, and the bound
\begin{align}
    \label{eq12.A.2.60}
    |(\partial_{p_j} + \frac{q_0}{p_0} \partial_{q_j}) \Phi^{i j} (P, Q)| + |\Phi^{i j} (P, Q)| \lesssim  q_0^7 (1+ |p-q|^{-1})
\end{align}
(see Lemma 2 on p. 277  in \cite{GS_03}),
we conclude that to handle $I_2$, it suffices to estimate the integral
$$
   \mathfrak{I} (t, x): =  \int \big(\partial_t^{k_1} f (t, x, p)\big)  \,  p_0^{n}  \, (\partial_{p_j}^{l} \cI (t, x, p)) \,  (\partial_{p_i}  + \frac{ p_i}{2 p_0}) \xi (p) \, dp,
$$
where $l, n \in \{0, 1\}$, and
$$
     \cI (t, x, p): = \int \Xi (p, q) J^{1/4} (q) \partial_t^{k_2} f (t, x, q) \, dq, \quad   |\Xi (p, q)| \lesssim  1+ |p-q|^{-1}.
$$
By the Cauchy-Schwarz inequality,
$$
    |\cI (t, x, p)| \lesssim \|\partial_t^{k_2} f (t, x, \cdot)  \|_{ L_2 (\bR^3)}.
$$
Then, integrating by parts in $p_j$ and using the Cauchy-Schwarz inequality,
we obtain
\begin{align}
\label{eq12.A.2.18}
    |\mathfrak{I}| \lesssim  \|\partial_t^{k_1} f (t, x, \cdot)  \|_{ W^1_2 (\bR^3)} \|\partial_t^{k_2} f (t, x, \cdot)  \|_{ L_2 (\bR^3)} \|\xi\|_{ W^2_{2, 1} (\bR^3)}.
\end{align}
We note that since $k_1 \le m/2$, the right-hand side of \eqref{eq12.A.2.18} is less than that in \eqref{eq12.A.2.4}.
Thus, by this and \eqref{eq12.A.2.5}, the inequality \eqref{eq12.A.2.4} is true, and hence, so is the desired estimate \eqref{eq12.A.2.3}.

$(ii)$ By the estimate \eqref{eq12.A.1.1} in Lemma \ref{lemma 12.A.1}  and the $L_{\infty}^{t, x}$-$L_2^{t, x}$-$L_2^{t, x}$ H\"older's inequality, the integral on the left-hand side of \eqref{eq12.A.2.1} is dominated by
\begin{align*}
  \|\partial_t^{k} f \zeta\|_{ L_2 ((s, t) \times \Omega) W^1_{2} (\bR^3) }
  \sum_{ l \le m/2}
     \bigg(\|\partial_t^{l} f \|_{  L_{\infty} ((s, t) \times \Omega) W^1_{2} (\bR^3)  }
    \|\partial_t^{k-l} f\|_{ L_2 ((s, t) \times \Omega) W^1_{2} (\bR^3) }\bigg).
 \end{align*}
We note that
\begin{itemize}
    \item[--]  since $\zeta, \nabla_p \zeta \in L_{\infty}^{x, p}$, we may drop this function from the above inequality,
    \item[--]  by \eqref{eq12.0.4} in Remark \ref{remark 12.1}, the factors of type
$$\|\partial_t^n f\|_{ L_2 ((s, t) \times \Omega)) W^1_2 (\bR^3)}, n \le m-2,$$ are bounded by
$
    N  \, (\int_s^t \cD  \, d\tau)^{1/2}.
$
    \item[--]  by \eqref{eq14.4.7} in Lemma \ref{lemma 5.1}, the  first factor inside the parenthesis is bounded by $N \|\cI\|_{ L_{\infty} ((s, t)) }^{1/2}$.
\end{itemize}
Thus, \eqref{eq12.A.2.1} is valid.

$(iii)$
First, we split the integral into
\begin{align}
    \label{eq12.A.2.100}
     &  \int_s^{t}  \int_{ \Omega  } \langle \partial_t^k \Gamma (f, f), P \partial_t^k f \rangle   \, dx d\tau \\
      & +  \int_s^{t}  \int_{ \Omega } \langle \partial_t^k \Gamma (P f, P f), (1-P) \partial_t^k f \rangle  \, dx d\tau \notag \\
     & + \int_s^{t}  \int_{ \Omega  } \langle\partial_t^k \Gamma ((1-P) f, P f), (1-P) \partial_t^k f \rangle  \, dx d\tau \notag \\
  &   +  \int_s^{t}  \int_{ \Omega  } \langle\partial_t^k \Gamma (P f, (1-P) f), (1-P) \partial_t^k f \rangle  \, dx d\tau  \notag\\
 & +   \int_s^{t}  \int_{ \Omega  } \langle \partial_t^k \Gamma ((1-P) f, (1-P) f), (1-P) \partial_t^k f \rangle \, dx d\tau=: I_1 + I_2 + I_3 + I_4 + I_5. \notag
       \end{align}
We note that $I_1$ vanishes due to the product rule and
the fact that
\begin{align*}
     \langle \Gamma (f_1, f_2), P f_3 \rangle  = 0,
\end{align*}
which is easily derived from the identities (see \cite{GS_03})
\begin{align*}
  &   \int_{\bR^3}   \cC (f^{\pm},  g^{\pm}) \, dp = 0,  \\
  &\int_{\bR^3} p \,  \cC (f^{\pm},  g^{\pm}) \, dp = 0, \quad \int_{\bR^3} p \,  \big(\cC (f^{+},  g^{-}) + \cC (f^{-},  g^{+})\big) \, dp = 0, \\
       & \int_{\bR^3} p_0^{\pm} \cC (f^{\pm}, g^{\pm}) \, dp = 0, \quad \int_{\bR^3} \big(p_0^{+} \cC (f^{+}, g^{-}) + p_0^{-} \cC (f^{-}, g^{+})\big) \, dp = 0.
\end{align*}

Next, by the product rule,
\begin{align*}
     &   |I_2| \lesssim \sum_{k_1+k_2  = k} \int_s^{t} \int_{\Omega} |\partial_t^{k_1} [a^{+}, a^{-}, b, c]|  \, |\partial_t^{k_2} [a^{+}, a^{-}, b, c]| \,  |W_k| \, dx d\tau,
\end{align*}
where $W_k (t, x)$ is a linear combination of terms
$$
    \int_{\bR^3} \widetilde \xi (p) (1-P^{\pm}) \partial^k_t f (t, x, p) \, dp,
$$
and $\widetilde \xi \in \{\Gamma_{\pm} (\chi_i^{\pm}, \chi_j^{\pm}), i, j = 1, \ldots, 6\}$.
We note that
\begin{itemize}
    \item[--] by \eqref{eq12.A.1.1} in Lemma \ref{lemma 12.A.1},
\begin{align*}
        |W_k (t, x)|  \lesssim \|(1-P) \partial_t^k f (t, x, \cdot)\|_{ W^1_2 (\bR^3)},
\end{align*}
\item[--] by the Cauchy-Schwarz inequality,  for any $l \le m$,
\begin{align}
    \label{eq12.A.2.10}
    |\partial_t^l [a^{\pm}, b, c]| (t, x) \le \|\partial_t^l f (t, x, \cdot)\|_{ L_2 (\bR^3) }.
\end{align}
\end{itemize}
Then, by
the $L_{2}^t L_{\infty}^x$-$L_{\infty}^t L_2^x$-$L_2^{t, x}$ H\"older's inequality and the 
estimate \eqref{eq14.4.7} in Lemma \ref{lemma 5.1}, we get
\begin{align*}
       &  |I_2| \lesssim \|\partial_t^k (1-P) f\|_{ L_2 ((s, t) \times \Omega) W^1_{2} (\bR^3) } \\
&\times  \big(\sum_{l \le m/2} \|\partial_t^{l} f\|_{ L_2 ((s, t)) L_{\infty} (\Omega)  L_2 (\bR^3) }
\|\partial_t^{k-l} f\|_{ L_{\infty} ((s, t)) L_{2} (\Omega \times \bR^3) }\big) \\
& \lesssim  \big(\int_s^t \cD_{||} \, d\tau\big)^{1/2} \big(\int_s^t \cD  \, d\tau\big)^{1/2} 
\|\cI_{||}\|_{ L_{\infty} ((s, t)) }^{1/2} \\
& \le \|\cI_{||}\|_{ L_{\infty} ((s, t)) }^{1/2} \int_s^t \cD  \, d\tau.
\end{align*}

Furthermore,     applying  \eqref{eq12.A.1.1} in Lemma \ref{lemma 12.A.1} first
 and then, using the $L_2^t L_{\infty}^x -  L_{\infty}^t L_2^x - L_2^{t, x}$
and  the $L_{\infty}^{t, x}$-$L_2^{t, x}$-$L_2^{t, x}$ H\"older's inequalities,
we have
\begin{align*}
  &  |I_3| \lesssim  \|\partial_t^k (1-P) f\|_{ L_2 ((s, t) \times \Omega) W^1_{2} (\bR^3) } \mathcal{J}, \\
&\mathcal{J}= \sum_{l \le m/2} \|\partial_t^{l}  f\|_{ L_2 ((s, t)) L_{\infty} (\Omega) W^1_2 (\bR^3) }
\| P  (\partial_t^{k-l} f)\|_{ L_{\infty} ((s, t)) L_{2} (\Omega)  W^1_2  (\bR^3) } \\
&+ \sum_{l \le m/2}  \|P \partial_t^{l} f\|_{ L_{\infty} ((s, t) \times \Omega)  W^1_2 (\bR^3) }
\|\partial_t^{k-l} (1-P) f\|_{ L_2 ((s, t) \times \Omega) W^1_{2} (\bR^3) } =: \mathcal{J}_1+\mathcal{J}_2.
    \end{align*}
We note that
\begin{itemize}
    \item[--] the factors involving $(1-P) f$ are bounded by $(\int_s^t \cD_{||} \, d\tau)^{1/2}$,
    \item[--] due to the estimate \eqref{eq14.4.7} in Lemma \ref{lemma 5.1} and the fact that $m > 16$, the first factor in $\mathcal{J}_1$  is bounded by
    $N  (\int_s^t \cD \, d\tau)^{1/2}$,
        \item[--] by \eqref{eq12.A.2.10}, the second factor in $\mathcal{J}_1$   is bounded by $N \,  \|\cI\|_{ L_{\infty} ((s, t)) }^{1/2}$,
    \item[--] again, by  the estimates \eqref{eq12.A.2.10} and \eqref{eq14.4.7}, the first factor in $\mathcal{J}_2$ is bounded by $N \,  \|\cI\|_{ L_{\infty} ((s, t)) }^{1/2}$.
\end{itemize}
Thus,  we conclude that
\begin{align*}
    |I_{3}| &\lesssim  \|\cI\|_{ L_{\infty} ((s, t)) }^{1/2}  \int_s^t \cD \, d\tau.
\end{align*}
Similarly, the above estimate holds with $I_3$ replaced with $I_4$ (see \eqref{eq12.A.2.100}).
Finally, by \eqref{eq12.A.1.1} in Lemma \ref{lemma 12.A.1}, the $L_{\infty}^{t, x}$-$L_2^{t, x}$-$L_2^{t, x}$ H\"older's inequality, the estimate \eqref{eq14.4.7}, and the fact that $m > 16$, we find
\begin{align*}
  &  |I_5| \lesssim  \|\partial_t^k (1-P) f\|_{ L_2 ((s, t) \times \Omega) W^1_{2} (\bR^3) } \\
&\times  \big(\sum_{l \le m/2} \|\partial_t^{l} f\|_{ L_{\infty} ((s, t) \times \Omega) W^1_{2} (\bR^3) }
\|\partial_t^{k-l} (1-P) f\|_{ L_2 ((s, t) \times \Omega) W^1_{2} (\bR^3) }\big) \\
      &  \lesssim   \|\cI\|_{ L_{\infty} ((s, t)) }^{1/2}  \int_s^t \cD_{||}  \, d\tau.
    \end{align*}
Thus, the desired bound \eqref{eq12.A.2.2} holds, and the lemma is proved.
\end{proof}

\begin{lemma}
        \label{lemma 12.A.3}
  Let $\xi \in W^1_2 (\bR^3)$, $\zeta$ be a function such that $\zeta, \nabla_p \zeta \in L_{\infty} (\Omega \times \bR^3)$, and
  $$
    H^{\pm}  =   (\bE  + \frac{p}{p_0^{\pm}} \times \bB) \cdot \nabla_p f^{\pm} \,  \text{or} \,   \frac{p}{p_0^{\pm}}  \cdot \bE \,  f.
  $$
  Then, under the assumptions of Lemma \ref{lemma 12.A.2},  the following estimates are valid:
\begin{align}
\label{eq12.A.3.2}
 &  I_1 =  \sum_{k=0}^m  \int_s^{t} \int_{\Omega} \bigg| \int_{\bR^3} \partial_t^k H^{\pm} \cdot \xi \, dp\bigg|^2 \, dx d\tau\\
 &   \lesssim_{\xi, \Omega, r_3, r_4, \theta} \|\cI_{ || }\|_{ L_{\infty} ((s, t)) }  \int_s^t \cD  \, d\tau, \notag\\
            \label{eq12.A.3.1}
    & I_2 =  \sum_{k = 0}^{m-2}   \bigg|\int_s^{t} \int_{ \Omega \times \bR^3  }    (\partial_t^{k} H^{\pm}) \cdot  (\partial_t^{k} f^{\pm}) \,  \zeta  \, dz\bigg| \\
    & \lesssim_{\zeta, \Omega, r_3, r_4, \theta} \|\cI_{ || }\|_{ L_{\infty} ((s, t)) }^{1/2}  \int_s^t \cD  \, d\tau. \notag
\end{align}
\end{lemma}

\begin{proof} 
\textit{Estimate of $I_1$.} We will consider the case when $H^{\pm} = \bE \cdot \nabla_p f^{+}$, as the remaining cases are handled similarly.
Integrating by parts in $p$ in $I_1$, we move the $p$-derivative to the factor $\xi$. Then, by using the Cauchy-Schwarz inequality in the $p$ variable first and the $L_{\infty}^t L_2^x$-$L_2^t L_{\infty}^x$ H\"older's inequality, we get
\begin{align*}
 & I_1 \lesssim_{\xi} \big(\sum_{ l  = 0}^{ m/2} \|\partial_t^{l} \bE\|^2_{ L_2 ((s, t)) L_{\infty} (\Omega) }\big)  \big(\sum_{l = m/2}^{ m} \|\partial_t^{l} f\|^2_{ L_{\infty} ((s, t)) L_2 (\Omega \times \bR^3) }\big) \\
 & + \big(\sum_{l = 0}^{ m/2}  \|\partial_t^{l} f\|^2_{ L_{2} ((s, t)) L_{\infty} (\Omega) L_2 (\bR^3) }\big)  \big(\sum_{l = m/2}^{ m}\|\partial_t^{l} \bE\|^2_{ L_{\infty} ((s, t)) L_{2} (\Omega) }\big).
\end{align*}
We note that
 by \eqref{eq14.4.7}--\eqref{eq14.4.6} in Lemma \ref{lemma 5.1}, the first factors in each term  are bounded by $N \, \int_s^t \cD \, d\tau$, and hence,
 \begin{align*}
     I_1 \lesssim  \|\cI_{||}\|_{ L_{\infty} (s, t) } \int_s^t \cD \, d\tau.
 \end{align*}

 \textit{Estimate of $I_2$.} Integrating by parts in $p$ in $I_2$,  using the Cauchy-Schwarz inequality in the $p$ variable first and then, 
 the $L_2^{t, x}$-$L_{\infty}^t L_2^x$-$L_2^t L_{\infty}^x$ H\"older's  inequality, and invoking \eqref{eq12.0.4} in Remark \ref{remark 12.1}, we have
\begin{align}
\label{eq12.A.3.5}
   & I_2  \lesssim_{\zeta}  \bigg(\sum_{k=0}^{m-2}  \|\partial_t^{ k }  f \zeta\|_{ L_2 ((s, t) \times \Omega) W^1_{2} (\bR^3) }\bigg)  \sum_{k \le m-2, k_1+k_2=k} M_{k_1, k_2}, \\
\label{eq12.A.3.6}
   & 
  M_{k_1, k_2} := \|\partial_t^{k_1} [\bE, \bB]\|_{L_{\infty} ((s, t)) L_2 (\Omega)}
\big(1_{ k_2 \le m/2} \|\partial_t^{k_2} f\|_{ L_2 ((s, t)) L_{\infty} (\Omega) L_2 (\bR^3) }\big)    \\
   &  +   \|\partial_t^{k_2} f\|_{L_{\infty} ((s, t)) L_2 (\Omega \times \bR^3)  } \big(1_{ k_1 \le m/2}  \|\partial_t^{k_1} [\bE, \bB]\|_{ L_2 ((s, t)) L_{\infty} (\Omega) }\big)=:M_{k_1, k_2, 1}+M_{k_1, k_2, 2}. \notag
\end{align}
Furthermore,
\begin{itemize}
    \item[--]  by the estimate \eqref{eq12.0.4} in Remark \ref{remark 12.1}, we may bound the first factor on the r.h.s. of \eqref{eq12.A.3.5}  by
    $N (\int_s^t \cD\,  d\tau)^{1/2}$,
   \item[--] by estimates \eqref{eq14.4.7}--\eqref{eq14.4.6} in Lemma \ref{lemma 5.1}, the second factors in  $M_{k_1, k_2, j}, j=1, 2,$ are bounded by $N (\int_s^t \cD\,  d\tau)^{1/2}$, which gives
   \begin{align}
    \label{eq12.A.3.7}
    M_{k_1, k_2} \lesssim_{\theta, \Omega, r_3, r_4} \|\cI_{||}\|_{ L_{\infty} ((s, t)) }^{1/2}  (\int_s^t \cD\,  d\tau)^{1/2}.
   \end{align}
\end{itemize}
Thus, we have
\begin{align*}
   I_2 \lesssim  \|\cI_{ || }\|_{ L_{\infty} ((s, t)) }^{1/2}  \int_s^t \cD\,  d\tau,
\end{align*}
and, hence, the desired estimate \eqref{eq12.A.3.1} is valid.
\end{proof}

\begin{lemma}
    \label{lemma 12.A.4} 
    Let $Q (x, x'), x, x' \in \bR^3,$ be a quadratic polynomial.
Then, under the assumptions of Lemma \ref{lemma 12.A.2},   we have 
\begin{align}
    \label{eq12.A.4.1}
   &  \sum_{k=0}^m \int_s^{t} \bigg|\int_{\Omega} |\partial_{\tau}^k Q (\bE (\tau, x), \bB (\tau, x))| \, dx\bigg|^2 \, d\tau \\
   & \lesssim_{\Omega, r_3}  \|\cI_{||}\|_{ L_{\infty} ((s, t)) }  \int_s^t \cD  \, d\tau,  \notag\\
      \label{eq12.A.4.2}
    &  \sum_{k=0}^m \bigg|\int_{\Omega}  |\partial_{\tau}^k Q (\bE (\tau, x), \bB (\tau, x))|   \, dx\bigg|^2  
    \lesssim_{\Omega}   \cI^{2}_{||} (\tau). 
\end{align}
\end{lemma}

\begin{proof}
For the sake of simplicity, let us consider the case when $Q (x, x') = x_i x_j'$. The general case can be handled by the same argument.
By the Cauchy-Schwarz inequality, we have
\begin{align}
    \label{eq12.A.4.5}
&\bigg| \int_{\Omega} |\partial_{t}^k \big(\bE_i (\tau, x) \bB_j  (\tau, x)\big)| \, dx\bigg|^2  \\
& \le \sum_{k_1 + k_2 = k}  \|\partial_{t}^{k_1} \bE (\tau,\cdot)\|^2_{L_2 (\Omega) } \, \|\partial_{t}^{k_2} \bB (\tau,\cdot)\|^2_{L_2 (\Omega) }  \le \cI_{||}^2 (\tau), \notag
\end{align}
which gives the desired estimate \eqref{eq12.A.4.2}.
By the definition of $\cD$ in \eqref{eq12.0.1},  we may also replace the r.h.s. of \eqref{eq12.A.4.5}  with $\cI_{||} (\tau) \cD (\tau)$. Integrating the last expression over $\tau \in (s, t)$  gives \eqref{eq12.A.4.1}.
\end{proof}


\begin{lemma}
    \label{lemma 12.A.7}
Invoke the assumptions of Lemma \ref{lemma 12.A.2}.
Assume that $\Omega$ is an axisymmetric domain such that its axis is parallel to some vector $\omega$ and contains a point $x_0$. Denote $R = \omega \times (x-x_0)$. Then, we have
\begin{align}
    \label{eq12.A.7.1}
  \sum_{k=0}^{m+1} \int_s^t  \bigg|\partial_{t}^k \int_{\Omega} R \cdot (\bE  \times \bB)   \, dx\bigg|^2  \, d\tau
 \lesssim_{ \Omega, \theta, r_3, r_4 }  \|\cI_{||}\|_{ L_{\infty} ((s, t)) } \,  \int_s^t \cD  \, d\tau.
\end{align}
\end{lemma}

\begin{proof}
Thanks to \eqref{eq12.A.4.1} in Lemma \ref{lemma 12.A.4}, we only need to estimate the term with $k=m+1$.
By using the angular momentum identity for Maxwell's equations in \eqref{eqH.4}, the $L_{\infty}^t L_2^x$-$L_2^t L_{\infty}^x$ H\"older's  inequality,  and  \eqref{eq12.A.3.6}--\eqref{eq12.A.3.7}, we obtain
\begin{align*}
    & \int_s^t \bigg|\partial_{t}^{ m+1 } \int_{\Omega} R \cdot (\bE  \times \bB)   \, dx\bigg|^2 \, d\tau  \lesssim
   \int_s^t  \bigg|\partial_{t}^{m}  \int_{\Omega} R \cdot (\rho \bE + \bm{j} \times \bB)  \, dx\bigg|^2 \, d\tau   \\
   & \le \sum_{k_1+k_2 =  m }  M_{k_1, k_2}^2
    \lesssim_{\Omega, \theta, r_3, r_4}\|\cI_{||}\|_{ L_{\infty} ((s, t)) } \int_s^t \cD  \, d\tau.
\end{align*}
\end{proof}

\section{Green's formula}

\label{appendix C}

For the sake of simplicity, we set all the physical constants to $1$.

The following assertion can be derived from Proposition 1 on p. 382 in \cite{BP_87} via polarization (see also Theorem 5.1.2 in \cite{U_86}).
\begin{lemma}[Green's identity]
        \label{lemma E.A.1}
Let  $\theta \ge 0$ be a number and $E_{2, \theta} (\Sigma^T)$ be the  class of functions $u$ such that 
\begin{itemize}
    \item $u$, $(\partial_t + \frac{p}{p_0} \cdot \nabla_x) u  \in L_{2, \theta} (\Sigma^T)$, 
    \item either $u_0$ or $u_T$ (see \eqref{eq1.2.25}) belong to $L_{2, \theta} (\Omega \times \bR^3)$, 
    \item  either
    $$
            \int_{ \Sigma^T_{  + }  } u_{+}^2 (t, x, p) \frac{|p\cdot n_x|}{p_0} \, dS_x dp dt < \infty
    $$
   or an analogous condition holds for $u_{-}$ on $\Sigma^T_{-}$.
    \end{itemize}
Then, for any $u, \phi  \in E_{2, \theta} (\Sigma^T)$, we have 
\begin{equation}
            \label{eqE.A.1.1}
\begin{aligned}
    &    \int_{\Omega \times \bR^3} (u_T \phi_T (x, p)  - u_0 \phi_0 (x, p)) \, p_0^{2\theta } dx dp\\
&
   +  \int_{\Sigma^T_{+}}    u_{+} \phi_{+}   p_0^{2\theta }\,  \frac{|p \cdot n_x|}{p_0}  \,  dS_x dp dt
    -  \int_{ \Sigma^T_{-} }     u_{-} \phi_{-}  p_0^{2\theta }\,  \frac{|p \cdot n_x|}{p_0} \,  dS_x dp dt\\
 &
  =
     \int_{\Sigma^T} \big((\partial_t u + \frac{p}{p_0} \cdot \nabla_x u) \phi + (\partial_t \phi + \frac{p}{p_0}  \cdot \nabla_x \phi) u\big)  p_0^{2\theta } \, dz.
 \end{aligned}
\end{equation}
\end{lemma}



For the proof of the following lemma, see Proposition 5.13 in \cite{VML}.
\begin{lemma}[Well-posedness and energy identity for finite energy solutions]
    \label{lemma 27.4}
Let 
\begin{itemize}[--]

\item $T > 0$ be a number, $\Omega$ be a $C^{  1, 1 }$ bounded domain,


\item 		for some  $\varkappa \in (0, 1]$,
\begin{align}
\label{eq27.4.35}
g \in L_{\infty} ((0, T))  C^{\varkappa/3, \varkappa}_{x, p} (\Omega \times \bR^3) \cap 
L_{\infty} ((0, T) \times \Omega) W^1_{\infty} (\bR^3),
\end{align}

\item \begin{equation}
    	\label{eq27.4.16}
\|g\|_{ L_{\infty} (\Sigma^T) }\le \frac{\varepsilon_{\star}}{2 N_0},
\end{equation}
where $\varepsilon_{\star} \in (0, 1)$ is a number such that 
\begin{align}
    \label{eq27.4.17}
    \sigma (p) := \int_{\bR^3} \Phi (P, Q) J (q) \, dq \ge \varepsilon_{\star} \bm{1}_3, \, \, p \in \bR^3
\end{align}
(see \cite{L_00}),
and $N_0$ is a constant in the estimate \eqref{eqB.2.1} with $k=0$ and $r=\infty$.
\end{itemize}
Then, for any $\theta \ge 0$ and
\begin{align*}
    f_0 \in L_{2, \theta} (\Omega \times \bR^3), \quad \eta \in L_2 ((0, T) \times \Omega) W^{-1}_{2, \theta} (\bR^3),
\end{align*}
the problem    
\begin{align}
   			\label{eq27.4.18}
	&	(\partial_t + \frac{p}{p_0} \cdot \nabla_x) f  - \nabla_p \cdot (\sigma_g \nabla_p f) = \eta, \\
\label{eq27.4.19}
 &f (t, x, p) = f (t, x, R_x p), \, z \in \Sigma^T_{-}, \, \, f (0, \cdot) = f_{0} (\cdot), 
  \end{align}
  with
\begin{align}
    \label{eq27.4.30}
     \sigma_g (t, x, p): =  \int_{\bR^3} \Phi (P, Q) \big(2 J (q)  + \sqrt{J (q)} g (t, x, q)\big)\, dq
  \end{align} 
has a unique finite energy solution (see Definition \ref{definition 27.1}) on the interval $[0, T]$. In addition, for any $t \in (0, T]$, the following energy identity is valid:
\begin{align}
\label{eq27.4.12}
       & \int_{\Omega \times \bR^3} \big(f^2 (t, x, p)  -  f^2_0 (x, p)\big) p_0^{2 \theta}\, dx dp \\
   & + \int_{\Sigma^t} (\nabla_p f)^T \sigma_g \nabla_p (f  p_0^{2 \theta})\, dz = \int_{(0, t)  \times \Omega }  \langle \eta, f p_0^{2 \theta} \rangle \notag \, dx d\tau. \notag
\end{align}
\end{lemma}

\begin{lemma}[Green's formula for finite energy solutions]
\label{lemma 30.1}
We invoke the assumptions of Lemma \ref{lemma 27.4} and let $f$ be the finite energy solution to \eqref{eq27.4.18}--\eqref{eq27.4.19}. Then, the following assertions hold.

$(i)$ (Green's identity) There exist  unique measurable functions $f_{\pm}$ on $\Sigma^T_{\pm}$, respectively, such that the following hold:
\begin{itemize}
    \item 
    \begin{align}
        \label{eq30.1.8}
                \int_{\Sigma^T_{\pm}} \frac{|p \cdot n_x|^2}{p_0^2} f^2_{\pm} \, dS_x dp dt  < \infty,  
      \end{align}
\item    for any function $\phi$ such that
\begin{align}
    \label{eq30.1.4}
&       (\partial_t + \frac{p}{p_0} \cdot \nabla_x)  \phi,  \nabla_p \phi \in L_2 (\Sigma^T), \quad \phi_{\pm} \in L_{2} (\Sigma^T_{\pm}),
\end{align}
one has for all $\tau \in [0, T]$,
\begin{align}
   \label{eq30.1.0}
    &- \int_{ \Sigma^{\tau} }  f (\partial_t \phi + \frac{p}{p_0} \cdot \nabla_x  \phi) \, dz  +  \int_{\Omega \times \bR^3} \big(f (\tau, x, p) \phi (\tau, x, p) -   f_0 (x, p) \phi (0, x, p)\big) \, dx dp \\
    & + \int_{ \Sigma^{\tau}_{+} } f_{+} \phi_{+}  \frac{|p \cdot n_x|}{p_0} \, dp dS_x dt
    - \int_{ \Sigma^{\tau}_{-} } f_{-} \phi_{-}  \frac{|p \cdot n_x|}{p_0} \, dp dS_x dt \notag\\
    &+ \int_{ \Sigma^{\tau} }(\nabla_p \phi)^T \sigma_g \nabla_p f \, dz
   = \int_{ 0 }^{\tau} \int_{\Omega} \langle \eta, \phi \rangle \, dx  dt. \notag
\end{align}
\end{itemize}

$(ii)$ (energy identity) Let $\phi (x, p)  \in W^1_{\infty} (\Omega \times \bR^3)$ be a function such that 
\begin{align}
    \label{eq30.1.10}
    |\phi (x, p)| \lesssim_{\Omega} \frac{|p \cdot n_x|}{p_0}, \, \, p \in \bR^3, x \in \partial \Omega.
\end{align}
Then, for all $0 \le s < t \le T$, we have
\begin{align}
   \label{eq30.1.7}
    & \frac{1}{2} \int_{\Omega \times \bR^3} \big(f^2 (t, x, p)-   f^2 (s, x, p) \big) \phi (x, p)\, dx dp \\
    & +  \frac{1}{2} \int_{ s }^t \int_{\gamma_{+} } f^2_{+} \phi \frac{|p \cdot n_x|}{p_0} \, dS_x dp  d\tau
    -  \frac{1}{2} \int_{ s }^t \int_{ \gamma_{-} } f^2_{-}\phi   \frac{|p \cdot n_x|}{p_0} \, dS_x  dp d\tau \notag\\
    &+\int_{ s }^t \int_{ \Omega \times \bR^3 }  \big(-\frac{1}{2} f^2 \frac{p}{p_0} \cdot \nabla_x \phi + (\nabla_p f)^T \sigma_g \nabla_p (f \phi)\big) \, dx dp d\tau
   = \int_{ s }^t \int_{ \Omega  } \langle \eta, f \phi \rangle \, dx d\tau. \notag
\end{align}
\end{lemma}

\begin{remark}
Similar results hold for a finite energy solution to a steady counterpart of \eqref{eq27.4.18}--\eqref{eq27.4.19}.
\end{remark}

\begin{remark}
\label{remark C.6.1}
The proof of the above lemma involves a weighted trace estimate, which is similar to a local trace estimate for elements of the so-called kinetic Sobolev space established in Proposition 4.3 of \cite{S_22}. In fact, the argument of the aforementioned paper involves a similar test function. 
\end{remark}

\begin{proof}[Proof of Lemma \ref{lemma 30.1}]
$(i)$ The uniqueness of traces is standard.
To prove the existence, we use an approximation argument.

Let 
$\eta_0, \eta_1 \in L_2 (\Sigma^T)$ be functions such that
$$
    \eta = \eta_0 + \nabla_p \cdot \eta_1,
$$
where $\eta_1$ is a vector field.
 Let $f_{0, n} \in C^{\infty}_0 (\Omega \times \bR^3), n \ge 1,$ and $\eta_{0, n}, \eta_{1, n}  \in C^{\infty}_0 (\Sigma^T), n \ge 1, j=1, 2,$ be  sequences such that
\begin{align}
    \label{eq30.1.17}
    f_{0, n} \to f_0 \,\, \text{in} \,\, L_2 (\Omega \times \bR^3), \quad  \eta_{0, n} \to \eta_0, \quad \eta_{1, n} \to \eta_1 \,\, \text{in} \,\, L_2 (\Sigma^T) \, \text{as} \, n \to \infty.
\end{align}
Then, by Proposition 5.4 in \cite{VML}, there exists a unique strong solution $f_n$ to the problem 
   \begin{align}
          \label{eq30.1.1}
 &	(\partial_t + \frac{p}{p_0} \cdot \nabla_x) f_{n}  - \nabla_p \cdot (\sigma_g \nabla_p f_{n})  =  \eta_{0, n} + \nabla_p \cdot  \eta_{1, n}, \\
 &f_{n} (t, x, p) = f_{n} (t, x, R_x p), \, z \in \Sigma^T_{-}, \quad f_n (0, \cdot) = f_{0, n} (\cdot). \notag
  \end{align}
  Furthermore, by Proposition 5.6 in \cite{VML}, 
  $$
    f_n \in L_{\infty, \theta} (\Sigma^T) \cap  L_{\infty} ((0, T)) C^{\varkappa/3, \varkappa}_{x, p} (\Omega \times \bR^3)
  $$
  for any $\theta > 0$. By this, we have
 \begin{align}
          \label{eq30.1.2}
    \int_{\Sigma^T_{\pm}} |(f_{n})_{\pm}|^2 \, \frac{|p \cdot n_x|}{p_0} dS_x dp dt < \infty.
 \end{align}
 Since  \eqref{eq30.1.2} holds,  a variant of energy identity   \eqref{eqE.A.1.1} is applicable.
 Furthermore, by the definition of $\sigma_g$ in \eqref{eq27.4.30}, the bound \eqref{eqB.2.1} with $k=0$, and \eqref{eq27.4.16}--\eqref{eq27.4.17}, 
 $$
    \sigma_g \ge \varepsilon_{\star} \bm{1}_3.
 $$
 Then, by a standard energy argument, we have
 \begin{align}
    \label{eq30.1.5}
    f_n \to f \, \text{in} \, C ([0, T]) L_2 (\Omega \times \bR^3) \cap L_2 ((0, T) \times \Omega) W^1_2 (\bR^3).
 \end{align}

Next, we derive an estimate for the kinetic traces of $u := f_n  - f_m$.
  To this end, we fix any  Lipschitz vector field $\nu$   such that $\nu (x) = n_x$ on $\partial \Omega$. Such a vector field can be constructed as follows. First, we define the signed distance function
$$
    \delta (x) =
    \begin{cases}
      \text{dist} (x, \partial \Omega), x \in \overline{\Omega}, \\
   -  \text{dist} (x, \partial \Omega), x \not \in \Omega.
    \end{cases}
$$
 Since $\Omega$ is a $C^{1, 1}$ domain,  $\delta$ is a $C^{1, 1}$ function in a sufficiently small neighborhood of $\partial \Omega$.
 Furthermore, there exists an open set $U$ containing $\partial \Omega$ such that $\nabla_x \delta \neq 0$.
 We set
 $$
        \nu (x)= -\frac{\nabla_x \delta (x)}{|\nabla_x \delta (x)|} \phi (x),
 $$
  where $\phi$ is a smooth cutoff function supported in $U$ such that $\phi = 1$ on $\partial \Omega$.
 It follows from the above discussion that $\nu$ is a Lipschitz function on $\bR^3$.
Furthermore, we denote
 \begin{equation}
    \label{eqE.2.14}
         \zeta (x, p) =  \frac{p}{p_0} \cdot \nu (x) 1_{ p \cdot \nu (x) > 0}
  \end{equation}
and note that 
$$
    ||\zeta| + |\nabla_{x, p} \zeta||_{L_{\infty} (\Omega \times \bR^3) } \lesssim_{\Omega} 1.
$$

Next, since $(\partial_t + \frac{p}{p_0} \cdot \nabla_x) [u, u \zeta] \in L_2 (\Sigma^T)$ and \eqref{eq30.1.2} is valid, we may use  \eqref{eqE.A.1.1}. We obtain
 \begin{align*}
  & \frac{1}{2} \int_{\Omega \times \bR^3} u^2 (\tau, x, p)  \zeta (x, p) \, dx dp
  + \frac{1}{2}\int_{\Sigma^{\tau}_{+}}   u^2_{+} \frac{|p \cdot n_x|^2}{p_0^2} \, dS_x dp dt  \\
      &  = \frac{1}{2} \int_{\Omega \times \bR^3} (f_{0, n} (x, p) - f_{0, m} (x, p))^2 \zeta (x, p) \, dx dp
      + \frac{1}{2}\int_{\Sigma^{\tau}}  u^2 \frac{p}{p_0} \cdot \nabla_x \zeta \, dz\\
      & -  \int_{\Sigma^{\tau}}  (\nabla_p u)^T \sigma_g \nabla_p (u\zeta) \, dz  + \int_{\Sigma^{\tau}} (\eta_{0, n} - \eta_{0, m}) u \zeta
      -  \nabla_p (u \zeta)  \cdot (\eta_{1, n} - \eta_{1, m}) \, \, dz.
\end{align*}
Using  the Cauchy-Schwarz inequality, we conclude that 
 \begin{align}
    \label{eq30.1.6}
& \int_{\Sigma^{\tau}_{+}}  \frac{|p \cdot n_x|^2}{p_0^2} u^2_{+} \, dS_x dp dt  \\ 
& \lesssim  \|f_{0, n} - f_{0, m}\|^2_{ L_2 (\Omega \times \bR^3) }
+  \sum_{j=0}^1 \|\eta_{j, n} - \eta_{j, m}\|^2_{L_2 (\Sigma^T) }
+ \|u\|^2_{L_2  ((0, T) \times \Omega) W^1_2 (\bR^3)} . \notag
\end{align}
Similarly, we obtain the same estimate for the trace  $u_{-}$.
Hence, by \eqref{eq30.1.6} and the convergences \eqref{eq30.1.17} and \eqref{eq30.1.5}, we conclude that $(f_n)_{\pm}$ converges to some function $f_{\pm}$ satisfying \eqref{eq30.1.8}.

Finally, to show that \eqref{eq30.1.0} holds, we note that $f \in E_{2} (\Sigma^T)$ (see Lemma \ref{lemma E.A.1}) due to \eqref{eq30.1.2}, and for any  $\phi$ satisfying \eqref{eq30.1.4}, one has $\phi \in E_{2} (\Sigma^T)$. Then,  for each $n$ and $\tau > 0$, we have 
\begin{align}
   \label{eq30.1.11}
    &- \int_{ \Sigma^{\tau} }  f_n (\partial_t \phi + \frac{p}{p_0} \cdot \nabla_x \phi) \, dz  +  \int_{\Omega \times \bR^3} \big(f_n (\tau, x, p) \phi (\tau, x, p) -   f_{0, n} (x, p) \phi (0, x, p)\big) \, dx dp \\
    & + \int_{ \Sigma^{\tau}_{+} } (f_n)_{+} \phi_{+}  \frac{|p \cdot n_x|}{p_0} \, dS_x dp dt
    - \int_{ \Sigma^{\tau}_{-} } (f_n)_{-} \phi_{-}  \frac{|p \cdot n_x|}{p_0} \, dS_x dp  d\tau \notag\\
    &+ \int_{ \Sigma^{\tau} }(\nabla_p \phi)^T \sigma_g \nabla_p f_n \, dz
   = \int_{ \Sigma^{\tau} } \langle \eta_n, \phi \rangle \, dz. \notag
\end{align}
Thanks to \eqref{eq30.1.5}, the second condition in \eqref{eq30.1.4}, and the trace estimate \eqref{eq30.1.6}, we may pass to the limit in \eqref{eq30.1.11} and obtain the desired identity \eqref{eq30.1.0}.


$(ii)$ First, we note that due to \eqref{eq30.1.2} and \eqref{eq30.1.10}, one can apply a variant of the energy identity \eqref{eqE.A.1.1} with $f_n$ and $f_n \phi$ in place of $f$ and $\phi$, respectively, and conclude that \eqref{eq30.1.7} holds with $f$ replaced with $f_n$. 
To pass to the limit in the aforementioned identity, we note that for $u = f_n  - f_m$, by the condition \eqref{eq30.1.10} and the estimate \eqref{eq30.1.6}, we have
$$
    \int_{\Sigma^T_{\pm}} u^2_{ \pm } |\phi| \frac{|p \cdot n_x|}{p_0} \, dS_x dp dt
    \lesssim_{\Omega} \int_{\Sigma^T_{\pm}} u^2_{ \pm }   \frac{|p \cdot n_x|^2}{p_0^2} \, dS_x dp dt \to 0
$$
as $n \to \infty$.  The convergence of the remaining terms follows from \eqref{eq30.1.17} and \eqref{eq30.1.5}. Thus, the desired identity \eqref{eq30.1.7} is valid.
\end{proof}

\section{Steady $S_p$ estimate for a linear relativistic Fokker-Planck equation}

\label{appendix D}

In this section, all the physical constants are set to $1$ for the sake of simplicity.

The first two results are taken from \cite{VML}.

\begin{theorem}[steady $S_r$ estimate in the presence of SRBC, cf. Proposition 5.11 in \cite{VML}]
\label{theorem 14.C.6}
 We invoke the assumptions of Lemma \ref{lemma 27.4}.
 Let $r \in [2, \infty)$ and $\kappa \in [0, 1)$ be numbers.



Then, for sufficiently small $\varepsilon_{\star} \in (0, 1)$ independent of $r, \Omega, \varkappa$, and $\kappa$, there exists a constant
$\theta = \theta (r, \varkappa, \kappa) > 0$
such that if, additionally,
\begin{equation}
    \label{eq14.C.6.1}
    	\eta \in L_{2, \theta} (\Omega \times \bR^3)  \cap   L_{r, \theta} (\Omega \times \bR^3),
     \end{equation}
and if $f$ is a strong solution to 
\begin{align*}
    	(\frac{p}{p_0} \cdot \nabla_x) f  - \nabla_p \cdot (\sigma_g \nabla_p f) = \eta
\end{align*}
with the SRBC (see Definition \ref{definition 3.3}), then, one has
\begin{align}
    \label{eq14.C.6.3}
f \in S_{2, \kappa \theta} (\Omega \times \bR^3) \cap S_{r, \kappa \theta}  (\Omega \times \bR^3),
 \end{align}
 and,
 \begin{align}
    \label{eq14.C.6.6}
	&\|f\|_{ S_{2, \kappa \theta} (\Omega \times \bR^3) } + \|f\|_{ S_{r, \kappa \theta} (\Omega \times \bR^3) } \\
& \lesssim_{\varkappa, \kappa,  r,    \theta, \Omega} \|\eta\|_{  L_{2, \theta}   (\Omega \times \bR^3) } +   \|\eta\|_{  L_{r, \theta}   (\Omega \times \bR^3) }   + \|f\|_{  L_{2, \theta}   (\Omega \times \bR^3) }. \notag
 \end{align}

Furthermore, in the case when $r < 6$,  we have
    \begin{align}
        \label{eq14.C.6.7}
        \|f\|_{L_{r_1, \kappa \theta} (\Omega \times \bR^3) } +   \|\nabla_p f\|_{L_{r_2, \kappa \theta} (\Omega \times \bR^3) } \lesssim \text{r.h.s. of \eqref{eq14.C.6.6}},
    \end{align}
    where $r_1, r_2 > 1$ are  numbers satisfying the relations
        \begin{align}
        \label{eq14.C.6.10}
        \frac{1}{r_1} > \frac{1}{r} - \frac{1}{6},  \quad   \frac{1}{r_2} >  \frac{1}{r}  - \frac{1}{12}.
        \end{align}
In the case when $r \in (6,  12)$,
\begin{align}
    \label{eq14.C.6.8}
    \|f\|_{ L_{\infty, \kappa \theta} (\Omega \times \bR^3)  }  + \|\nabla_p f\|_{L_{r_2, \kappa \theta} (\Omega \times \bR^3) } \lesssim \text{r.h.s. of \eqref{eq14.C.6.6}},
\end{align}
where $r_2$ satisfies \eqref{eq14.C.6.10}.
Finally, in the case when $r > 12$,
 \begin{align}
    \label{eq14.C.6.9}
& \|[f, \nabla_p f]\|_{ L_{\infty, \kappa \theta} (\Omega\times \bR^3)  } +	\|[f, \nabla_p f]\|_{  C^{\alpha/3,  \alpha}_{x, p} (\Omega \times \bR^3) } \\
& \lesssim \text{r.h.s. of \eqref{eq14.C.6.6}},  \notag
\end{align}
where $\alpha \in  (0,  1 - \frac{12}{r})$. In  \eqref{eq14.C.6.7}, \eqref{eq14.C.6.8}, and \eqref{eq14.C.6.9}, one needs to take into account the dependence of  the constants on the right-hand sides  on the additional parameters  $r_1, r_2$, and $\alpha$.
\end{theorem}

\begin{corollary}[embedding in a bounded domain, cf. Corollary 5.12 in \cite{VML}]
\label{corollary 14.C.8}
Let $r \in [2, \infty)$, $\kappa \in (0, 1)$, $\theta > 0$, and $f \in S_{r, \theta} (\Omega \times \bR^3)$ be a function satisfying the SRBC.
    Then, for sufficiently large $\theta = \theta (r, \varkappa, \kappa) > 0$, the following assertions hold.

    $(i)$ If $r \in [2, 6)$, for $r_1$ and $r_2$ satisfying \eqref{eq14.C.6.10}, one has
    \begin{align}
        \label{eq14.C.8.1}
       & \|f\|_{L_{r_1, \kappa \theta} (\Omega \times \bR^3) } +   \|\nabla_p f\|_{L_{r_2, \kappa \theta} (\Omega \times \bR^3) }\\
      &  \lesssim_{\Omega, \theta,  \kappa, \varkappa,   r, r_1, r_2}
         \|f\|_{ S_{r, \theta} (\Omega \times \bR^3) }. \notag
    \end{align}

    $(ii)$ If $r \in (6, 12)$, then for $r_2$ satisfying \eqref{eq14.C.6.10},
    \begin{align}
        \label{eq14.C.8.2}
       & \|f\|_{L_{\infty, \kappa \theta} (\Omega \times \bR^3) } +   \|\nabla_p f\|_{L_{r_2, \kappa \theta} (\Omega \times \bR^3) } \\
       & \lesssim_{\Omega, \theta, \kappa,  \varkappa,  r,  r_2}
        \|f\|_{ S_{r, \theta} (\Omega \times \bR^3) }. \notag
    \end{align}

    $(iii)$ If $r > 12$,  then for any $\alpha \in (0,  1 - \frac{12}{r})$,
    \begin{align}
        \label{eq14.C.8.20}
       & \|[f, \nabla_p f]\|_{L_{\infty, \kappa \theta} (\Omega \times \bR^3) } + \|[f, \nabla_p f]\|_{  C^{\alpha/3,  \alpha}_{x, p} (\Omega \times \bR^3) }   \\
       & \lesssim_{\Omega, \theta, \varkappa,  r, \alpha}   \|f\|_{ S_{r, \theta} (\Omega \times \bR^3) }. \notag
    \end{align}
\end{corollary}

To prove the final result of this section, we need a simple commutator estimate.
 \begin{lemma}
         \label{lemma 22.2}
Let $r \in (1, \infty)$,   $0 < \beta < \alpha  \le 1$ be numbers. For any  $f \in L_r (\bR^d)$ and $g \in  C^{\alpha} (\bR^d)$, we set
\begin{equation}
    \label{eq22.2.2}
    \text{Com}_{\beta} (f, g) =  (-\Delta_x)^{\beta/2} (f g) - ((-\Delta_x)^{\beta/2} f) g,
\end{equation}
where the above expression is understood in the sense of distributions.
Then, $\text{Com}_{\beta} (f, g) \in L_r (\bR^d)$, and
\begin{equation}
    \label{eq22.2.1}
     \|\text{Com}_{\beta} (f, g)\|_{ L_r (\bR^d) } \lesssim_{d, \alpha, \beta, r}  \|g\|_{ C^{\alpha} (\bR^d) } \|f\|_{ L_r (\bR^d) }.
\end{equation}
\end{lemma}
The above estimate can be proved, for example, by testing $\text{Com}_{\beta} (f, g)$ with  $\phi \in C^{\infty}_0 (\bR^d)$ and using a pointwise formula for $\text{Com}_{\beta} (\phi, g)$ combined with H\"older's and Minkowski inequalities.

\begin{remark}
Invoke the assumptions of Lemma \ref{lemma 22.2} and assume, additionally, that $f \in H^{\beta}_r (\bR^d)$. Then,
\begin{equation}
    \label{eq22.3.1}
    \|f g\|_{ H^{\beta}_r (\bR^d) } \lesssim_{d, r, \alpha, \beta } \|g\|_{ C^{\alpha} (\bR^d) } \|f\|_{ H^{\beta}_r (\bR^d) }.
\end{equation}
Indeed, by using  \eqref{eq22.2.1}, we obtain
\begin{align*}
    &\|f g\|_{ H^{\beta}_r (\bR^d) }  \lesssim_{d, r}  \|f g\|_{ L_r (\bR^d) }  + \|\text{Com}_{\beta} (f, g)\|_{ L_r (\bR^d) } + \|g (-\Delta_x)^{\beta/2} f\|_{ L_r (\bR^d) } \\
    & \lesssim_{d, r, \alpha, \beta} \|g\|_{ L_{\infty} (\bR^d) }  \|f\|_{ L_r (\bR^d) } + \|g\|_{ C^{\alpha} (\bR^d) } \|f\|_{ L_r (\bR^d) } + \|g\|_{ L_{\infty} (\bR^d) }  \|(-\Delta_x)^{\beta/2} f\|_{ L_r (\bR^d) } \\
    & \lesssim  \|g\|_{ C^{\alpha} (\bR^d) } \|f\|_{ H^{\beta}_r (\bR^d) }.
\end{align*}
\end{remark}

\begin{lemma}[higher regularity in $x$]
            \label{lemma 32.1}
Let
\begin{itemize}
    \item[--]  $d \ge 1$ and $r \in (1, \infty)$, $\varkappa, \alpha \in (0, 1]$,
\item[--] $a$ be a  $d \times d$ symmetric matrix-valued  function such that  for some $\delta \in (0, 1)$,
    $$
        \delta \bm{1}_d \le a \le \delta^{-1} \bm{1}_d,
    $$
\item[--] 
for some $K> 0$,
\begin{align}
    \label{eq32.1.30}
        \|a\|_{ C^{\varkappa/3,  \varkappa}_{x, v} (\bR^{2d}) } +  \|[a, \nabla_v a]\|_{ L_{\infty} (\bR^d_v) C^{\alpha/3}  (\bR^d_x) }  \le K,
\end{align}
\item[--] $U \in S_r^N (\bR^{2d})$ (see \eqref{eq14.C.20})  is a strong solution to 
   \begin{equation}
        \label{eq32.1.0}
          v \cdot \nabla_x U  - \nabla_v \cdot (a   \nabla_{v} U)  = \eta.
\end{equation}
\end{itemize} 
Then,  if $\eta \in L_r (\bR^d_v) H^{s}_r (\bR^{d}_x)$ for some $s \in (0, \alpha/3)$, one has 
\begin{align}
\label{eq32.1.1}
  & \| (-\Delta_x)^{1/3 + s/2} U \|_{   L_r (\bR^{2d}) }  + \| (-\Delta_x)^{s/2} U\|_{ S_r^N (\bR^{2d}) }  \\
   & \lesssim_{d, r, \varkappa, \alpha, s,  K} \delta^{-\beta}  \big(\|(-\Delta_x)^{s/2}  \eta \|_{ L_r (\bR^{2d}) } +   \| U\|_{ S_r^N (\bR^{2d}) }\big), \notag
\end{align}
where $\beta = \beta (d, r, \varkappa) > 0$.
\end{lemma}

\begin{proof}

To prove the lemma, we will  apply the operator $(-\Delta_x)^{s/2}$ to Eq. \eqref{eq32.1.0} and use the regularity results from \cite{DY_21a} and \cite{DY_22}.
First, for $\lambda > 0$, we consider the equation
\begin{align}
\label{eq32.1.2}
     & v \cdot \nabla_x U_1 - a^{i j}  \partial_{v_i v_j} U_1 - \partial_{v_i} a_{ i j} \partial_{v_j} U_1 + \lambda U_1  = (-\Delta_x)^{s/2}  \eta \\
    & + \text{Com}_{s} (\partial_{v_i v_j} U, a_{ i j}) +   \text{Com}_{s}  (\partial_{v_j} U, \partial_{v_i} a_{ i j}) + \lambda (-\Delta_x)^{s/2} U, \notag
 \end{align}
 where $\text{Com}_{s}$ is the operator defined in \eqref{eq22.2.2}.
By the inequality \eqref{eq22.2.1} in  Lemma \ref{lemma 22.2}, we have
\begin{align}
    \label{eq32.1.9}
   & \|\text{Com}_{s} (\partial_{v_i v_j} U, a_{ i j})\|_{ L_r (\bR^{2d})} +  \|\text{Com}_{s}  (\partial_{v_j} U, \partial_{v_i} a_{ i j})\|_{ L_r (\bR^{2d}) } \\
   & \lesssim_{ \alpha, s,   r, K}  \||\nabla_v U| + |D^2_v U|\|_{ L_r (\bR^{2d})}. \notag
\end{align}
Furthermore, we recall that $(-\Delta_x)^{1/3} u \in L_r (\bR^{2d})$  because $u \in S_r^N (\bR^{2d})$ (see Theorem 2.6   and Remark 2.11 in \cite{DY_21a}). Then, since $s < 1/3$,
\begin{align}
    \label{eq32.1.12}
   \|(-\Delta_x)^{s/2} U \|_{ L_r (\bR^{2d}) }  \lesssim_{r, d, s} \| U \|_{ S_r^N (\bR^{2d}) }.
\end{align}
Thus, the right-hand side of Eq. \eqref{eq32.1.2} belongs to $L_r (\bR^{2d})$. By this and the assumptions $a \in C^{\varkappa/3, \varkappa}_{x, v} (\bR^{2d})$,  $\nabla_v a \in L_{\infty} (\bR^{2d})$, we may apply the stationary counterpart of Theorem 2.6 in \cite{DY_21a} (see Remark 2.11 therein). We conclude that for sufficiently large $\lambda = \lambda (d, \varkappa,  K, r) > 0$, Eq. \eqref{eq32.1.2} has  a unique strong solution $U_1 \in S_r^N (\bR^{2d})$.
We will  show that $U_1 =  (-\Delta_x)^{s/2} U.$

Next, we denote
\begin{align}
    \label{eq32.1.7}
 &\bS_r (\bR^{2d}): = \{u:  u, \nabla_v u \in L_r (\bR^{2d}),  v \cdot \nabla_x u \in L_r (\bR^d_x) W^{-1}_r (\bR^d_v)\}.
\end{align}
We say that $U$ is a  $\bS_r (\bR^{2d})$ solution  to Eq. \eqref{eq32.1.0}   (see Definition 1.10 in \cite{DY_22})
if $U \in \bS_r (\bR^{2d})$, and the  identity \eqref{eq32.1.0} holds in the sense of distributions, that is, for any $\psi \in C^{\infty}_0 (\bR^{2d})$, we have
\begin{equation}
    \label{eq32.1.4}
        - \int_{\bR^{2d}} (v \cdot \nabla_x \psi) U \, dx dv + \int_{\bR^{2d}} (\nabla_v \psi)^T a \nabla_v U \, dx dv  = \int_{\bR^d_x}  \langle \eta,  \psi\rangle \, dx.
\end{equation}
    Since $C^{\infty}_0 (\bR^{2d})$ is dense in $S_r (\bR^{2d})$ for any $r \in (1, \infty)$ (cf. Lemma 4.4 in \cite{DY_21a}), the  identity \eqref{eq32.1.4} holds for any  $\psi \in S_{r/(r-1)} (\bR^{2d})$ by an approximation argument.

Furthermore, for any $\phi \in C^{\infty}_0 (\bR^{2d})$ we replace formally $\psi$ with $(-\Delta_x)^{s/2} \phi \in S_{r/(r-1)} (\bR^{2d})$ in the integral formulation \eqref{eq32.1.4} and obtain
\begin{align}
        \label{eq32.1.3}
  &  - \int_{\bR^{2d}} (v \cdot \nabla_x \phi)  (-\Delta_x)^{s/2}  U \, dx dv + \int_{\bR^{2d}}  (\nabla_v \phi)^T   (-\Delta_x)^{s/2} \big(a \nabla_v U\big) \, dx dv\\
    &= \int_{\bR^{d}_x} \langle \phi, (-\Delta_x)^{s/2} \eta \rangle \, dx dv. \notag
\end{align}
We claim that
\begin{align}
 \label{eq32.1.10}
   U_2: =   (-\Delta_x)^{s/2} U \in \bS_r (\bR^{2d})
\end{align}
(see \eqref{eq32.1.7}).
We recall that by the stationary counterpart of Theorem 2.6  in \cite{DY_21a},  since  $u \in S_r^N (\bR^{2d})$, we have
\begin{align}
 \label{eq32.1.5}
    \nabla_v  (-\Delta_x)^{1/6} U \in L_r (\bR^{2d}),
\end{align}
and hence, due to  $s < 1/3$, one has $\nabla_v U_2 \in L_r (\bR^{2d})$. Furthermore, by this and \eqref{eq22.3.1}--\eqref{eq32.1.30}, we obtain
\begin{align}
 \label{eq32.1.6}
    (-\Delta_x)^{s/2} \big(a \nabla_v U\big) \in L_r (\bR^{2d}).
\end{align}
Gathering \eqref{eq32.1.3} and \eqref{eq32.1.5}--\eqref{eq32.1.6} and using the fact that $\eta \in L_r (\bR^{d}_v) H^{s}_r (\bR^d_x)$, we conclude that
$$v \cdot \nabla_x  U_2 \in L_r (\bR^d_x) W^{-1}_r (\bR^d_v),$$
and, thus, \eqref{eq32.1.10} is true.
By this and  \eqref{eq32.1.3}, we find that $U_2$ is a $\bS_r (\bR^{2d})$ solution to the equation
\begin{align*}
     & v \cdot \nabla_x U_2 - \nabla_v \cdot  (a \nabla_{v } U_2)  + \lambda U_2  = (-\Delta_x)^{s/2}  \eta \\
    & + \text{Com}_{s} (\partial_{v_i v_j} U, a_{ i j}) +   \text{Com}_{s}  (\partial_{v_j} U, \partial_{v_i} a_{ i j}) + \lambda (-\Delta_x)^{s/2} U. \notag
 \end{align*}
 By the uniqueness theorem for divergence form kinetic Fokker-Planck equations in the class of $\bS_r (\bR^{2d})$ solutions (see Theorem 1.15 in \cite{DY_22}),  we conclude that  $U_1 = U_2$ provided that $\lambda = \lambda (d, \varkappa, r, K)$  is sufficiently large.
 Finally, we cancel the term $\lambda U_1$ on both sides in Eq. \eqref{eq32.1.2} and apply the stationary counterpart of the estimate $(2.8)$ in Corollary 2.8 of \cite{DY_21a}. Then, there exists $\beta = \beta (d, r, \varkappa) > 0$ such that
 \begin{align*}
 &   \|(-\Delta_x)^{1/3} U_1\|_{ L_r  (\bR^{2d}) } + \| U_1\|_{ S_r (\bR^{2d}) } \\
 &
  \lesssim_{r, d, \varkappa,  K} \delta^{-\beta} \bigg(\|(-\Delta_x)^{s/2}  \eta \|_{ L_r  (\bR^{2d}) }\\
 & +\||\text{Com}_{s} (\partial_{v_i v_j} U, a_{ i j})|+|\text{Com}_{s}  (\partial_{v_j} U, \partial_{v_i} a_{ i j})|\|_{ L_r (\bR^{2d})}  +  \|U_1\|_{ L_r (\bR^{2d})}\bigg).
\end{align*}
By \eqref{eq32.1.9} and \eqref{eq32.1.12}, we may replace the terms involving $\text{Com}_{s}$ and $U_1$ on the r.h.s. with $\|U\|_{S_r^N  (\bR^{2d}) }$.
The desired estimate \eqref{eq32.1.1} is proved.
\end{proof}

\section{H\"older estimates of $a_f, C_f$, and $K$. }

\label{appendix E}



For the sake of simplicity, we set all the physical constants  to $1$.

\begin{lemma}
    \label{lemma B.3}
For $g \in L_{r} (\bR^3), r \in (3/2, \infty]$,  we denote
\begin{align}
        \label{eqB.3.0}
  & I (p) =  \int  \frac{P \cdot Q}{p_0 q_0} \bigg((P \cdot Q)^2 - 1\bigg)^{-1/2}  J^{1/2} (q) g (q) \, dq, \\
  & \mathfrak{I} (p)  =  \int  \Phi^{i j} (P, Q) J^{1/2} (q)   g (q)  \, dq, \notag \\
  & \mathcal{I} (p) = \int \big((\partial_{p_k} + \frac{q_0}{p_0} \partial_{q_k}) \Phi^{i j} (P, Q)\big)  J^{1/2} (q)  g (q)  \, dq. \notag
\end{align}
Then, we have
    \begin{align}
            \label{eqB.3.5}
  &  \|I\|_{ L_{\infty} (\bR^3) } + \|\mathcal{I}\|_{ L_{\infty} (\bR^3) }  \lesssim \|g\|_{ L_{r} (\bR^3) }.
   \end{align}
   Furthermore, if $r = \infty$, then, for any $\alpha \in (0, 1),$
       \begin{align}
        \label{eqB.3.1}
& [I]_{ C^{\alpha} (\bR^3) } \lesssim_{\alpha}  \|g\|_{ L_{\infty} (\bR^3) } \\
        \label{eqB.3.9}
    &  [\mathfrak{I}]_{ C^{\alpha} (\bR^3) } \lesssim_{\alpha}   \|g\|_{ L_{\infty} (\bR^3) }, \\
        \label{eqB.3.1.1}
        &  [\mathcal{I}]_{ C^{\alpha} (\bR^3) } \lesssim_{\alpha}   \|g\|_{ L_{\infty} (\bR^3) }.
\end{align}
\end{lemma}


\begin{proof}
\textbf{Proof of \eqref{eqB.3.5}.}
First, we note that the estimate of $\cI$ in \eqref{eqB.3.5} follows from \eqref{eq12.A.2.60}. 
We denote
\begin{align*}
 &   \cR (p, q) =   \frac{P \cdot Q}{p_0 q_0} \bigg(P \cdot Q + 1\bigg)^{-1/2} J^{1/4} (q), \\
 & \cS (p, q) = \bigg(P \cdot Q - 1\bigg)^{-1/2},
\quad
 \, \, \text{where} \, \,  P \cdot Q =  p_0 q_0 - p \cdot q,
\end{align*}
so that
\begin{equation}
        \label{eqB.3.8}
   I (p) =    \int  \cR (p, q) \cS (p, q) J^{1/4} (q)   g (q) \, dq.
\end{equation}
By using a simple bound
\begin{equation}
            \label{eqB.3.7}
    |\nabla_p (P \cdot Q)| \lesssim q_0
\end{equation}
and the  inequalities (see the formula $(32)$  on p. 277 in \cite{GS_03})
\begin{align}
            \label{eqB.3.6}
  & \frac 1 2 |p-q|^2 \ge     P \cdot Q -  1  \gtrsim \frac{|p-q|^2}{q_0^2} 1_{ |p-q| < (|p|+1)/2 } + \frac{p_0}{q_0}  1_{ |p-q| \ge (|p|+1)/2 },
\end{align}
we obtain the following useful estimates: 
\begin{equation}
            \label{eqB.3.2}
      |\cR (p, q)| +  |\nabla_p \cR (p, q)| \lesssim  J^{  (1/4)- } (q), \quad  (P \cdot Q -  1)^{-1/2}  \lesssim  \frac{  q_0  }{|p-q|}  + \frac{q_0^{1/2}}{p_0^{1/2}}.
\end{equation}
We note that the $L_{\infty}$-norm estimate \eqref{eqB.3.5}  follows from \eqref{eqB.3.2} and the local integrability of the function $|p-q|^{-1}$.

\textbf{Proof of \eqref{eqB.3.1}.}
We fix arbitrary  $p^1, p^2 \in \bR^3$ such that $|p^1 - p^2| < 1$
 and  split the domain of integration into
$$
    A_1 = \{q: |p^1-q| \ge |p^2-q|\}, \quad A_2 = \bR^3 \setminus A_1.
$$
By symmetry, we may replace $g$  with $g 1_{A_1}$, so that
\begin{equation}
        \label{eqB.3.16}
    |p^1- q| \ge  |p^2 - q|.
\end{equation}
By the triangle inequality, it suffices to estimate
\begin{align*}
      &  I_1 (p^1, p^2) :=  \int  |\cR (p^1, q) - \cR (p^2, q)| \cS (p^1, q) J^{1/4} (q)   |g (q)| \, dq, \\
      & I_2 (p^1, p^2) :=   \int   \cR (p^2, q) \,  |\cS (p^1, q) - \cS (p^2, q)| \, J^{1/4} (q)   |g (q)| \, dq.
\end{align*}
By the mean-value theorem and \eqref{eqB.3.2},
\begin{align}
\label{eqB.3.3}
    I_1 (p^1, p^2)  & \lesssim  |p^1-p^2|     \int \big(1+\frac{1}{|p^1 - q|}\big) J^{1/8} (q) |g (q)| \, dq \\
  & \lesssim |p^1 - p^2|  \, \|g\|_{L_{\infty} (\bR^3)}. \notag
\end{align}
Next, by the identity
$$
    a_1^{-1/2} - a_2^{-1/2} = \frac{a_2 - a_1}{a_1 a_2^{1/2}  + a_2 a_1^{1/2}}
$$
with $a_j := P^j \cdot Q - 1 = (p^j)_0  q_0 - p \cdot q - 1$ and the bounds \eqref{eqB.3.7} and \eqref{eqB.3.2}, we conclude that to estimate $I_2 (p^1, p^2)$,
it suffices to show that
\begin{align}
    \label{eqB.3.10}
  &  |p^1-p^2| \int  J^{1/8} (q) |g (q)| \big(1 + \frac{1}{|p^1 - q|^{2}}\big) \big(1+\frac{1}{|p^2 - q|}\big) \, dq \\
    &\lesssim |p^1-p^2|^{\alpha} \|g\|_{L_{\infty} (\bR^3)}.\notag
\end{align}
To prove \eqref{eqB.3.10}, we note that by the triangle inequality and \eqref{eqB.3.16}, we have
$$|p^1 - p^2| \le \min\{1,  2 |p^1 - q|\},$$
and hence,   the l.h.s. of \eqref{eqB.3.10} is dominated by
\begin{align}
\label{eqB.3.4}
 & |p^1-p^2|^{\alpha}  \|g\|_{L_{\infty} (\bR^3)}   \\
&  \times     \int J^{1/8} (q) \big(1 + \frac{1}{|p^1 - q|^{1+\alpha} }  + \frac{1}{|p^2 - q| } + \frac{1}{|p^1 - q|^{1+\alpha} |p^2-q| }   \big) \, dq.  \notag
\end{align}
We will show that the last integral (involving the product $|p^1 - q|^{1+\alpha} |p^2-q|$) is finite since the remaining terms are simpler. To this end, we use  H\"older's inequality with the exponents $\beta \in (\frac{3}{2-\alpha}, 3)$ and $\beta' =\frac{\beta}{\beta-1}$
and the fact that
$$
     \int \frac{J^{\beta/8} (q) }{|p^2-q|^{\beta} } \, dq,   \int \frac{  J^{  \beta'/8  } (q) } {|p_1 - q|^{\beta' (1+\alpha)}  } \, dq < \infty,
$$
which  is true since $\beta, \beta' (1+\alpha) < 3$.
The desired estimate now follows from  \eqref{eqB.3.3}--\eqref{eqB.3.10}.

\textbf{Proof of \eqref{eqB.3.9}.}
By the definition of $\Phi (P, Q)$ (see \eqref{eq1.11}),
\begin{align}
        \label{eqB.3.24}
       &  \Phi (P, Q) = p_0^{-1} q_0^{-1} (P \cdot Q)^2 \big((P \cdot Q)^2 - 1\big)^{-1/2}  \bm{1}_3 \\
       & +  p_0^{-1} q_0^{-1} (P \cdot Q)^2  \big((P \cdot Q) + 1\big)^{-3/2}  (p \otimes q + q \otimes p)   \big((P \cdot Q) - 1\big)^{-1/2} \notag \\
              & -  p_0^{-1} q_0^{-1} (P \cdot Q)^2   (p-q) \otimes (p-q)  \big((P \cdot Q)^2 - 1\big)^{-3/2} \notag \\
              &=:  \Phi_1 (P, Q) +  \Phi_2 (P, Q)  - \Phi_3 (P, Q). \notag
\end{align}
We will focus on the integral involving $\Phi_3$, as the remaining terms can be handled in the same way.
Next, as in \eqref{eqB.3.8}, we write
\begin{equation}
        \label{eqB.3.14}
        \int \Phi_3 (P, Q) J^{1/2} (q) g (q) \, dq = \int \mathfrak{R} (p, q) \mathfrak{S} (p, q)  g (q) \, dq,
\end{equation}
where
\begin{align*}
    	&\mathfrak{R} (p, q) = p_0^{-1} q_0^{-1} (P \cdot Q)^2   \big((P \cdot Q) + 1\big)^{-3/2} J^{1/4} (q), \\
     &\mathfrak{S} (p, q)  = (p-q) \otimes (p-q) \big((P \cdot Q) - 1\big)^{-3/2} J^{1/4} (q).
\end{align*}
Since
$$
	|\mathfrak{R} (p, q)| + |\nabla_p \mathfrak{R} (p, q)|   \lesssim   J^{  (1/4)- } (q),
$$
it suffices to estimate an increment of $\mathfrak{S}$. As in the proof of \eqref{eqB.3.1}, we fix arbitrary $|p^1-p^2| \le 1$ and assume that \eqref{eqB.3.16} holds.

By direct calculations,
\begin{align*}
   & \mathfrak{S} (p^1, q) - \mathfrak{S} (p^2, q) = \big((p^1-q) \otimes (p^1-q) - (p^2-q) \otimes (p^2-q)\big)  \big((P^1 \cdot Q) - 1\big)^{-3/2} J^{1/4} (q) \\
   & + (p^2-q) \otimes (p^2-q)  \frac{   \big((P^2 \cdot Q) - 1)^{3/2}-(P^1 \cdot Q) - 1)^{3/2}\big)  }{((P^1 \cdot Q) - 1)^{3/2} \,  ((P^2 \cdot Q) - 1)^{3/2}}  J^{1/4} (q) =: \mathfrak{S}_1 + \mathfrak{S}_2.
\end{align*}
To handle $\mathfrak{S}_1$, we note that by the mean-value theorem, the triangle inequality, and \eqref{eqB.3.16},
\begin{align}
    \label{eqE.17.1}
    |(p^1-q) \otimes (p^1-q) - (p^2-q) \otimes (p^2-q)| \lesssim |p^1-p^2| |p^1-q|.
\end{align}
By this and the bound \eqref{eqB.3.2}, we get
\begin{align*}
   &  |\mathfrak{S}_1| \lesssim |p^1-p^2| |p^1-q| \big(\frac{1}{|p^1-q|^{3}} + \frac{1}{   (p^1_0)^{3/2} }\big) J^{1/8} (q) \\
   & \lesssim  |p^1-p^2| \big(\frac{1}{|p^1-q|^{2}} + 1\big) J^{1/16} (q).
\end{align*}
Hence, we have
\begin{equation}
        \label{eqB.3.13}
        \int_{A_1} \mathfrak{R}  |\mathfrak{S}_1| \, |g (q)| \, dq \lesssim |p^1-p^2| \|g\|_{ L_{\infty} (\bR^3)}.
\end{equation}

Next, to handle $\mathfrak{S}_2$, we observe  that by \eqref{eqB.3.7}--\eqref{eqB.3.6},
$$
    |\nabla_p ((P \cdot Q) - 1)^{3/2}| \lesssim q_0 ((P \cdot Q) - 1)^{1/2} \lesssim q_0  |p-q|,
$$
so that by the mean-value theorem and \eqref{eqB.3.16}, the absolute value of the nominator of the fraction in $\mathfrak{S}_2$ is bounded by (cf. \eqref{eqE.17.1})
$$
    q_0 |p^1-p^2|   |p^1-q|.
$$
By this and the bound \eqref{eqB.3.2},
\begin{align*}
&  |\mathfrak{S}_2| \lesssim  |p^1-p^2|  |p^2-q|^2   |p^1-q| (\frac{1}{  (p^1_0)^{3/2} }+\frac{1}{|p^1-q|^{3}})  (\frac{1}{   (p^2_0)^{3/2} }+\frac{1}{|p^2-q|^{3}}) J^{1/8} (q) \\
& =:  \mathfrak{S}_2^1 + \mathfrak{S}_2^2 + \mathfrak{S}_2^3 + \mathfrak{S}_2^4,
\end{align*}
where
\begin{align*}
  &  \mathfrak{S}_2^1 = |p^1-p^2| |p^2-q|^2   |p^1-q| \frac{1}{  (p^1_0)^{3/2}   (p^2_0)^{3/2} } J^{1/8} (q), \\
    & \mathfrak{S}_2^2 = |p^1-p^2| |p^2-q|^2    |p^1-q|   \frac{1}{  (p^2_0)^{3/2} } \frac{1}{|p^1-q|^{3}} J^{1/8} (q), \\
  & \mathfrak{S}_2^3 = |p^1-p^2| |p^2-q|^2   |p^1-q|  \frac{1}{ (p^1_0)^{3/2} }  \frac{1}{|p^2-q|^{3}} J^{1/8} (q), \\
    &  \mathfrak{S}_2^4 = |p^1-p^2|  \frac{1}{|p^1-q|^{2}  \, |p^2-q| } J^{1/8} (q).
\end{align*}
By the triangle inequality and \eqref{eqB.3.16},
\begin{align*}
&
    \mathfrak{S}_2^1  \lesssim  |p^1-p^2|   \frac{|p^2-q|^{3/2}   |p^1-q|^{3/2}}{   (p^1_0)^{3/2}   (p^2_0)^{3/2}  } J^{1/8} (q)
    \lesssim |p^1-p^2| J^{1/16} (q), \\
& \mathfrak{S}_2^2 \lesssim  |p^1-p^2|   J^{1/16} (q), \\
& \mathfrak{S}_2^3 \lesssim  |p^1-p^2|   \frac{1}{|p^2-q|} J^{1/16} (q).
\end{align*}
We note that due to \eqref{eqB.3.10},
$$
    \int  \mathfrak{R} \, \mathfrak{S}_2^4 \, |g (q)| \, dq \lesssim |p^1-p^2|^{\alpha} \|g\|_{L_{\infty} (\bR^3)}.
$$
The remaining integrals involving $\mathfrak{S}_2^j, j  = 1, 2, 3,$ are  handled in a similar way.
Thus, \eqref{eqB.3.9} holds with $\Phi$ replaced with $\Phi_3$ (see \eqref{eqB.3.24}).

Finally, we note that the integral involving $\Phi_1$ (see \eqref{eqB.3.24}) is estimated in the same way as the integral $I$ (see \eqref{eqB.3.0}), and the integral involving $\Phi_2$ is handled by splitting $p \otimes q = (p-q) \otimes q + q \otimes q$ and inspecting the proof   of the estimate of $\Phi_3$.
Thus, we conclude
\begin{equation*}
    \bigg\|\int \Phi_j (\cdot, Q) J^{1/2} (q) g (q) \, dq\bigg\|_{ C^{\alpha} (\bR^3) } \lesssim_{\alpha} \|g\|_{ L_{\infty} (\bR^3) }, j = 1, 2,
\end{equation*}
and hence, the desired estimate \eqref{eqB.3.9} holds.

\textbf{Proof of \eqref{eqB.3.1.1}. }
We denote
$$
    \Theta_k = \partial_{p_k} + \frac{q_0}{p_0} \partial_{q_k}.
$$
It was shown in the proof of Lemma 2 in \cite{GS_03} that
\begin{align}
    \label{eqB.3.17}
  &  \Theta_k \Phi^{i j} (P, Q) = \tilde \Phi_1^{i j} (P, Q) + \tilde \Phi_2^{i j} (P, Q) + \tilde \Phi_3^{i j} (P, Q), \\
  & \tilde \Phi_1^{i j} (p, q)= \frac{(P \cdot Q)^2}{\big((P \cdot Q)^2 - 1\big)^{1/2}}  \Theta_k \big(\frac{\delta_{i j}}{p_0 q_0}\big), \notag \\
  & \tilde \Phi_2^{i j} (p, q) =  \frac{(P \cdot Q)^2}{\big((P \cdot Q)^2 - 1\big)^{3/2}} \big((P \cdot Q) - 1\big) \Theta_k \big(\frac{p_i q_j + p_j q_i}{p_0 q_0}\big), \notag \\
  & \tilde \Phi_3^{i j}  (p, q) =  -  \frac{(P \cdot Q)^2}{\big((P \cdot Q)^2 - 1\big)^{3/2}} \Theta_k \big(\frac{(p_i - q_i) (p_j  - q_j)}{p_0 q_0}\big). \notag
\end{align}
By direct calculations,
\begin{align*}
&
    \Theta_k \big(\frac{1}{p_0 q_0}\big) = - (\frac{p_k}{p_0}  + \frac{q_k}{q_0}) \frac{ 1 }{p_0^2 q_0 }, \\
    &  \Theta_k \big(\frac{p_i q_j}{p_0 q_0}\big) = p_i q_j \Theta_k \big(\frac{1}{p_0 q_0}\big) + \frac{ \Theta_k (p_i q_j)  }{p_0 q_0} \\
   & =  -  \frac{p_i}{p_0} \frac{q_j}{q_0} (\frac{p_k}{p_0}  + \frac{q_k}{q_0}) \frac{1}{p_0 }
     + \frac{1}{p_0 } (\delta_{i k} \frac{q_j}{q_0} + \delta_{j k} \frac{p_i}{p_0}).
\end{align*}
Hence, for any multi-index $\beta$, we have
\begin{align}
    \label{eqB.3.18}
  &   |D^{\beta}_p  \Theta_k \big(\frac{1}{p_0 q_0}\big)| \lesssim_{\beta} p_0^{-2-|\beta|} q_0^{-1}, \\
      \label{eqB.3.19}
  &  |D^{\beta}_p \Theta_k \big(\frac{p_i q_j}{p_0 q_0}\big)| \lesssim_{\beta} p_0^{-1-|\beta|}.
\end{align}
Furthermore,  by the formulas $(42)$--$(43)$ on p. 278 in \cite{GS_03},
\begin{align}
      \label{eqB.3.20}
    \Theta_k \big(\frac{(p_i - q_i)(p_j - q_j)}{p_0 q_0}\big) =  \sum_{r, s = 1}^3 (p_{r} - q_r)(p_s - q_s) \phi^{i j}_{k, r  s} (p, q),
\end{align}
where $\phi^{i j}_{k, r s }$ are smooth functions satisfying the estimate
\begin{align}
      \label{eqB.3.21}
    |D^{\beta_1}_p D^{\beta_2}_q \phi^{i j}_{k, r s} (p, q)| \lesssim_{\beta_1, \beta_2} p_0^{-2-|\beta_1|} q_0^{-1-|\beta_2|}
\end{align}
for any multi-indexes $\beta_j, j = 1, 2$.
Combining \eqref{eqB.3.17}--\eqref{eqB.3.21}, we find that $\tilde \Phi_i$ (see \eqref{eqB.3.17} resembles $\Phi_i$ defined in \eqref{eqB.3.24} for $i=1, 2, 3$, and hence, their H\"older norms are estimated by repeating the argument we used to justify \eqref{eqB.3.9}. Thus, the lemma is proved.
\end{proof}

\begin{remark}
    \label{remark B.6}
    We can replace $J^{1/2} (q)$ in the integrals in the statement of Lemma \ref{lemma B.3} with any function $\xi = \xi (q)$ such that $\xi$ and $\nabla_q \xi$ decay
    sufficiently fast at infinity.
\end{remark}

The following lemma is an immediate corollary of  Lemma 4 on p. 287 in \cite{GS_03}. 
\begin{lemma}
        \label{lemma B.5}
    For $r \in (3/2, \infty]$, $g \in W^{  1 }_r (\bR^3)$, the following identity holds in the sense of distributions:
   \begin{align}
        \label{eqB.5.1}
     &  \partial_{p_i} \int    \Phi^{ i j } (P, Q)  J^{1/2} (q)   \partial_{q_j}  g ( q) \, dq  \\
     & =   \partial_{p_i}   \int    \Phi^{ i j } (P, Q)  J^{1/2} (q)  \frac{ q_j}{2 q_0}   g (q)  \, dq  \notag\\
    &  - 4  \int  \frac{P \cdot Q}{p_0 q_0} \bigg((P \cdot Q)^2 - 1\bigg)^{-1/2}  J^{1/2} (q)   g (q) \, dq - \kappa (p) J^{1/2} (p) g (p), \notag
     \end{align}
     where $\kappa (p) = 2^{7/2} \pi p_0 \int_0^{\pi} (1+|p|^2 \sin^2 \theta)^{-3/2} \sin (\theta) \, d\theta$.
\end{lemma}


\begin{lemma}
        \label{lemma B.1}
Let $r \in (3/2, \infty]$, $g = (g^{+}, g^{-}) \in W^{  1 }_r (\bR^3) \cap L_{\infty} (\bR^3)$ and $a_g$, $C_g$ and $K g$ be given by \eqref{eq1.1.5},  \eqref{eq1.1.6}, and \eqref{eq1.1.7}, respectively.
    Then, one has
        \begin{align}
        \label{eqB.1.5}
                  &  \|a_g\|_{ L_{\infty } (\bR^3) } \lesssim_r     \|g\|_{ W^1_{r} (\bR^3) },  \\
                             \label{eqB.1.3}
          & \|C_g\|_{ L_{\infty } (\bR^3) } \lesssim_r   1 +      \|g\|_{ W^1_{r} (\bR^3) }, \\
               \label{eqB.1.1}
             &   |K g| (p) \lesssim_r J^{1/4} (p) \|g\|_{W^1_r (\bR^3)}.
          \end{align}
         If $r = \infty$, then, for any $\alpha \in (0, 1)$,
         \begin{align}
         \label{eqB.1.6}
            &  [a_g]_{ C^{\alpha}  (\bR^3) }  \lesssim_{\alpha}   \||g| + |\nabla_p g|\|_{ L_{\infty} (\bR^3)},   \\
                                          \label{eqB.1.4}
         &   [C_g ]_{ C^{\alpha} (\bR^3) } \lesssim_{\alpha}   1 + \||g| + |\nabla_p g|\|_{ L_{\infty} (\bR^3)}, \\
        \label{eqB.1.2}
       & [K g]_{ C^{\alpha}  (\bR^3) }  \lesssim_{\alpha}   \||g| + |\nabla_p g|\|_{ L_{\infty} (\bR^3)}.
        \end{align}
\end{lemma}

\begin{proof}
The $L_{\infty}$ bounds \eqref{eqB.1.5}--\eqref{eqB.1.1} were proved in Lemma B.5 in \cite{VML}.

\textit{Estimate of $a_g$.}
Applying the estimate \eqref{eqB.3.9}  to the integral
$$
    \int \Phi^{ i j } (P, Q)   J^{1/2} (q)   \partial_{q_j}  g (q) \cdot  (1, 1)\, dq
$$
and estimating the H\"older seminorm of
$$
    \frac{p_i}{2 p_0}  \int \Phi^{ i j } (P, Q)   J^{1/2} (q)     g (q)\, dq
$$
by interpolating between the estimates \eqref{eqB.2.1}  with $k = 0$ and $k = 1$,
we prove the validity of \eqref{eqB.1.6}.

\textit{Estimate of $C_g$.}
 By the bound of the $\sigma$-function (see \eqref{eq27.4.17}) established in Lemma 5 in \cite{GS_03},
 $$
    |D^{\beta}_p \sigma| \lesssim_{|\beta|}  p_0^{ - |\beta| },
 $$
  it suffices to estimate the integral term in \eqref{eq1.1.6} given by
  \begin{align}
    \label{eqB.1.15}
   & \partial_{p_i} \int    \Phi^{ i j } (P, Q)  J^{1/2} (q)   \partial_{q_j}  g ( q) \cdot (1, 1)\, dq \\
   &  - \frac{p_i}{2 p_0}  \int  	   \Phi^{ i j } (P, Q)  J^{1/2} (q)  \partial_{q_j}  g (q) \cdot (1, 1)\, dq =: C_{g, 1} + C_{g, 2}. \notag
\end{align}
Next, by the identity \eqref{eqB.5.1},
  \begin{align}
      \label{eqB.1.16}
& C_{g, 1} =   \partial_{p_i} \int    \Phi^{ i j } (P, Q)  J^{1/2} (q)  \frac{ q_j}{2 q_0}   g (q) \cdot (1, 1) \, dq \\
    &  - 4  \int  \frac{P \cdot Q}{p_0 q_0} \bigg((P \cdot Q)^2 - 1\bigg)^{-1/2}  J^{1/2} (q)   g (q) \cdot (1, 1)\, dq \notag\\
 &    - \kappa (p) J^{1/2} (p) g (p) \cdot (1, 1)  =:   C_{g, 1, 1}  +   C_{g, 1, 2}  +   C_{g, 1, 3}. \notag
\end{align}

First, by \eqref{eqB.2.1} and \eqref{eqB.3.9},
\begin{align}
    \label{eqB.1.10}
    [C_{g, 2}]_{ C^{\alpha} (\bR^3)} \lesssim_{\alpha} \|\nabla_p g\|_{ L_{\infty} (\bR^3)}.
\end{align}
Furthermore, by \eqref{eqB.3.5} and \eqref{eqB.3.1}, one has
\begin{align}
        \label{eqB.1.11}
    [C_{g, 1, 2}]_{ C^{\alpha} (\bR^3)} \lesssim_{\alpha} \| g\|_{ L_{\infty} (\bR^3)}.
\end{align}
Due to the product rule inequality and  the bound
\begin{align*}
        |D^{\beta}  \big(\kappa (p) J^{1/2} (p)\big)| \lesssim_{|\beta|} J^{1/4} (p),
\end{align*}
 we have
\begin{align}
        \label{eqB.1.12}
    [C_{g, 1, 3}]_{ C^{\alpha} (\bR^3)} \lesssim_{\alpha} \|g\|_{ C^{\alpha} (\bR^3)}.
\end{align}
To estimate $C_{g, 1, 1}$, we recall the  identity \eqref{eq12.A.2.15}:
\begin{align}
    \label{eqB.1.8}
	& \partial_{p_i} \int_{ \bR^3 } \Phi^{i j} (P, Q) J^{1/2} (q) h (q) \, dq\\
&	=  \int   \Phi^{i j} (P, Q) J^{1/2} (q) \frac{q_0}{p_0}\partial_{q_i} h (q) \, dq \notag \\
&\quad	+ \int   \Phi^{i j} (P, Q) J^{1/2} (q) \big(\frac{q_i}{q_0 p_0}  - \frac{q_i}{2 p_0}\big) h (q) \, dq  \notag \\
 &\quad	+  \int  (\partial_{p_i} + \frac{q_0}{p_0} \partial_{q_i}) \Phi^{i j} (P, Q) J^{1/2} (q) h (q) \, dq. \notag
 \end{align}
 By using \eqref{eqB.3.9} again and Remark \ref{remark B.6}, we find that the $C^{\alpha}$-seminorms of the first two integrals on the r.h.s of \eqref{eqB.1.8} are bounded by
 $$
     N (\alpha) (\||h| + |\nabla_p h|\|_{ L_{\infty} (\bR^3)}).
 $$
 Furthermore, by \eqref{eqB.3.1.1}, the $C^{\alpha}$-seminorm  of the third term on the r.h.s. of \eqref{eqB.1.8} is dominated by
 $$
    N (\alpha) \|h\|_{ L_{\infty} (\bR^3)}.
 $$
 Replacing $h (q)$ with $\frac{q_j}{2 q_0} g (q)$ in the above argument, we conclude that
 \begin{align}
         \label{eqB.1.13}
    [C_{g, 1, 1}]_{ C^{\alpha} (\bR^3)} \lesssim_{\alpha} \||g| + |\nabla_p g|\|_{ L_{\infty} (\bR^3)}.
  \end{align}
 Combining \eqref{eqB.1.15}--\eqref{eqB.1.13} and using the interpolation inequality, we prove the desired estimate \eqref{eqB.1.4}.

\textit{Estimate of $K g$.}
First, we split  the integral in \eqref{eq1.1.7} as follows:
	\begin{align*}
	&	K g  = (\partial_{p_i} p_0) J^{1/2} (p) \int \Phi^{i j} (P, Q) J^{1/2} (q) (\partial_{q_j} g (q) + \frac{q_j}{2} g (q)) \cdot (1, 1) \, dq \,  (1, 1)\\
	&
		 -  J^{1/2} (p) \partial_{p_i}  \int \Phi^{i j} (P, Q) J^{1/2} (q)   \frac{q_j}{2 q_0} g (q) \cdot (1, 1) \, dq\,  (1, 1)  \\
   & -  J^{1/2} (p) \partial_{p_i} \int  \Phi^{i j} (P, Q) J^{1/2} (q)  \partial_{q_j} g (q) \cdot (1, 1) \, dq\,  (1, 1) =: K_1 + K_2 + K_3.
	\end{align*}
 We observe that the following terms are similar:
 \begin{itemize}
     \item[--] $K_1$ and $C_{g, 2}$ (see \eqref{eqB.1.15}),
     \item[--] $K_2$ and $C_{g, 1, 1}$ (see \eqref{eqB.1.16}),
     \item[--] $K_3$ and $C_{g, 1}$ (see \eqref{eqB.1.15}).
 \end{itemize}
  Hence, the estimate \eqref{eqB.1.2}  is proved by repeating the argument we used for $C_g$.
\end{proof}

\section{Regularity of a velocity average}

\label{appendix G}

The following result is a slightly generalized version of the averaging lemma in \cite{DLM_91} proved by inspecting the argument of the aforementioned reference. In particular, in the lemma below, $\psi$ does not need to be smooth and compactly supported. 

\begin{lemma}[cf. Theorem 2 in \cite{DLM_91}]
            \label{lemma 14.1}
Let 
\begin{itemize}
    \item[--] $d \ge 1$, $p \in [2, \infty)$, $\alpha \in [0, 1)$,
    \item[--] $f, g \in L_p (\bR^{2d})$ satisfy 
\begin{align*}
        v \cdot \nabla_x f  = (1-\Delta_x)^{\alpha/2} g,
\end{align*}
\item[--] $\chi \in L_1 (\bR^d_v)$ be a function such that for some $\beta > \frac{d-1}{2}$ and $K > 0$,
\end{itemize}
\begin{align*}
    |\chi (v)| \le K (1+|v|^2)^{-\beta} \,  \text{a.e.} \,  v \in \bR^d, \quad \|\chi\|_{L_1 (\bR^d)} \le K.
\end{align*}
Then, we have
\begin{align}
  \label{eq14.1.0}
   \bigg\| \int_{\bR^d} f (\cdot, v) \chi (v) \, dv \bigg\|_{ W^{\gamma}_p (\bR^d_x) } \lesssim_{d,  p, \alpha, \beta,   K }  \|f\|_{ L_p (\bR^{2d}) } + \|g\|_{ L_p (\bR^{2d}) },
\end{align}
where
$$
    \gamma = \frac{1-\alpha}{p}.
$$
\end{lemma}

\section{Sobolev regularity of even and odd functions}
\label{appendix F}
\begin{lemma}[cf. Lemma 5.2 in \cite{NPV_12}]
    \label{lemma F.1}
Let $p \in [1, \infty)$, $s \in (0, 1)$, and $\Omega \subset \bR^d$ be a domain symmetric with respect to $x_d$.
For a function $u \in W^s_p (\Omega)$ (see \eqref{eq1.2.21}), we denote
$$
 u_{\text{even}} (x)
  = \begin{cases}
& u (x), x_d \ge 0, \\
& u (x_1, \ldots, x_{d-1}, -x_d), x_d < 0.
\end{cases}
$$
Then, $u_{\text{even}} \in W^s_p (\Omega)$, and
\begin{align}
    \label{eqF.1.1}
    \|u_{\text{even}}\|_{W^s_p (\Omega)} \le 4 \|u\|_{W^s_p (\Omega \cap \bR^d_{+})}.
\end{align}
\end{lemma}

\begin{lemma}
        \label{lemma F.2}
Let  $p \in [1, \infty), s \in (0, 1/p)$, and $u \in W^s_p (\bR^d_{+})$. For
\begin{equation}
        \label{eqF.2.1}
     u_{\text{odd}} (x)
     = \begin{cases}
& u (x), x_d \ge 0, \\
& - u (x_1, \ldots, x_{d-1}, -x_d), x_d < 0,
     \end{cases}
\end{equation}
   we have
\begin{align*}
      [ u_{\text{odd}}]_{ W^s_p (\bR^d) }\lesssim_{d, s, p} [u]_{ W^s_p (\bR^d_{+}) }.
\end{align*}
\end{lemma}

\begin{proof}
We denote $x' = (x_1, \ldots, x_{d-1})$, $\bar x = (x', -x_d)$.
We note that by the change of variable $x_d \to -x_{d}$, we have
\begin{align}
    \label{eqF.2.2}
  & [u_{\text{odd}} ]^p_{W^s_p (\bR^d) }  = \int_{\bR^d} \int_{\bR^d} \frac{ |u_{\text{odd}} (x) -  u_{\text{odd}} (y)|^p }{|x-y|^{d+ s p}} \, dx dy \notag\\
   & = 2 \int_{\bR^d_{+}} \int_{\bR^d_{+}} \frac{ |u(x) -  u (y)|^p}{|x- y|^{d +  s p}} \, dx dy
   + 2 \int_{\bR^d_{+}} \int_{\bR^d_{+}}  \frac{ |u(x) +  u (y)|^p}{|x-\bar y|^{d+  s p}} \, dx dy \notag \\
   & \lesssim  [ u ]^p_{W^s_p (\bR^d_{+}) }  + \int_{\bR^d_{+}} |u (x)|^p \bigg(\int_{\bR^d_{+}} \frac{dy}{|x-\bar y|^{d + s p}} \bigg) \,  dx
   + \int_{\bR^d_{+}} |u (y)|^p \bigg(\int_{\bR^d_{+}} \frac{dx}{|x-\bar y|^{d + s p}}\bigg) \,  dy.
\end{align}
Furthermore, by changing variables $y' \to y'-x'$, $y_d \to y_d + x_d$, and $y' \to \frac{y'}{y_d}$, we get
\begin{align*}
        \int_{\bR^d_{+}} \frac{dy}{|x-\bar y|^{d + s p}}
    =  \int_{\bR^{d-1}} \frac{dy'}{(1+|y'|^2)^{(d + s p)/2}} \int_{x_d}^{\infty} \frac{d y_d}{y_d^{1+sp}}  = N (d, s, p) \frac{1}{x_d^{sp}}.
\end{align*}
Next, using a fractional variant of Hardy's inequality for $W^s_p (\bR^d_{+})$ functions (see, for example, \cite{FS_08}), which is valid for $s \in (0, 1/p)$,
we conclude
$$
    \int_{\bR^d_{+}} \frac{|u (x)|^p}{x_d^{sp}}\, dx \lesssim_{d, s, p}
    [u]^p_{W^s_p (\bR^d_{+}) }.
$$
The same bound holds for the last term on the r.h.s of \eqref{eqF.2.2}.
The lemma is proved.
\end{proof}

\begin{lemma}
        \label{lemma F.3}
Let $p \in [1, \infty)$, $s \in (0, 1/p)$, $\Omega$ be a Lipschitz domain, and $u \in W^s_p (\Omega)$.
Let  $\tilde u$ be the function defined as $u$ inside $\Omega$ and $0$ outside. Then, $\tilde u \in W^s_p (\bR^d)$, and
\begin{align}
    \label{eqF.3.1}
    \|\tilde u\|_{W^s_p (\bR^d)} \lesssim_{d, s, p, \Omega}  \|u\|_{W^s_p (\Omega)}.
\end{align}
\end{lemma}

\begin{proof}
By localization and boundary flattening  (see, for example, the proof of Theorem 5.4 in \cite{NPV_12}), we may assume that $\Omega = \bR^d_{+}$. The desired assertion now follows from
the fact that $\tilde u = \frac{1}{2}(u_{\text{even}} + u_{\text{odd}})$ combined with Lemmas  \ref{lemma F.1}--\ref{lemma F.2}.
\end{proof}

The next assertion is a direct corollary of Lemmas \ref{lemma F.2}--\ref{lemma F.3}.
\begin{lemma}
        \label{lemma F.4}
Let $p \in [1, \infty)$, $s \in (0, 1/p)$, and $\Omega$ be a  Lipschitz domain symmetric with respect to  $x_d$. For a function $u \in W^s_p (\Omega \cap \bR^d_{+})$, we denote $u_{\text{odd}}$  as in  \eqref{eqF.2.1}.
Then, $u_{\text{odd}} \in W^s_p (\Omega)$, and
\begin{align}
    \label{eqF.4.1}
      [u_{\text{odd}}]_{ W^s_p (\Omega) }\lesssim_{d, s, p, \Omega} [u]_{ W^s_p (\Omega \cap \bR^d_{+}) }.
\end{align}
\end{lemma}

\section{Regularity of the solution to the Lam\'e system with the Navier boundary condition}

\label{appendix H}

 We consider the Lam\'e system with the Navier boundary condition:
\begin{equation}
            \label{eqA.1}
\left\{\begin{aligned}
 &- \nabla \cdot S (\bm{u}) = \bm{f},\\
 & (\bm{u} \cdot n_x)_{| \partial \Omega} = 0,\\
 & \big((S (\bm{u}) n_x) \times n_x\big)_{| \partial \Omega} = 0,
\end{aligned}
\right.
\end{equation}
where $S$ is defined in \eqref{eq1.2.23}.

\begin{lemma}
For any
$\bm{u}, \bm{w} \in  W^2_{2 }  (\Omega)$ satisfying the Navier boundary condition, the following Green's identity  holds:
\begin{equation}
        \label{eqA.2.0}
   -  \int_{\Omega} \bm{w}_i \partial_{x_j} S_{i j} (\bm{u})  \, dx
    = \sum_{i, j =1}^3   \int_{\Omega}  S_{i j} (\bm{u}) S_{i j} (\bm{w})   \, dx.
\end{equation}
\end{lemma}

\begin{remark}
    \label{remark A.1}
Let $\bm{u}$ be a solution to \eqref{eqA.1} with $\bm{f} = 0$. Then, by \eqref{eqA.2.0}, $S (\bm{u}) = 0$. By this and the boundary condition $\bm{u} \cdot n_x = 0$, we conclude that $\bm{u} \in \mathcal{R} (\Omega)$ (see \eqref{eq1.1}).
\end{remark}


The proof of the following variant of Korn's inequality is standard (cf., for example, \cite{KO_88}).
\begin{lemma}
        \label{lemma A.2}
Let $\Omega$ be a  $C^{1}$ domain.
Then,  for any $\bm{u} \in W^1_{2 } (\Omega)$ such that  $\bm{u} \cdot n_x = 0$ on $\partial \Omega$ and  $\bm{u} \perp \mathcal{R} (\Omega)$ (see \eqref{eq1.1}) in the $L_2 (\Omega)$ sense,  we have
\begin{align}
    \label{eqA.2.1}	
      \|\bm{u}\|_{ W^1_2 (\Omega) } \lesssim_{\Omega} \|S (\bm{u})\|_{ L_2 (\Omega) }.
\end{align}
\end{lemma}

\begin{lemma}
            \label{lemma A.1}
Let $\bm{f} \in L_2 (\Omega)$ be a function such that $\bm{f} \perp  \mathcal{R} (\Omega)$ in the $L_2 (\Omega)$ sense.
Then,  the system \eqref{eqA.1} has a  unique strong solution $\bm{u} \in W^2_{2} (\Omega)$ satisfying $\bm{u} \perp \mathcal{R} (\Omega)$ in $L_2 (\Omega)$, and
\begin{align}
    \label{eqA.1.10}
    \|\bm{u}\|_{W^2_2 (\Omega)} \lesssim_{\Omega} \|\bm{f}\|_{ L_2 (\Omega)}.
\end{align}
\end{lemma}




\begin{proof}
\textbf{Step 1: well-posedness in $W^2_2$  for large $\lambda$. }
We will prove that there exists a constant $\lambda_0 = \lambda_0 (\Omega) > 0$ such that the system
\begin{equation}
      \label{eqA.1.1}
\left\{\begin{aligned}
 &- \nabla \cdot S (\bm{u}) + \lambda \bm{u} = \bm{f},\\
 & (\bm{u} \cdot n_x)_{| \partial \Omega} = 0,\\
 & \big((S (\bm{u}) n_x) \times n_x\big)_{| \partial \Omega} = 0,
\end{aligned}
\right.
\end{equation}
has a unique strong solution $\bm{u} \in W^2_2 (\Omega)$, and
\begin{align}
    \label{eqA.1.4}
        \|  \lambda |\bm{u}|  + \lambda^{1/2} |\nabla \bm{u}| + |D^2 \bm{u}|\|_{ L_2 (\Omega) } \lesssim_{\Omega} \|- \nabla \cdot S (\bm{u}) + \lambda \bm{u}\|_{ L_2 (\Omega) }.
\end{align}
To this end, we invoke the classical elliptic regularity theory established by Agmon, Douglis, and Nirenberg (see \cite{ADN_1}--\cite{ADN_2}) and further developed by many researchers.

We check the Lopatinskii-Shapiro (L-S) condition for our system \eqref{eqA.1} (see \cite{ADN_1}--\cite{ADN_2},  \cite{DG_20}). We may assume that the domain is the half-space $\bR^3_{-}$. Then, according to Section 10 in \cite{ADN_1}, the L-S condition is a necessary condition for the  $W^2_2 (\bR^3_{-})$ a priori estimate to hold. To verify this a priori estimate, we fix any $\bm{u} \in W^2_2 (\bR^3_{-})$ satisfying the system \eqref{eqA.1} on $\bR^3_{-}$ and observe the Navier boundary condition becomes
$$
    \bm{u}_3 (x_1, x_2, 0) = 0, \quad \partial_{x_3 }  \bm{u}_i (x_1, x_2, 0) = 0, i = 1, 2.
$$
Furthermore, let
\begin{itemize}
\item[--] $\overline{\bm{ u}}_3$ and $\overline{\bm{f}}_3$ be the odd extensions of $\bm{u}_3$ and $\bm{f}_3$, respectively, across $\{x_3 = 0\}$,
\item[--] $\overline{\bm{u}_i}, \overline{\bm{f}_i},  i = 1, 2,$ be the even extensions of $\bm{u}_i, \bm{f}_i, i = 1, 2$, respectively.
\end{itemize}
We note that $\overline{\bm{u}} = (\overline{\bm{u}}_1,  \overline{ \bm{u}}_2,  \overline{ \bm{u}}_3) \in W^2_2 (\bR^3)$, so that $\overline{\bm{u}}$ satisfies the identity
\begin{align}
    \label{eqA.1.3}
        -\nabla_x \cdot S (\overline{\bm{u}}) = \overline{\bm{f}}
\end{align}
on $\bR^3$.
We observe that the Lam\'e system \eqref{eqA.1.3} is strongly elliptic in the Legendre sense, and hence, the $W^2_2 (\bR^3)$ a priori estimate is true for
$\overline{ \bm{u} }$ (see, for example, Theorem 3.1 in \cite{DK_11}).  Thus, on $\bR^3_{-}$, the system \eqref{eqA.1} satisfies the $W^2_2 (\bR^3_{-})$ a priori estimate, and therefore, the L-S condition holds.

Next, due to Theorem 3.2 and Remark 3.4 $(i)$ in \cite{DG_20}, there exists $\lambda_0 = \lambda_0 (\Omega) > 0$ such that for any $\lambda \ge \lambda_0$, the system  \eqref{eqA.1.1} has a unique strong solution $\bm{u} \in W^2_2 (\Omega)$, and the estimate \eqref{eqA.1.4} holds.

\textbf{Step 2: $W^2_2$ solvability of the original system.}
Finally, the desired assertion of the lemma follows from the unique solvability of the system \eqref{eqA.1.1} in the class of $W^2_2 (\Omega)$ solutions for large $\lambda$ via a standard argument involving the Fredholm alternative. See, for example, the proof of Theorem 6.2.4  in \cite{E_10}.
\end{proof}

\begin{remark}
  See  \cite{CK_23}  for a different proof of Lemma \ref{lemma A.1}.
\end{remark}

\section{Derivation of the angular momentum conservation}
\label{appendix I}
Let $F, \bE, \bB$ be a sufficiently regular solution to the RVML system in \eqref{RVML} and $\Omega$ be an axisymmetric domain such that its axis is parallel to $\omega$ and contains a point $x_0$.  The goal of this section is to verify the conservation of angular momentum identity \eqref{eq4.3.4}.

We claim that the following momentum identity for the RVML system is true (cf. $(7)$ in Section 9 in \cite{WS_89}):
\begin{align}
    \label{eqH.7}
 &\partial_t \big(\int_{\bR^3} p (F^{+} + F^{-})  \, dp  + \frac{1}{4 \pi } (\bE \times \bB)\big)  \\
 &   +  \nabla_{x} \cdot    \big(\int_{\bR^3} p \otimes (\frac{p}{p_0^{+}} F^{+} + \frac{p}{p_0^{-}} F^{-}) \, dp - \bm{T} \big) = 0, \notag
\end{align}
where
\begin{align*}
    \bm{T}_{i j} = \frac{1}{4 \pi } \big(\bE_i \bE_j + \bB_i \bB_j  - \frac{1}{2} \delta_{i j} (|\bE|^2+|\bB|^2) \big)
\end{align*}
is the Maxwell stress tensor.
The above identity is derived by multiplying the Landau equations by $p$ and using the momentum identity for Maxwell's equations given by  (see Section 5.3 in \cite{S_87})
\begin{align*}
  \frac{1}{4 \pi }  \partial_t  (\bE \times \bB) - \nabla_x \cdot \bm{T}  =  - (\rho \bE  + \bm{j} \times \bB).
\end{align*}

Next, we will verify that
\begin{equation}
        \label{eqH.2}
     \int_{\Omega} R  (x)\cdot \nabla_x \cdot \bm{T} \, dx =  0.
    \end{equation}
We may assume that $\omega = e_1$ and $x_0 = 0$, so that $R (x): = \omega \times (x-x_0) = (-x_2, x_1, 0)^T$.
By the divergence theorem,  the r.h.s. of \eqref{eqH.2} equals
\begin{align*}
       \int_{\partial \Omega}  R^T \bm{T} n_x \, dS_x -  \int_{\Omega}  (\bm{T}_{1 2} -  \bm{T}_{2 1})  \, dx.
\end{align*}
Clearly, the second integral on the r.h.s. is $0$. Since $\bE$ and $\bB$ satisfy the perfect conductor boundary condition, $\bm{T} n_x$ is parallel to $n_x$, and hence, since
\begin{align}
    \label{eqH.14}
    R \cdot n_x = 0,
\end{align} the surface integral in the above identity vanishes. Thus, \eqref{eqH.2} is valid, and we obtain the ``angular momentum identity" for the electromagnetic field:
\begin{align}
        \label{eqH.4}
 \frac{1}{4 \pi }   \partial_t  \int_{\Omega}  R  \cdot (\bE \times \bB) \, dx  =  -   \int_{\Omega} R \cdot (\rho \bE  + \bm{j} \times \bB) \, dx.
\end{align}

Next, by  using the identity
$$
p =  P_{||} p + p_{\perp} n_x
 $$
(see \eqref{eq3.1.4}) and  \eqref{eqH.14}, we get
\begin{align}
    \label{eqH.6}
   & \int_{\Omega} R (x) \cdot \nabla_x \cdot \bigg(\int_{\bR^3} p \otimes (\frac{p}{p_0^{+}} F^{+} + \frac{p}{p_0^{-}} F^{-}) \, dp\bigg) \, dx \\
  &  =  \int_{\partial \Omega}  \int_{\bR^3} \big(R \cdot (P_{||} p)\big) p_{\perp}  (\frac{1}{p_0^{+}} F^{+} + \frac{1}{p_0^{-}} F^{-}) \, dp dx. \notag
\end{align}
The last integral vanishes because $(p_0^{\pm})^{-1} F^{\pm}$ satisfy the SRBC.

Finally, we obtain the desired conservation law \eqref{eq4.3.4} by multiplying the momentum identity \eqref{eqH.7} by  $R (x)$, integrating the result over $\Omega$,  using the identities \eqref{eqH.2} and \eqref{eqH.6},
and the assumption on the initial data \eqref{eq6.1.3}.

\section{Construction of the  functions $B_{i j}$ }

\label{section J}

For convenience, we set all the physical constants to $1$.

\subsection{Conditions on the function $h$.} First, we derive certain conditions on $h$ in \eqref{eq3.1.40} that imply  that the desired properties \eqref{eq3.1.6}--\eqref{eq3.1.8} hold.

\textit{The property \eqref{eq3.1.8}.} We will show that for \eqref{eq3.1.8} to be valid,  it suffices to impose two conditions on $h$ (see  \eqref{eq51.5} and \eqref{eq51}).
To compute
\begin{align*}
  \mathcal{A}:= \big(\int_{\bR^3} B_{i j} \frac{p_k}{p_0} p_l \, dp\big) \partial_{x_k} S_{i j} (\bm{\phi}) \xi_l, \, \xi \in \bR^3,
\end{align*}
in terms of $h$, we denote
$$
    \lambda_1 = \int \frac{p_1^2 p_2^2}{p_0} h (|p|) \sqrt{J} \, dp, 
    \quad 
    \lambda_2 = \int \frac{p_1^4}{p_0} h (|p|) \sqrt{J} \, dp, \quad \lambda_3  = \int  \frac{p_1^2}{p_0} h (|p|)  \sqrt{J}\, dp.
$$
By direct calculations, 
\begin{align}
    \label{eq51.2}
& \int B_{i j}  \frac{p_k}{p_0} p_l \sqrt{J} \, dp  = \int  (p_i p_j    - \delta_{i j})  h (|p|)   \frac{p_k}{p_0} p_l \sqrt{J} \, dp \\
&= 1_{i \neq j}\big(1_{k=i, l=j}
 + 1_{k=j, l=i}\big) \lambda_1
 + 1_{ i=j \neq  k =  l } (\lambda_1 - \lambda_3) 
  +1_{i=j=k=l} (\lambda_2 - \lambda_3). \notag
\end{align}
Imposing the condition $\lambda_1 = \lambda_3$, that is,
\begin{align}
    \label{eq51.5}
    \int_{\bR^3} \frac{p_1^2 p_2^2}{p_0} h (|p|) \sqrt{J} \, dp = \int_{\bR^3}  \frac{p_1^2}{p_0}  h (|p|) \sqrt{J} \, dp,
\end{align}
we may cancel the second term on the r.h.s. of \eqref{eq51.2} and obtain
\begin{align}
    \label{eq51.3}
& \mathcal{A} = \lambda_1 \sum_{i\neq j} \big(\partial_{x_i} S_{i j} (\bm{\phi}) \xi_j 
+ \partial_{x_j} S_{i j} (\bm{\phi}) \xi_i\big) 
+ (\lambda_2 - \lambda_1) \sum_{i=1}^3 \partial_{x_i}^2 \bm{\phi}^i \xi_i  \\
& =  2 \lambda_1 \partial_{x_i} S_{i j} (\bm{\phi}) \xi_j 
+ (\lambda_2 - 3 \lambda_1) \sum_{i=1}^3 \partial_{x_i}^2 \bm{\phi}^i \xi_i.  \notag
\end{align}
To cancel the last term, we impose the condition
\begin{equation}
    \label{eq51}
\lambda_1  = \frac{1}{2}.
\end{equation}
We claim that under \eqref{eq51},
\begin{align}
    \label{3times}
    \lambda_2 = 3 \lambda_1 = \frac{3}{2}.
\end{align}
To verify this, we use the spherical coordinates:
\begin{align*}
     &   \lambda_1 = \Lambda \int_0^{2\pi} \cos^2 (\theta) \sin^2 (\theta) \, d\theta, \quad \lambda_2 = \Lambda \int_0^{2\pi} \cos^4 (\theta)  \, d\theta, \\
    &\text{where} \,\,  \Lambda:= \int_0^{\infty} r^6 h (r) \sqrt{J} \, dr \int_0^{ \pi} \sin^5 (\phi) \, d\phi.
\end{align*}
Hence,  the identity \eqref{3times}  follows from 
$$
    \int_0^{2\pi} \cos^4 (\theta)  \, d\theta =  \frac{3 \pi}{4}, 
    \quad \int_0^{2\pi} \cos^2 (\theta) \sin^2 \theta \, d\theta = \frac{\pi}{4}.
$$
Thus, by \eqref{eq51.3}--\eqref{3times}
\begin{align*}
    \mathcal{A}    = \partial_{x_i} S_{i j} (\bm{\phi}) \xi_j,
    \end{align*}
which is the desired identity \eqref{eq3.1.8}.

\textit{The properties \eqref{eq3.1.6}--\eqref{eq3.1.7}.}
We note that by oddness,  the first identity in \eqref{eq3.1.7} is true. Furthermore, the second one in \eqref{eq3.1.7} follows from the second condition in \eqref{eq3.1.6}. 

Next, we note that
 \begin{align*} 
 B_{i j} (p)  = (p_i p_j - \frac{1}{3}\delta_{i j} |p|^2) h (|p|)  
 +\delta_{i j} \frac{1}{3} (|p|^2-3)   h (|p|) =: \tilde B_{i j} (p) + \dtilde B_{i j } (p),
 \end{align*}
 and by symmetry and oddness,
 \begin{align*}
   \tilde   B_{i j} \perp \sqrt{J}, p_k \sqrt{J}, p_0^{\pm} \sqrt{J}, \quad \dtilde B_{i j} \perp  p_k \sqrt{J}.
 \end{align*}
 Hence, for \eqref{eq3.1.6} to hold, it suffices to impose the conditions
\begin{align}
    \label{eq54}
&\int h (|p|)  (|p|^2-3) \sqrt{J} \, dp = 0, \\
\label{eq55}
&\int h (|p|)   (|p|^2-3) p_0 \sqrt{J} \, dp = 0.
\end{align}

\subsection{Construction of the function $h$} 
Gathering the conditions  \eqref{eq54}, \eqref{eq55}, \eqref{eq51}, and \eqref{eq51.5}), plugging the ansatz \eqref{eq3.1.41}
 \begin{align*}
    h (r) = \mu (r) \sqrt{1+r^2} e^{\frac{1}{2} \sqrt{1+r^2}} (k_1 r^2 +  k_2 r^4 + k_3 r^6 + k_4 r^8), \quad \mu (r) = \frac{1}{\sqrt{2\pi}}  e^{-r^2/2},
 \end{align*}
 and using spherical coordinates, we obtain a set of equations on $k_j, j = 1, \ldots, 4$:
\begin{align}
\label{eq52.5}
   &k_j  \int_0^{\infty}  r^2 r^{2j} \mu (r)\, dr = \text{const}_1, \\
  \label{eq52.6}
     &k_j  \int_0^{\infty} r^4 r^{2j}  \mu (r)\, dr = \text{const}_2,\\
  \label{eq52.7}
   &k_j  \int_0^{\infty} \underbrace{ r^2 (r^2-3) (r^2+1)}_{ = r^6 - 2  r^4 - 3  r^2 } r^{2j} \mu (r) \, dr = 0, \\
  \label{eq52.8}
   &k_j  \int_0^{\infty} r^2 (r^2-3) r^{2j}  \sqrt{1+r^2}  \mu (r) \, dr = 0.
\end{align}
By evenness, we may assume that all the above integrals are over $\bR$.

Next, by $m_n$, we denote the even Gaussian moments
$$
    m_n = \int_{\bR} \mu (r) r^{2n} \, dr, \, \,  n = 0, 1, 2, \ldots.
$$
It is well known that
\begin{align}
    \label{eq52.15}
   m_0 = 1, \quad  m_{n} = (2n-1)!!=(2n-1) (2n-3) \ldots 1.
\end{align}
Then,  the system \eqref{eq52.5}--\eqref{eq52.8} can be rewritten as
\begin{align}
    \label{eq52.9}
& \sum_{j=1}^4 k_j m_{j+1} = \text{const}_1 \neq 0, \\
    \label{eq52.10}
&\sum_{j=1}^4  k_j m_{j+2} = \text{const}_2 \neq 0, \\
        \label{eq52.12}
& \sum_{j=1}^4 k_j (m_{j+3}-2 m_{j+2}-3m_{j+1}) = 0, \\
    \label{eq52.11}
&\sum_{j=1}^4  k_j  (i_{j+2}-3i_{j+1}) = 0,
\end{align}
where 
\begin{align*}
    i_j := \int_{\bR} r^{2j} \sqrt{1+r^2}  \mu (r) \, dr.
\end{align*} 
The next lemma shows that all the integrals $i_j, j \ge 2$, can be computed in terms of $i_0$ and $i_1$.

\begin{lemma}
    \label{lemma 52.1}
\begin{align}
    \label{eq52.1.1}
 &   i_{j}= (2j-1) i_{j-1}+(2j-3) i_{j-2}, \, j \ge 2, \\
    \label{eq52.1.2}
 &  i_1 > i_0.
\end{align}

\end{lemma}

\begin{proof}
\textit{Proof of \eqref{eq52.1.1}.} By direct calculation involving integration by parts, we have
\begin{align*}
 &   i_n = \int_{\bR} r^{2n} \sqrt{1+r^2}  \mu (r) \, dr
  =  \int_{\bR} \frac{r^{2n} (1+r^2) }{\sqrt{1+r^2}}  \mu (r) \, dr \\
&  = \int_{\bR} (\sqrt{1+r^2})'  (r^{2n-1} +r^{2n+1}) \mu (r) \, dr
  = -  \int_{\bR} \big((2n-1) r^{2n-2} + (2n+1) r^{2n}\big) \sqrt{1+r^2} \mu (r) \, dr \\ 
 &+ \int_{\bR}  (r^{2n} +r^{2n+2}) \sqrt{1+r^2} \mu (r) \, dr 
  = -(2n-1) i_{n-1} - (2n+1)  i_n + i_n + i_{n+1}.
    \end{align*}
    Hence, 
    \begin{align*}
        i_{n+1} = (2n+1) i_n + (2n-1) i_{n-1}.
    \end{align*}
    Replacing $n$ with $j-1$, we obtain \eqref{eq52.1.1}.

\textit{Proof of \eqref{eq52.1.2}.} Integrating by parts gives
\begin{align*}
    i_1 = - \int_{\bR} \sqrt{1+r^2}  \, r \mu' \, dr= \int_{\bR} (r \sqrt{1+r^2})' \mu \, dr =  i_0 + \int_{\bR} \frac{r^2}{\sqrt{1+r^2}}  \mu \, dr  > i_0.
\end{align*}
\end{proof}

Finally, we compute the coefficient matrix of the system \eqref{eq52.9}--\eqref{eq52.11} explicitly. First, by  \eqref{eq52.15} and \eqref{eq52.1.1}, 
\begin{align*}
  &  m_2 = 3, \quad m_3= 15, \quad m_4= 105,  \quad m_5= 945, \quad m_6= 10395, \quad m_7 =  135135, \\
  &i_2 = i_0 + 3 \, i_1, \quad i_3 =  5 \, i_0 + 18 \, i_1, \quad i_4 =  40 \, i_0 + 141 \, i_1, \\
  & i_5=  395 \, i_0 + 1395 \, i_1, \quad  i_6 = 4705 \, i_0 + 16614 \, i_1.
\end{align*}
Hence, by \eqref{eq52.9}--\eqref{eq52.11},  the aforementioned  coefficient matrix  is given by
\begin{align*}
    C = \begin{pmatrix}       
    3      &       15   &             105        &         945\\
    15     &       105  &             945        &      10395 \\
    66    &       690   &         8190           & 111510 \\
    2 \, i_0 + 9 \, i_1 & 25 \, i_0 + 87\, i_1 & 275 \, i_0 + 972\, i_1 & 3520\, i_0 + 12429\, i_1.
\end{pmatrix}
\end{align*}
Due to \eqref{eq52.1.2}, we conclude
\begin{align*}
\text{det} \, C = 14364000 \, i_0 - 15649200 \, i_1 < 0,
\end{align*}
as desired.

{}

\end{document}